\documentclass[a4paper]{article}
\usepackage[bottom=1in, left = 1in , right = 1in, top = 1in]{geometry}
\usepackage[utf8]{inputenc}
\usepackage{mathrsfs,mathtools}
\usepackage{amsmath,amssymb,amsthm,amsfonts,bbm}
\usepackage{amstext}
\usepackage{amsopn}
\usepackage{mathabx}
\usepackage[T1]{fontenc} 
\usepackage{manfnt}
\usepackage{array}
\usepackage{blindtext}
\usepackage{tikz}
\usepackage{easyReview}
\usepackage{float}

\usepackage{enumerate}
\usepackage{subcaption}

\newtheorem{theorem}{Theorem}
\newtheorem{proposition}{Proposition}
\newtheorem{assumption}{Assumption}
\newtheorem{corollary}{Corollary}
\newtheorem{lemma}[theorem]{Lemma}
\newtheorem{remark}{Remark}

\usepackage[maxbibnames=99]{biblatex}
\title{Sticky coupling as a control variate for sensitivity analysis}
\author{S. Darshan$^{1,2}$, A. Eberle$^3$ and G. Stoltz$^{1,2}$ \\
	\small $^1$: CERMICS, Ecole des Ponts, IP Paris, Marne-la-Vall\'ee, France \\
	\small $^2$: MATHERIALS team-project, Inria Paris, Paris, France\\
	\small $^3$: Institute for Applied Mathematics, University of Bonn, Bonn, Germany \\
}

\begin{document}

\maketitle
\begin{abstract}
	We present and analyze a control variate strategy based on couplings to reduce the variance of finite difference estimators of sensitivity coefficients, called transport coefficients in the physics literature. We study the bias and variance of a sticky-coupling and a synchronous-coupling based estimator as the finite difference parameter \(\eta\) goes to zero.  For diffusions with elliptic additive noise, we show that when the drift is contractive outside a compact the bias of a sticky-coupling based estimator is bounded as \(\eta \to 0\) and its variance  behaves like \(\eta^{-1}\), compared to the standard estimator whose bias and variance behave like \(\eta^{-1}\) and \(\eta^{-2}\), respectively. Under the stronger assumption that the drift is contractive everywhere, we additionally show that the bias and variance of the synchronous-coupling based estimator are both bounded as \(\eta \to 0\). Our hypotheses include overdamped Langevin dynamics with many physically relevant non-convex potentials. We illustrate our theoretical results with numerical examples, including overdamped Langevin dynamics with a highly non-convex Lennard-Jones potential to demonstrate both failure of synchronous coupling and the effectiveness of sticky coupling in the not globally contractive setting.
\end{abstract}

\section{Introduction}
Statistical physics provides a means of deducing the macroscopic properties of a system from a microscopic description of its dynamics. Its numerical realization molecular dynamics, i.e the simulation of the dynamics of molecular and atomistic systems, provides scientists a "numerical microscopic" to conduct computer experiments allowing them to test physical theories and to make precise quantitative measurements of simulated systems. It has flourished in the last 70 years---see \cite{BattimelliCiccottiGreco} for a historical perspective. An important problem in molecular dynamics is the computation of transport coefficients. These coefficients relate an external forcing on a system to the average response of some observable. The theory of statistical physics gives two ways of quantifying these coefficients either by integrating the equilibrium correlations via the Green--Kubo formula or in the limit of small perturbations to the equilibrium system. We focus on this second perspective, the so-called "non-equilibrium molecular dynamics" method \cite{Ciccotti, StatMech, Tuckerman}. 

At a microscopic level, the external forcing is modeled by adding a perturbation of size \(\eta \in \mathbb{R}\) to the reference dynamics. Of particular interest is the case when the perturbation is not given by the gradient of some potential function. One expects that the response of the system \(\mathbb{E}_\eta[R]\) for some physically relevant observable \(R\) to be roughly proportional to the size of the perturbation, namely
\[\mathbb{E}_\eta[R] - \mathbb{E}_0[R] \approx \alpha_R \eta,\]
 when \(|\eta|\) is small---the so-called linear response regime.
The proportionality coefficient \(\alpha_R\) is called the transport coefficient, see Section~\ref{sec:lin_resp} for the formal definition. It is seldom possible to analytically compute this coefficient necessitating the consideration of finite difference approximations. Furthermore, when the perturbation is not of gradient form one does not have an explicit expression for the (unnormalized) density of the invariant measure. Thus the difference of the steady state averages has to be replaced with the time averages of ergodic processes:
\begin{equation}\label{eq:intro_erg_avg}
	\alpha_R \approx \frac{1}{\eta} \left(\frac{1}{t}\int_0^t R\left(X_s^\eta\right)ds - \frac{1}{t}\int_0^t R\left(Y_s^0\right)ds\right),
\end{equation}
where \(\left(X_t^\eta\right)_{t \geq 0}\) is a stochastic process following the reference dynamics perturbed by an external forcing of magnitude \(\eta\) and \(\left(Y_t^0\right)_{t\geq 0}\) is a process following the unperturbed reference dynamics.
As we will see in Section~\ref{sec:lin_resp}, this approximation suffers from a large noise to signal ratio as dividing by small \(\eta\) greatly increases the variance but taking \(\eta\) small is necessary to remain in the linear response regime. Long computational times are therefore necessary to guarantee that such estimators converge. For a more in-depth discussion of the difficulties around computing transport coefficients we refer the interested reader to \cite{Gabriel_rev}.

The large variance and the long computational times necessary to compensate highlight the need for variance reduction strategies. A general discussion of variance reduction strategies for Monte Carlo methods may be found in standard references and textbooks such as \cite{Caflisch, Liu, KroeseRubinstein}. Among these strategies are control variate methods which involve subtracting off a mean-zero random variable from the summand of the Monte Carlo estimator with the hope that the difference has lower variance than the original summand.

To build a suitable control variate one may use the same simulation trajectory and subtract off \(\Phi\left(X_t^\eta\right)\) from time averages such as the ones appearing in \eqref{eq:intro_erg_avg} where \(\Phi\) is some well-chosen function with expectation zero. The zero variance principle \cite{AssarafCaffarel} suggests that the optimal choice of \(\Phi\) is the solution of a Poisson equation. However, solving this equation is intractable in practice. A strategy of using the solution of an approximate tractable Poisson equation was proposed and analyzed in \cite{Roussel}. Their strategy requires model specific tricks to construct a good approximate equation that is solvable in high dimensions. 

Alternatively, one may construct a trajectory of another stochastic process and use the observable evaluated at this process as a control variate. We propose to constructing such a process using an intelligently coupled version of the reference dynamics. Couplings and coupling methods have a long history of applications in probability theory, see for example the books \cite{Lindvall, Thorisson}. They have been particularly useful in proving non-asymptotic rates convergence of solutions to a stochastic differential equation to its invariant probability measure, see for instance \cite{LindvallRogers, HairerMattinglyScheutzow,Eberle}. In recent years, they have as well proved useful in the development Monte Carlo methods including generating unbiased samples without rejection \cite{GlynnRhee, Jacob, Chada}, testing for convergence \cite{BiswasJacobVanetti, LiWang, DobsonLiZhai}, and numerically exploring the landscape of high-dimensional potential function \cite{LiTaoWang}. 

Two works most directly related to the current work also propose coupling based control variates \cite{PintoNeal, GoodmanLin}. In the first work \cite{PintoNeal}, a Markov chain intended to sample from a Bayesian posterior distribution is coupled to a second chain whose invariant measure is a Gaussian approximation of the posterior distribution. In our framework, this corresponds to coupling two "equilibrium" processes with different potentials since the invariant measures are known up to normalizing constants. This allows the authors to make use of the unnormalized density of the target distribution, which will not work in our setting. In the second work \cite{GoodmanLin}, the authors suggest a coupling based control variate strategy for Markov jump processes under the assumption that a good coupling between the perturbed and reference process exists. They then illustrate their method on the simulation of two lattice models of heat transport. The target process is the non-equilibrium process driven at the boundaries and it is coupled to a jump process sampling from the local thermal equilibrium. The two processes are coupled by forcing them to make the same jump whenever possible---their coupling is effectively a synchronous coupling of the two processes. Numerical results suggest a dramatic reduction in variance. However no rigorous quantification of the variance reduction is proven nor are any hypotheses given under which their coupling is guaranteed to work. In their conclusion, the authors suggest that such coupling based control variates could be useful for the computation of sensitivity coefficients.
The current article validates their intuition and applies it to solutions of stochastic differential equations with additive noise and their discretizations.

\paragraph{Outline.} The article is organized as follows. In Section~\ref{sec:lin_resp}, we recall the functional framework in which we will work and rigorously define the linear response and transport coefficients. In Section~\ref{sec:coupling}, we present the general idea of coupling based control variates and a synchronous-coupling based estimator. We prove a central limit theorem for the synchronous-coupling based estimator and bounds on its bias and variance under the global contractivity hypothesis. In Section~\ref{sec:sticky_coupling}, for technical reasons we work in discrete time. We start the section by recalling facts about the discrete-time dynamics and linear response in discrete time. We then present the discrete-time sticky coupling and a central limit theorem for the sticky-coupling based estimator along with some quantitative bounds on its bias and variance. We then present and prove certain properties of the discrete-time sticky-coupled dynamics. Finally, we use these properties to prove the announced results for the sticky-coupling based estimator. In Section~\ref{sec:numerics}, we provide some numerical illustrations. We defer various ancillary results to the appendices. In Appendix~\ref{sec:Kopec_ext}, we prove Proposition~\ref{prop:exist_poisson_eq} on regularity of solutions to the Poisson equation. In Appendix~\ref{sec:sync_contractive}, we prove the ergodicity of the synchronously-coupled dynamics. In Appendix~\ref{sec:discrete_lin_resp_proofs}, we prove two technical lemmas on discrete-time solutions to the Poisson equation and the linear response in discrete time. In Appendix~\ref{sec:equivalence_of_disc_MR_coupling}, we show that two definitions of the meeting probability for sticky coupling are equivalent. 
\section{Linear Response and Transport Coefficients}\label{sec:lin_resp}

In this section, we present the stochastic dynamics we will work with in this article and the assumptions that will hold throughout. We as well introduce the functional framework in which we will be working. We then present a rigorous definition of linear response and transport coefficient in this framework. Next we present the standard NEMD estimator and finish the section with a proposition qualifying the bias and asymptotic variance of this estimator.

\subsection{Definition of the Dynamics and Linear Response}
We consider the following family of SDEs with values in \(\mathbb{R}^d\) and additive noise:
\begin{equation}\label{eq:sde_model}
	dX_t^\eta = \left(b\left(X_t^\eta\right) + \eta F\left(X_t^\eta\right)\right)dt + \sqrt{\frac{2}{\beta}}dW_t,
\end{equation}
where \(b,F : \mathbb{R}^d \to \mathbb{R}^d\) are smooth functions and \(\left(W_t\right)_{t\geq 0}\) is a standard \(d\)-dimensional Brownian motion. Suppose that this dynamics admits a unique invariant probability measure, denoted by \(\nu_\eta\). Then for a given observable \(R \in L^1\left(\nu_\eta\right)\) for all \(\eta \in \mathbb{R}\), we define the transport coefficient by 
\begin{equation}\label{eq:transport_coef}
	\alpha_R = \lim_{\eta \to 0} \frac{1}{\eta}\left(\int_{\mathbb{R}^d} R \, d\nu_\eta - \int_{\mathbb{R}^d}R \, d\nu_0 \right),
\end{equation}
provided this limit is well defined.

\begin{assumption}\label{ass:drift}
The function \(F\) is Lipschitz with Lipschitz coefficient \(L_F\) and uniformly bounded:
\begin{equation}\label{eq:F_bounded}
	\sup_{x \in \mathbb{R}^d} \left|F(x)\right| < +\infty.	
\end{equation}
The drift \(b\) is Lipschitz with Lipschitz coefficient \(L_b\) and contractive at infinity, i.e. there exist constants \(m > 0\) and \(M \geq 0\) such that 
\begin{equation}\label{eq:contractive_at_inf}
	\forall \left|x - y\right| \geq M, \qquad \left\langle x-y, b(x) - b(y)\right\rangle \leq -m \left|x - y\right|^2.
\end{equation} 
\end{assumption}

The archetypal dynamics we will be considering is overdamped Langevin dynamics with a potential energy function~\(U\) that coincides with a \(m\)-strongly convex function outside a ball centered around the origin and perturbed by a non-gradient bounded forcing. This corresponds to \eqref{eq:sde_model} with \(b = -\nabla U\).
\begin{remark}
It should be possible to weaken the Lipschitz assumption on \(b\) and \(F\) to only local Lipschitz by using~\eqref{eq:contractive_at_inf} to make Lyapunov type arguments to show that it is exponentially rare for the processes to be far from to origin. Then assuming that the derivatives of \(b\) and \(F\) grow at most polynomial (which we will assume in the sequel) should ensure that all the arguments that follow still work. However this would greatly encumber the exposition so we maintain the Lipschitz assumption for clarity's sake.
\end{remark}

For a measurable function \(V:\mathbb{R}^d \to \left[1, \infty \right)\), we define the \(V\)-norm \(\|\cdot\|_V\) for functions \(f:\mathbb{R}^d \to \mathbb{R}\) and finite measures \(\mu\) on \(\mathbb{R}^d\) by
\begin{equation}\label{eq:v_norm}
	\begin{aligned}
		\left\|f\right\|_V &= \sup_{x \in \mathbb{R}^d} \frac{\left|f(x)\right|}{V(x)},\\
		\left\|\mu\right\|_V &= \frac{1}{2}\sup_{\left\|f \right\|_V \leq 1} \int_{\mathbb{R}^d}f d\mu.
	\end{aligned}
\end{equation}
The factor \(\frac{1}{2}\) in the definition of \(\|\mu\|_V\) is motivated by the fact that when \(V \equiv 1\), these norms correspond to the supremum norm for functions and the total variation norm for measures, which we denote respectively by \(\left\|\cdot\right\|_\infty\) and \(\left\|\cdot\right\|_{\mathrm{TV}}\). We also write \(d_{\mathrm{TV}}\) for the distance induced by \(\left\|\cdot\right\|_{\mathrm{TV}}\). 
For \(\eta \in \mathbb{R}\), we define the following projection operator on \(L^1\left(\nu_\eta\right)\):
\[\Pi_\eta \varphi = \varphi - \int_{\mathbb{R}^d} \varphi \, d\nu_\eta. \]
We denote the space of measurable functions with finite \(V\)-norm by \[B_{V}^{\infty} := \left\{\varphi \text{ measurable } \left| \left\|\varphi\right\|_{V} < \infty \right. \right\}.\] 
We will use in particular the \(\mathcal{K}_n\)--norms \(\|\cdot\|_{\mathcal{K}_n}\) defined with respect to the functions \(\mathcal{K}_n := 1+|x|^n\) for~\(n \in \mathbb{N}\), and write \(B_n^{\infty}\) for the corresponding spaces of functions with finite \(\mathcal{K}_n\)-norms. For a Banach space~\(E\), we denote by \(\mathcal{B}\left(E\right)\) the space of bounded linear operators on \(E\).

We denote the transition semi-group of (\ref{eq:sde_model}) by \(\left(P_t^\eta\right)_{t\geq 0}\) and its generator by \(\mathcal{L}_\eta = \mathcal{L}_0 + \eta\widetilde{\mathcal{L}}\), with
\begin{equation}
	\mathcal{L}_0 = b \cdot \nabla + \frac{1}{\beta}\Delta, \qquad \widetilde{\mathcal{L}} = F \cdot \nabla.
\end{equation}

The contractivity at infinity \eqref{eq:contractive_at_inf} implies that, for any \(n\geq 2\), the function \(\mathcal{K}_n\left(x\right) = 1 + \left|x\right|^{n}\) is a Lyapunov function for the dynamics~(\ref{eq:sde_model}). Standard results show that \eqref{eq:sde_model} admits a unique global-in-time strong solution for any \(\eta \in \mathbb{R}\) (see for example \cite{IkedaWatanabe}) and has a unique invariant probability measure \(\nu_\eta\) with a smooth positive density with respect to the Lebesgue measure \cite{Rey-Bellet}. Moreover, the following estimates hold:
\begin{align}
	\forall \eta_\star \in \left(0, \infty\right), \quad \forall n \geq 1, &\qquad \sup_{\left|\eta\right| \leq \eta_\star} \nu_{\eta}\left(\mathcal{K}_n\right) < \infty, \label{eq:moment_bounds} \\
	\forall \eta_\star \in \left(0, \infty\right), \quad \forall n \geq 1, &\qquad \sup_{\left|\eta\right|\leq \eta_\star}\sup_{t\geq 0} \sup_{x \in \mathbb{R}^d} \left|\frac{\left(P_t^\eta \mathcal{K}_n \right)(x)}{\mathcal{K}_n(x)}\right| \leq S_{n, \eta_\star} < \infty, \label{eq:semigroup_estimates}\\ 
	\forall \eta_\star \in \left(0, \infty\right), \quad \forall n \geq 1, &\qquad \sup_{\left|\eta\right| \leq \eta_\star} \left\|\mathcal{L}^{-1}_\eta\right\|_{\mathcal{B}\left(\Pi_\eta B_n^\infty\right)} < \infty. \label{eq:inverse_bounds}
\end{align}
The above estimates are obtained from a minorization condition and Lyapunov conditions of the form: for any \(n \geq 2\) and \(\eta_{\star} > 0\), there exist \(a_{n, \eta_{\star}} > 0\) and \(b_{n, \eta_{\star}} \in \mathbb{R}\) such that \(\mathcal{L}_\eta \mathcal{K}_n \leq - a_{n, \eta_{\star}} \mathcal{K}_n + b_{n, \eta_{\star}}\) uniformly in \(\eta \in \left[-\eta_\star, \eta_\star\right]\); see \cite{Spacek}. These Lyapunov conditions and minorization condition also imply that the dynamics is geometrically ergodic with respect to the \(\mathcal{K}_n\)-norm uniformly in \(\eta\) for any \(n \in \mathbb{N}\); that is to say for any \(\eta_\star > 0\) and \(n \in \mathbb{N}\), there exist constants \(C_{n, \eta_{\star}}, \lambda_{n, \eta_{\star}} > 0\) such that 
\begin{equation}\label{eq:geo_ergodicity}
	\forall \eta \in \left[-\eta_\star, \eta_\star\right], \qquad \forall x \in \mathbb{R}^d, \qquad \left\|\delta_x P_t^\eta - \nu_\eta\right\|_{\mathcal{K}_n} \leq C_{n, \eta_{\star}} \mathcal{K}_n(x)\mathrm{e}^{-\lambda_{n, \eta_{\star}} t}.
\end{equation}
The proof of this fact follows along the same lines as the proof of \cite[Theorem 4.4]{Mattingly}, for example. 

We introduce the space \(\mathscr{S}\) of smooth function which grow at most polynomially and whose derivatives also grow at most polynomially. Denoting by \(\partial^{k} = \partial_{x_1}^{k_1}\cdots \partial_{x_d}^{k_d}\) for \(k = \left(k_1, \dots, k_d\right) \in \mathbb{N}^d\),
\begin{equation}
	\mathscr{S} = \left\{\varphi \in C^{\infty}\left(\mathbb{R}^d\right) \, \left|\, \forall k \in \mathbb{N}^d, \exists n \in \mathbb{N}, \partial^k\varphi \in B_n^\infty \right. \right\}.
\end{equation}
We also consider the subspace \(\mathscr{S}_\eta = \Pi_\eta \mathscr{S}\) of functions in \(\mathscr{S}\) with average 0 with respect to \(\nu_\eta\). We make the following assumption on the drift and \(\nu_0\) (below we write \(\nu_0\) for both the measure and its density with respect to the Lebesgue measure).
\begin{assumption}\label{ass:ref_measure}
	The functions \(b, F\) belong to \(\mathscr{S}\) in the sense that each of their components belongs to~\(\mathscr{S}\). The density of the invariant measure of the dynamics when \(\eta = 0\) is such that \(\log\left(\nu_0\right) \in \mathscr{S}\). 
\end{assumption}
A consequence of this assumption is that \(\widetilde{\mathcal{L}}\) stabilizes \(\mathscr{S}\), i.e. \(\widetilde{\mathcal{L}}\mathscr{S} \subset \mathscr{S}\), and that \(\widetilde{\mathcal{L}}^*\mathbf{1}\in \mathscr{S}_0\), where here and in the remainder of this work, adjoints are taken with respect to \(L^2\left(\nu_0\right)\). This can be be checked using the explicit formulae for \(\widetilde{\mathcal{L}}\) and \(\widetilde{\mathcal{L}}^* = -\widetilde{\mathcal{L}} - \mathrm{div}\left(F\right) - F \cdot \nabla \left(\log \nu_0\right)\).
Furthermore, as a consequence of Assumption~\ref{ass:drift}, we have the following proposition whose proof is postponed to Appendix~\ref{sec:Kopec_ext}.
\begin{proposition}\label{prop:exist_poisson_eq}
	Suppose that Assumptions~\ref{ass:drift} and \ref{ass:ref_measure} hold. Then for any \(\varphi \in \mathscr{S}\), the Poisson equation \(-\mathcal{L}_\eta\widetilde{\varphi}_\eta = \Pi_\eta \varphi\) has a unique solution in \(\mathscr{S}_\eta\).
\end{proposition}

These assumptions ensure that our system has a well defined linear response in the sense that the following lemma adapted from \cite{Spacek} holds.
\begin{lemma}\label{lm:lin_resp}
	Suppose that Assumptions~\ref{ass:drift} and \ref{ass:ref_measure} hold. The linear response of the invariant measure is given by
	\begin{equation}\label{eq:lin_resp_factor}
		\mathfrak{f} = -(\mathcal{L}_0^*)^{-1}\widetilde{\mathcal{L}}^* \mathbf{1},
	\end{equation}
	with \(\mathfrak{f}\in \Pi_0 B_n^\infty\) for \(n \in \mathbb{N}\) large enough.
	More precisely, fix \(\eta_\star > 0\). Then there exists \(p \in \mathbb{N}\) such that, for any \(n \in  \mathbb{N}\) there exists a constant \(M_{n, \eta_\star} > 0\) such that, for any \(\varphi \in \mathscr{S}\) satisfying \(\partial^\alpha \varphi \in B_n^\infty\) for any multi-index \(|\alpha| = \alpha_1 + \dots + \alpha_d \leq p\)
	\begin{equation}\label{eq:lin_resp}
		\int_{\mathbb{R}^d} \varphi \, d\nu_\eta = \int_{\mathbb{R}^d} \varphi\left(1+ \eta \mathfrak{f} \right) d\nu_0 + \eta^2 \mathscr{R}_{\eta,\varphi},
	\end{equation}
	with \(|\mathscr{R}_{\eta,\varphi}| \leq M_{n, \eta_\star}\sum_{|\alpha| \leq p}\left\|\partial^\alpha \varphi\right\|_{\mathcal{K}_n}\) for all \(\eta \in [-\eta_\star,\eta_\star]\).
\end{lemma}

\begin{proof}
	For \(n \in \mathbb{N}\) large enough, the function~$\mathfrak{f}$ is a well defined element of~\(\Pi_0B_n^\infty\) in view of Assumption~\ref{ass:ref_measure} and \eqref{eq:inverse_bounds}.
	Using the expression for \(\mathfrak{f}\), we have, for \(\phi \in \mathscr{S}\),
	\[\begin{aligned}
		\int_{\mathbb{R}^d} \mathcal{L}_\eta \phi \, (1+\eta\mathfrak{f}) \, d\nu_0 &= \int_{\mathbb{R}^d} \left(\mathcal{L}_0 + \eta \widetilde{\mathcal{L}}\right)\phi \, d\nu_0 + \eta \int_{\mathbb{R}^d} \left[\left(\mathcal{L}_0 + \eta\widetilde{\mathcal{L}}\right)\phi \right] \mathfrak{f}\, d\nu_0\\
		&= \int_{\mathbb{R}^d} \mathcal{L}_0 \phi \, d\nu_0 + \eta \int_{\mathbb{R}^d} \widetilde{\mathcal{L}} \phi \, d\nu_0 + \eta \int_{\mathbb{R}^d} \mathcal{L}_0 \phi \left(-\mathcal{L}_0^*\right)^{-1}\widetilde{\mathcal{L}}^*\mathbf{1} \, d\nu_0 + \eta^2\int_{\mathbb{R}^d} \left(\widetilde{\mathcal{L}} \phi\right) \mathfrak{f}\, d\nu_0\\
		&= \int_{\mathbb{R}^d} \mathcal{L}_0 \phi \, d\nu_0 + \eta \int_{\mathbb{R}^d} \widetilde{\mathcal{L}} \phi \, d\nu_0 - \eta \int_{\mathbb{R}^d} \widetilde{\mathcal{L}} \phi \, d\nu_0 + \eta^2\int_{\mathbb{R}^d} \left(\widetilde{\mathcal{L}} \phi\right) \mathfrak{f}\, d\nu_0\\
		&= \eta^2 \int_{\mathbb{R}^d} \left(\widetilde{\mathcal{L}} \phi\right) \mathfrak{f}\, d\nu_0,
	\end{aligned} \]
	where we use for the last equality the fact that \(\mathcal{L}_0\) is the generator of dynamics with invariant probability measure \(\nu_0\).
	We next replace~\(\phi\) by \(Q_\eta \varphi\), where \(Q_\eta\) is some operator on~\(\mathscr{S}\) and \(\varphi \in \mathscr{S}\). This leads to
	\begin{equation}\label{eq:integral_equal_zero}
		\int_{\mathbb{R}^d} \mathcal{L}_\eta Q_\eta \varphi \, d\nu_\eta = 0 = \int_{\mathcal{R}^d} \mathcal{L}_\eta Q_\eta \varphi \, (1+\eta\mathfrak{f}) \, d\nu_0 - \eta^2 \int_{\mathbb{R}^d} \left(\widetilde{\mathcal{L}} Q_\eta \varphi\right)\mathfrak{f} \, d\nu_0.
	\end{equation}
	The choice \(Q_\eta = \Pi_0 \mathcal{L}_0^{-1} \Pi_0 - \eta \Pi_0 \mathcal{L}_0^{-1} \Pi_0 \widetilde{\mathcal{L}} \Pi_0 \mathcal{L}_0^{-1} \Pi_0\) ensures that 
	\[\begin{aligned}
		\mathcal{L}_\eta Q_\eta &= \left(\mathcal{L}_0 + \eta \widetilde{\mathcal{L}}\right) \left(\Pi_0 \mathcal{L}_0^{-1}\Pi_0 - \eta \Pi_0\mathcal{L}_0^{-1}\Pi_0\widetilde{\mathcal{L}}\Pi_0 \mathcal{L}_0^{-1}\Pi_0\right)\\
		&= \Pi_0 + \eta \left(\widetilde{\mathcal{L}}\Pi_0\mathcal{L}_0^{-1}\Pi_0 - \Pi_0\widetilde{\mathcal{L}}\Pi_0\mathcal{L}_0^{-1}\Pi_0\right) - \eta^2\left(\widetilde{\mathcal{L}}\Pi_0\mathcal{L}_0^{-1}\Pi_0\right)^2,
	\end{aligned}\] 
	since \(\mathcal{L}_0\Pi_0\mathcal{L}_0^{-1}\Pi_0 = \mathcal{L}_0\mathcal{L}_0^{-1}\Pi_0 = \Pi_0\) as \(\mathcal{L}_0^{-1}\) is an operator on \(\mathscr{S}_0\). Observe that, for any \(\psi \in \mathscr{S}\), the function \(\left(\widetilde{\mathcal{L}}\Pi_0\mathcal{L}_0^{-1}\Pi_0 - \Pi_0\widetilde{\mathcal{L}}\Pi_0\mathcal{L}_0^{-1}\Pi_0\right)\psi\) is constant. Plugging this expression for \(\mathcal{L}_\eta\mathcal{Q}_\eta\) to the left hand side of~\eqref{eq:integral_equal_zero} gives
	\[\int_{\mathbb{R}^d} \mathcal{L}_\eta \mathcal{Q}_\eta \varphi \, d\nu_\eta = \int_{\mathbb{R}^d}\varphi \, d\nu_\eta - \int_{\mathbb{R}^d} \varphi \, d\nu_0 + \eta \left(\widetilde{\mathcal{L}}\Pi_0\mathcal{L}_0^{-1}\Pi_0 - \Pi_0\widetilde{\mathcal{L}}\Pi_0\mathcal{L}_0^{-1}\Pi_0\right)\varphi - \eta^2\int_{\mathbb{R}^d}\left(\widetilde{\mathcal{L}}\Pi_0\mathcal{L}_0^{-1}\Pi_0\right)^2\varphi \, d\nu_\eta. \]
	Similarly, plugging in to the right hand side of \eqref{eq:integral_equal_zero} gives
	\[\begin{aligned}
		&\int_{\mathcal{R}^d} \mathcal{L}_\eta Q_\eta \varphi \, (1+\eta\mathfrak{f}) \, d\nu_0 - \eta^2 \int_{\mathbb{R}^d} \left(\widetilde{\mathcal{L}} Q_\eta \varphi\right)\mathfrak{f} \, d\nu_0\\ 
		&\quad = \int_{\mathbb{R}^d} \varphi\left(1 + \eta \mathfrak{f}\right) \, d\nu_0 - \int_{\mathbb{R}^d} \varphi \, d\nu_0 + \eta \left(\widetilde{\mathcal{L}}\Pi_0\mathcal{L}_0^{-1}\Pi_0 - \Pi_0\widetilde{\mathcal{L}}\Pi_0\mathcal{L}_0^{-1}\Pi_0\right)\varphi\\
		&\qquad-\eta^2 \int_{\mathbb{R}^d}\left[\left(\widetilde{\mathcal{L}}\Pi_0\mathcal{L}_0^{-1}\Pi_0\right)^2\varphi \left(1 + \eta \mathfrak{f}\right) + \left(\widetilde{\mathcal{L}}\mathcal{Q}_\eta\varphi\right)\mathfrak{f}\right]\, d\nu_0,
	\end{aligned} \]
	where we use the fact that \(\mathfrak{f}\) has average zero with respect to \(\nu_0\).
	Consequently, rearranging the equality~\eqref{eq:integral_equal_zero} gives
	\[
	\int_{\mathbb{R}^d} \varphi \, d\nu_\eta = \int_{\mathbb{R}^d} \varphi \, (1+\eta\mathfrak{f}) \, d\nu_0 + \eta^2 \mathscr{R}_{\eta,\varphi}, 
	\]
	with the remainder term
	\[
	\mathscr{R}_{\eta,\varphi} = \int_{\mathbb{R}^d} \left(\widetilde{\mathcal{L}} \Pi_0 \mathcal{L}_0^{-1} \Pi_0\right)^2 \varphi \, d\nu_\eta - \int_{\mathbb{R}^d} \left( \left[\left(\widetilde{\mathcal{L}} \Pi_0 \mathcal{L}_0^{-1} \Pi_0\right)^2\varphi\right] (1+\eta\mathfrak{f})+\left(\widetilde{\mathcal{L}} Q_\eta \varphi\right)\mathfrak{f} \right) d\nu_0.
	\]
	The result then follows from the fact that \(\mathscr{R}_{\eta,\varphi}\) involves only derivatives up to some finite order of \(\varphi\), so that, in view of \eqref{eq:moment_bounds} (implied by Assumption~\ref{ass:drift}) and Proposition~\ref{prop:exist_poisson_eq}, this term is uniformly bounded for~\(\eta \in \left[-\eta_\star,\eta_\star\right]\).
\end{proof}
\noindent Lemma~\ref{lm:lin_resp} implies that \eqref{eq:transport_coef} is well-defined when \(R\in \mathscr{S}\) and Assumptions~\ref{ass:drift} and \ref{ass:ref_measure} hold. Furthermore the lemma gives the following expression for \(\alpha_R\):
\[\alpha_R = \int_{\mathbb{R}^d} R \mathfrak{f} \, d\nu_0.\]

\subsection{Estimator of Linear Response}
For a response function \(R \in \mathscr{S}_0\), a standard estimator of \(\alpha_R\) in \eqref{eq:transport_coef} from non-equilibrium molecular dynamics is
\begin{equation}
	\label{eq:naive_estimator}
	\widehat{\Phi}_{\eta,t} = \frac{1}{\eta t}\int_0^t R(X_s^\eta) \, ds,
\end{equation}
which converges almost surely to
\begin{equation}
	\label{eq:alpha_R_eta}
	\alpha_{R,\eta} = \frac1\eta \int_{\mathbb{R}^d} R \, d\nu_\eta = \alpha_R + \mathrm{O}(\eta),
\end{equation}
where the last equality comes from Lemma~\ref{lm:lin_resp}. However, the estimator~\(\widehat{\Phi}_{\eta,t}\) suffers both from a large asymptotic variance, of order~\(\sigma_{\mathrm{ref}, R}^2/\eta^{2}\), with 
\begin{equation}\label{eq:ref_variance}
	\sigma_{\mathrm{ref}, R}^2 = 2 \int_{\mathbb{R}^d} \left(-\mathcal{L}_0^{-1}R\right)R d\nu_0,
\end{equation}
the asymptotic variance for time averages \(\frac{1}{t}\int_0^t R\left(X_s^0\right)ds\) of~\(R\) computed with the reference dynamics, i.e the limit of the variance \(\frac{1}{\sqrt{t}}\int_0^t R\left(X_s^0\right)ds\) as \(t\) goes to infinity. It also has a large finite time sampling bias, of order~\(1/(\eta t)\). This is made precise in the following result adapted from \cite{Spacek}, and proved in Section 2.3

\begin{proposition}
	\label{prop:std_estimator}
	Fix \(\eta_\star > 0\) and \(R\in\mathscr{S}_0\). Suppose that Assumptions~\ref{ass:drift} and \ref{ass:ref_measure} hold, and that \(X_0^\eta \sim \nu_{\rm init}\) for some initial probability measure~\(\nu_{\rm init}\) such that \(\nu_{\rm init}(\mathcal{K}_n) < +\infty\) for any \(n \geq 1\). Then the estimator~\(\widehat{\Phi}_{\eta,t}\) converges almost surely to \(\alpha_{R,\eta}\) as \(t \to +\infty\), and the following central limit theorem holds:
	\[
	\sqrt{t} \left(\widehat{\Phi}_{\eta,t} - \alpha_{R,\eta}\right) \xrightarrow[t \to +\infty]{\mathrm{law}} \mathcal{N}\left(0,\sigma_{\mathrm{std}, R,\eta}^2\right).
	\]
	Moreover, there exists \(C,K \in \mathbb{R}_+\) (which depend on~\(\eta_\star\) and~\(R\)) such that
	\begin{equation}
		\label{eq:std_variance_bounds}
		\forall \eta \in [-\eta_\star, \eta_\star], \qquad \left|\sigma_{\mathrm{std}, R,\eta}^2 - \frac{\sigma_{\mathrm{ref}, R}^2}{\eta^2}\right| \leq \frac{C}{\eta},
	\end{equation}
	and
	\begin{equation}
		\label{eq:std_bias_bounds}
		\forall \eta \in [-\eta_\star, \eta_\star], \quad \forall t > 0, \qquad \left| \mathbb{E}\left(\widehat{\Phi}_{\eta,t}\right) - \alpha_{R,\eta} \right| \leq \frac{K}{\eta t}.
	\end{equation} 
\end{proposition}

This result shows that simulation times of order~\(t \sim \eta^{-2}\) should be considered in order for the variance of the standard estimator~\eqref{eq:naive_estimator} to be of order~1, and for its bias to be of order~\(\eta\), \emph{i.e.} of the same order of magnitude as the bias \(\alpha_R - \alpha_{R,\eta}\) arising from choosing \(\eta > 0\). The estimators introduced in Sections~\ref{sec:coupling}~and~\ref{sec:sticky_coupling} use couplings of the perturbed dynamics to the reference dynamics to achieve better scaling of the bias and variance with \(\eta\) compared to \eqref{eq:std_bias_bounds} and \eqref{eq:std_variance_bounds}.

\subsection{Proof of Proposition~\ref{prop:std_estimator}}

We adapt the proof of Proposition~\ref{prop:std_estimator} from \cite[Proposition 1]{Spacek}. We start with the following lemma adapted from  \cite{Spacek}. We give a simplified statement and proof as we do not need the same level of generality as in that work.

\begin{lemma}\label{lm:poisson_sol_approx}
	Suppose that Assumptions \ref{ass:drift}~and~\ref{ass:ref_measure} hold true. Fix \(\eta_\star > 0\) and \(\varphi \in \mathscr{S}_0\). Consider for any~\(\eta \in \mathbb{R}\) the unique solution~\(\phi_\eta \in \mathscr{S}_\eta\) of the Poisson equation \(-\mathcal{L}_\eta \phi_\eta = \Pi_\eta \varphi\). There exist \(n \geq 1\), \(C \geq 0 \), and \(\widecheck{\phi}_{\eta} \in \mathscr{S}\) such that
	\begin{equation}\label{eq:poisson_sol_diff}
		\phi_\eta - \phi_0 = \eta \widecheck{\phi}_\eta,
	\end{equation}
	with
	\begin{equation}\label{eq:bound_diff_term}
		\forall \eta \in \left[-\eta_\star, \eta_\star\right], \qquad \left\|\widecheck{\phi}_\eta \right\|_{\mathcal{K}_n} \leq C. 
	\end{equation}
\end{lemma}

\begin{proof}
	The following computations show that \(\widecheck{\phi}_\eta = \left(-\mathcal{L}_\eta\right)^{-1}\Pi_\eta \widetilde{\mathcal{L}}\left(-\mathcal{L}_0\right)^{-1} \varphi + \frac{1}{\eta}\int_{\mathbb{R}^d} \left(-\mathcal{L}_0\right)^{-1}\varphi \, d\nu_\eta\):
	\begin{equation*}
		\begin{aligned}
			\phi_\eta - \phi_0 &= \left(-\mathcal{L}_\eta\right)^{-1}\Pi_\eta \varphi - \left(-\mathcal{L}_0\right)^{-1} \varphi\\
			&= \left(-\mathcal{L}_\eta\right)^{-1}\Pi_\eta \varphi + \left(-\mathcal{L}_\eta \right)^{-1}\Pi_\eta \left(\mathcal{L}_0 + \eta \widetilde{\mathcal{L}}\right)\left(-\mathcal{L}_0\right)^{-1} \varphi + \left(\mathrm{Id} - \Pi_\eta\right) \left(-\mathcal{L}_0\right)^{-1}\varphi \\
			&=\left(-\mathcal{L}_\eta\right)^{-1}\Pi_\eta \left(\varphi + \left(\mathcal{L}_0 + \eta\widetilde{\mathcal{L}}\right)\left(-\mathcal{L}_0\right)^{-1} \varphi\right) + \int_{\mathbb{R}^d} \left(-\mathcal{L}_0\right)^{-1}\varphi \, d\nu_\eta\\
			&= \eta \left[\left(-\mathcal{L}_\eta\right)^{-1}\Pi_\eta \widetilde{\mathcal{L}}\left(-\mathcal{L}_0\right)^{-1} \varphi + \frac{1}{\eta}\int_{\mathbb{R}^d} \left(-\mathcal{L}_0\right)^{-1}\varphi \, d\nu_\eta \right].
		\end{aligned}
	\end{equation*}
	Since \(\varphi \in \mathscr{S}_0 \subset \mathscr{S}\) and the operators \(\widetilde{\mathcal{L}}\), \(\mathcal{L}_0^{-1}\Pi_0\), and \(\mathcal{L}_\eta^{-1}\Pi_\eta\) stabilize \(\mathscr{S}\) (using Proposition~\ref{prop:exist_poisson_eq} for the result for the latter two operators), it holds that \(\widecheck{\phi}_\eta \in \mathscr{S}\). Using Lemma~\ref{lm:lin_resp}, we bound the constant term on the right hand side as
	\[\left|\frac{1}{\eta}\int_{\mathbb{R}^d} \left(-\mathcal{L}_0\right)^{-1}\varphi \, d\nu_\eta \right| \leq \left|\frac{1}{\eta}\int_{\mathbb{R}^d} \left(-\mathcal{L}_0\right)^{-1}\varphi \, d\nu_0 \right|  + \left|\int_{\mathbb{R}^d} \left[\left(-\mathcal{L}_0\right)^{-1}\varphi\right] \mathfrak{f} \, d\nu_0 \right| + \eta_{\star}\left|\mathscr{R}_{\eta, \left(-\mathcal{L}_0\right)^{-1}\varphi}\right|, \]
	where \(\mathscr{R}_{\eta, \cdot}\) is the remainder term in \eqref{eq:lin_resp}. The first term vanishes since \(\left(-\mathcal{L}_0\right)^{-1}\varphi \in \mathscr{S}_0\). Furthermore, for \(n \in \mathbb{N}\) large enough such that \(\partial^\alpha \mathcal{L}_0^{-1}\varphi \in B_n^\infty\) for all \(|\alpha| \leq p\), the remainder term is bound 
	\[\left|\mathscr{R}_{\eta, \left(-\mathcal{L}_0\right)^{-1}\varphi}\right| \leq M_{\eta, \eta_{\star}}\sum_{|\alpha | \leq p} \left\|\partial^\alpha \varphi\right\|_{\mathcal{K}_n},\] 
	uniformly in \(\eta \in \left[-\eta_\star, \eta_\star\right]\). 
	
	Furthermore, up to taking \(n \in \mathbb{N}\) larger, we have that \(\widetilde{\mathcal{L}}\left(-\mathcal{L}_0\right)^{-1}\varphi \in B_n^{\infty}\). Thus the bound \eqref{eq:inverse_bounds} implies
	\[\left\|\left(-\mathcal{L}_\eta\right)^{-1}\Pi_\eta \widetilde{\mathcal{L}}\left(-\mathcal{L}_0\right)^{-1} \varphi\right\|_{\mathcal{K}_n} \leq  \left\|\mathcal{L}_\eta^{-1}\right\|_{\mathcal{B}\left(\Pi_\eta B_n^\infty\right)} \left\|\Pi_\eta \left( \widetilde{\mathcal{L}}\left(-\mathcal{L}_0\right)^{-1} \phi \right)\right\|_{\mathcal{K}_n}.\]
	We can bound the right had side using the fact that for any~\(u \in {B}_n^\infty\),
	\begin{equation}\label{eq:bdd_norm_projection}
		\left\|\Pi_\eta u\right\|_{\mathcal{K}_n} \leq \left\|u\right\|_{\mathcal{K}_n} + \int_{\mathbb{R}^d} \left|u\right|\, d\nu_\eta \leq \left\|u\right\|_{\mathcal{K}_n} + \left\|u\right\|_{\mathcal{K}_n}\nu_\eta\left(\mathcal{K}_n\right) = \left(1 + \nu_\eta\left(\mathcal{K}_n\right)\right)\left\|u\right\|_{\mathcal{K}_n},
	\end{equation}
	which is uniformly bounded for \(\eta \in \left[-\eta_\star, \eta_\star\right]\) by \eqref{eq:moment_bounds}. Putting this all together, it is clear that 
	\[\left\|\widecheck{\phi}_\eta\right\|_{\mathcal{K}_n} \leq C,\]
	with \(C\) independent of \(\eta\). 
\end{proof}

We can now provide the proof of Proposition~\ref{prop:std_estimator}.
\begin{proof}[Proof of Proposition~\ref{prop:std_estimator}]
	The ellipticity of the noise and the assumptions on the drift allow us to apply the results of \cite{Kliemann}, which imply that, for any \(x \in \mathbb{R}^d\), the process \(\left(X_t^\eta\right)_{t\geq 0}\) started at \(X_0^\eta = x\) satisfies almost surely
	\[\frac{1}{\eta t}\int_{0}^t R\left(X_s^\eta\right)ds \xrightarrow[t \to +\infty]{} \frac{1}{\eta}\int_{\mathbb{R}^d}R\, d\nu_\eta.\]
	Since the almost sure convergence holds for any deterministic initial condition, it holds for any initial probability measure \(\nu\) such that \(\nu\left(|R|\right) < + \infty\), which is the case by our assumptions on \(\nu_{\rm init}\).
	
	A central limit theorem holds for \(\widehat{\Phi}_{\eta, t}\) by the results of~\cite{Bhattacharya} since the Poisson equation  $-\mathcal{L}_\eta \widetilde{R}_\eta = \Pi_\eta R$ has a solution in~$\mathscr{S}_\eta \subset L^2(\nu_\eta)$. We write the asymptotic variance as
	\[
	\sigma_{\mathrm{std}, R,\eta}^2 = \frac{2}{\eta^2} \int_{\mathbb{R}^d} R \widetilde{R}_\eta \, d\nu_\eta.
	\]
	In view of Lemma~\ref{lm:poisson_sol_approx}, there is an \(\widecheck{R}_\eta \in \mathscr{S}\) such that
 	\begin{equation}\label{eq:r_rhat_nu_eta}
		\int_{\mathbb{R}^d} R \widetilde{R}_\eta \, d\nu_\eta = \int_{\mathbb{R}^d} R \widetilde{R}_0 \, d\nu_\eta + \eta \int_{\mathbb{R}^d} R \widecheck{R}_\eta \, d\nu_\eta.
	\end{equation}
	We next use Lemma~\ref{lm:lin_resp} for the two integrals on the right hand side to get
	\begin{equation}\label{eq:r_rhat_nu_eta2}
		\int_{\mathbb{R}^d} R \widetilde{R}_\eta \, d\nu_\eta = \int_{\mathbb{R}^d} R \widetilde{R}_0 \, d\nu_0 + \eta\int_{\mathbb{R}^d} R \left(\widetilde{R}_0 \mathfrak{f} + \widecheck{R}_\eta\right)\, d\nu_0 + \eta^2\mathcal{R}_\eta,
	\end{equation}
	where \(\mathfrak{f}\) is the function introduced in Lemma~\ref{lm:lin_resp} and the remainder \(\mathcal{R}_\eta\) is given by
	\[\mathcal{R}_\eta = \mathscr{R}_{\eta, R\widetilde{R}_0} + \int_{\mathbb{R}^d} R \widecheck{R}_\eta\mathfrak{f} \, d\nu_0 + \eta\mathscr{R}_{\eta, R \widecheck{R}_\eta},\]
	with \(\mathscr{R}_{\eta, \cdot}\) the remainder term from Lemma~\ref{lm:lin_resp}. By Lemma~\ref{lm:poisson_sol_approx}, \(\left\|\widecheck{R}_\eta\right\|_{\mathcal{K}_n}\) is bounded for some \(n \in \mathbb{N}\) uniformly in \(\eta\in \left[-\eta_\star, \eta_\star\right]\); likewise for the remainder terms \(\mathscr{R}_{\eta, R\widetilde{R}_0}\) and \(\mathscr{R}_{\eta, R \widecheck{R}_\eta}\) by Lemma~\ref{lm:lin_resp}. Thus the second integral and remainder term on the right hand side \eqref{eq:r_rhat_nu_eta2} are uniformly bounded in \(\eta\in \left[-\eta_\star, \eta_\star\right]\). Using the expression of  \(\sigma_{\mathrm{ref}, R}^2\) \eqref{eq:ref_variance}, we have the following equality for the asymptotic variance,
	\[\sigma_{\mathrm{std}, R,\eta}^2 = \frac{\sigma_{\mathrm{ref}, R}^2}{\eta^2} + \frac{1}{\eta}\int_{\mathbb{R}^d} R \left(\widetilde{R}_0 \mathfrak{f} + \widecheck{R}_\eta\right)\, d\nu_0 + \mathcal{R}_\eta,\]
	proving~\eqref{eq:std_variance_bounds}.
	
	We finally prove~\eqref{eq:std_bias_bounds}. Applying It\^o's formula to \(\widetilde{R}_\eta\left(X_t^\eta\right)\), we obtain
	\[
	\widehat{\Phi}_{\eta,t} - \frac1\eta \int_{\mathbb{R}^d} R \, d\nu_\eta = \frac{\widetilde{R}_\eta(X_0^\eta) - \widetilde{R}_\eta(X_t^\eta)}{\eta t} + \frac{1}{\eta t} \int_0^t \nabla\widetilde{R}_\eta(X_s^\eta)^{\top } dW_s,
	\]
	where the Itô integral is a true martingale by the fact that \(\widetilde{R}_\eta \in \mathscr{S}\) by Proposition~\ref{prop:exist_poisson_eq} and the moment bounds implied by \eqref{eq:semigroup_estimates} together with the hypotheses on the initial measure. Therefore,
	\[
	\mathbb{E}\left(\widehat{\Phi}_{\eta,t}\right) - \alpha_{R,\eta} = \frac{1}{\eta t} \mathbb{E}\left[\widetilde{R}_\eta(X_0^\eta) - \widetilde{R}_\eta(X_t^\eta)\right],
	\]
	which leads to~\eqref{eq:std_bias_bounds} with
	\[
	K = \left(1 + S_{m, \eta_{\star}}\right) \left\|\widetilde{R}_\eta\right\|_{B^\infty_m} \nu_{\rm init}\left(\mathcal{K}_m\right) < +\infty,
	\]
	where \(m \in \mathbb{N}\) is such that \(\widetilde{R}_\eta \in B^\infty_m\) and where we used~\eqref{eq:semigroup_estimates} to control \(\nu_{\rm init}\left(P_t^\eta \mathcal{K}_m\right)\) by \(S_{m, \eta_{\star}}\nu_{\rm init}\left(\mathcal{K}_m\right)\), the latter quantity being finite by the moment conditions on~\(\nu_{\rm init}\). 
\end{proof}
\section{Control Variates Based on Synchronous Couplings}\label{sec:coupling}
To construct an estimator of \(\alpha_R\) with lower variance than that of the standard estimator, we use a control variate approach, relying on the knowledge that \(R \in \mathscr{S}_0\) vanishes under the stationary probability measure of the reference dynamics. For this approach to be efficient, we want the trajectories of~\(\left(X_t^\eta\right)_{t\geq 0}\) and \(\left(X_t^0\right)_{t\geq 0}\) to remain as close as possible, even for long times. This motivates considering two dynamics~\(\left(X_t^\eta\right)_{t\geq 0}\)~and~\(\left(Y_t^0\right)_{t\geq 0}\):
\begin{equation}\label{eq:coupled_dynamics}
	\begin{aligned}
		dX_t^\eta &= \left(b\left(X_t^\eta\right) + \eta F\left(X_t^\eta\right)\right)dt + \sqrt{\frac{2}{\beta}}dW_t,\\	
		dY_t^0 &= b\left(Y_t^0\right)dt + \sqrt{\frac{2}{\beta}} d\widetilde{W}_t,
	\end{aligned}
\end{equation}
with coupled driving noises \(W\) and \(\widetilde{W}\); and constructing the following estimator 
\begin{equation}\label{eq:coupled_estimator}
	\widehat{\Psi}_{\eta, t} = \frac{1}{\eta t}\int_0^t\left[R\left(X_s^\eta\right) - R\left(Y_s^0\right)\right]ds.
\end{equation}
In this section we present synchronous coupling which works exceedingly well when the dynamics is strong contractive everywhere, i.e. \(M = 0\) in \eqref{eq:contractive_at_inf}. The dynamics are coupled by choosing the same Brownian motion to drive \(X^\eta\) and \(Y^0\), i.e. set \(W = \widetilde{W}\) in \eqref{eq:coupled_dynamics}. When we have strong contractivity everywhere, the synchronous coupling based estimator, denoted by \(\widehat{\Psi}_{\eta, t}^{\mathrm{sync}}\), greatly outperforms the standard estimator in a sense made precise by Theorem~\ref{thm:sync_bias_var} below. However contractivity is essential. Sticky coupling, presented in Section~\ref{sec:sticky_coupling}, does not require as strong assumptions on the contractivity of the dynamics at the price of a worse scaling of the variance with \(\eta\) compared to synchronous coupling. 

\paragraph{Notation.}
For a continuous Markov process with values in \(\mathbb{R}^d\) and a probability measure \(\mu\) on \(\mathbb{R}^d\), we denote by \(\mathbb{P}_\mu\) the law on \(C\left(\mathbb{R}_+, \mathbb{R}^d\right)\) of the process with initial condition \(X_0 \sim \mu\) and by \(\mathbb{E}_{\mu}\) the expectation with respect to this measure. When there is no risk of ambiguity or it is not relevant we may suppress the dependence on the initial measure and simply write \(\mathbb{E}\). Furthermore, we may write~\(\mathbb{E}_\eta\)~and~\(\mathbb{E}_x\) for expectations with respect to the laws with \(\nu_\eta\) and \(\delta_x\) respectively as initial probability measures. We use an analogous notation for the coupled process on \(\mathbb{R}^d\times \mathbb{R}^d\).

In what follows, we use \(C\) as a generic constant that may change from line to line. This will not be the case for constants with sub- or superscripts or using other letters. We use \(\oplus\) to denote the direct sum of vector spaces. Furthermore, for two functions \(f, g : \mathbb{R}^d \to \mathbb{R}\), we denote by \(f \oplus g\) the function on \(\mathbb{R}^d \times \mathbb{R}^d\) defined by \(\left(f \oplus g\right)\left(x, y\right)= f(x) + g(y)\); and for two operators \(\mathcal{A}, \mathcal{B}\) acting on some subset of functions on~\(\mathbb{R}^d\), we denote by \(\mathcal{A} \oplus \mathcal{B}\) the operator on some suitable subset of functions on \(\mathbb{R}^d \times \mathbb{R}^d\) defined by 
\[\left(\mathcal{A} \oplus \mathcal{B}\right)\varphi\left(x,y\right) = \mathcal{A}_x \varphi(x, y) + \mathcal{B}_y \varphi(x,y)\]
where \(\mathcal{A}_x\) is \(\mathcal{A}\) acting on functions of the first component \(x\) (\(y\) being considered as a parameter in this setting) and \(\mathcal{B}_y\) is \(\mathcal{B}\) acting on functions of the second component \(y\).

\subsection{Quantitative Results}
Like the standard estimator \eqref{eq:naive_estimator}, the synchronous coupling based estimator \(\widehat{\Psi}_{\eta, t}^\mathrm{sync}\) satisfies the following central limit theorem.
\begin{proposition}\label{prop:sync_clt}
	Fix \(\eta \in \mathbb{R}\) and suppose that \eqref{eq:contractive_at_inf} holds with \(M = 0\). Then, the synchronously coupled dynamics admits a unique ergodic invariant probability measure \(\mu_{\mathrm{sync},\eta}\). Furthermore, for any \(R\in \mathscr{S}_0\), the estimator~\(\widehat{\Psi}_{\eta, t}^{\mathrm{sync}}\) converges almost surely to \(\alpha_{R, \eta}\) as \(t \to \infty\), and the following central limit theorem holds:
	\begin{equation}\label{eq:sync_clt}
		\sqrt{t}\left(\widehat{\Psi}_{\eta, t}^{\mathrm{sync}} - \alpha_{R, \eta}\right) \xrightarrow[t \to \infty]{\mathrm{law}} \mathcal{N}\left(0, \sigma^2_{\mathrm{sync}, R, \eta}\right),
	\end{equation}
	with asymptotic variance \(\sigma^2_{\mathrm{sync}, R, \eta} \in \left(0, \infty\right)\). 
\end{proposition}
During the proof of this result, which can be read in Section~\ref{sec:sync_proofs}, we show the existence of a unique invariant ergodic probability measure~\(\mu_{\mathrm{sync},\eta}\) for the coupled process. 
We also derive in the proof an explicit expression for the asymptotic variance \(\sigma_{\mathrm{sync}, R, \eta}^2\). As the following theorem makes clear, this variance  and the finite-time integration bias are uniformly bounded as \(\eta\) goes to \(0\).

\begin{theorem}\label{thm:sync_bias_var}
	Fix \(\eta_\star > 0\) and \(R \in \mathscr{S}_0\). Suppose that Assumptions~\ref{ass:drift}~and~\ref{ass:ref_measure} hold and that the contractivity condition \eqref{eq:contractive_at_inf} is satisfied with \(M = 0\). Assume that \(\left(\mu_{\mathrm{init}, \eta}\right)_{\eta \in \left[-\eta_\star, \eta_\star\right]}\) is a family of initial probability measures on \(\mathbb{R}^d \times \mathbb{R}^d\) such that \(\left(X_0^\eta, Y_0^0\right) \sim \mu_{\mathrm{init}, \eta}\) for each \(\eta \in \left[-\eta_\star, \eta_\star\right]\), with  \(\mu_{\mathrm{init}, \eta}\left(\mathcal{K}_n \oplus \mathcal{K}_n\right) \leq C_{n, \eta_{\star}}\) uniformly in \(\eta \in \left[-\eta_\star, \eta_\star\right]\) for any \(n\in \mathbb{N}\),  and that there exist~\(p > 1\) and \(C_p \in \mathbb{R}_+\) such that 
	\begin{equation}\label{eq:sync_init_assump}
		\forall \eta \in \left[-\eta_\star, \eta_\star\right], \qquad \int_{\mathbb{R}^d \times \mathbb{R}^d} \left|x - y\right|^p \mu_{\mathrm{init}, \eta}\left(dx \, dy\right) \leq C_p \eta^p.
	\end{equation}
	Then there exist \(K_1, K_2  > 0\) (depending on \(\eta_\star, R\); \(K_2\) depending on \(\mu_{\mathrm{init}}\) as well) such that the asymptotic variance and bias respectively satisfy 
	\begin{equation}\label{eq:sync_var_bound}
		\forall \eta \in \left[-\eta_\star, \eta_\star\right], \qquad \left|\sigma_{\mathrm{sync}, R, \eta}^2\right| \leq K_1,
	\end{equation} 
	and
	\begin{equation}\label{eq:sync_bias_bound}
		\left|\mathbb{E}\left[\widehat{\Psi}_{\eta, t}^{\mathrm{sync}}\right] - \alpha_{R,\eta}\right| \leq \frac{K_2}{t}.
	\end{equation}
\end{theorem}
In practice the assumption \eqref{eq:sync_init_assump} on the initial distributions of the coupled process is easy to satisfy as a natural initialization is \(\nu_0\left(dx\right)\delta_x\left(dy\right)\), i.e. start the two components of the coupled process from the same configuration sampled according to the reference stationary measure.

Unlike the bias and variance of the standard estimator \eqref{eq:std_bias_bounds} and \eqref{eq:std_variance_bounds}, the bias and variance of the synchronously coupled estimator are bounded uniformly in \(\eta\). This makes it an attractive option when the drift is strongly contractive everywhere. This assumption is essential, without it Lemma~\ref{lm:sync_coupling_dist} below fails and we can no longer control the distance between the two coupled trajectories. Many problems of interest involve locally non-convex potentials. In non-convex regions, the trajectories can separate and particularly in high dimensions have trouble coming back together. Thus without contractivity everywhere we cannot solely rely on the deterministic part of the dynamics to bring the trajectories together, and need to resort to a more educated coupling (see Section~\ref{sec:sticky_coupling}).
\subsection{Proofs of Proposition~\ref{prop:sync_clt} and Theorem~\ref{thm:sync_bias_var}}\label{sec:sync_proofs}
We start with two lemmas. The first result quantifies the evolution of the distance between the two components of the coupled process.
\begin{lemma}\label{lm:sync_coupling_dist}
	Let \(\left(X_t^\eta, Y_t^0\right)_{t\geq 0}\) be a solution to the synchronously coupled dynamics (\ref{eq:coupled_dynamics}) and \(m > 0\) such that (\ref{eq:contractive_at_inf}) is satisfied for all \(x, y \in \mathbb{R}^d\). Then 
	\begin{equation}\label{eq:sync_coupling_dist}
		\left|X_t^\eta - Y_t^0\right| \leq \left(\left|X_0^\eta - Y_0^0\right| - \frac{\eta \left\|F\right\|_\infty}{m}\right)\mathrm{e}^{-m t} + \frac{\eta \left\|F\right\|_\infty}{m}.
	\end{equation}
\end{lemma}
\begin{proof}
	Using Itô's formula, we get
	\begin{align*}
		d\left(\left|X_t^\eta - Y_0^0\right|^2\right) &= 2 \left<X_t^\eta - Y_t^0, dX_t^\eta - dY_t^0\right> \\
		&= 2\left< X_t^\eta - Y_t^0, b\left(X_t^\eta\right) - b\left(Y_t^0\right)\right>dt + 2\eta \left<X_t^\eta - Y_t^0, F\left(X_t^\eta\right)\right>dt\\
		&\leq  \left(-2m\left|X_t^\eta - Y_t^0\right|^2 + 2\eta\left|X_t^\eta - Y_t^0\right| \left|F\left(X_t^\eta\right)\right|\right)dt\\
		& \leq \left(-2 m\left|X_t^\eta - Y_t^0\right|^2 + 2\eta\left|X_t^\eta - Y_t^0\right|\left\|F\right\|_{\infty}\right)dt,
	\end{align*}
	where the first inequality is due to (\ref{eq:contractive_at_inf}) and a Cauchy--Schwarz inequality. We therefore have that \(\left|X_t^\eta - Y_t^0\right|^2 \leq \psi\left(t\right)^2\) where \(\psi^2\) satisfies
	\begin{equation*}
		\left(\psi^2\right)' = - 2m \psi^2 + 2\eta \left\|F\right\|_\infty\psi, \qquad \psi(0)^2 = \left|X_0^\eta - Y_0^0\right|^2,
	\end{equation*}
	and consequently \(\psi\) satisfies
	\begin{equation*}
		\psi' = -m\psi + \eta \left\|F\right\|_\infty, \qquad \psi(0) = \left|X_0^\eta - Y_0^0\right|.
	\end{equation*}
	Solving this equation, we then obtain
	\begin{equation}
		\left|X_t^\eta - Y_t^0\right| \leq \left(\left|X_0^\eta - Y_0^0\right| - \frac{\eta \left\|F\right\|_{\infty}}{m}\right)\mathrm{e}^{-m t} + \frac{\eta \left\|F\right\|_{\infty}}{m},
	\end{equation}
	giving the desired bound.
\end{proof}
The second lemma quantifies the local Lipschitzness of functions in \(\mathscr{S}\). 
\begin{lemma}\label{lm:pseudo_lip}
	Let \(\varphi \in \mathscr{S}\). Then,  there exist \(C > 0\) and \(n \in \mathbb{N}\) such that
	\begin{equation}\label{eq:pseudo_lip_cond}
		\forall x,y \in \mathbb{R}^d \qquad \left|\varphi(x) - \varphi(y)\right| \leq C \left(1 + |x|^n + |y|^n\right) \left|x - y\right|.
	\end{equation}
\end{lemma}
\begin{proof}
	As \(\varphi \in \mathscr{S}\), there exists \(n \in \mathbb{N}\) such that \(\partial_{x_i} \varphi \in B_n^\infty\) for all \(i = 1, \dots, d\). By the Mean Value Theorem, it holds
	\begin{equation}
		\begin{aligned}
			\left|\varphi(x) - \varphi(y)\right| &\leq \sup_{t \in [0,1]} \left|\nabla \varphi\left(x + t \left(x-y\right)\right)\right| |x - y| \\
			&\leq C \sup_{t \in [0,1]} \mathcal{K}_n\left(x + t \left(x-y\right)\right) |x - y|\\
			&\leq C \left(1 + |x|^n + |y|^n\right) |x-y|,
		\end{aligned}
	\end{equation}
	which is the claimed result.
\end{proof}

\begin{proof}[Proof of Propostion \ref{prop:sync_clt}]
	As in Proposition~\ref{prop:std_estimator}, the results of \cite{Kliemann} allow us to conclude that almost surely
	\[\frac{1}{\eta t}\int_0^t R\left(X_s^\eta\right)ds \xrightarrow[t\to \infty]{} \frac{\nu_\eta\left(R\right)}{\eta}, \quad \text{and} \quad \frac{1}{\eta t} \int_0^t R\left(Y_s^0\right)ds \xrightarrow[t\to \infty]{} \frac{\nu_0\left(R\right)}{\eta} = 0.\]
	Thus as the difference of two processes that converge almost surely, 
	\[\widehat{\Psi}_{\eta, t}^{\mathrm{sync}} \xrightarrow[t \to \infty]{} \frac{\nu_\eta\left(R\right)}{\eta} = \alpha_{R, \eta}.\]

	To show that the synchronously coupled dynamics admit a unique ergodic invariant probability measure, we construct a coupling of two solutions to \eqref{eq:coupled_dynamics} with \(W = \widetilde{W}\), denoted by \(\left(Z^\eta_t\right)_{t\geq 0} = \left(X^\eta_t, Y_t^0\right)_{t\geq 0}\) and \(\left(\widecheck{Z}^\eta_t\right)_{t\geq 0} = \left(\widecheck{X}^\eta_t, \widecheck{Y}_t^\eta\right)_{t\geq 0}\), for which there exists \(\gamma > 0\) such that 
	\[\mathbb{E}\left|Z_t^\eta - \widecheck{Z}_t^\eta\right| \leq C\mathrm{e}^{-\gamma t},\]
	where the constant \(C\) depends only on the marginals of the initial conditions \(\left(Z_0^\eta, \widecheck{Z}_0^\eta\right)\). Denoting by \(\left(T_t^\eta\right)_{t \geq 0}\) the semi-group of the synchronously coupled process, this then implies that, for \(\mu\) and \(\widetilde{\mu}\) belonging to some class of probability measures that will be specified in Appendix~\ref{sec:sync_contractive},
	\[\mathcal{W}_1\left(\mu T_t^\eta, \widetilde{\mu} T_t^\eta\right) \leq C\left(\mu, \widetilde{\mu}\right) \mathrm{e}^{-\gamma t}, \]
	where \(\mathcal{W}_1\) is the 1-Wasserstein distance. The existence and uniqueness of an ergodic invariant probability measure then follows as a corollary. We defer this construction and the proof of the existence and uniqueness of an ergodic invariant probability measure to Appendix~\ref{sec:sync_contractive}.
	
	Admitting for the moment that it exists, we denote by \(\mu_{\mathrm{sync},\eta}\) the invariant probability measure of the synchronously coupled dynamics, which is a coupling of \(\nu_\eta\) and \(\nu_0\) since the marginal processes admit respectively \(\nu_\eta\) and \(\nu_0\) as invariant probability measures. Applying Lemma~\ref{lm:sync_coupling_dist} and using ergodicity to take the limit as \(t\to \infty\) in \(\mathbb{E}\left[\left|X_t^\eta - Y_t^0\right|^\alpha\right]\) for any \(\alpha \in \mathbb{R}_+\), we obtain
	\begin{equation}\label{eq:stationary_sync_dist}
		\int_{\mathbb{R}^d\times \mathbb{R}^d}\left|x - y\right|^\alpha \mu_{\mathrm{sync},\eta}\left(dx\,dy\right) \leq \left(\frac{\left\|F\right\|_\infty}{m}\right)^\alpha \eta^\alpha.
	\end{equation}
	In fact, taking the limit as \(t \to \infty\) in \eqref{eq:sync_coupling_dist} implies that the measure \(\mu_{\mathrm{sync}, \eta}\) is supported on a tube in \(\mathbb{R}^d\times \mathbb{R}^d\) centered along the diagonal of diameter \(\frac{\|F\|_\infty}{m}\eta\), i.e.
	\[\mathrm{supp}\left(\mu_{\mathrm{sync}, \eta}\right) \subset \left\{ (x, y) \in \mathbb{R}^d\times \mathbb{R}^d \left| \, |x - y| \leq \frac{\|F\|_\infty}{m}\eta \right. \right\}.\]
	
	The generator of the coupled process is given by
	\begin{equation}
		\mathcal{A}_\eta = \left(b(x) + \eta F(x)\right)\cdot \nabla_x + \frac{1}{\beta} \Delta_x + b(y)\cdot \nabla_y + \frac{1}{\beta} \Delta_y + \frac{2}{\beta}\sum_{i=1}^d \partial_{x_i}\partial_{y_i}
		=\mathcal{L}_\eta \oplus \mathcal{L}_0 + \frac{2}{\beta}\sum_{i=1}^d \partial_{x_i}\partial_{y_i}.
	\end{equation}
	For a given function \(R \in \mathscr{S}_0\), we can therefore construct an explicit solution to the Poisson equation corresponding to the coupled process, namely
	\begin{equation}\label{eq:sync_coupled_poisson_eq}
		-\mathcal{A}_\eta u = \left(\Pi_\eta R\right) \oplus \left(-\Pi_0 R\right).
	\end{equation}
	To this end, we denote by \(\widetilde{R}_\eta\) and \(\widetilde{R}_0\) the solutions to the Poisson equations corresponding the marginal processes, namely \(-\mathcal{L}_\eta \widetilde{R}_\eta = \Pi_\eta R\) and \(-\mathcal{L}_0 \widetilde{R}_0 = \Pi_0 R\). Then \(u(x,y) = \widetilde{R}_\eta(x) - \widetilde{R}_0(y)\) is a solution to the Poisson equation \eqref{eq:sync_coupled_poisson_eq}, as can easily be checked by computing \(\mathcal{A}_\eta u\). The corresponding Poisson equation has a solution in \(\mathscr{S}\oplus \mathscr{S}\), which is a subset of \( L^2\left(\mu_{\eta, \text{sync}}\right)\) since the moment bounds \eqref{eq:moment_bounds} imply that \(\mathscr{S} \subset L^2\left(\nu_\eta\right)\) for any \(\eta \in \left[-\eta_\star, \eta_\star\right]\). Therefore, by Bhattacharya's Central Limit Theorem \cite{Bhattacharya}, \(\widehat{\Psi}_{\eta, t}^{\mathrm{sync}}\) satisfies (\ref{eq:sync_clt}) with asymptotic variance
	\begin{equation}
		\sigma_{\mathrm{sync}, R, \eta}^2 = \frac{2}{\eta^2}\int_{\mathbb{R}^d \times \mathbb{R}^d} \left(\widetilde{R}_\eta(x) - \widetilde{R}_0(y)\right)\left(\Pi_\eta R(x)- \Pi_0 R(y)\right)\mu_{\eta, \text{sync}}\left(dx\,dy\right),
	\end{equation}
	concluding the proof of the CLT for the synchronous coupling based estimator.
\end{proof}

\begin{proof}[Proof of Theorem \ref{thm:sync_bias_var}]
	By Itô's formula we have 
	\begin{equation}\label{eq:sync_bias_var_eq1}
		\begin{aligned}
			\widehat{\Psi}_{\eta, t}^{\mathrm{sync}} - \alpha_{R, \eta} = &\frac{1}{\eta t} \left[\widetilde{R}_\eta\left(X_0^\eta\right) - \widetilde{R}_0\left(Y_0^0\right) - \widetilde{R}_\eta\left(X_t^\eta\right) + \widetilde{R}_0\left(Y_t^0\right)\right]\\
			&\quad+ \frac{1}{\eta t} \left[\int_0^t \nabla \widetilde{R}_\eta \left(X_s^\eta\right) \cdot dW_s - \int_0^t \nabla \widetilde{R}_0\left(Y_s^0\right)\cdot dW_s\right],
		\end{aligned}
	\end{equation}
	where the stochastic integrals are true martingales as \(\nabla \widetilde{R}_\eta\) and \(\nabla \widetilde{R}_0\) have at most polynomial growth by Proposition~\ref{prop:exist_poisson_eq} and the moments of \(X_t^\eta\) and \(Y_t^0\) are uniformly bounded as implied by \eqref{eq:semigroup_estimates} and the hypothesis on initial probability measures. Therefore,
	\begin{equation}\label{eq:sync_bias_decomp}
		\begin{aligned}
			\left|\mathbb{E}\left[\widehat{\Psi}_{\eta, t}^{\mathrm{sync}} - \alpha_{R, \eta}\right]\right|  &\leq \frac{1}{t}\left(\mathbb{E}\left|\frac{\widetilde{R}_{\eta}\left(X_0^\eta\right) - \widetilde{R}_0\left(Y_0^0\right)}{\eta}\right| + \mathbb{E}\left|\frac{\widetilde{R}_\eta\left(X_t^\eta\right) - \widetilde{R}_0\left(Y_t^0\right)}{\eta}\right|\right)\\
			&= \frac{1}{t}\left(\int_{\mathbb{R}^d\times\mathbb{R}^d}\frac{\left|\widetilde{R}_\eta(x) - \widetilde{R}_0(y)\right|}{\eta}\mu_{\text{init}, \eta}\left(dx\,dy\right) + \mathbb{E}\left|\frac{\widetilde{R}_\eta\left(X_t^\eta\right) - \widetilde{R}_0\left(Y_t^0\right)}{\eta}\right|\right).
		\end{aligned}
	\end{equation}
	We bound the first term on the right hand side as 
	\begin{equation}\label{eq:sync_bias_var_eq3}
		\begin{aligned}
			\int_{\mathbb{R}^d\times\mathbb{R}^d}\frac{\left|\widetilde{R}_\eta(x) - \widetilde{R}_0(y)\right|}{\eta}\mu_{\text{init}, \eta}\left(dx\,dy\right) &\leq \int_{\mathbb{R}^d\times \mathbb{R}^d}\frac{\left|\widetilde{R}_\eta(x) - \widetilde{R}_0(x)\right|}{\eta}\mu_{\text{init}, \eta}\left(dx\,dy\right)\\
			&+ \int_{\mathbb{R}^d\times \mathbb{R}^d}\frac{\left|\widetilde{R}_0(x) - \widetilde{R}_0(y)\right|}{\eta}\mu_{\text{init}, \eta}\left(dx\,dy\right).
		\end{aligned}
	\end{equation}
	Applying Lemma~\ref{lm:poisson_sol_approx} gives, for a sufficiently large \(n \in \mathbb{N}\),
	\begin{equation}\label{eq:sync_bias_var_eq4}
		\int_{\mathbb{R}^d\times \mathbb{R}^d}\frac{\left|\widetilde{R}_\eta(x) - \widetilde{R}_0(x)\right|}{\eta}\mu_{\mathrm{init}, \eta}\left(dx\,dy\right)\leq C \mu_{\mathrm{init}, \eta}\left(\mathcal{K}_n \oplus \mathbf{0}\right),
	\end{equation}
	where the term on the right hand side is in uniformly bounded in \(\eta\) by the hypothesis on \(\mu_{\mathrm{init}, \eta}\). The second term in \eqref{eq:sync_bias_var_eq3} is bounded as follows:
	\begin{equation}\label{eq:sync_bias_var_eq5}
		\begin{aligned}
			&\int_{\mathbb{R}^d\times \mathbb{R}^d}\frac{\left|\widetilde{R}_0(x) - \widetilde{R}_0(y)\right|}{\eta}\mu_{\text{init}, \eta}\left(dx\,dy\right) 
			\leq \frac{C}{\eta} \int_{\mathbb{R}^d \times \mathbb{R}^d} \left(1 + \left|x\right|^n + \left|y\right|^n\right)\left|x - y\right|\mu_{\text{init}, \eta}\left(dx\,dy\right)\\
			&\leq \frac{C}{\eta}\left(\int_{\mathbb{R}^d \times \mathbb{R}^d} \left(1 + \left|x\right|^n + \left|y\right|^n\right)^{\frac{p}{p-1}}\mu_{\text{init}, \eta}\left(dx\,dy\right)\right)^{\frac{p-1}{p}}
			\left(\int_{\mathbb{R}^d \times \mathbb{R}^d}\left|x - y\right|^p\mu_{\text{init}, \eta}\left(dx\,dy\right)\right)^{1/p} \leq\frac{C}{\eta}\eta,
		\end{aligned}
	\end{equation}
	where the first inequality follows from \eqref{eq:pseudo_lip_cond} in Lemma~\ref{lm:pseudo_lip}, the second from Hölder's inequality, and the third from the hypotheses on the initial distribution. 
	
	For the second term on the right-hand side of (\ref{eq:sync_bias_decomp}), we similarly write 
	\begin{equation}\label{eq:sync_bias_var_eq6}
		\mathbb{E}\left|\frac{\widetilde{R}_\eta\left(X_t^\eta\right) - \widetilde{R}_0\left(Y_t^0\right)}{\eta}\right| \leq \mathbb{E}\left|\frac{\widetilde{R}_\eta\left(X_t^\eta\right) - \widetilde{R}_0\left(X_t^\eta\right)}{\eta}\right| + \mathbb{E}\left|\frac{\widetilde{R}_0\left(X_t^\eta\right) - \widetilde{R}_0\left(Y_t^0\right)}{\eta}\right|.
	\end{equation}
	The first term can be bounded using Lemma~\ref{lm:poisson_sol_approx} with \(n\in \mathbb{N}\) large enough:
	\begin{equation}\label{eq:sync_bias_var_eq7}
		\mathbb{E}\left|\frac{\widetilde{R}_\eta\left(X_t^\eta\right) - \widetilde{R}_0\left(X_t^\eta\right)}{\eta}\right| \leq C \mathbb{E}\left[\mathcal{K}_n\left(X_t^\eta\right)\right] \leq C \mu_{\mathrm{init}, \eta}\left(\mathcal{K}_n \oplus \mathbf{0}\right),
	\end{equation}
	where the second inequality is due to \eqref{eq:semigroup_estimates}. The right hand side is finite due to the hypotheses on the initial distribution. The second term in \eqref{eq:sync_bias_var_eq6} is bounded using \eqref{eq:pseudo_lip_cond} in Lemma~\ref{lm:pseudo_lip}:
	\begin{equation}\label{eq:sync_bias_var_eq9}
		\begin{aligned}
			\mathbb{E}\left|\frac{\widetilde{R}_0\left(X_t^\eta\right) - \widetilde{R}_0\left(Y_t^0\right)}{\eta}\right| 
			&\leq \frac{C}{\eta} \mathbb{E}\left[\left(1 + |X_t^\eta|^n + |Y_t^0|^n \right) \left|X_t^\eta - Y_t^0\right|\right]\\
			&\leq \frac{C}{\eta} \mathbb{E}\left[\left(1 + |X_t^\eta|^n + |Y_t^0|^n \right)^{\frac{p}{p-1}}\right]^{\frac{p-1}{p}} \mathbb{E}\left[\left|X_t^\eta - Y_t^0\right|^p\right]^{1/p}\\
			&\leq \frac{C}{\eta} \mathbb{E}\left[\left|\left(\left|X_0^\eta - Y_0^0\right| - \frac{\eta \left\|F\right\|_{\infty}}{m}\right)\mathrm{e}^{-m t} + \frac{\eta \left\|F\right\|_{\infty}}{m}\right|^p\right]^{1/p}\\
			&\leq \frac{C}{\eta} \left[\mathrm{e}^{-m t}\mathbb{E}\left[\left|X_0^\eta - Y_0^0\right|^p\right]^{1/p} + \left(1 - \mathrm{e}^{-m t}\right)\frac{\eta \left\|F\right\|_{\infty}}{m}\right]\\
			&\leq \frac{C}{\eta}\left[C_p^{1/p}\mathrm{e}^{-m t}\eta + \left(1 - \mathrm{e}^{-m t}\right)\frac{\eta \left\|F\right\|_{\infty}}{m}\right] \leq C,
		\end{aligned}
	\end{equation}
	where the third inequality follows by bounding the first factor in the second line with the moment bounds on the marginal process implied by \eqref{eq:semigroup_estimates} and the hypotheses on the initial distribution, and the second factor with \eqref{eq:sync_coupling_dist} from Lemma~\ref{lm:sync_coupling_dist}. The fifth follows from \eqref{eq:sync_init_assump}. Altogether these bounds imply (\ref{eq:sync_bias_bound}).
	
	For the asymptotic variance, it holds by the proof of Proposition~\ref{prop:sync_clt} and a Cauchy--Schwarz inequality that 
	\begin{equation}\label{eq:sync_bias_var_eq10}
		\begin{aligned}
			&\sigma_{\mathrm{sync}, R, \eta}^2 = \frac{2}{\eta^2}\int_{\mathbb{R}^d \times \mathbb{R}^d} \left(\widetilde{R}_\eta(x) - \widetilde{R}_0(y)\right)\left(\Pi_\eta R(x)- \Pi_0 R(y)\right)\mu_{\eta, \text{sync}}\left(dx\,dy\right) \\
			&\leq \frac{2}{\eta^2}\left(\int_{\mathbb{R}^d \times \mathbb{R}^d} \left(\widetilde{R}_\eta(x) - \widetilde{R}_0(y)\right)^2\mu_{\eta, \text{sync}}\left(dx\,dy\right)\right)^{1/2}\left(\int_{\mathbb{R}^d \times \mathbb{R}^d} \left(\Pi_\eta R(x)- \Pi_0 R(y)\right)^2\mu_{\eta, \text{sync}}\left(dx\,dy\right)\right)^{1/2}.
		\end{aligned}
	\end{equation}
	We can bound the first integral on the right hand side as 
	\begin{equation}\label{eq:sync_bias_var_eq11}
		\begin{aligned}
			&\int_{\mathbb{R}^d \times \mathbb{R}^d} \left(\widetilde{R}_\eta(x) - \widetilde{R}_0(y) \right)^2\mu_{\eta, \text{sync}}\left(dx\,dy\right)\\
			&\quad\leq 2\int_{\mathbb{R}^d \times \mathbb{R}^d} \left(\widetilde{R}_\eta(x) - \widetilde{R}_0(x) \right)^2\mu_{\eta, \text{sync}}\left(dx\,dy\right)
			+ 2\int_{\mathbb{R}^d \times \mathbb{R}^d} \left(\widetilde{R}_0(x) - \widetilde{R}_0(y) \right)^2\mu_{\eta, \text{sync}}\left(dx\,dy\right).
		\end{aligned}
	\end{equation}
	We can control the first term on the right hand side by \(C\eta^2\) using Lemma~\ref{lm:poisson_sol_approx} and the moment bounds on the marginals \eqref{eq:moment_bounds}. For the second term we use the estimate \eqref{eq:pseudo_lip_cond} in Lemma~\ref{lm:pseudo_lip} on \(\widetilde{R}_0\) and Hölder's inequality with an arbitrary \(q > 1\) to get
	\begin{equation}\label{eq:sync_bias_var_eq12}
		\begin{aligned}
			&\int_{\mathbb{R}^d \times \mathbb{R}^d} \left(\widetilde{R}_0(x) - \widetilde{R}_0(y) \right)^2\mu_{\eta, \mathrm{sync}}\left(dx\,dy\right) \leq \int_{\mathbb{R}^d \times \mathbb{R}^d} \left(1 + \left|x\right|^n + \left|y\right|^n\right)^2\left|x-y\right|^2 \mu_{\eta, \mathrm{sync}}\left(dx\,dy\right)\\
			&\leq \left(\int_{\mathbb{R}^d \times \mathbb{R}^d} \left(1 + |x|^n + |y|^n\right)^{2q}\mu_{\eta, \mathrm{sync}}\left(dx\,dy\right)\right)^{1/q} \left(\int_{\mathbb{R}^d \times \mathbb{R}^d} \left|x-y\right|^{\frac{2q}{q-1}}\mu_{\eta, \mathrm{sync}}\left(dx\,dy\right)\right)^{\frac{q-1}{q}}.
		\end{aligned}
	\end{equation}
	The first factor is bounded by the estimates on the moments of the marginals \eqref{eq:moment_bounds} and the second factor by \eqref{eq:stationary_sync_dist}. Together this allows us to bound the first integral in \eqref{eq:sync_bias_var_eq10} by \(C\eta^2\). For the second integral in~\eqref{eq:sync_bias_var_eq10} we write 
	\begin{equation}\label{eq:sync_bias_var_eq13}
		\begin{aligned}
			&\int_{\mathbb{R}^d \times \mathbb{R}^d} \left(\Pi_\eta R(x) - \Pi_0 R(y) \right)^2 \mu_{\eta, \text{sync}}\left(dx\,dy\right) \\
			&\leq 2\int_{\mathbb{R}^d \times \mathbb{R}^d} \left(\left(\Pi_\eta - \Pi_0\right)R(x)\right)^2 \mu_{\eta, \text{sync}}\left(dx\,dy\right)
			+2 \int_{\mathbb{R}^d \times \mathbb{R}^d} \left(\Pi_0 R(x)  - \Pi_0 R(y)\right)^2\mu_{\eta, \text{sync}}\left(dx\,dy\right).
		\end{aligned}
	\end{equation}
	The first term is bounded by \(C\eta^2\) by Lemma \ref{lm:lin_resp}. For the second term, again applying \eqref{eq:pseudo_lip_cond} in Lemma~\ref{lm:pseudo_lip}, Hölder's inequality for an arbitrary \(q > 1\), and the fact that \(\Pi_0 R = R\) since \(R \in \mathscr{S}_0\), gives
	\begin{equation}\label{eq:sync_bias_var_eq14}
		\begin{aligned}
			&\int_{\mathbb{R}^d \times \mathbb{R}^d} \left(\Pi_0R(x) - \Pi_0R(y)\right)^2 \mu_{\eta, \text{sync}}\left(dx\,dy\right) = \int_{\mathbb{R}^d \times \mathbb{R}^d} \left(R(x) - R(y)\right)^2 \mu_{\eta, \text{sync}}\left(dx\,dy\right)\\
			& \leq \left(\int_{\mathbb{R}^d \times \mathbb{R}^d} \left(1 + |x|^n + |y|^n\right)^{2q}\mu_{\eta, \text{sync}}\left(dx\,dy\right)\right)^{1/q}
			\left(\int_{\mathbb{R}^d \times \mathbb{R}^d} \left|x-y\right|^{\frac{2q}{q-1}}\mu_{\eta, \text{sync}}\left(dx\,dy\right)\right)^{\frac{q-1}{q}}.
		\end{aligned}
	\end{equation}
	Exactly as above the first factor is bounded using the moment estimates \eqref{eq:moment_bounds} since \(\mu_{\mathrm{sync},\eta}\) is a coupling of~\(\nu_\eta\) and \(\nu_0\), so that \(\mu_{\mathrm{sync},\eta}\left(\mathcal{K}_m \oplus \mathcal{K}_m\right) = \nu_\eta\left(\mathcal{K}_m\right) + \nu_0\left(\mathcal{K}_m\right)\). The second factor is bounded using \eqref{eq:stationary_sync_dist}. Therefore the second term in \eqref{eq:sync_bias_var_eq13} is bounded by~\(C\eta^2\). The second integral in \eqref{eq:sync_bias_var_eq10} is then bounded by \(C\eta^2\). 
	Putting these bounds together finally gives (\ref{eq:sync_var_bound}).
\end{proof}

\section{Sticky Coupling}\label{sec:sticky_coupling}
\begin{figure}[h]
	\includegraphics[width = 0.45\linewidth]{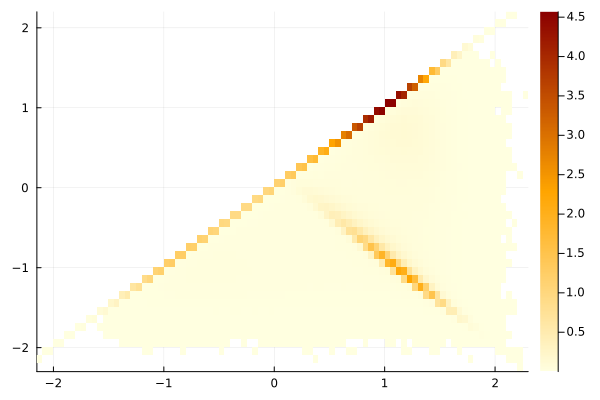}
	\includegraphics[width = 0.45\linewidth]{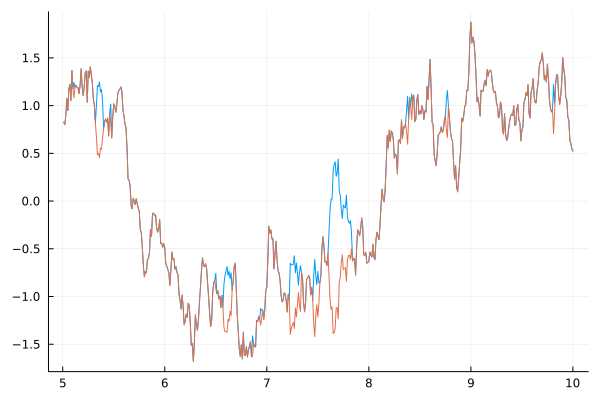}
	\caption{\small Sticky coupling of a one-dimensional particle in a double well potential perturbed by a constant force to the right, i.e. \(b(x) = -4x\left(x^2-1\right)\) and \(\eta F(x) = 1\). Left: histogram of the occupation density of the coupled process; Right: segment of trajectory of the coupled process with the perturbed marginal \(X^\eta\) plotted in blue and the reference marginal \(Y^0\) plotted in orange.}\label{fig:sticky_coupling}
\end{figure}

To overcome the shortcomings of synchronous coupling in the setting without strong contractivity everywhere, we turn to sticky coupling \cite{EberleZimmer} to build our coupled estimator. Sticky coupling couples the driving noise so that the processes are brought together by a reflection coupling, see \cite{LindvallRogers, Eberle}, and are "sticky" when they meet---in the sense that the set \(\left\{t \geq 0 \left| X_t = Y_t \right. \right\}\) has positive Lebesgue measure with probability 1. As suggested by Figure~\ref{fig:sticky_coupling}, we expect the invariant measure of the sticky coupled process to put positive mass on the diagonal \(\left\{\left. (x,y) \in \mathbb{R}^d \times \mathbb{R}^d \right| x = y\right\}\). 

Unfortunately, the continuous-time sticky coupled process is highly degenerate. The noise coefficients are discontinuous and the zero set of the one-dimensional process \(\left(r_t\right)_{t\geq 0}\) that bounds the difference between \(X^\eta\) and \(Y^0\) is a random fat Cantor set \cite{Howitt}. We believe that the sticky-coupled process should have an invariant probability measure with respect to which it is ergodic, but the degeneracies of the dynamics present technical obstacles to proving this. 

To sidestep these difficulties, we work with the discrete-time sticky coupled process.   
In practice, it is anyway this discrete-time sticky-coupled process that is simulated. In subsection 4.1, we present the discretized version of the dynamics and linear response. Next in subsection 4.2, we present discrete-time sticky coupling and an estimator based on the discretized sticky-coupled process. We then state in subsection 4.3 a central limit theorem for our estimator and some bounds on its bias and variance that are uniform in the time step and \(\eta\). We next present in subsection 4.4 some properties of sticky coupling useful for the proofs. We conclude in subsection 4.5 with a proof of the main result.

\subsection{Discrete-Time Dynamics and Linear Response}\label{subsec:disc_dynamics}
We discretize \eqref{eq:sde_model} via an Euler--Maruyama scheme with time step \(\Delta t > 0\).  Let \(\left(G_k\right)_{k\geq 1}\) be an independent and identically distributed (i.i.d.) sequence of standard Gaussian random variables in \(\mathbb{R}^d\). The discretized process is a Markov chain satisfying for, \(k \in \mathbb{N}^*\),
\begin{equation}\label{eq:X_markov_chain_update}
	X_{k+1}^{\eta, \Delta t} = X_k^{\eta, \Delta t} + \left(b\left(X_k^{\eta, \Delta t}\right) + \eta F\left(X_k^{\eta, \Delta t}\right)\right)\Delta t + \sqrt{\frac{2\Delta t}{\beta}}G_{k+1}.
\end{equation}
We denote the transition kernel of the Markov chain \eqref{eq:X_markov_chain_update} by \(P^{\eta, \Delta t}\) and its invariant probability measure by \(\nu_{\eta, \Delta t}\). We below sketch a proof of the existence of a unique ergodic invariant probability measure. For \(\eta \in \mathbb{R}\) and \(\Delta t > 0 \) we denote by \(\Pi_{\nu_{\eta, \Delta t}}\) the projection operator onto the space of function with mean zero with respect to \(\nu_{\eta, \Delta t}\), i.e for \(\phi \in L^1\left(\nu_{\eta, \Delta t}\right)\),
\[\Pi_{\nu_{\eta, \Delta t}} \phi = \phi - \int_{\mathbb{R}^d}\phi \, d\nu_{\eta, \Delta t}.\]

\paragraph{Linear Response} As a consequence of Lemma~\ref{lm:disc_lin_resp} below, we have a discrete time analogue of \eqref{eq:transport_coef}:
\begin{equation}\label{eq:disc_transport_coef}
	\alpha_{R, \Delta t} = \lim_{\eta \to 0} \frac{1}{\eta}\left(\int_{\mathbb{R}^d} R \, d\nu_{\eta, \Delta t} - \int_{\mathbb{R}^d} R \, d\nu_{0, \Delta t}\right),
\end{equation}
at least for \(\Delta t\) small enough such that the hypotheses of the lemma hold true.
The finite difference and discrete time approximation of this limit is then given by
\begin{equation}\label{eq:disc_alpha_eta}
	\alpha_{R, \eta, \Delta t} = \frac{1}{\eta}\left(\int_{\mathbb{R}^d} R \, d\nu_{\eta, \Delta t} - \int_{\mathbb{R}^d} R \, d\nu_{0, \Delta t}\right).
\end{equation}
As a consequence of Lemma~\ref{lm:disc_lin_resp} below, there exists \(\eta_\star, \Delta t^\star > 0\) and a positive constant \(C_{\eta_\star, \Delta t^\star}\) such that
\[\forall \eta \in \left[-\eta_\star, \eta_\star\right], \quad \forall \Delta t \in \left(0, \Delta t^\star\right), \qquad \left|\alpha_{R, \eta, \Delta t} - \alpha_R\right| \leq C_{\eta_\star, \Delta t^\star, k}\left(\Delta t + \eta \right).\]
Indeed, by \eqref{eq:disc_lin_resp},
\[
	\begin{aligned}
		\int_{\mathbb{R}^d} R\, d\nu_{\eta, \Delta t} &= \int_{\mathbb{R}^d} R \left(1 + \Delta t \mathfrak{f}_1 + \eta \mathfrak{f}_2 + \eta\Delta t\mathfrak{f}_3\right)d\nu_0 + \eta^2 \mathcal{A}_{1, \Delta t, \eta}R + \Delta t^2 \left(\mathcal{A}_{2, \Delta t} + \eta\mathcal{A}_{3, \Delta t, \eta}\right)R.
	\end{aligned}
\]
Thus, we can write \(\alpha_{R, \eta, \Delta t}\) as 
\[\begin{aligned}
	\alpha_{R, \eta, \Delta t} &= \frac{1}{\eta}\left(\int_{\mathbb{R}^d} R\, d\nu_{\eta, \Delta t} - \int_{\mathbb{R}^d} R\, d\nu_{0, \Delta t}\right)\\ 
	&= \int_{\mathbb{R}^d} R \mathfrak{f}_2 \, d\nu_0 + \Delta t\int_{\mathbb{R}^d} R \mathfrak{f}_3 \, d\nu_0 + \eta \mathcal{A}_{1, \Delta t, \eta}R + \Delta t^2 \mathcal{A}_{3, \Delta t, \eta}R.
\end{aligned}\]
It will be seen in Lemma~\ref{lm:disc_lin_resp} that \(\mathfrak{f}_2\) has the same expression as the linear response factor in the continuous time case \eqref{eq:lin_resp_factor} thus the first integral on the right hand side is equal to the transport coefficient \(\alpha_R\). The remaining terms are bounded by the conclusion of the lemma.

\paragraph{Ergodicity} We next sketch a uniform ergodicity result for the discretized marginal process. Fix $\eta_\star > 0$ and set \(\widetilde{M} = \max\left\{M, 4 \eta_\star \left\|F\right\|_\infty / m\right\}\). For \(\left|x - \widetilde{x}\right| \leq \widetilde{M}\), the fact \(b\) and \(F\) are Lipschitz implies that, for \(\eta \in \left[-\eta_\star, \eta_\star\right]\),
\begin{equation}
	\left|x + \Delta t \left(b\left(x\right) + \eta F\left(x\right)\right) - \widetilde{x} - \Delta t \left(b\left(\widetilde{x}\right) + \eta F\left(\widetilde{x}\right)\right)\right| \leq \left(1 + \Delta t\left(L_b + \eta_\star L_F\right)\right) \left|x - \widetilde{x}\right|. 
\end{equation}
Furthermore for \(\left|x - \widetilde{x}\right|\geq \widetilde{M}\), we have using Assumption~\ref{ass:drift} that 
\begin{equation}
	\begin{aligned}
		&\left|x + \Delta t \left(b\left(x\right) + \eta F\left(x\right)\right) - \widetilde{x} - \Delta t \left(b\left(\widetilde{x}\right) + \eta F\left(\widetilde{x}\right)\right)\right|^2\\
		& = \left|x - \widetilde{x}\right|^2 + 2\Delta t\left\langle x - \widetilde{x}, b\left(x\right) - b\left(\widetilde{x}\right) + \eta F\left(x\right) - \eta F\left(\widetilde{x}\right)\right\rangle + \Delta t^2\left|b\left(x\right) + \eta F\left(x\right) - b\left(\widetilde{x}\right) - \eta F\left(\widetilde{x}\right)\right|^2\\
		&\leq \left|x - \widetilde{x}\right|^2 - 2\Delta t m\left|x - \widetilde{x}\right|^2 + 4\Delta t\eta \left\|F\right\|_{\infty} \left|x - \widetilde{x}\right| + \Delta t^2 \left(L_b + \eta_\star L_F\right)^2\left|x - \widetilde{x}\right|^2\\
		&\leq \left|x - \widetilde{x}\right|^2 - 2\Delta t m\left|x - \widetilde{x}\right|^2 + \frac{4\Delta t\eta \left\|F\right\|_{\infty}}{\widetilde{M}} \left|x - \widetilde{x}\right|^2 + \Delta t^2 \left(L_b + \eta_\star L_F\right)^2\left|x - \widetilde{x}\right|^2\\
		&\leq \left(1 - \Delta t m + \Delta t^2 \left(L_b + \eta_\star L_F\right)^2\right)\left|x - \widetilde{x}\right|^2,
	\end{aligned}
\end{equation}
where we used the contractivity of \(b\) (see \eqref{eq:contractive_at_inf}), the boundedness of \(F\) (see \eqref{eq:F_bounded}), and the Lipschitzness of \(b\) and \(F\) for the first inequality, and our particular definition of \(\widetilde{M}\) for the last one.  
If \(\Delta t \in \left(0, \frac{m}{2\left(L_b + \eta_\star L_F\right)^2}\right)\), then the factor in front \(\left|x - \widetilde{x}\right|^2\) in the last line can be bounded by \(1 - \frac{m\Delta t}{2}\) which in turn can be bounded by \(\left(1- \frac{m\Delta t}{4}\right)^2\). These calculations suggest defining for \(\Delta t \in \left(0, \frac{m}{2\left(L_b + \eta_\star L_F\right)^2}\right)\) a continuous piecewise affine function \(\tau_{\eta_\star, \Delta t}\) by
\begin{equation}\label{eq:discrete_contraction}
	\tau_{\eta_\star, \Delta t}\left(r\right) = \left\{\begin{aligned}
		&\left[1 + \left(L_b + \eta_\star L_F\right)\Delta t\right]r && \text{ if } r \in [0, \widetilde{M}],\\
		&\left[1 + \left(L_b + \eta_\star L_F\right)\Delta t\right]\widetilde{M} + \left(1 - \frac{\Delta t m}{4}\right)\left(r -\widetilde{M}\right) && \text{ if } r \geq \widetilde{M},
	\end{aligned}\right.
\end{equation}
to control the discretized drift:
\begin{equation}\label{eq:disc_drift_control}
	\forall (x,\widetilde{x})\in \mathbb{R}^d \times \mathbb{R}^d, \qquad \left| \left(x - \widetilde{x}\right) + \Delta t \left(b(x) - b(\widetilde{x}) + \eta \left(F(x) - F(\widetilde{x})\right)\right) \right| \leq \tau_{\eta_\star, \Delta t}\left(\left|x - \widetilde{x}\right|\right).
\end{equation}
Let \(\widetilde{V}_c(x) := \mathrm{e}^{c|x|^2}\) be a Lyapunov function for the discretized process.
Following the proof of \cite[Proposition 1]{Durmus_etal}, the control \eqref{eq:disc_drift_control} on the discretized drift implies that \(P^{\eta, \Delta t}\) is uniformly geometrically ergodic in both \(\Delta t\) and \(\eta\). More precisely, fix \(\eta_\star\) and \(\Delta t^\star \in \min\left\{\frac{1}{m}, \frac{m}{2\left(L_b + \eta_\star L_F\right)^2}\right\}\). 
Then there exist constants \(\gamma \in (0, 1)\) and \(A, C, c > 0\) such that
\begin{equation}\label{eq:marginal_lyapunov_ineq}
	\forall \Delta t \in \left(0, \Delta t^*\right), \quad \forall \eta \in \left[-\eta_\star, \eta_\star\right], \quad \forall x \in \mathbb{R}^d, \qquad P^{\eta, \Delta t}\widetilde{V}_c(x) \leq \gamma^{\Delta t}\widetilde{V}_c\left(x\right) + \Delta t A
\end{equation}
and
\[\forall \Delta t \in \left(0, \Delta t^*\right), \quad \forall \eta \in \left[-\eta_\star, \eta_\star\right], \quad \forall x \in \mathbb{R}^d, \qquad \left\|\delta_x \left(P^{\eta, \Delta t}\right)^n - \nu_{\eta, \Delta t}\right\|_{\widetilde{V}_c} \leq C \gamma^{n\Delta t}\widetilde{V}_c(x).\]
Furthermore we have discrete time analogues of \eqref{eq:moment_bounds}, \eqref{eq:semigroup_estimates}, and \eqref{eq:inverse_bounds}, namely
\begin{align}
	\forall \eta_\star \in \left(0, \infty\right),  &\qquad \sup_{\Delta t \in \left(0, \Delta t^\star\right)}\sup_{\left|\eta\right| \leq \eta_\star} \nu_{\eta, \Delta t}(\widetilde{V}_c) < \infty, \label{eq:disc_moment_bounds} \\
	\forall \eta_\star \in \left(0, \infty\right),  &\qquad \sup_{\Delta t \in \left(0, \Delta t^\star\right)}\sup_{\left|\eta\right|\leq \eta_\star}\sup_{n \in \mathbb{N}} \sup_{x \in \mathbb{R}^d} \left|\frac{\left(\left(P^{\eta, \Delta t}\right)^n \widetilde{V}_c \right)(x)}{\widetilde{V}_c(x)}\right| < \infty, \label{eq:disc_semigroup_estimates}\\ 
	\forall \eta_\star \in \left(0, \infty\right), &\qquad \sup_{\Delta t \in \left(0, \Delta t^\star\right)}\sup_{\left|\eta\right| \leq \eta_\star} \left\|\left(\frac{P^{\eta, \Delta t} - \mathrm{Id}}{\Delta t}\right)^{-1}\right\|_{\mathcal{B}\left(\Pi_{\eta, \Delta t} B_{\widetilde{V}_c}^\infty\right)} < \infty \label{eq:disc_inverse_bounds}.
\end{align}
\paragraph{Perturbation results} We conclude this subsection with two lemmas that are the discrete time analogues of Lemmas~\ref{lm:lin_resp}~and~\ref{lm:poisson_sol_approx}. They are adapted from \cite[Theorem 3.4]{Leimkuhler} and its proof. There is a slight error in the proof in the original article. We give a corrected proof for our setting in Appendix~\ref{sec:discrete_lin_resp_proofs}.
\begin{lemma}\label{lm:disc_lin_resp}
	Let \(\eta_\star > 0\) and \(\Delta t^* > 0\). There exist function \(\mathfrak{f}_1, \mathfrak{f}_2, \mathfrak{f}_3 \in \Pi_0B_n^\infty\) for \(n \in \mathbb{N}\) large enough, with \(\mathfrak{f}_2 = -\left(\mathcal{L}_0^*\right)^{-1}\widetilde{\mathcal{L}}^*\mathbf{1}\), such that for any \(\varphi \in \mathscr{S}\),
	\begin{equation}\label{eq:disc_lin_resp}
		\begin{aligned}
			\int_{\mathbb{R}^d}\varphi \, d\nu_{\eta, \Delta t} &= \int_{\mathbb{R}^d} \varphi \left(1 + \Delta t \mathfrak{f}_1 + \eta \mathfrak{f}_2 + \eta \Delta t\mathfrak{f}_3\right)d\nu_0 + \eta^2 \mathcal{A}_{1, \Delta t, \eta}\varphi + \Delta t^2 \left(\mathcal{A}_{2, \Delta t} + \eta\mathcal{A}_{3, \Delta t, \eta}\right)\varphi
		\end{aligned}
	\end{equation}
	where \(\mathcal{A}_{1, \Delta t}\), \(\mathcal{A}_{2, \Delta t, \eta}\), and \(\mathcal{A}_{3, \Delta t, \eta}\) are linear functionals from \(\mathscr{S}\) to \(\mathbb{R}\) for which there exists \(p \in \mathbb{N}\) such that for any \(n \in \mathbb{N}\) there exist constant \(C_{n, \eta_\star} > 0\) for which, for any \(\varphi \in \mathscr{S}\) satisfying \(\partial^\alpha \varphi \in B_n^\infty\) for any multi-index \(|\alpha| \leq p\), 
	\[\begin{aligned}
		&\left|\mathcal{A}_{1, \Delta t}\varphi \right| \leq C_{n, \eta_\star} \sum_{\left|\alpha\right| \leq p} \left\|\partial^\alpha \varphi\right\|_{\mathcal{K}_n},\\
		&\left|\mathcal{A}_{2, \Delta t}\varphi \right| \leq C_{n, \eta_\star} \sum_{\left|\alpha\right| \leq p} \left\|\partial^\alpha \varphi\right\|_{\mathcal{K}_n},\\ 
		&\left|\mathcal{A}_{3, \Delta t, \eta}\varphi \right| \leq C_{n, \eta_\star} \sum_{\left|\alpha\right| \leq p} \left\|\partial^\alpha \varphi\right\|_{\mathcal{K}_n},
		\end{aligned} \]
		uniformly in \(\Delta t \in \left(0, \Delta t^*\right)\) and \(\eta \in \left[-\eta_\star, \eta_\star\right]\). 
\end{lemma}
The final lemma of this subsection gives a bound on the difference between the solutions of discrete time Poisson equation corresponding to the perturbed and reference processes. It shows that this difference is of the order of the perturbation, up to an error term related to the time step discretization and which can be made as small as wanted. This error term related to the time step discretization appears for technical reasons in the proof.
\begin{lemma}\label{lm:disc_poisson_sol_approx}
	Fix \(\eta_\star > 0\), \(\Delta t^\star > 0\), and \(R \in \mathscr{S}\). For any \(\eta \in \left[-\eta_\star, \eta_\star\right]\), consider the solution \(\widehat{R}_{\eta, \Delta t}\) in \(\Pi_{\eta, \Delta t}B_{\widetilde{V}_c}^\infty\) where \(c > 0\) is such that \eqref{eq:marginal_lyapunov_ineq} holds to the discrete the Poisson equation
	\[\left(\frac{\mathrm{Id} - P^{\eta, \Delta t}}{\Delta t}\right)\widehat{R}_{\eta, \Delta t} = \Pi_{\nu_{\eta, \Delta t}}R.\] 
	For any \(n \in \mathbb{N}\), there exists a constant \(C_n > 0\) such that 
	\begin{equation}
		\forall \eta \in \left[-\eta_\star, \eta_\star\right], \quad \forall \Delta t \in \left(0, \Delta t^*\right), \qquad \left\| \widehat{R}_{\eta, \Delta t} - \widehat{R}_{0, \Delta t}\right\|_{\widetilde{V}_c} \leq C_n \left(\eta + \Delta t^{2n}\right).
	\end{equation} 
\end{lemma}
\subsection{Discrete-Time Sticky Coupling}\label{subsec:sticky_coupling}
Following \cite{Durmus_etal}, we construct a discrete-time sticky coupling of the Euler--Maruyama discretizations of the two SDEs in (\ref{eq:coupled_dynamics})---see Figure~\ref{fig:sticky_coupling_schema} for a schematic illustration of the coupling. Let \(\left(U_k\right)_{k\geq 1}\) be a sequence of i.i.d. uniform \(\left[0,1\right]\) random variables independent from \(\left(G_k\right)_{k \geq 1}\). The first component of the couple \(\left(X_k^{\eta,\Delta t}, Y_k^{0, \Delta t}\right)\) evolves according to a standard Euler--Maruyama discretization,
\[X_{k+1}^{\eta, \Delta t} = X_k^{\eta, \Delta t} + \Delta t\left[b\left(X_k^{\eta, \Delta t}\right) + \eta F\left(X_k^{\eta, \Delta t}\right)\right] + \sqrt{\frac{2\Delta t}{\beta}}G_{k+1}.\]
For the second component we do one of two things at each step: with probability \(p_{\Delta t, \beta}\left(X_k^{\eta, \Delta t}, Y_k^{0, \Delta t}, G_{k+1}\right)\) the trajectories are forced together, with
\begin{gather}\label{eq:meeting_prob}
	p_{\Delta t, \beta}\left(x,y,g\right) = \min\left\{1, \frac{\varphi_d\left(\sqrt{\displaystyle \frac{\beta}{2\Delta t}}\mathbf{E}\left(x,y\right) + g  \right)}{\varphi_d\left(g\right)}\right\},
\end{gather}
where \(\varphi_d\) is the density of a \(d\)-dimensional standard Gaussian distribution and
\[\mathbf{E}\left(x,y\right) = y-x + \Delta t \left[b(y) - b(x) - \eta F(x)\right].\]
If they are not forced together, then \(Y_k^{0, \Delta t}\) is evolved with the same Gaussian noise as \(X_k^{\eta, \Delta t}\) but reflected over the hyperplane separating the two trajectories. More precisely, denote by \(\mathbf{e}(x,y)\) the normalization of the difference vector \(\mathbf{E}(x,y)\)
\[\mathbf{e}\left(x,y\right) = \left\{
\begin{aligned}
	&\frac{\mathbf{E}\left(x, y\right)}{\left|\mathbf{E}\left(x,y\right)\right|} \;\; && \text{ if } \mathbf{E}\left(x,y\right) \neq 0,\\
	& 0 && \text{ otherwise.}
\end{aligned}
\right. \]
To lighten the notation, we often use \(\mathbf{e}_k := \mathbf{e}\left(X_k^{\eta, \Delta t}, Y_k^{0, \Delta t}\right)\). Thus, \(\left[\mathrm{Id} - 2\mathbf{e}_k\mathbf{e}_k^T\right]G_{k+1}\) is the reflection of the Gaussian vector \(G_{k+1}\) used to drive the first component. 

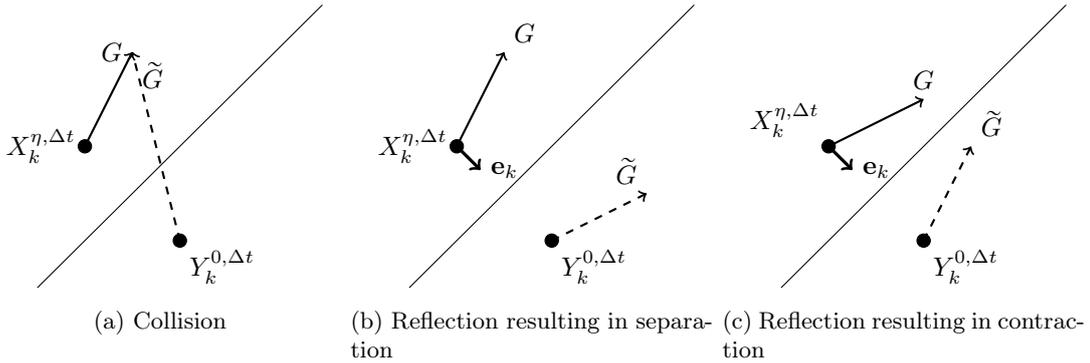
\begin{figure}[h]
	\centering
	\begin{subfigure}[t]{0.30\textwidth}
		\centering
		\begin{tikzpicture}[scale=1.25]
			\draw[black, thick, ->] (0.5, 1.5) -- (1, 2.5) node[anchor =  east] {\(G\)};
			\filldraw[black] (0.5, 1.5) circle (2pt) node[anchor = east] {\(X_{k}^{\eta, \Delta t}\)};
			\draw[black, thick, dashed, ->] (1.5, 0.5) -- (1, 2.5) node[anchor =   north west] {\(\widetilde{G}\)};
			\filldraw[black] (1.5, 0.5) circle (2pt) node[anchor = north west] {\(Y_k^{0, \Delta t}\)};
			\draw[black, thin] (0, 0) -- (3,3);
		\end{tikzpicture}
		\caption{Collision}\label{collision}
	\end{subfigure}
	\begin{subfigure}[t]{0.30\textwidth}
		\centering
		\begin{tikzpicture}[scale=1.25]
			\draw[black, thick, ->] (0.5, 1.5) -- (1, 2.5) node[anchor = south west] {\(G\)};
			\filldraw[black] (0.5, 1.5) circle (2pt) node[anchor = east] {\(X_k^{\eta, \Delta t}\)};
			\draw[black, very thick, ->](0.5, 1.5) -- (0.75, 1.25) node[anchor = west] {\(\mathbf{e}_k\)};
			\draw[black, thick, dashed, ->] (1.5, 0.5) -- (2.5, 1) node[anchor = south east] {\(\widetilde{G}\)};
			\filldraw[black] (1.5, 0.5) circle (2pt) node[anchor = north west] {\(Y_k^{0, \Delta t}\)};
			\draw[black, thin] (0, 0) -- (3,3);
		\end{tikzpicture}
		\caption{Reflection resulting in separation}\label{separation}
	\end{subfigure}
	\begin{subfigure}[t]{0.30\textwidth}
		\centering
		\begin{tikzpicture}[scale=1.25]
			\draw[black, thick, ->] (0.5, 1.5) -- (1.5, 2) node[anchor = south] {\(G\)};
			\filldraw[black] (0.5, 1.5) circle (2pt) node[anchor = south east] {\(X_k^{\eta, \Delta t}\)};
			\draw[black, very thick, ->](0.5, 1.5) -- (0.75, 1.25) node[anchor = west] {\(\mathbf{e}_k\)};
			\draw[black, thick, dashed, ->] (1.5, 0.5) -- (2, 1.5) node[anchor = south west] {\(\widetilde{G}\)};
			\filldraw[black] (1.5, 0.5) circle (2pt) node[anchor = north west] {\(Y_{k}^{0, \Delta t}\)};
			\draw[black, thin] (0, 0) -- (3,3);
		\end{tikzpicture}
		\caption{Reflection resulting in contraction}\label{contraction}
	\end{subfigure}
	\caption{The possible behavior of the coupled trajectories in one-step}\label{fig:sticky_coupling_schema}
\end{figure}
The definition of the meeting probability \eqref{eq:meeting_prob} we use (taken from \cite[Chapter 2]{Jacob_notes}) is equivalent to the one given in \cite[Equation~(8)]{Durmus_etal}, see Appendix~\ref{sec:equivalence_of_disc_MR_coupling}. A computation similar to the one in \cite[Section 4.1]{Durmus_Moulines} shows that this construction is indeed a coupling of the discretizations of the two SDEs in \eqref{eq:coupled_dynamics}. 
Furthermore, this coupling is maximal in the sense that it maximizes the probability of the two trajectories meeting in one step. A computation indeed shows that
\[\begin{aligned}
	\int_{\mathbb{R}^d} p_{\Delta t, \beta}\left(x, y, z\right) \varphi_d(z)dz &= \int_{\mathbb{R}^d} \min\left\{\varphi_d\left(\sqrt{\frac{\beta}{2\Delta t}} \mathbf{E}(x,y) + z\right), \varphi_d(z)\right\} dz\\
	&=  1 - d_{\mathrm{TV}}\left(P^{\eta, \Delta t}(x, \cdot), P^{0, \Delta t}(y, \cdot)\right).
\end{aligned}\]
As a consequence, this construction saturates the coupling inequality (see \cite[Section 2.2]{Jacob_notes})
\[\mathbb{P}\left(X_{k+1}^{\eta, \Delta t} = Y_{k+1}^{0, \Delta}\left| X_k^{\eta, \Delta t} = x, Y_k^{0, \Delta t} = y\right.\right) \leq 1 - d_{\mathrm{TV}}\left(P^{\eta, \Delta t}(x, \cdot), P^{0, \Delta t}(y, \cdot)\right).\] 
If \(X_{k+1}^{\eta, \Delta t}\) and \(Y_{k+1}^{0, \Delta t}\) are not forced together, then \(Y^{0, \Delta t}\) evolves with a noise that is reflected across the hyperplane dividing the two trajectories. Inspired by \cite{Eberle}, we expect this reflection coupling to induce a contraction in average. 

Putting this altogether, the update rule for the discrete-time sticky coupled process is given for \(k \in \mathbb{N}\) by
\begin{equation}\label{eq:disc_time_sticky}
	\begin{aligned}
		X_{k+1}^{\eta, \Delta t} &= X_k^{\eta, \Delta t} + \Delta t\left[b\left(X_k^{\eta, \Delta t}\right) + \eta F\left(X_k^{\eta, \Delta t}\right)\right] + \sqrt{\frac{2\Delta t}{\beta}} G_{k+1},\\
		Y_{k+1}^{0, \Delta t} &= 
		\left\{
		\begin{aligned}
			&X_{k+1}^{\eta, \Delta t}, && \text{ if } U_{k+1} \leq p_{\Delta t, \beta}\left(X_k^{\eta, \Delta t}, Y_k^{0, \Delta t}, G_{k+1}\right),\\
			&Y_k^{0, \Delta t} + \Delta t b\left(Y_k^{0, \Delta t}\right) + \sqrt{\frac{2\Delta t}{\beta}}\left[\mathrm{Id} - 2\mathbf{e}_k\mathbf{e}_k^T\right]G_{k+1} && \text{ otherwise}.
		\end{aligned}
		\right. 
	\end{aligned}
\end{equation}
We denote by \(T^{\eta, \Delta t}\) the transition kernel of the coupled process.
\subsection{Quantitative Results}
For discrete sticky coupling, the discrete-time estimator of \(\alpha_R\) is
\begin{equation}
	\widehat{\Psi}_{\eta, N}^{\Delta t, \mathrm{sticky}} = \frac{1}{\eta N}\sum_{k = 0}^{N-1}\left[R\left(X_k^{\eta, \Delta t}\right) - R\left(Y_k^{0, \Delta t}\right)\right].
\end{equation}
To state the results of this section, we introduce a Lyapunov function for the coupled process: 
\begin{equation}\label{eq:coupled_lyapunov_func}
	V_c(x,y) = \mathrm{e}^{c|x|^2} + \mathrm{e}^{c|y|^2} =: \widetilde{V}_c(x) + \widetilde{V}_c(y),
\end{equation}
with \(c > 0\) having the same value as in \eqref{eq:marginal_lyapunov_ineq}. The Lyapunov function \(V_c\) can be seen as a sum of Lyapunov functions for the marginals.
Like the previously introduced estimators, the sticky coupling based estimator satisfies a central limit theorem.

\begin{proposition}\label{prop:sticky_clt}
	Let \(\eta \in \mathbb{R}\) and \(R \in \mathscr{S}\). Assume that \(b\) and \(F\) satisfy Assumption~\ref{ass:drift} and that \(\left(X_0^{\eta, \Delta t}, Y_0^{0, \Delta t}\right) \sim \mu_{\mathrm{init}}\) for some probability measure \(\mu_{\mathrm{init}}\) such that \(\mu_{\mathrm{init}}\left(V_c\right) < + \infty\). Then there exists \(\Delta t^\star\) such that, for any \(\Delta t \in \left(0, \Delta t^\star\right)\), the estimator \(\widehat{\Psi}_{\eta, N}^{\Delta t, \mathrm{sticky}}\) converges almost surely as \(N \to \infty\) to 
	\[\alpha_{R, \eta, \Delta t} = \frac{\nu_{\eta, \Delta t}\left(R\right) - \nu_{0, \Delta t}\left(R\right)}{\eta},\] 
	and the following central limit theorem holds:
	\begin{equation}
		\sqrt{N\Delta t}\left(\widehat{\Psi}_{\eta, N}^{\Delta t, \mathrm{sticky}} - \alpha_{R, \eta, \Delta t}\right) \xrightarrow[N \to \infty]{\mathrm{law}} \mathcal{N}\left(0, \sigma^2_{\mathrm{sticky}, R, \eta, \Delta t}\right),
	\end{equation}
	with asymptotic variance \(\sigma^2_{\mathrm{sticky}, R, \eta, \Delta t} \in \left(0, \infty\right)\).
\end{proposition}

The sources of error of the estimator \(\widehat{\Psi}_{\eta, N}^{\Delta t, \mathrm{sticky}}\), in roughly decreasing order of severity, are the variance, finite-time integration bias, and time-discretization bias. The following theorem shows that the finite-time integration and discretization biases are uniformly bounded as the perturbation strength \(\eta\) goes to zero. Furthermore the asymptotic variance \(\sigma^2_{\mathrm{sticky}, R, \eta, \Delta t}\) is essentially of order \(1/\eta\). 

\begin{theorem}\label{thm:sticky_bias_var}
	Fix \(\eta_\star > 0\) and suppose that Assumptions~\ref{ass:drift} and \ref{ass:ref_measure} hold. Let \(\left(\mu_{\mathrm{init}, \eta}\right)_{\eta \in \left[-\eta_\star, \eta_\star\right]}\) be a family of initial probability measures on \(\mathbb{R}^d \times \mathbb{R}^d\) such that \(\left(X_0^{\eta, \Delta t}, Y_0^{0, \Delta t}\right) \sim \mu_{\mathrm{init}, \eta}\) for any \(\eta \in \left[-\eta_\star, \eta_\star\right]\), with \(\mu_{\mathrm{init}, \eta}\left(V_c\right) \leq C\) uniformly in \(\eta \in \left[-\eta_\star, \eta_\star\right]\), and
	\begin{equation}\label{eq:intial_hypo_W_n_bound}
		\int_{\mathbb{R}^d \times \mathbb{R}^d} V_c(x,y)\mathbf{1}_{\left\{x\neq y\right\}} \mu_{\mathrm{init}, \eta}\left(dx \, dy\right) \leq C \eta.
	\end{equation}
	Then there exist \(\Delta t^\star > 0\) and  \(K > 0\) (depending on \(\eta_\star\), \(R\), and \(\mu_{\mathrm{init}, \eta}\)) such that, for any \(\Delta t \in \left(0, \Delta t^\star\right)\),
		\begin{equation}\label{eq:sticky_bias}
		\left|\mathbb{E}\left[\widehat{\Psi}_{\eta, N}^{\Delta t}\right] - \alpha_{R,\eta}\right| \leq K\left(\frac{1}{\Delta t N} + \Delta t\right).
	\end{equation}
	Moreover, for any \(n \in \mathbb{N}\), there exists \(K_n > 0\) (depending on \(\eta_\star\), \(R\), and \(\Delta t^\star\)) such that, for any \(\Delta t \in \left(0, \Delta t^\star\right)\),
	\begin{equation}\label{eq:sticky_var}
		\forall \eta \in \left[-\eta_\star, \eta_\star\right], \qquad \sigma^2_{\mathrm{sticky}, R, \eta, \Delta t} \leq \frac{K_n}{\eta}\left(1 + \frac{\Delta t^{4n}}{\eta}\right).
	\end{equation}
\end{theorem}

The proof of this result can be read in Section 4.5. Theorem~\ref{thm:sticky_bias_var} reveals that the performance of \(\widehat{\Psi}_{\eta, t}^{\mathrm{sticky}}\) improves on that of the standard estimator \(\widehat{\Phi}_{\eta, t}\) by reducing the bias and variance by a factor \(1/\eta\) compared to \eqref{eq:std_bias_bounds} and \eqref{eq:std_variance_bounds}, up to an order \(\Delta t^{4n}\) term in the variance bound appearing for technical reasons related to Lemma~\ref{lm:disc_poisson_sol_approx}. The sticky-coupled estimator is unable to achieve the uniformly bounded variance of the synchronously coupled estimator~\(\widehat{\Psi}_{\eta, t}^{\mathrm{sync}}\). It can however be used in a more general setting where the contractivity condition \eqref{eq:contractive_at_inf} need not hold everywhere.  

\subsection{Properties of Discrete-Time Sticky Coupling}
We gather in this section various results useful to prove Proposition~\ref{prop:sticky_clt} and Theorem~\ref{thm:sticky_bias_var}. To start, as the following proposition shows, the sticky coupled process is geometrically ergodic.
\begin{proposition}\label{prop:discrete_sticky_invariant_measure}
	Fix \(\eta_\star > 0\), suppose that \(b\) and \(F\) satisfy Assumption~\ref{ass:drift}, and let 
	\begin{equation}\label{eq:time_step_cond}
		\Delta t^* = \min\left\{\frac{1}{m}, \frac{m}{2\left(L_b + \eta_\star L_F\right)^2}\right\}.
	\end{equation} 
	Then, for \(\Delta t \in \left(0, \Delta t^*\right)\) and \(\eta \in \left[-\eta_\star, \eta_\star\right]\), the Markov chain \(\left\{X_k^{\eta, \Delta t}, Y_k^{0, \Delta t}\right\}_{k \in \mathbb{N}}\) admits a unique invariant probability measure \(\mu_{\eta, \Delta t}\) and is geometrically ergodic with respect to this measure. Specifically, there exist \(C, c > 0\) and \(\gamma_{\Delta t} \in (0, 1)\) such that, for any measurable function \(\phi : \mathbb{R}^d \times \mathbb{R}^d \to \mathbb{R}\) with \(\left\|\phi \right\|_{V_c} < \infty\), 
	\begin{equation}\label{eq:ref_sticky_geometric_ergodicity}
		\forall n \geq 0, \qquad \left\|\left(T^{\eta, \Delta}\right)^{n}\phi - \mu_{\eta, \Delta t}\left(\phi\right)\right\|_{V_c}  \leq C \gamma^{n}_{\Delta t} \left\|\phi - \mu_{\eta, \Delta t}\left(\phi\right)\right\|_{V_c},
	\end{equation}
	with \(V_c(x,y)\) defined in \eqref{eq:coupled_lyapunov_func}.
\end{proposition}
\begin{proof}
	The proposition follows from Harris's ergodic theorem as soon as we can find a suitable Lyapunov function and minorization measure (see \cite[Theorem 1.2]{HairerMattingly}). 
	
	Since \(T^{\eta, \Delta t}\) is a coupling of \(P^{\eta, \Delta t}\) and \(P^{0, \Delta t}\), it is clear that 
	\[T^{\eta, \Delta t}V_c(x,y) = P^{\eta, \Delta t}\widetilde{V}_c(x) + P^{0, \Delta t}\widetilde{V}_c(y).\]
	Therefore, due to \eqref{eq:marginal_lyapunov_ineq},
	\[T^{\eta, \Delta t} V_c \leq \gamma^{\Delta t}V_c + 2\Delta t A,\]
	giving us a Lyapunov function. 
	
	We define the reflection matrix \(\mathfrak{R}(x, \widetilde{x}) = \mathrm{Id} - 2\mathbf{e}(x, \widetilde{x})\mathbf{e}(x, \widetilde{x})^T\). Observe that for any two sets \(A, B \subset \mathbb{R}^d\), we have that \[\left(A\times B\right) \cap \left\{\left. (x,y) \in \mathbb{R}^d\times\mathbb{R}^d \right| x = y\right\} = \left(A\cap B\right)\times \left(A\cap B\right).\]
	Consequently, for any two measurable sets \(A, B \subset \mathbb{R}^d\),
	\begin{equation*}
		\begin{aligned}
			&T^{\eta, \Delta t}\left(\left(x, \widetilde{x}\right), A \times B \right)\\ 
			& = 
			\int_{\mathbb{R}^d} \mathbf{1}_{A\times B}\left(x + \Delta t \left(b(x) + \eta F(x)\right) + \sqrt{\frac{2\Delta t}{\beta}}z, x + \Delta t \left(b(x) + \eta F(x)\right) + \sqrt{\frac{2\Delta t}{\beta}}z \right) p_{\Delta t, \beta}\left(x, \widetilde{x}, z\right)\varphi_d(z)\,dz\\
			&+ \int_{\mathbb{R}^d} \mathbf{1}_{A\times B}\left(x + \Delta t \left(b(x) + \eta F(x)\right) + \sqrt{\frac{2\Delta t}{\beta}}z, \widetilde{x} + \Delta t b(\widetilde{x}) + \sqrt{\frac{2\Delta t}{\beta}}\mathfrak{R}\left(x, \widetilde{x}\right)z\right) \left(1 - p_{\Delta t, \beta}\left(x, \widetilde{x}, z\right)\right) \varphi_d(z)\,dz\\
			&\geq \int_{\mathbb{R}^d} \mathbf{1}_{A\cap B}\left(x + \Delta t \left(b(x) + \eta F(x)\right) + \sqrt{\frac{2\Delta t}{\beta}}z\right) p_{\Delta t, \beta}\left(x, \widetilde{x}, z\right) \varphi_d(z)\,dz \\ 
			&= \int_{\mathbb{R}^d} \mathbf{1}_{A \cap B}\left(x + \Delta t \left(b(x) + \eta F(x)\right) + \sqrt{\frac{2\Delta t}{\beta}}z\right) \min\left\{\varphi_d\left(z + \sqrt{\frac{\beta}{2\Delta t}}\mathbf{E}(x, \widetilde{x})\right), \varphi_d(z)\right\} dz.\\
		\end{aligned}
	\end{equation*}
	With the change of variable \(z' = x + \Delta t \left(b(x) + \eta F(x)\right) + \sqrt{\frac{2\Delta t}{\beta}}z\) the above line becomes:
	\begin{equation*}
		\begin{aligned}
			&\int_{A\cap B}\!\!\! \min\left\{\frac{\exp\left(-\frac{\beta}{4\Delta t}\left|z' + \widetilde{x}- 2x +\Delta t\left(b(\widetilde{x}) - 2b(x)\right) - 2\eta \Delta tF(x) \right|^2\right)}{\left(4\pi\Delta t/\beta\right)^{d/2}}, \frac{\exp\left(-\frac{\beta}{4\Delta t}\left|z' - x -\Delta t\left(b(x) + \eta F(x)\right) \right|^2\right)}{\left(4\pi\Delta t/\beta\right)^{d/2}}\right\} dz'\\
			&\geq\exp{\left(-2\beta \Delta t^\star\eta_\star^2\left\|F\right\|_\infty^2\right)}\int_{A\cap B}\!\!\! \min\left\{\frac{\exp\left(-\frac{2\beta}{4\Delta t}\left|z + \widetilde{x}- 2x +\Delta t\left(b(\widetilde{x}) - 2b(x)\right) \right|^2\right)}{\left(2\pi\right)^{d/2}\left(2\Delta t/\beta\right)^{d/2}}, \frac{\exp\left(-\frac{2\beta}{4\Delta t}\left|z - x -\Delta tb(x) \right|^2\right)}{\left(2\pi\right)^{d/2}\left(2\Delta t/\beta\right)^{d/2}}\right\} dz\\
			&\geq \left(\frac{2\Delta t^\star}{\beta}\right)^{-d/2}\exp{\left(-2\beta \Delta t^\star\eta_\star^2\left\|F\right\|_\infty^2\right)}\\
			&\qquad \times\int_{A\cap B}\!\!\! \mathbf{1}_{\left\{|z| \leq K\right\}}
			\min\left\{\varphi_d\left(\sqrt{2}\frac{z + \widetilde{x}- 2x +\Delta t\left(b(\widetilde{x}) - 2b(x)\right)}{\sqrt{2\Delta t /\beta}}\right), \varphi_d\left(\sqrt{2}\frac{z- x -\Delta t b(x)}{\sqrt{2\Delta t /\beta}}\right)\right\}dz,
		\end{aligned}
	\end{equation*}
	where for the second inequality we used the fact that \(t \mapsto \exp(-t)\) is a decreasing function on \(\mathbb{R}_+\) and the inequality \(\left(a+b\right)^2 \leq 2a^2 + 2b^2\) for any \(a, b \in \mathbb{R}\). 
	
	For \(K > 0\), denote by \(\overline{B}_K := \left\{\left. (x,y) \in \mathbb{R}^d \times \mathbb{R}^d \,\right| \, \max\left\{|x|, |y|\right\} \leq K \right\}\) the closed ball of radius~\(K\) on \(\mathbb{R}^d \times \mathbb{R}^d\) for the max product distance.  Let \(\mathcal{C}_K\) be the diagonal intersected with \(\overline{B}_K\), namely \(\mathcal{C}_K := \left\{\left. (x,y) \in \overline{B}_K \right| x = y\right\}\), and \(\xi_K\) the uniform probability measure on \(\mathcal{C}_K\). 
	The function 
	\[\left(x, \widetilde{x}, z\right) \mapsto \min\left\{\varphi_d\left(\sqrt{2}\frac{z + \widetilde{x}- 2x +\Delta t\left(b(\widetilde{x}) - 2b(x)\right)}{\sqrt{2\Delta t /\beta}}\right), \varphi_d\left(\sqrt{2}\frac{z + \widetilde{x}- x -\Delta t b(x)}{\sqrt{2\Delta t /\beta}}\right)\right\}\]
	is strictly positive on the compact set~\(\overline{B}_K\times \left\{z\in \mathbb{R}^d \left| |z| \leq K \right.\right\}\) and thereby bounded below by a strictly positive constant \(\rho_{K, \Delta t}\). Consequently, the integral is lower bounded by a strictly positive constant \(\rho_{K, \Delta t}\) times the Lebesgue measure of \(\left(A \cap B\right) \cap \left\{\left. z\in \mathbb{R}^d \right| |z| \leq K\right\}\) which is equal to \(\rho_{K, \Delta t} V_d(K) \xi_{K}\left(A \times B\right)\), where \(V_d(K) = \frac{\pi^{d/2}}{\Gamma\left(\frac{n}{2} + 1\right)}K^d\) is the Lebesgue measure of a \(d\)-dimensional ball of radius \(K\). The monotone class theorem (see for example \cite{RogersWilliams}) then implies that this lower bounded holds for any Borel set \(S \subset \mathbb{R}^d \times \mathbb{R}^d\):
	\[\inf_{(x,y)\in \overline{B}_K} T^{\eta, \Delta t}\left(\left(x,y\right), S\right) \geq \rho_{K, \Delta t}V_d(K) \xi_{K}\left(S\right). \]
	Observe that any sublevel set of \(V_c\) is contained in \(\overline{B}_K\) for~\(K\) large enough, i.e, for any \(K' > 0\), it holds \(\left\{\left. (x, y) \in \mathbb{R}^d \times \mathbb{R}^d \,\right|\, V_c(x,y) \leq K' \right\} \subset \overline{B}_K\) for a large enough \(K > 0\). The previous inequality implies that for \(K\) large enough \(\xi_K\) is a suitable minorization measure allowing us to apply Harris's theorem to deduce the claim of the proposition.
\end{proof}
\begin{remark}
In our proof, the contraction rate \(\gamma_{\Delta t}\) depends on the time step, because our minorization condition is not uniform in \(\Delta t\). We believe that one might be able to derive a uniform in \(\Delta t\) contraction rate, i.e. \(\gamma_{\Delta t} = \gamma^{\Delta t}\) for some \(\gamma \in \left(0, 1\right)\),
perhaps by analyzing the iterated kernel \(\left(T^{\eta, \Delta t}\right)^{\left\lfloor t/ \Delta t\right\rfloor}\) as in \cite{BouRabeeHairer,deBortoli,DurmusEnfroyStoltz} for example. A key technical difficulty is that, as \(\Delta t\) goes to zero, the process jumps rapidly onto and off of the diagonal. Indeed, in the continuous-time limit we expect the process to spend a positive amount of time on the diagonal but never spend any interval of time on the diagonal. 
\end{remark}

The central limit theorem for the estimator \(\widehat{\Psi}_{\eta, N}^{\Delta t, \mathrm{sticky}}\) follows from the geometric ergodicity of \(T^{\eta, \Delta t}\) and a central limit theorem for Markov chains.
\begin{proof}[Proof of Proposition \ref{prop:sticky_clt}]
	Let \(\Delta t^\star > 0\) be given by \eqref{eq:time_step_cond}. Then Proposition~\ref{prop:discrete_sticky_invariant_measure} ensurse that \(\mu_{\eta, \Delta t}\) is the unique invariant probability measure and that \(\left\|\mu_{\mathrm{init}}\left(T^{\eta, \Delta t}\right)^n - \mu_{\eta, \Delta t}\right\|_{\mathrm{TV}}\) converges to zero. The convergence in total variation is due to the fact that \(V_c\)--norm dominates the total variation norm. Therefore, by \cite[Proposition 5.2.14]{Douc_book}, the ergodicity of the dynamics with respect to \(\mu_{\eta, \Delta t}\) and the fact that \(\mu_{\mathrm{init}}\left(T^{\eta, \Delta t}\right)^n \) converges to \(\mu_{\eta, \Delta t}\) in total variation imply that the estimator \(\widehat{\Psi}_{\eta, N}^{\Delta t, \mathrm{sticky}}\) almost surely converges to
	\[\frac{1}{\eta} \int_{\mathbb{R}^d \times \mathbb{R}^d}\!\! \left(R(x) - R(y)\right)\mu_{\eta, \Delta t}\left(dx\, dy\right) = \frac{1}{\eta}\int_{\mathbb{R}^d} R \, d\nu_{\eta, \Delta t} - \frac{1}{\eta}\int_{\mathbb{R}^d} R \, d\nu_{0, \Delta t} = \alpha_{R, \eta, \Delta t}.\]
	The above equality holds because \(\mu_{\eta, \Delta t}\) is a coupling of \(\nu_{\eta, \Delta t}\) and \(\nu_{0, \Delta t}\) since \(T^{\eta, \Delta t}\) is a Markov coupling of the kernels \(P^{\eta, \Delta t}\) and \(P^{0, \Delta t}\), which admit \(\nu_{\eta, \Delta t}\) and \(\nu_{0, \Delta t}\) respectively as invariant probability measures. 
	
	Denote by \(\Pi_{\mu_{\eta, \Delta t}}\) the projection operator onto the space of functions with mean zero with respect to~\(\mu_{\eta, \Delta t}\), i.e. for \(\phi \in L^1\left(\mu_{\eta, \Delta t}\right)\),
	\[\Pi_{\mu_{\eta, \Delta t}} \phi = \phi - \int_{\mathbb{R}^d\times \mathbb{R}^d}\!\! \phi\, d\mu_{\eta, \Delta t}.\]
	Furthermore, denote by \(u(x,y) = R(x) - R(y)\).
	By \cite[Theorem 21.2.5 and Proposition 21.1.3]{Douc_book}, the estimator \(\widehat{\Psi}_{\eta, N}^{\Delta t, \mathrm{sticky}}\) satisfies a central limit theorem with asymptotic variance
	\begin{equation}\label{eq:clt_asymp_variance}
		\sigma^2_{\mathrm{sticky}, R, \eta, \Delta t} = \frac{1}{\eta^2}\mathbb{E}_{\mu_{\eta, \Delta t}}\left[\left(\widehat{u}_{\eta, \Delta t}\left(X_1^{\eta, \Delta t}, Y_1^{0, \Delta t}\right) - T^{\eta, \Delta t}\widehat{u}_{\eta, \Delta t}\left(X_0^{\eta, \Delta t}, Y_0^{0, \Delta t}\right)\right)^2\right],
	\end{equation}
	as soon as the discrete Poisson equation
	\begin{equation}\label{eq:sticky_disc_poisson_eq}
		\left(\frac{\mathrm{Id} - T^{\eta, \Delta t}}{\Delta t}\right)\widehat{u}_{\eta, \Delta t} = \Pi_{\mu_{\eta, \Delta t}}u
	\end{equation}
	admits a solution in \(L^2\left(\mu_{\eta, \Delta t}\right)\). Let \(V_c\) be the Lyapunov function defined in \eqref{eq:coupled_lyapunov_func} with \(c\) from Proposition~\ref{prop:discrete_sticky_invariant_measure}. By Proposition~\ref{prop:discrete_sticky_invariant_measure}, \(\mathrm{Id} - T^{\eta, \Delta t}\) is invertible on \(\Pi_{\mu_{\eta, \Delta t}}B_{V_c}^{\infty}\), since 
	\begin{equation}\label{eq:disc_poisson}
		\left\|\left(\mathrm{Id} - T^{\eta, \Delta t}\right)^{-1}\right\|_{\mathcal{B}\left(\Pi_{\mu_{\eta, \Delta t}}B_{V_c}^\infty\right)} \leq \sum_{k \geq 0} \left\|\left(T^{\eta, \Delta t}\right)^{k}\right\|_{\mathcal{B}\left(\Pi_{\mu_{\eta, \Delta t}}B_{V_c}^\infty\right)} \leq C \sum_{k \geq 0} \gamma^{k}_{\Delta t} = \frac{C}{1 - \gamma_{\Delta t}}.
	\end{equation}
	Since \(u \in \mathscr{S} \oplus \mathscr{S} \subset B_{V_c}^\infty\) for any \(c > 0\), it also holds that \(\widehat{u}_{\eta, \Delta t} \in B_{V_c}^\infty\). Upon replacing \(c\) with \(c/2\), it can be assumed that \(B_{V_c}^\infty \subset L^2\left(\mu_{\eta, \Delta t}\right)\), since \(\mu_{\eta, \Delta t}\) is a coupling of \(\nu_{\eta, \Delta t}\) and \(\nu_{0, \Delta t}\) and 
	\[\mu_{\eta, \Delta t}\left(V_c\right) = \nu_{\eta, \Delta t}\left(\widetilde{V}_c\right) + \nu_{0, \Delta t}\left(\widetilde{V}_c\right) < \infty. \]
	As a consequence, \(\widehat{u}_{\eta, \Delta t} \in L^2\left(\mu_{\eta, \Delta t}\right)\). Applying \cite[Theorem 21.2.5]{Douc_book} then gives the desired result for \(\mu_{\mathrm{init}, \eta} = \mu_{\eta, \Delta t}\). Applying \cite[Proposition 21.1.3]{Douc_book} lets us extend this result to any initial condition satisfying out hypotheses since \(\mu_{\mathrm{init}}\left(T^{\eta, \Delta t}\right)^n \) converges to \(\mu_{\eta, \Delta t}\) in total variation.   
\end{proof}
\begin{remark}
	We can in fact write the solution \(\widehat{u}_{\eta, \Delta t}\) to the Poisson equation \eqref{eq:sticky_disc_poisson_eq} as the difference of the solutions of the Poisson equations corresponding to the marginal processes, namely
	\begin{equation}\label{eq:sol_sticky_poisson_eq}
		\widehat{u}_{\eta, \Delta t}\left(x, y\right) = \widehat{R}_{\eta, \Delta t}(x) -\widehat{R}_{0, \Delta t}(y),
	\end{equation}
	with 
	\[\widehat{R}_{\eta, \Delta t} := \left(\frac{\mathrm{Id} - P^{\eta, \Delta t}}{\Delta t}\right)^{-1}\Pi_{\nu_{\eta, \Delta t}}R  \quad\text{and}\quad \widehat{R}_{0, \Delta t} := \left(\frac{\mathrm{Id} - P^{0, \Delta t}}{\Delta t}\right)^{-1}\Pi_{\nu_{0, \Delta t}}R.\] 
	We verify this equality by applying \(\left(\mathrm{Id} - T^{\eta, \Delta t}\right)/\Delta t\) to this proposed solution: 
	\begin{equation*}
		\begin{aligned}
			\left(\frac{\mathrm{Id} - T^{\eta, \Delta t}}{\Delta t}\right)\widehat{u}_{\eta, \Delta t}(x,y) &= \left(\frac{\mathrm{Id} - T^{\eta, \Delta t}}{\Delta t}\right)\widehat{R}_{\eta, \Delta t}(x) - \left(\frac{\mathrm{Id} - T^{\eta, \Delta t}}{\Delta t}\right)\widehat{R}_{0, \Delta t}(y)\\
			&= \left(\frac{\mathrm{Id} - P^{\eta, \Delta t}}{\Delta t}\right)\widehat{R}_{\eta, \Delta t}(x) - \left(\frac{\mathrm{Id} - P^{0, \Delta t}}{\Delta t}\right)\widehat{R}_{0, \Delta t}(y) \\
			&= \Pi_{\nu_{\eta, \Delta t}}R(x) - \Pi_{\nu_{0, \Delta t}}R(y) = \Pi_{\mu_{\eta, \Delta t}}\left(R\oplus\left(-R\right)\right)(x,y).
		\end{aligned}
	\end{equation*}
	The second equality is due to the fact that \(T^{\eta, \Delta t}\) is a coupling of the two Markov kernels \(P^{\eta, \Delta t}\) and \(P^{0, \Delta t}\), while the fourth one is due to the fact that \(\mu_{\eta, \Delta t}\) is a coupling of \(\nu_{\eta, \Delta t}\) and \(\nu_{0, \Delta t}\). We can then conclude that~\(\widehat{u}_{\eta, \Delta t}\) is indeed the unique solution by uniqueness of solutions to the discrete Poisson equation \eqref{eq:sticky_disc_poisson_eq} in~\(\Pi_{\mu_{\eta, \Delta t}}B^{\infty}_{V_c}\). 
\end{remark}

We can control the amount of time the sticky coupled process spends off the diagonal using the following proposition which is a straightforward consequence of the results of \cite[Section 2]{Durmus_etal}.
\begin{proposition}\label{prop:discrete_Wn_norm_bound}
	Fix \(\eta_\star > 0\), assume that \(b\) and \(F\) satisfy Assumption~\ref{ass:drift} and let \(\Delta t^\star > 0\) be given by \eqref{eq:time_step_cond}. Let \(\mu_{\mathrm{init}}\) be some probability measure on \(\mathbb{R}^d\times \mathbb{R}^d\) such that \(\left(X_0^{\eta, \Delta t}, Y_0^{0, \Delta t}\right) \sim \mu_{\mathrm{init}}\). Then, there exists \(C > 0\) such that for any measurable non-negative function \(f:\mathbb{R}^d\times \mathbb{R}^d \to \mathbb{R}_+\) of the form \(f(x, y) = \widetilde{f}(x) + \widetilde{f}(y)\) with \(\widetilde{f}: \mathbb{R}^d \to \mathbb{R}_+\) a non-negative measurable function, and for any \(\eta \in \left[-\eta_\star, \eta_\star\right]\), \(\Delta t \in \left(0, \Delta t^* \right)\) and \(k \in \mathbb{N}\),
	\begin{equation}\label{eq:finitetime_discrete_weightedtv_bound}
		\begin{aligned}
			\mathbb{E}_{\mu_{\mathrm{init}}}\left[\mathbf{1}_{\left\{X_k^{\eta, \Delta t}\neq Y_k^{0, \Delta t}\right\}}f\left(X_k^{\eta, \Delta t}, Y_k^{0, \Delta t}\right)\right] \leq  \left[C\eta + \mu_{\mathrm{init}}\left(\mathbf{1}_{\left\{x\neq y\right\}}\right)\right]\mathbb{E}_{\mu_{\mathrm{init}}}\left[f\left(X_k^{\eta, \Delta t}, Y_k^{0, \Delta t}\right)\right].
		\end{aligned}
	\end{equation}
	Moreover, for the invariant probability measure, it holds
	\begin{equation}\label{eq:stationary_discrete_weightedtv_bound}
		\int_{\mathbb{R}^d \times \mathbb{R}^d} f\left(x,y\right)\mathbf{1}_{\left\{x\neq y\right\}} \mu_{\eta, \Delta t}\left(dx\,dy\right) \leq C\eta\left[\nu_{\eta, \Delta t}\left(\widetilde{f}\right) + \nu_{0,\Delta t}\left(\widetilde{f}\right)\right].
	\end{equation}
\end{proposition}
\begin{proof}
	Following \cite{Durmus_etal}, we prove \eqref{eq:finitetime_discrete_weightedtv_bound} by constructing a Markov chain \(\left(W_k^{\eta, \Delta t}\right)_{k \in \mathbb{N}}\) on \(\mathbb{R}_+\) such that, almost surely for all \(k \in \mathbb{N}\), it holds \(\left|X_k^{\eta, \Delta t} - Y_k^{0, \Delta t}\right| \leq W_k^{\eta, \Delta t}\). Then the event \(\left\{X_k^{\eta, \Delta t} \neq Y_k^{0, \Delta t}\right\}\) is a subset of \(\left\{W_k^{\eta , \Delta t} > 0\right\}\) modulo null sets. 
	
	To this end, we use \cite[Proposition 6]{Durmus_etal} to write
	\begin{equation}
		\left|X_{k+1}^{\eta, \Delta t} - Y_{k+1}^{0, \Delta t}\right| \leq \mathscr{G}_{\Delta t}\left(\left|X_k^{\eta, \Delta t} - Y_k^{0, \Delta t}\right|, \mathcal{G}_{k+1}, U_{k+1}\right),
	\end{equation}
	where \(\left(\mathcal{G}_k\right)_{k \geq 1}\) is defined by
	\begin{equation}
		\mathcal{G}_k = \left\langle \mathbf{e}_{k-1}, G_k \right\rangle. 
	\end{equation} 
	and \(\mathscr{G}_{\Delta t}\) by
	\begin{equation}
		\mathscr{G}_{\Delta t}\left(w, g, u\right) = \left\{
		\begin{aligned}
			&\tau_{0, \Delta t}(w) + \eta \left\|F\right\|_\infty \Delta t - 2\sqrt{\frac{2\Delta t}{\beta}}g && \text{ if } u \geq \widebar{p}_{\Delta t, \beta}\left(\tau_{0, \Delta t}(w) + \eta\left\|F\right\|_\infty \Delta t, g\right),\\
			&0  && \text{ otherwise},
		\end{aligned}\right.
	\end{equation}
	for \(w \in \left[0, \infty\right)\), \(g \in \mathbb{R}\), and \(u \in \left[0, 1\right]\) with \(\tau_{0, \Delta t}\) defined in \eqref{eq:discrete_contraction}
	\begin{equation}
		\bar{p}_{\Delta t, \beta}\left(a,g\right) = \min\left\{1, \frac{\varphi_1\left(a\sqrt{\frac{\beta}{2\Delta t}} - g\right)}{\varphi_1\left(g\right)}\right\},
	\end{equation}
	where \(\varphi_1\) is the density of a one-dimensional standard Gaussian distribution. 
	Furthermore, for any \(g\in \mathbb{R}\) and \(u \in \left[0, 1\right]\), the function \(w \mapsto \mathscr{G}_{\Delta t}\left(w, g, u\right)\) is non-decreasing, 
	
	We construct the Markov chain \(\left(W_k^{\eta, \Delta t}\right)\) as in \cite{Durmus_etal} by setting \(W_0^{\eta, \Delta t} = \left|X_0^{\eta, \Delta t} - Y_0^{0, \Delta t}\right|\) and, for \(k \geq 0\),
	\begin{equation}\label{eq:bounding_chain}
		W_{k+1}^{\eta, \Delta t} = \mathscr{G}_{\Delta t}\left(W_k^{\eta, \Delta t}, \mathcal{G}_{k+1}, U_{k+1}\right).
	\end{equation}
	We denote the Markov kernel of this chain as \(Q^{\eta, \Delta t}\) and its stationary measure as \(\rho_{\eta, \Delta t}\); the fact that \(w \mapsto \mathscr{G}_{\Delta t}\left(w, g, u\right)\) is non-decreasing ensures that this kernel is stochastically monotone and that we can bound \(\left|X_{k}^{\eta, \Delta t} - Y_{k}^{0, \Delta t}\right|\) by \(W_k^{\eta, \Delta t}\) for each \(k \geq 0\). 
	
	We can write the left-hand side of \eqref{eq:finitetime_discrete_weightedtv_bound} as 
	\begin{equation}\label{eq:finite_disc_weightedtv_bound_2}
		\begin{aligned}
		&\mathbb{E}_{\mu_{\mathrm{init}}}\left[\mathbf{1}_{\left\{X_k^{\eta, \Delta t} \neq Y_k^{0, \Delta t}\right\}}f\left(X_k^{\eta, \Delta t},Y_k^{0, \Delta t}\right)\right] \\
		& \qquad = \mathbb{P}_{\mu_{\mathrm{init}}}\left(X_k^{\eta, \Delta t} \neq Y_k^{0, \Delta t}\right)\mathbb{E}_{\mu_{\mathrm{init}}}\left[\widetilde{f}\left(X_k^{\eta, \Delta t}\right) + \widetilde{f}\left(Y_k^{0,\Delta t}\right)\left| X_k^{\eta, \Delta t} \neq Y_k^{0, \Delta t}\right. \right]\\
		& \qquad = \mathbb{P}_{\mu_{\mathrm{init}}}\left(X_k^{\eta, \Delta t} \neq Y_k^{0, \Delta t}\right)\left( \mathbb{E}_{\mu_{\mathrm{init}}}\left[\widetilde{f}\left(X_k^{\eta, \Delta t}\right)\right] +  \mathbb{E}_{\mu_{\mathrm{init}}}\left[\widetilde{f}\left(Y_k^{0, \Delta t}\right)\right]\right).
		\end{aligned}
	\end{equation}
	The second equality follows from the fact that the event \(\left\{X_k^{\eta, \Delta t} \neq Y_k^{0, \Delta t}\right\}\) is equivalent up to a negligible set to the event that the noise was reflectively coupled at the \(k\)-step. As a consequence, letting 
	\[A = Y_{k-1}^{0, \Delta t} + \Delta t b\left(Y_{k-1}^{0, \Delta t}\right) + \sqrt{\frac{2\Delta t}{\beta}}\left[\mathrm{Id} - 2\mathbf{e}_{k-1}\mathbf{e}_{k-1}^T\right]G_k, \]
	we have
	\begin{equation*}
		\begin{aligned}
			\mathbb{E}_{\mu_{\mathrm{init}}}\left[\widetilde{f}\left(X_k^{\eta, \Delta t}\right) + \widetilde{f}\left(Y_k^{0,\Delta t}\right)\left| X_k^{\eta, \Delta t} \neq Y_k^{0, \Delta t}\right. \right] 
			&= \mathbb{E}_{\mu_{\mathrm{init}}}\left[\widetilde{f}\left(X_k^{\eta, \Delta t}\right) + \widetilde{f}\left(A\right)\right]\\
			&=\mathbb{E}_{\mu_{\mathrm{init}}}\left[\widetilde{f}\left(X_k^{\eta, \Delta t}\right)\right] + \mathbb{E}_{\mu_{\mathrm{init}}}\left[\widetilde{f}\left(A\right)\right]\\
			&=\mathbb{E}_{\mu_{\mathrm{init}}}\left[\widetilde{f}\left(X_k^{\eta, \Delta t}\right)\right] + \mathbb{E}_{\mu_{\mathrm{init}}}\left[\widetilde{f}\left(Y_k^{0, \Delta t}\right)\right],
		\end{aligned}
	\end{equation*}
	where the third equality follows from the fact that \(A\) and \(Y_k^{0, \Delta t}\) have the same law because reflecting a standard normal random variable does not change its law.
	
	We then use \(\left\{W_k^{\eta, \Delta t}\right\}_{k \in \mathbb{N}}\) to bound the right-hand side of \eqref{eq:finite_disc_weightedtv_bound_2} (and thereby left-hand side of \eqref{eq:finitetime_discrete_weightedtv_bound}) as
	\begin{equation}\label{eq:finite_disc_weightedtv_bound_3}
		\begin{aligned}
			\mathbb{E}_{\mu_{\mathrm{init}}}\left[\mathbf{1}_{\left\{X_k^{\eta, \Delta t} \neq Y_k^{0, \Delta t}\right\}}f\left(X_k^{\eta, \Delta t}, Y_k^{0, \Delta t}\right)\right] \leq \mathbb{P}_{\mu_{\mathrm{init}}}\left(W_k^{\eta, \Delta t} > 0 \right) \mathbb{E}_{\mu_{\mathrm{init}}}\left[f\left(X_k^{\eta, \Delta t}, Y_k^{0, \Delta t}\right)\right]&\\
			\leq \left(\delta_0 \left(Q^{\eta, \Delta t}\right)^k \left((0, \infty)\right) + \mu_{\mathrm{init}}\left(\mathbf{1}_{\{x \neq y\}}\right)\right)\mathbb{E}_{\mu_{\mathrm{init}}}\left[f\left(X_k^{\eta, \Delta t}, Y_k^{0, \Delta t}\right)\right],&
		\end{aligned}
	\end{equation}
	where the first the inequality is due to the fact that \(\left\{X_k^{\eta, \Delta t} \neq Y_k^{0, \Delta t}\right\} \subset \left\{W_k^{\eta, \Delta t} > 0\right\}\). The second inequality is due to 
	\begin{equation*}
		\begin{aligned}
			\mathbb{P}_{\mu_{\mathrm{init}}}\left(W_k^{\eta, \Delta t} > 0 \right) &= \mathbb{P}_{\mu_{\mathrm{init}}}\left(W_k^{\eta, \Delta t} > 0 \left| W_0^{\eta, \Delta t} = 0\right. \right)\mathbb{P}_{\mu_{\mathrm{init}}}\left(W_0^{\eta, \Delta t} = 0\right) \\
			&\qquad + \mathbb{P}_{\mu_{\mathrm{init}}}\left(W_k^{\eta, \Delta t} > 0 \left| W_0^{\eta, \Delta t} > 0\right. \right)\mathbb{P}_{\mu_{\mathrm{init}}}\left(W_0^{\eta, \Delta t} > 0\right)\\
			&\leq \mathbb{P}_{\mu_{\mathrm{init}}}\left(W_k^{\eta, \Delta t} > 0 \left| W_0^{\eta, \Delta t} = 0\right. \right) +\mathbb{P}_{\mu_{\mathrm{init}}}\left(W_0^{\eta, \Delta t} > 0\right)\\
			&= \delta_0 \left(Q^{\eta, \Delta t}\right)^k \left((0, \infty)\right) + \mu_{\mathrm{init}}\left(\mathbf{1}_{\{x \neq y\}}\right).
		\end{aligned}
	\end{equation*}
	We show below in Lemma~\ref{lm:W_from_zero} below that the difference in total variation norm between \(\delta_0 \left(Q^{\eta, \Delta t}\right)^k \) and~\(\rho_{\eta, \Delta t}\) is bounded by \(\rho_{\eta, \Delta t}\left((0, \infty)\right)\) uniformly in \(k\). Thus \(\delta_0 \left(Q^{\eta, \Delta t}\right)^k \left((0, \infty)\right) \leq 2 \rho_{\eta, \Delta t}\left((0, \infty)\right)\).
	By \cite[Theorem~11]{Durmus_etal}, we can bound \(\rho_{\eta, \Delta t}\left((0, \infty)\right)\) by \(C\eta\) where the authors give an explicit constant \(C\) that is independent of \(\Delta t\) and \(\eta\). Inserting this bound into the right hand side of the inequality \eqref{eq:finite_disc_weightedtv_bound_3} leads to~\eqref{eq:finitetime_discrete_weightedtv_bound}.
	
	If we choose as our initial probability measure, the invariant measure for the coupled process \(\mu_{\eta, \Delta t}\), the first inequality in \eqref{eq:finite_disc_weightedtv_bound_3}, namely
	\[\mathbb{E}_{\mu_{\mathrm{init}}}\left[\mathbf{1}_{\left\{X_k^{\eta, \Delta t} \neq Y_k^{0, \Delta t}\right\}}f\left(X_k^{\eta, \Delta t}, Y_k^{0, \Delta t}\right)\right] \leq \mathbb{P}_{\mu_{\mathrm{init}}}\left(W_k^{\eta, \Delta t} > 0 \right) \mathbb{E}_{\mu_{\mathrm{init}}}\left[f\left(X_k^{\eta, \Delta t}, Y_k^{0, \Delta t}\right)\right],\]
	becomes
	\[\int_{\mathbb{R}^d \times \mathbb{R}^d} f\left(x,y\right)\mathbf{1}_{\left\{x\neq y\right\}} \mu_{\eta, \Delta t}\left(dx\,dy\right) \leq \mathbb{P}_{\mu_{\eta, \Delta t}}\left(W_k^{\eta, \Delta t} > 0 \right) \int_{\mathbb{R}^d \times \mathbb{R}^d} f\left(x,y\right) \mu_{\eta, \Delta t}\left(dx\,dy\right). \]
	Since \(Q^{\eta, \Delta t}\) is geometrically ergodic \cite[Proposition 9]{Durmus_etal}, we can let \(k\) tend to infinity. The probability \(\mathbb{P}_{\mu_{\eta, \Delta t}}\left(W_k^{\eta, \Delta t} > 0 \right)\) converges to \(\rho_{\eta, \Delta t}\left((0, \infty)\right)\). Applying \cite[Theorem 11]{Durmus_etal} again, we obtain
	\[\int_{\mathbb{R}^d \times \mathbb{R}^d} f\left(x,y\right)\mathbf{1}_{\left\{x\neq y\right\}} \mu_{\eta, \Delta t}\left(dx\,dy\right) \leq C\eta \int_{\mathbb{R}^d \times \mathbb{R}^d} f\left(x,y\right) \mu_{\eta, \Delta t}\left(dx\,dy\right),\]
	from which \eqref{eq:stationary_discrete_weightedtv_bound} follows since \(\mu_{\eta, \Delta t}\) is a coupling of \(\nu_{\eta, \Delta t}\) and \(\nu_{0, \Delta t}\). 
\end{proof}
We end this section with proof of the estimate for \(\left\|\delta_0 \left(Q^{\eta, \Delta t}\right)^k - \rho_{\eta, \Delta t}\right\|_{\mathrm{TV}}\) we used above.
\begin{lemma}\label{lm:W_from_zero}
	Fix \(\eta_\star > 0 \) and \(\Delta t^\star > 0\) as in \eqref{eq:time_step_cond} and suppose that the assumptions of Proposition~\ref{prop:discrete_Wn_norm_bound} hold. Let \(\left\{W^{\eta, \Delta t}_k\right\}_{k \in \mathbb{N}}\) be the bounding Markov chain defined in \eqref{eq:bounding_chain} with \(Q^{\eta, \Delta t}\) and \(\rho_{\eta, \Delta t}\) its Markov transition kernel and invariant probability measure, respectively.
	Then, for any \(\eta \in \left[-\eta_\star, \eta_\star\right]\), \(\Delta t \in \left(0, \Delta t^* \right)\) and~\(k \in \mathbb{N}\),
	\begin{equation}\label{eq:W_from_zero}
		\left\|\delta_0 \left(Q^{\eta, \Delta t}\right)^k - \rho_{\eta, \Delta t}\right\|_{\mathrm{TV}} \leq \rho_{\eta, \Delta t}\left((0, \infty)\right). 
	\end{equation}
\end{lemma}
\begin{proof}
	To prove the desired bound, we synchronously couple two versions of the bounding process, \(\left\{W_k^{\eta, \Delta t}\right\}_{k \in \mathbb{N}}\) and \(\left\{\widecheck{W}_k^{\eta, \Delta t}\right\}_{k \in \mathbb{N}}\), by driving them with the same sequences of uniform random variables \(\left\{U_k\right\}_{k \geq 1}\) and of Gaussian random variables \(\left\{\mathcal{G}_{k}\right\}_{k\geq 1}\). We start \(W^{\eta, \Delta t}\) at zero, i.e. \(W^{\eta, \Delta t}_0 = 0\), and \(\widecheck{W}^{\eta, \Delta t}\) at its invariant measure, \(\widecheck{W}^{\eta, \Delta t}_0 \sim \rho_{\eta, \Delta t}\). 
	
	As we are synchronously coupling the two chains, if they meet at some time \(i_0 \in \mathbb{N}\), i.e. \(W_{i_0}^{\eta, \Delta t} = \widecheck{W}_{i_0}^{\eta, \Delta t} \), they remain together for all times after, i.e. \(W_\ell^{\eta, \Delta t} = \widecheck{W}_\ell^{\eta, \Delta t}\) for all \(\ell \geq i_0\). Furthermore. since \(Q^{\eta, \Delta t}\) is stochastically monotone, \(W_k^{\eta, \Delta t} \leq \widecheck{W}_k^{\eta, \Delta t}\) for all \(k\) almost surely. Consequently, if \(\widecheck{W}_i^{\eta, \Delta t} = 0 \) for some time \(i \in \mathbb{N}\) then \(W_i^{\eta, \Delta t} = \widecheck{W}_i^{\eta, \Delta t} = 0\). Thus, we can upper bound the probability that the two chains have not met by the probability that \(\widecheck{W}_k^{\eta, \Delta t}\) has not yet hit zero:
	\[\left\|\delta_0 \left(Q^{\eta, \Delta t}\right)^k - \rho_{\eta, \Delta t}\right\|_{\mathrm{TV}} \leq \mathbb{P}\left(W_k^{\eta, \Delta t} \neq \widecheck{W}^{\eta, \Delta t}_k\right) \leq \mathbb{P}\left(\min_{0\leq i \leq k} \widecheck{W}^{\eta, \Delta t}_i > 0\right).\]
	By \cite[Lemma 35]{Durmus_etal}, we can bound the right hand side by
	\[\mathbb{P}\left(\min_{0\leq i \leq k} \widecheck{W}^{\eta, \Delta t}_i > 0\right) \leq \int_{\left(0, \infty\right)}\left[1 - 2\Phi\left(-\frac{\tau_{0, \Delta t}(w) + \alpha_k}{2\beta_k}\right)\right]\rho_{\eta, \Delta t}\left(dw\right) \leq \rho_{\eta, \Delta t}\left((0, \infty)\right),\]
	where \(\Phi\) is the cumulative distribution function of the standard one-dimensional Gaussian distribution and \(\alpha_k\) and \(\beta_k\) are explicit constants given in \cite[Equation (44)]{Durmus_etal}. The last inequality is due to the fact that the integrand is bounded by \(1\). This finally gives \eqref{eq:W_from_zero}.
\end{proof}

\subsection{Proof of Theorem \ref{thm:sticky_bias_var}}
We start with an analog of inequality~\eqref{eq:pseudo_lip_cond} in Lemma~\ref{lm:pseudo_lip} that works conveniently with sticky coupling. We will use this inequality repeatedly in the proof and highlight it as a lemma for the sake of clarity.
\begin{lemma}\label{lm:W_n_bound}
	Let \(V: \mathbb{R}^d \to \left[1, \infty\right)\) be a measurable function. Then, for any \(\varphi \in B_V^\infty\),
	\begin{equation}\label{eq:W_n_bound}
		\forall x, y\in \mathbb{R}^d, \qquad \left|\varphi(x) - \varphi(y)\right| \leq \left\|\varphi\right\|_{V}\left(V(x) + V(y)\right)\mathbf{1}_{\left\{x\neq y\right\}}.
	\end{equation}
\end{lemma}
\begin{proof}
	By the triangle inequality and since \(\varphi \in B_V^\infty\), we have
	\begin{equation*}
		\begin{aligned}
			\left|\varphi(x) - \varphi(y)\right| &= \left|\varphi(x) - \varphi(y)\right|\mathbf{1}_{\left\{x\neq y\right\}} \leq \left(\left|\varphi(x)\right| + \left|\varphi(y)\right|\right)\mathbf{1}_{\left\{x \neq y\right\}} \leq \left\|\varphi\right\|_V\left(V(x) + V(y)\right)\mathbf{1}_{\left\{x\neq y\right\}},
		\end{aligned}
	\end{equation*}
	which gives the desired inequality.
\end{proof}
\subsubsection{Control of the bias}
	We adapt the proof of \cite[Proposition 5]{Plechac} to control the bias of our estimator. As in the proof of Proposition~\ref{prop:std_estimator}, denote by \(\widetilde{R}_\eta \in \mathscr{S}_\eta\) and \(\widetilde{R}_0 \in \mathscr{S}_0\) the solutions of the continuous time Poisson equations
	\begin{equation}\label{eq:cont_poisson_eq}
		-\mathcal{L}_\eta \widetilde{R}_\eta = \Pi_\eta R, \qquad -\mathcal{L}_0 \widetilde{R}_0 = \Pi_0 R.
	\end{equation}
	Since \(\widetilde{R}_\eta, \widetilde{R}_0 \in \mathscr{S}\), the functions \(\widetilde{R}_\eta, \widetilde{R}_0\) and all their derivatives belong to \(B^\infty_{\widetilde{V}_c}\) as \(\mathcal{K}_n \leq c_n \widetilde{V}_c\) for all \(n \in \mathbb{N}\).
	A Taylor expansion gives
	\begin{equation}\label{eq:bias_development_taylor}
		\begin{gathered}
			\mathbb{E}_{\mu_{\mathrm{init}, \eta}}\left[\widetilde{R}_\eta\left(X_{k+1}^{\eta, \Delta t}\right) - \widetilde{R}_0\left(Y_{k+1}^{0, \Delta t}\right)\right] = \mathbb{E}_{\mu_{\mathrm{init}, \eta}}\left[\widetilde{R}_\eta\left(X_{k}^{\eta, \Delta t}\right) - \widetilde{R}_0\left(Y_{k}^{0, \Delta t}\right)\right] \\
			+ \Delta t \, \mathbb{E}_{\mu_{\mathrm{init}, \eta}}\left[\mathcal{L}_\eta\widetilde{R}_\eta\left(X_k^{\eta, \Delta t}\right) - \mathcal{L}_0 \widetilde{R}_0\left(Y_k^{0, \Delta t}\right)\right] + \mathscr{E}_k^{\eta, \Delta t}.
		\end{gathered}
	\end{equation}
	We show below that the error term \(\mathscr{E}_k^{\eta, \Delta t}\) is of order at most \(\eta \Delta t^2\) uniformly in \(k\). Admitting this fact for the moment, we sum \eqref{eq:bias_development_taylor} for \(k = 0, 1, \dots, N-1\) and divide by \(\frac{1}{\eta N\Delta t}\). After rearranging, we then obtain
	\begin{equation*}
		\begin{aligned}
			&\mathbb{E}_{\mu_{\mathrm{init}, \eta}}\left[\frac{1}{\eta N}\sum_{k = 0}^{N-1} \left(\mathcal{L}_{\eta}\widetilde{R}_\eta\left(X_k^{\eta, \Delta t}\right) - \mathcal{L}_0\widetilde{R}_0\left(Y_k^{0, \Delta t}\right)\right)\right]\\
			&= \frac{1}{\eta N \Delta t}\mathbb{E}_{\mu_{\mathrm{init}, \eta}}\left[\widetilde{R}_{\eta}\left(X_{N}^{\eta, \Delta t}\right) - \widetilde{R}_{0}\left(Y_N^{0, \Delta t}\right) - \left(\widetilde{R}_{\eta}\left(X_0^{\eta, \Delta t}\right) - \widetilde{R}_0\left(Y_0^{0, \Delta t}\right)\right)\right] + \mathrm{O}\left(\Delta t\right).
		\end{aligned}
	\end{equation*}
	Using the fact that \(\widetilde{R}_\eta\) and \(\widetilde{R}_0\) are solutions to the Poisson equations \eqref{eq:cont_poisson_eq}, we can rewrite the above equality as
	\[\mathbb{E}_{\mu_{\mathrm{init}, \eta}}\left[\widehat{\Psi}_{\eta, N}^{\Delta t, \mathrm{sticky}}\right] - \alpha_{R, \eta} = \frac{1}{\eta \Delta t N}\mathbb{E}_{\mu_{\mathrm{init}, \eta}}\left[\widetilde{R}_{\eta}\left(X_{0}^{\eta, \Delta t}\right) - \widetilde{R}_{0}\left(Y_0^{0, \Delta t}\right) - \left(\widetilde{R}_{\eta}\left(X_N^{\eta, \Delta t}\right) - \widetilde{R}_0\left(Y_N^{0, \Delta t}\right)\right)\right] + \mathrm{O}\left(\Delta t\right).\]
	Taking absolute values gives
	\begin{equation}\label{eq:bias_bound}
		\begin{aligned}
			&\left|\mathbb{E}_{\mu_{\mathrm{init}, \eta}}\left[\widehat{\Psi}_{\eta, N}^{\Delta t, \mathrm{sticky}}\right] - \alpha_{R, \eta}\right|\\
			&\leq \frac{1}{\Delta t N}\mathbb{E}_{\mu_{\mathrm{init}, \eta}}\left[\left|\frac{\widetilde{R}_{\eta}\left(X_{0}^{\eta, \Delta t}\right) - \widetilde{R}_{0}\left(Y_0^{0, \Delta t}\right)}{\eta}\right|+ \left|\frac{\widetilde{R}_{\eta}\left(X_N^{\eta, \Delta t}\right) - \widetilde{R}_0\left(Y_N^{0, \Delta t}\right)}{\eta}\right|\right] + \mathrm{O}\left(\Delta t\right).
		\end{aligned}
	\end{equation}
	For the first term in the expectation on the right hand side we have
	\[\mathbb{E}_{\mu_{\mathrm{init}, \eta}}\left[\left|\frac{\widetilde{R}_\eta\left(X_0^{\eta, \Delta t}\right) -\widetilde{R}_0\left(Y_0^{0, \Delta t}\right)}{\eta}\right|\right] \leq \mathbb{E}_{\mu_{\mathrm{init}, \eta}}\left[\left|\frac{\widetilde{R}_\eta\left(X_0^{\eta, \Delta t}\right) -\widetilde{R}_0\left(X_0^{\eta, \Delta t}\right)}{\eta}\right| + \left|\frac{\widetilde{R}_0\left(X_0^{\eta, \Delta t}\right) -\widetilde{R}_0\left(Y_0^{0, \Delta t}\right)}{\eta}\right|\right].\]
	The function \(\left|\widetilde{R}_\eta\left(X_0^{\eta, \Delta t}\right) -\widetilde{R}_0\left(X_0^{\eta, \Delta t}\right)\right|\) in the first term on the right hand side is bounded by \(C\eta \mathcal{K}_n\) for \(n \in \mathbb{N}\) large enough by Lemma~\ref{lm:poisson_sol_approx}; as a result the first term is uniformly bounded in \(\eta\) due to the assumptions on \(\mu_{\mathrm{init}, \eta}\). Using \eqref{eq:W_n_bound} in Lemma~\ref{lm:W_n_bound}, we can bound the second term as 
	\[\mathbb{E}_{\mu_{\mathrm{init}, \eta}}\left[\left|\frac{\widetilde{R}_0\left(X_0^{\eta, \Delta t}\right) -\widetilde{R}_0\left(Y_0^{0, \Delta t}\right)}{\eta}\right|\right] \leq \frac{C}{\eta}\mathbb{E}_{\mu_{\mathrm{init}, \eta}}\left[\mathbf{1}_{\left\{X_0^{\eta, \Delta t} \neq Y_0^{0, \Delta t}\right\}}\left(\mathcal{K}_n\left(X_0^{\eta, \Delta t}\right) + \mathcal{K}_n\left(Y_0^{0, \Delta t}\right)\right)\right],\]
	with the right hand side uniformly bounded in \(\eta\) due to the assumption \eqref{eq:intial_hypo_W_n_bound} on \(\mu_{\mathrm{init}, \eta}\). For the second term in the expectation in \eqref{eq:bias_bound}, the argument is similar; we have 
	\[\mathbb{E}_{\mu_{\mathrm{init}, \eta}}\left[\left|\frac{\widetilde{R}_\eta \left(X_N^{\eta, \Delta t}\right) - \widetilde{R}_0\left(Y_N^{\eta, \Delta t}\right)}{\eta}\right|\right] \leq \mathbb{E}_{\mu_{\mathrm{init}, \eta}}\left[\left|\frac{\widetilde{R}_\eta\left(X_N^{\eta, \Delta t}\right) - \widetilde{R}_0\left(X_N^{\eta, \Delta t }\right)}{\eta}\right| + \left|\frac{\widetilde{R}_0\left(X_N^{\eta, \Delta t}\right) - \widetilde{R}_0\left(Y_N^{0, \Delta t}\right)}{\eta}\right|\right].\]
	The first term is bounded in the same way as before using Lemma~\ref{lm:poisson_sol_approx} and \eqref{eq:disc_semigroup_estimates}. For the second, we use~\eqref{eq:W_n_bound} in Lemma~\ref{lm:W_n_bound} to get
	\[\mathbb{E}_{\mu_{\mathrm{init}, \eta}}\left[\left|\frac{\widetilde{R}_0\left(X_N^{\eta, \Delta t}\right) - \widetilde{R}_0\left(Y_N^{0, \Delta t}\right)}{\eta}\right|\right] \leq \frac{C}{\eta}\mathbb{E}_{\mu_{\mathrm{init}, \eta}}\left[\mathbf{1}_{\left\{X_N^{\eta, \Delta t} \neq Y_N^{0, \Delta t}\right\}}\left(\mathcal{K}_n\left(X_N^{\eta, \Delta t}\right) + \mathcal{K}_n\left(Y_N^{0, \Delta t}\right)\right)\right].\]
	Proposition \ref{prop:discrete_Wn_norm_bound} implies that the expectation on the right hand side is of order \(\eta\), so the entire right hand side is uniformly bounded in \(\eta\). Therefore putting these controls together, we get that there exists a constant \(K_1 > 0\) such that
	\[\left|\mathbb{E}_{\mu_{\mathrm{init}, \eta}}\left[\widehat{\Psi}_{\eta, N}^{\Delta t, \mathrm{sticky}}\right] - \alpha_{R, \eta}\right| \leq K_1 \left(\frac{1}{\Delta t N} + \Delta t\right),\]
	proving the desired result for the bias.

	\paragraph{Controlling the error term in \eqref{eq:bias_development_taylor}.} To conclude the proof of the estimate on the bias, it remains to show that the error term \(\mathscr{E}^{\eta, \Delta t}_k\) is uniformly bounded by \(C\eta\Delta t^2\). To lighten the notation in what follows we write \(b^{k, \eta} := b\left(X_k^{\eta, \Delta t}\right)\), \(b^{k, 0} := b\left(Y_k^{0, \Delta t}\right)\), and \(F^{k, \eta} := F\left(X_k^{\eta, \Delta t}\right)\) and for the \(i\)-th component of these vector-valued functions we write \(b^{k, \eta}_i\), \(b^{k, 0}_i\), and \(F^{k,\eta}_i\). We further denote by \(\Xi^{\eta, k} := \left(b\left(X_k^{\eta, \Delta t}\right) + \eta F\left(X_k^{\eta, \Delta t}\right)\right)\Delta t + \sqrt{\frac{2\Delta t}{\beta}}G_{k+1}\), which can be viewed as the function \(\Xi^{\eta}(x, g) = \left(b(x) + \eta F(x)\right)\Delta t + \sqrt{\frac{2\Delta t}{\beta}}g\) evaluated at \(x = X_k^{\eta, \Delta t}\) and \(g = G_{k+1}\); in particular, \(\Xi^{0, k} :=  \Xi^{0}\left(X_k^{\eta, \Delta t}, G_{k+1}\right) = b\left(X_k^{\eta, \Delta t}\right)\Delta t + \sqrt{\frac{2\Delta t}{\beta}}G_{k+1}\) still depends on \(\eta\) through the argument of the function \(X_k^{\eta, \Delta t}\). This definition is motivated by the fact that \(X_{k+1}^{\eta, \Delta t} = X_k^{\eta, \Delta t} + \Xi^{\eta, k}\) and in particular \(\Xi^{0, k}\) is the update of \(X_k^{\eta, \Delta t}\) without the perturbation. We similarly define the update of the second marginal \(Y_k^{0, \Delta t}\) as \(\widetilde{\Xi}^{0, k} := \Xi\left(Y_k^{0, \Delta t}, \widetilde{G}_{k+1}\right) =  b\left(Y_k^{0, \Delta t}\right)\Delta t + \sqrt{\frac{2\Delta t}{\beta}}\widetilde{G}_{k+1}\), with 
	\[\widetilde{G}_{k+1} = \left\{
	\begin{aligned}
		&\sqrt{\frac{\beta}{2\Delta t}}\left[X_{k+1}^{\eta, \Delta t} - Y_k^{0, \Delta t} -\Delta t b\left(Y_k^{0, \Delta t}\right)\right] && \text{ if } U_{k+1} \leq p_{\Delta t, \beta}\left(X_k^{\eta, \Delta t}, Y_k^{0, \Delta t}, G_{k+1}\right),\\
		&\left[\mathrm{Id} - 2\mathbf{e}_k\mathbf{e}_k^T\right]G_{k+1} && \text{ otherwise. }
	\end{aligned}
	\right. \] Note that the random variable \(\widetilde{G}_{k+1}\) has the law of a standard \(d\)-dimensional Gaussian distribution. The error term is then given by
	\begin{equation}\label{eq:bias_error_terms}
		\begin{gathered}
			\mathscr{E}_k^{\eta, \Delta t} = \frac{\Delta t^2}{2} \mathbb{E}_{\mu_{\mathrm{init}, \eta}}\left(\nabla^2 \widetilde{R}_\eta\left(X_k^{\eta, \Delta t}\right)\left[b^{k, \eta} + \eta F^{k, \eta}, b^{k, \eta} + \eta F^{k, \eta}\right] - \nabla^2\widetilde{R}_0\left(Y_k^{0, \Delta t}\right)\left[b^{k, 0}, b^{k, 0}\right]\right)\\
			+ \frac{1}{6}\mathbb{E}_{\mu_{\mathrm{init}, \eta}}\Bigg(\int_0^1 \left(1-\theta\right)^2 \nabla^3\widetilde{R}_{\eta}\left(X_k^{\eta, \Delta t} + \theta\Xi^{\eta, k} \right)\left[\Xi^{\eta, k}, \Xi^{\eta, k}, \Xi^{\eta, k}\right]d\theta\\
			\qquad \qquad \qquad- \int_0^1 \left(1- \theta\right)^2 \nabla^3 \widetilde{R}_0 \left(Y_k^{0, \Delta t} + \theta\widetilde{\Xi}^{0, k}\right)\left[\widetilde{\Xi}^{0, \Delta t}, \widetilde{\Xi}^{0, \Delta t}, \widetilde{\Xi}^{0, \Delta t}\right]d\theta\Bigg),
		\end{gathered}
	\end{equation}
	where, for any \(x \in \mathbb{R}^d\), we view \(\nabla^2\widetilde{R}_\eta(x)\left[\cdot, \cdot\right]\) and  \(\nabla^2\widetilde{R}_0(x)\left[\cdot, \cdot\right]\) as bilinear forms on \(\mathbb{R}^d\) and \(\nabla^3\widetilde{R}_\eta(x)\left[\cdot, \cdot, \cdot\right]\) and \(\nabla^3\widetilde{R}_0(x)\left[\cdot, \cdot, \cdot\right]\) as trilinear forms on \(\mathbb{R}^d\). 
	The general strategy of our argument to bound \(\mathscr{E}^{\eta, \Delta t}_k\) is expanding the terms of \eqref{eq:bias_error_terms} and controlling the resulting differences of terms individually. In particular we frequently use \eqref{eq:W_n_bound} in Lemma~\ref{lm:W_n_bound} to reduce bounding various differences to controlling terms of the from 
	\begin{equation}
		(\dots) \underset{\text{by \eqref{eq:W_n_bound}}}{\leq} \mathbb{E}_{\mu_{\mathrm{init}, \eta}}\left[\mathbf{1}_{\left\{X_k^{\eta, \Delta t} \neq Y_k^{0, \Delta t}\right\}}\left(\widetilde{f}(X_k^{\eta, \Delta t}) + \widetilde{f}(Y_k^{0, \Delta t})\right)\right].
	\end{equation}
	We can then control the right hand side using \eqref{eq:finitetime_discrete_weightedtv_bound} in Proposition~\ref{prop:discrete_Wn_norm_bound}:
	\begin{equation}
		\begin{aligned}
			&\mathbb{E}_{\mu_{\mathrm{init}, \eta}}\left[\mathbf{1}_{\left\{X_k^{\eta, \Delta t} \neq Y_k^{0, \Delta t}\right\}}\left(\widetilde{f}(X_k^{\eta, \Delta t}) + \widetilde{f}(Y_k^{0, \Delta t})\right)\right]\\
			&\qquad\qquad\qquad\leq \left(C\eta + \mu_{\mathrm{init}, \eta}\left(\mathbf{1}_{\left\{x\neq y\right\}}\right)\right)\left(\mathbb{E}_{\mu_{\mathrm{init}, \eta}}\left[\widetilde{f}(X_k^{\eta, \Delta t})\right] + \mathbb{E}_{\mu_{\mathrm{init}, \eta}}\left[\widetilde{f}(Y_k^{0, \Delta t})\right]\right).
		\end{aligned}
	\end{equation}
	The prefactor is then of order \(\eta\) due to the hypotheses on \(\left(\mu_{\mathrm{init}, \eta}\right)_{\eta \in \left[-\eta_\star, \eta_\star\right]}\) and the two expectations on the right hand side are uniformly bounded in \(k\) due to the moment growth bounds \eqref{eq:disc_semigroup_estimates} and the hypotheses on the initial probability measures. 
	
	For the first term in \eqref{eq:bias_error_terms}, expanding the bilinear forms gives
	\begin{equation}\label{eq:bias_error_term1}
		\begin{aligned}
			&\mathbb{E}_{\mu_{\mathrm{init}, \eta}}\left(\nabla^2 \widetilde{R}_\eta\left(X_k^{\eta, \Delta t}\right)\left[b^{\eta, k} + \eta F^{\eta, k}, b^{\eta, k} + \eta F^{\eta, k}\right] - \nabla^2\widetilde{R}_0\left(Y_k^{0, \Delta t}\right)\left[b^{0, k}, b^{0, k}\right]\right) \\
			&\qquad= \eta^2 \mathbb{E}_{\mu_{\mathrm{init}, \eta}}\left(\nabla^2\widetilde{R}_\eta\left(X_k^{\eta, \Delta t}\right)\left[F^{\eta, k}, F^{\eta, k}\right]\right) + 2\eta \mathbb{E}_{\mu_{\mathrm{init}, \eta}}\left(\nabla^2 \widetilde{R}_\eta\left(X_k^{\eta, \Delta t}\right)\left[b^{\eta, k}, F^{\eta, k}\right]\right) \\
			&\qquad\qquad +\mathbb{E}_{\mu_{\text{init}, \eta}}\left(\nabla^2\widetilde{R}_\eta\left(X_k^{\eta, \Delta t}\right)\left[b^{\eta, k}, b^{\eta, k}\right] - \nabla^2\widetilde{R}_0\left(Y_k^{0, \Delta t}\right)\left[b^{0,k}, b^{0,k}\right]\right)
		\end{aligned}
	\end{equation}
	Since \(x \mapsto \nabla^2 \widetilde{R}_\eta(x)\left[F(x), F(x)\right]\) and \(x \mapsto \nabla^2\widetilde{R}_\eta(x)\left[b(x), F(x)\right]\) are in \(\mathscr{S}\) as the products and sums of functions which are in \(\mathscr{S}\) by Assumption~\ref{ass:ref_measure} and Proposition~\ref{prop:exist_poisson_eq}, what is inside the first two expectation is bounded by \(C\mathcal{K}_n\left(X_n^{\eta, \Delta t}\right)\) for some \(n \in \mathbb{N}\). Consequently, the assumptions on the initial measure and the moment growth bound \eqref{eq:disc_semigroup_estimates} ensure that the first term is of order \(\eta^2\) and the second of order \(\eta\) uniformly in~\(k\). To show that the third term is also of order \(\eta\), we decompose it into three terms:
	\begin{equation}\label{eq:bias_error_third_of_first_term}
		\begin{aligned}
			\mathbb{E}_{\mu_{\text{init}, \eta}}&\left(\nabla^2\widetilde{R}_\eta\left(X_k^{\eta, \Delta t}\right)\left[b^{\eta, k}, b^{\eta, k}\right] - \nabla^2\widetilde{R}_0\left(Y_k^{0, \Delta t}\right)\left[b^{0,k}, b^{0,k}\right]\right)\\
			= & \,\, \mathbb{E}_{\mu_{\text{init}, \eta}}\left(\nabla^2\widetilde{R}_\eta\left(X_k^{\eta, \Delta t}\right) \left[b^{\eta, k}, b^{\eta, k}\right] - \nabla^2\widetilde{R}_{\eta}\left(X_k^{\eta, \Delta t}\right)\left[b^{0, k}, b^{0, k}\right]\right)\\
			&+\mathbb{E}_{\mu_{\text{init}, \eta}}\left(\nabla^2\widetilde{R}_{\eta}\left(X_k^{\eta, \Delta t}\right)\left[b^{0, k}, b^{0,k}\right] - \nabla^2\widetilde{R}_0\left(X_k^{\eta, \Delta t}\right)\left[b^{0, k}, b^{0,k}\right]\right)\\
			&+\mathbb{E}_{\mu_{\text{init}, \eta}}\left(\nabla^2\widetilde{R}_0\left(X_k^{\eta, \Delta t}\right)\left[b^{0,k}, b^{0,k}\right] - \nabla^{2}\widetilde{R}_0\left(Y_k^{0, \Delta t}\right)\left[b^{0, k}, b^{0, k}\right]\right)\\
			=: &\,\, (\mathrm{I}) + (\mathrm{II}) + (\mathrm{III}).
		\end{aligned}
	\end{equation}
	Expanding the first term and using the triangle inequality we get
	\begin{equation*}
		\begin{aligned}
			\left|(\mathrm{I})\right| &\leq \mathbb{E}_{\mu_{\text{init}, \eta}}\left|\nabla^2\widetilde{R}_\eta\left(X_k^{\eta, \Delta t}\right)\left[b^{\eta, k}, b^{\eta, k}\right] - \nabla^2\widetilde{R}_\eta\left(X_k^{\eta, \Delta t}\right)\left[b^{\eta, k}, b^{0, k}\right]\right|\\
			&\quad + \mathbb{E}_{\mu_{\text{init}, \eta}}\left|\nabla^2\widetilde{R}_\eta\left(X_k^{\eta, \Delta t}\right)\left[b^{\eta, k}, b^{0, k}\right] - \nabla^2\widetilde{R}_\eta\left(X_k^{\eta, \Delta t}\right)\left[b^{0, k}, b^{0, k}\right]\right|\\
			&\leq \mathbb{E}_{\mu_{\mathrm{init}, \eta}}\left[\left\|\nabla^2\widetilde{R}_{\eta}\left(X_k^{\eta, \Delta t}\right)\right\|\left|b^{\eta, k}\right|\left|b^{\eta, k} - b^{0, k}\right| + \left\|\nabla^2\widetilde{R}_{\eta}\left(X_k^{\eta, \Delta t}\right)\right\|\left|b^{0, k}\right|\left|b^{\eta, k} - b^{0, k}\right|\right]\\
			&=\mathbb{E}_{\mu_{\mathrm{init}, \eta}}\left[\left\|\nabla^2\widetilde{R}_{\eta}\left(X_k^{\eta, \Delta t}\right)\right\|\left(\left|b^{\eta, k}\right| + \left|b^{0, k}\right|\right)\left|b^{\eta, k} - b^{0, k}\right|\right].
		\end{aligned}
	\end{equation*}
	Since \(\widetilde{R}_\eta\) and \(b\) are in \(\mathscr{S}\), we can bound \(\left\|\nabla^2 \widetilde{R}_\eta(\cdot)\right\|\) and \(|b|\) by \(C\mathcal{K}_m\) for \(m\) large enough. This bound and~\eqref{eq:W_n_bound} in Lemma \ref{lm:W_n_bound} with \(n\) large enough then gives
	\begin{equation*}
		\begin{aligned}
			\left|(\mathrm{I})\right| &\leq C\mathbb{E}_{\mu_{\mathrm{init}, \eta}}\left[\mathcal{K}_{m}\left(X_k^{\eta, \Delta t}\right)\left(\mathcal{K}_m\left(X_k^{\eta, \Delta t}\right) + \mathcal{K}_m\left(Y_k^{0, \Delta t}\right)\right)\left|b^{\eta, k} - b^{0, k}\right|\right]\\
			&\leq C\mathbb{E}_{\mu_{\mathrm{init}, \eta}}\left[\left(\mathcal{K}_{m'}\left(X_k^{\eta, \Delta t}\right) + \mathcal{K}_{m'}\left(Y_k^{0, \Delta t}\right)\right)\left|b^{\eta, k} - b^{0, k}\right|\right]\\
			&\leq C \mathbb{E}_{\mu_{\mathrm{init}, \eta}}\left[\left(\mathcal{K}_{m'}\left(X_k^{\eta, \Delta t}\right) + \mathcal{K}_{m'}\left(Y_k^{0, \Delta t}\right)\right)\mathbf{1}_{\left\{X_k^{\eta, \Delta t} \neq Y_k^{0, \Delta t}\right\}}\left(\mathcal{K}_n\left(X_k^{\eta, \Delta t}\right) + \mathcal{K}_n\left(Y_k^{0, \Delta t}\right)\right)\right]\\
			&\leq C \mathbb{E}_{\mu_{\mathrm{init}, \eta}}\left[\mathbf{1}_{\left\{X_k^{\eta, \Delta t} \neq Y_k^{0, \Delta t}\right\}}\left(\mathcal{K}_{n'}\left(X_k^{\eta, \Delta t}\right) + \mathcal{K}_{n'}\left(Y_k^{0, \Delta t}\right)\right)\right],
		\end{aligned}
	\end{equation*}
	where the second and fourth inequality follow from choosing \(m' \geq m\) and \(n' \geq n + m'\) large enough. The right hand side is order \(\eta\) by Proposition~\ref{prop:discrete_Wn_norm_bound}, the moment growth bounds \eqref{eq:disc_semigroup_estimates}, and the hypotheses on the initial probability measures.
	For the second term in \eqref{eq:bias_error_term1} we obtain using Lemma~\ref{lm:poisson_sol_approx} 
	\begin{equation*}
		\begin{aligned}
		\left|(\mathrm{II})\right| &\leq \mathbb{E}_{\mu_{\mathrm{init}, \eta}} \left|\left(\nabla^2\widetilde{R}_{\eta}\left(X_k^{\eta, \Delta t}\right) - \nabla^2\widetilde{R}_0\left(X_k^{0, \Delta t}\right)\right)\left[b^{0, k},b^{0, k} \right]\right|\\
		&\leq C\mathbb{E}_{\mu_{\mathrm{init}, \eta}}\left[\max_{1\leq i, j\leq d}\left|\partial_{x_i x_j}^2\left(\widetilde{R}_\eta - \widetilde{R}_0\right)\left(X_k^{\eta, \Delta t}\right)\right| \left|b^{0, k}\right|^2\right]\\
		&\leq C\eta \mathbb{E}_{\mu_{\mathrm{init}, \eta}}\left[\max_{1\leq i, j\leq d}\left|\partial_{x_i x_j}^2\widecheck{R}_\eta\left(X_k^{\eta, \Delta t}\right)\right| \left|b^{0, k}\right|^2\right]\\
		&\leq C\eta \mathbb{E}_{\mu_{\mathrm{init}, \eta}}\left[\mathcal{K}_{n}\left(X_k^{\eta, \Delta t}\right)\mathcal{K}_{n}\left(Y_k^{0, \Delta t}\right)\right]\\
		&\leq C\eta \mathbb{E}_{\mu_{\mathrm{init}, \eta}}\left[\mathcal{K}_{2n}\left(X_k^{\eta, \Delta t}\right)\right]^{1/2}\mathbb{E}_{\mu_{\mathrm{init}, \eta}}\left[\mathcal{K}_{2n}\left(Y_k^{0, \Delta t}\right)\right]^{1/2},
		\end{aligned}
	\end{equation*}
	where \(\widecheck{R}_\eta \in \mathscr{S}\) is the function from Lemma~\ref{lm:poisson_sol_approx}. The fourth inequality is due to the fact that second derivatives of \(\widecheck{R}_\eta\) can be bounded  by \(C\mathcal{K}_n\) for \(n\) large enough and the fifth follows from the Cauchy--Schwarz inequality. The two expectations in the final line are uniformly bounded by the moment growth bounds \eqref{eq:disc_semigroup_estimates} and the hypotheses in the initial measure. For the last term in \eqref{eq:bias_error_third_of_first_term}, observe that \(z \mapsto \nabla^2\widetilde{R}_0(z)\) is in \(\mathscr{S}\) so as before we can apply \eqref{eq:W_n_bound} in Lemma~\ref{lm:W_n_bound} to get
	\[\begin{aligned}
		\left|(\mathrm{III})\right| &\leq C\mathbb{E}_{\mu_{\mathrm{init}, \eta}}\left[\left|b^{0,k}\right|^2\mathbf{1}_{\left\{X_k^{\eta, \Delta t} \neq Y_k^{0, \Delta t}\right\}}\left(\mathcal{K}_n\left(X_k^{\eta, \Delta t}\right) + \mathcal{K}_n\left(Y_k^{0, \Delta t}\right)\right)\right]\\
		&\leq C\mathbb{E}_{\mu_{\mathrm{init}, \eta}}\left[\mathcal{K}_{m}\left(Y_k^{0, \Delta t}\right)\mathbf{1}_{\left\{X_k^{\eta, \Delta t} \neq Y_k^{0, \Delta t}\right\}}\left(\mathcal{K}_n\left(X_k^{\eta, \Delta t}\right) + \mathcal{K}_n\left(Y_k^{0, \Delta t}\right)\right)\right]\\
		&\leq C\mathbb{E}_{\mu_{\mathrm{init}, \eta}}\left[\mathbf{1}_{\left\{X_k^{\eta, \Delta t} \neq Y_k^{0, \Delta t}\right\}}\left(\mathcal{K}_{n'}\left(X_k^{\eta, \Delta t}\right) + \mathcal{K}_{n'}\left(Y_k^{0, \Delta t}\right)\right)\right],
	\end{aligned}\]
	where the second inequality follows from the fact that \(b \in \mathscr{S}\) and the third inequality from choosing \(n' \in \mathbb{N}\) large enough and resorting to a Cauchy--Schwarz inequality. 
	The last line is of order \(\eta\) by Proposition~\ref{prop:discrete_Wn_norm_bound}, the moment growth bounds \eqref{eq:disc_semigroup_estimates}, and the hypotheses on the initial probability measures.
	
	To bound the second expectation in \eqref{eq:bias_error_terms}, we write it as 
	\begin{equation*}
		\begin{aligned}
			&\mathbb{E}_{\mu_{\mathrm{init}, \eta}}\left[\int_0^1\left(1-\theta\right)^2\left\{\nabla^3\widetilde{R}_\eta \left(X_k^{\eta, \Delta t} + \theta\Xi^{\eta, k}\right) -\nabla^3\widetilde{R}_0 \left(X_k^{\eta, \Delta t} + \theta\Xi^{\eta, k}\right)\right\}\left[\Xi^{\eta,k}, \Xi^{\eta,k}, \Xi^{\eta, k}\right]d\theta \right]\\
			&+\mathbb{E}_{\mu_{\mathrm{init}, \eta}}\left[\int_0^1 \left(1 - \theta\right)^2\left\{\nabla^3\widetilde{R}_0 \left(X_k^{\eta, \Delta t} + \theta\Xi^{\eta, k}\right)\left[\Xi^{\eta,k}, \Xi^{\eta,k}, \Xi^{\eta, k}\right] - \nabla^3\widetilde{R}_0\left(X^{\eta, \Delta t} + \theta\Xi^{\eta, k}\right)\left[\Xi^{0, k}, \Xi^{0,k}, \Xi^{0,k}\right]\right\}d\theta\right]\\
			&+\mathbb{E}_{\mu_{\mathrm{init}, \eta}}\left[\int_0^1 \left(1 - \theta\right)^2\left(\nabla^3\widetilde{R}_0 \left(X_k^{\eta, \Delta t} + \theta\Xi^{\eta, k}\right) - \nabla^3\widetilde{R}_0\left(X_k^{\eta, \Delta t} + \theta\Xi^{0, k}\right)\right)\left[\Xi^{0, k}, \Xi^{0,k}, \Xi^{0,k}\right]d\theta\right]\\
			&+\mathbb{E}_{\mu_{\mathrm{init}, \eta}}\left[\int_0^1 \left(1 - \theta\right)^2\left\{\nabla^3\widetilde{R}_0 \left(X_k^{\eta, \Delta t} + \theta\Xi^{0, k}\right)\left[\Xi^{0,k}, \Xi^{0,k}, \Xi^{0, k}\right] - \nabla^3\widetilde{R}_0\left(Y_k^{0, \Delta t} + \theta\widetilde{\Xi}^{0, k}\right)\left[\widetilde{\Xi}^{0, k}, \widetilde{\Xi}^{0,k}, \widetilde{\Xi}^{0,k}\right]\right\}d\theta\right]\\
			& =: (\mathrm{I}) + (\mathrm{II}) + (\mathrm{III}) + (\mathrm{IV}).
		\end{aligned}
	\end{equation*}
	For the first term, we use the triangle inequality to write
	\begin{equation*}
		\begin{aligned}
			&\left|(\mathrm{I})\right| \leq \sum_{1\leq i, j, \ell\leq d}\left|\mathbb{E}_{\mu_{\mathrm{init}, \eta}}\left[\int_0^1 \left(1 - \theta\right)^2 \partial_{x_i x_j x_\ell}^3\left(\widetilde{R}_\eta - \widetilde{R}_0\right)\left(X_k^{\eta,\Delta t} + \theta\Xi^{\eta, k}\right)\Xi_i^{\eta,k}\Xi_j^{\eta,k}\Xi_\ell^{\eta,k} d\theta\right]\right|\\
			&\leq \Delta t^3\!\!\!\sum_{1\leq i, j, \ell\leq d}\left|\mathbb{E}_{\mu_{\mathrm{init}, \eta}}\left[\int_0^1 \left(1 - \theta\right)^2 \partial_{x_i x_j x_\ell}^3\left(\widetilde{R}_\eta - \widetilde{R}_0\right)\left(X_k^{\eta,\Delta t} + \theta\Xi^{\eta, k}\right)\left(b_i^{\eta, k}\!\! + \eta F_i^{\eta, k}\right)\left(b_j^{\eta, k}\!\! + \eta F_j^{\eta, k}\right)\left(b_\ell^{\eta, k}\!\! + \eta F_\ell^{\eta, k}\right)d\theta\right]\right|\\
			&+3\Delta t^{5/2}\sqrt{\frac{2}{\beta}}\sum_{1\leq i, j, \ell\leq d}\left|\mathbb{E}_{\mu_{\mathrm{init}, \eta}}\left[\int_0^1 \left(1 - \theta\right)^2 \partial_{x_i x_j x_\ell}^3\left(\widetilde{R}_\eta - \widetilde{R}_0\right)\left(X_k^{\eta,\Delta t} + \theta\Xi^{\eta, k}\right)\left(b_i^{\eta, k}\!\! + \eta F_i^{\eta, k}\right)\left(b_j^{\eta, k}\!\! + \eta F_j^{\eta, k}\right)G_{k+1,\ell}\,d\theta\right]\right|\\
			&+3\Delta t^{2}\left(\frac{2}{\beta}\right)\sum_{1\leq i, j, \ell\leq d}\left|\mathbb{E}_{\mu_{\mathrm{init}, \eta}}\left[\int_0^1 \left(1 - \theta\right)^2 \partial_{x_i x_j x_\ell}^3\left(\widetilde{R}_\eta - \widetilde{R}_0\right)\left(X_k^{\eta,\Delta t} + \theta\Xi^{\eta, k}\right)\left(b_i^{\eta, k}\!\! + \eta F_i^{\eta, k}\right)G_{k+1,j}G_{k+1,\ell}\,d\theta\right]\right|\\
			&+\Delta t^{3/2}\left(\frac{2}{\beta}\right)^{3/2}\!\!\!\!\!\!\!\sum_{1\leq i, j, \ell\leq d}\left|\mathbb{E}_{\mu_{\mathrm{init}, \eta}}\left[\int_0^1 \left(1 - \theta\right)^2 \partial_{x_i x_j x_\ell}^3\left(\widetilde{R}_\eta - \widetilde{R}_0\right)\left(X_k^{\eta,\Delta t} + \theta\Xi^{\eta, k}\right)G_{k+1, i}G_{k+1, j}G_{k+1, \ell}\,d\theta\right]\right|.\\
		\end{aligned}
	\end{equation*}
	Using Lemma~\ref{lm:poisson_sol_approx} to control the difference between solutions to the Poisson equation, we bound \(\left|\partial_{x_i x_j x_\ell}^3\left(\widetilde{R}_\eta - \widetilde{R}_0\right)\right|\) by \(C\eta\mathcal{K}_n\) for an \(n\) large enough. Using the Cauchy--Schwarz inequality to separate away the Gaussian random variables (see how we do it for the fourth sum below) and the fact that \(b, F \in \mathscr{S}\), we bound the first three sums by \(C\eta \Delta t^2 \mathbb{E}_{\mu_{\mathrm{init}, \eta}}\left[\mathcal{K}_n\left(X_k^{\eta, \Delta t}\right)\right]\) for an \(n\) large enough. 
	For the final term, we exchange the expectation operator and the integral to obtain the following expectation in the integrand of a generic term in the sum
	\[\mathbb{E}_{\mu_{\mathrm{init}, \eta}}\left[\partial_{x_i x_j x_\ell}^3\left(\widetilde{R}_\eta - \widetilde{R}_0\right)\left(X_k^{\eta,\Delta t} + \theta\Delta t\left(b^{\eta,k}\! + \eta F^{\eta, k}\right) + \theta \sqrt{\frac{2\Delta t}{\beta}}G_{k+1}\right)G_{k+1,i}G_{k+1,j}G_{k+1,\ell}\right].\]
	Performing another first order Taylor expansion gives 
	\begin{align*}
		&\mathbb{E}_{\mu_{\mathrm{init}, \eta}}\left[\partial_{x_i x_j x_\ell}^3\left(\widetilde{R}_\eta - \widetilde{R}_0\right)\left(X_k^{\eta,\Delta t} + \theta\Delta t\left(b^{\eta,k}\!\! + \eta F^{\eta, k}\right) + \theta \sqrt{\frac{2\Delta t}{\beta}}G^{k+1}\right)G_{k+1,i}G_{k+1,j}G_{k+1,\ell}\right] \\
		&= \mathbb{E}_{\mu_{\mathrm{init}, \eta}}\left[\partial_{x_i x_j x_\ell}^3\left(\widetilde{R}_\eta - \widetilde{R}_0\right)\left(X_k^{\eta,\Delta t} + \theta\Delta t\left(b^{\eta,k}\!\! + \eta F^{\eta, k}\right)\right)G_{k+1,i}G_{k+1,j}G_{k+1,\ell}\right]\\
		&\,\, + \theta \sqrt{\frac{2\Delta t}{\beta}}\!\!\sum_{q = 1}^d \!\! \mathbb{E}_{\mu_{\mathrm{init}, \eta}}\!\left[\!\int_0^1\!\!\!\partial_{x_i x_j x_\ell x_q}^4\!\!\left(\widetilde{R}_\eta - \widetilde{R}_0\right)\!\!\left(\! X_k^{\eta,\Delta t}\!\! + \theta\Delta t\left(b^{\eta,k}\!\! + \eta F^{\eta, k} \!\right) + \theta\vartheta \sqrt{\frac{2\Delta t}{\beta}}G_{k+1}\right)\!\! G_{k+1,i}G_{k+1,j}G_{k+1,\ell} G_{k+1,q}d\vartheta\right].
	\end{align*}
	The first expectation vanishes since \(G_{k+1}\) is independent of \(X_k^{\eta, \Delta t}\). Using Lemma~\ref{lm:poisson_sol_approx} to bound \(\left|\partial_{x_i x_j x_\ell x_q}^4\left(\widetilde{R}_\eta - \widetilde{R}_0\right)\right|\) for any \(q\) by \(C\eta\mathcal{K}_n\) for an \(n\) large enough, we control the second expectation as:
	\begin{equation*}
	\begin{aligned}
		&\left|\theta \sqrt{\frac{2\Delta t}{\beta}}\!\sum_{q = 1}^d \! \mathbb{E}_{\mu_{\mathrm{init}, \eta}}\left[\int_0^1\!\! \partial_{x_i x_j x_\ell x_q}^4 \!\left(\widetilde{R}_\eta - \widetilde{R}_0\right)\!\!\left(\!\! X_k^{\eta,\Delta t}\!\! + \theta\Delta t\left(b^{\eta,k}\!\! + \eta F^{\eta, k}\right) + \theta\vartheta \sqrt{\frac{2\Delta t}{\beta}}G_{k+1}\!\!\right)G_{k+1,i}G_{k+1,j}G_{k+1,\ell} G_{k+1,q}d\vartheta\right]\right|\\
		& \leq C\eta\theta \sqrt{\frac{2\Delta t}{\beta}}\!\sum_{q = 1}^d\! \mathbb{E}_{\mu_{\mathrm{init}, \eta}}\left[\int_0^1\mathcal{K}_n\left(\! X_k^{\eta,\Delta t}\!\! + \theta\Delta t\left(b^{\eta,k}\!\! + \eta F^{\eta, k}\right) + \theta\vartheta \sqrt{\frac{2\Delta t}{\beta}}G_{k+1}\!\right)\left|G_{k+1,i}G_{k+1,j}G_{k+1,\ell} G_{k+1,q}\right|d\vartheta\right]\\
		&= C\eta\theta \sqrt{\frac{2\Delta t}{\beta}}\!\sum_{q = 1}^d\int_0^1\mathbb{E}_{\mu_{\mathrm{init}, \eta}}\left[\mathcal{K}_n\left(\! X_k^{\eta,\Delta t}\!\! + \theta\Delta t\left(b^{\eta,k}\!\! + \eta F^{\eta, k}\right) + \theta\vartheta \sqrt{\frac{2\Delta t}{\beta}}G^{k+1} \!\right)\left|G_{k+1,i}G_{k+1,j}G_{k+1,\ell} G_{k+1,q}\right|\right]d\vartheta\\
		&\leq C\eta\theta \sqrt{\frac{2\Delta t}{\beta}}\!\int_0^1\!\!\!\mathbb{E}_{\mu_{\mathrm{init}, \eta}}\!\! \left[\!\mathcal{K}_n\left(\!\! X_k^{\eta,\Delta t}\!\! + \theta\Delta t\left(b^{\eta,k}\!\! + \eta F^{\eta, k}\right) + \theta\vartheta \sqrt{\frac{2\Delta t}{\beta}}G_{k+1}\!\right)^2\!\! \right]^{1/2}\!\!\!\!\!\!\!\! d\vartheta \!\!\sum_{q = 1}^d \!\! \mathbb{E}_{\mu_{\mathrm{init}, \eta}}\! \left[\left|G_{k+1,i}G_{k+1,j}G_{k+1,\ell} G_{k+1,q}\right|^2\right]^{1/2}\\
		&\leq C\eta\theta \sqrt{\frac{2\Delta t}{\beta}} \int_0^1\mathbb{E}_{\mu_{\mathrm{init}, \eta}}\left[\mathcal{K}_n\left(X_k^{\eta,\Delta t}\!\! + \theta\Delta t\left(b^{\eta,k}\!\! + \eta F^{\eta, k}\right) + \theta\vartheta \sqrt{\frac{2\Delta t}{\beta}}G_{k+1}\right)^2\right]^{1/2}\!\!\!\!\!\! d\vartheta\\
		&\leq C\eta\theta \sqrt{\frac{2\Delta t}{\beta}} \int_0^1\mathbb{E}_{\mu_{\mathrm{init}, \eta}}\left[\mathcal{K}_n\left(X_k^{\eta,\Delta t}\!\! + \theta\Delta t\left(b^{\eta,k}\!\! + \eta F^{\eta, k}\right) + \theta\vartheta \sqrt{\frac{2\Delta t}{\beta}}G_{k+1}\right)^2\right]d\vartheta\\
		&\leq C\eta\theta \sqrt{\frac{2\Delta t}{\beta}} \int_0^1\mathbb{E}_{\mu_{\mathrm{init}, \eta}}\left[\mathcal{K}_{2n}\left(X_k^{\eta,\Delta t}\!\! + \theta\Delta t\left(b^{\eta,k}\!\! + \eta F^{\eta, k}\right) + \theta\vartheta \sqrt{\frac{2\Delta t}{\beta}}G_{k+1}\right)\right]d\vartheta\\
		&\leq C\eta\theta \sqrt{\frac{2\Delta t}{\beta}} \left(\mathbb{E}_{\mu_{\mathrm{init}, \eta}}\left[\mathcal{K}_{2n}\left(X_k^{\eta,\Delta t}\!\! + \theta\Delta t\left(b^{\eta,k}\!\! + \eta F^{\eta, k}\right)\right)\right] + \int_0^1\mathbb{E}_{\mu_{\mathrm{init}, \eta}}\left[\mathcal{K}_{2n}\left(\theta\vartheta \sqrt{\frac{2\Delta t}{\beta}}G_{k+1}\right)\right]d\vartheta\right),
	\end{aligned}
	\end{equation*}
	where we used the Cauchy--Schwarz inequality for the second inequality and the fact that \(\mathcal{K}_n \geq 1\) and that \(x \geq \sqrt{x}\) when \(x \geq 1\) for the fourth one. The last inequality is due to the fact that \(\mathcal{K}_m(x+y) \leq C\left(\mathcal{K}_m(x) + \mathcal{K}_m(y)\right)\) for any \(x, y\in \mathbb{R}^d\). Since \(b, F \in \mathscr{S}\), we can bound the last line by 
	\[C\eta \theta \sqrt{\frac{2\Delta t}{\beta}}\mathbb{E}_{\mu_{\mathrm{init}, \eta}}\left[\mathcal{K}_{n'}\left(X_k^{\eta, \Delta t}\right)\right],\]
	with \(n' \geq 2n\) sufficiently large. Putting this bound together with previously obtained bounds, we get for~\(n\) large enough that 
	\[\left|(\mathrm{I})\right| \leq C\eta \Delta t^2 \mathbb{E}_{\mu_{\mathrm{init}, \eta}}\left[\mathcal{K}_n\left(X_k^{\eta, \Delta t}\right)\right]. \]
	We can then bound the expectation using moment growth bound \eqref{eq:disc_semigroup_estimates} and the hypotheses on the initial measures.
	
	Expanding the trilinear form in the integrand of the term \((\mathrm{II})\) gives:
	\begin{equation*}
		\begin{aligned}
			&\nabla^3\widetilde{R}_0 \left(X_k^{\eta, \Delta t} + \theta\Xi^{\eta, k}\right)\left[\Xi^{\eta,k}, \Xi^{\eta,k}, \Xi^{\eta, k}\right] - \nabla^3\widetilde{R}_0\left(X^{\eta, \Delta t} + \theta\Xi^{\eta, k}\right)\left[\Xi^{0, k}, \Xi^{0,k}, \Xi^{0,k}\right] \\
			&\quad =3\eta \Delta t \nabla^3 \widetilde{R}_0\left(X_k^{\eta, \Delta t} + \theta\Xi^{\eta, k}\right)\left[F^{\eta, k}, b^{\eta, k}\Delta t + \sqrt{\frac{2\Delta t}{\beta}}G_{k+1}, b^{\eta, k}\Delta t + \sqrt{\frac{2\Delta t}{\beta}}G_{k+1}\right]\\
			&\quad + 3\eta^2 \Delta t^2 \nabla^3 \widetilde{R}_0 \left(X_k^{\eta, \Delta t} + \theta \Xi^{\eta, k}\right)\left[F^{\eta, k}, F^{\eta, k}, b^{\eta, k} \Delta t + \sqrt{\frac{2\Delta t}{\beta}}G_{k+1}\right]\\
			&\quad + \eta^3 \Delta t^3 \nabla^3\widetilde{R}_0\left(X_k^{\eta, \Delta t} + \theta \Xi^{\eta, k}\right) \left[F^{\eta, k}, F^{\eta, k}, F^{\eta, k}\right]\\
			&=3\eta \Delta t^2 \nabla^3 \widetilde{R}_0\left(X_k^{\eta, \Delta t} + \theta\Xi^{\eta, k}\right)\left[F^{\eta, k}, b^{\eta, k}\sqrt{\Delta t} + \sqrt{\frac{2}{\beta}}G_{k+1}, b^{\eta, k}\sqrt{\Delta t} + \sqrt{\frac{2}{\beta}}G_{k+1}\right]\\
			&\quad + 3\eta^2 \Delta t^2 \nabla^3 \widetilde{R}_0 \left(X_k^{\eta, \Delta t} + \theta \Xi^{\eta, k}\right)\left[F^{\eta, k}, F^{\eta, k}, b^{\eta, k}\Delta t + \sqrt{\frac{2\Delta t}{\beta}}G_{k+1}\right]\\
			&\quad + \eta^3 \Delta t^3 \nabla^3\widetilde{R}_0\left(X_k^{\eta, \Delta t} + \theta \Xi^{\eta, k}\right) \left[F^{\eta, k}, F^{\eta, k}, F^{\eta, k}\right].
		\end{aligned}
	\end{equation*}
	Each the above terms involves functions in \(\mathscr{S}\) and Gaussian random variables, whose expectations can be analytically estimated. Following a manipulation analogous to what we did for term \((\mathrm{I})\) above to take care of the Gaussian random variable in the argument of \(\nabla^3 \widetilde{R}_0\), we can conclude that for \(n \in \mathbb{N}\) large enough
	\begin{equation*}
		\left|(\mathrm{II})\right| \leq C\eta \Delta t^2 \mathbb{E}_{\mu_{\mathrm{init}, \eta}}\left[\mathcal{K}_n\left(X_k^{\eta, \Delta t}\right)\right].
	\end{equation*}
	We can control the quantity on the right hand side of the inequality using the assumptions on the initial condition and the moment growth bound \eqref{eq:disc_semigroup_estimates}.
	
	For the  term \((\mathrm{III})\), we exchange the integral and expectation to obtain in the integrand the following expectation:
	\[\mathbb{E}_{\mu_{\mathrm{init}, \eta}}\left(\left[\nabla^3\widetilde{R}_0 \left(X_k^{\eta, \Delta t} + \theta\Xi^{\eta, k}\right) - \nabla^3\widetilde{R}_0\left(X_k^{\eta, \Delta t} + \theta\Xi^{0, k}\right)\right]\left[\Xi^{0, k}, \Xi^{0,k}, \Xi^{0,k}\right]\right).\]
	Expanding the trilinear form gives
	\begin{align*}
		&\mathbb{E}_{\mu_{\mathrm{init}, \eta}}\left(\left[\nabla^3\widetilde{R}_0 \left(X_k^{\eta, \Delta t} + \theta\Xi^{\eta, k}\right) - \nabla^3\widetilde{R}_0\left(X_k^{\eta, \Delta t} + \theta\Xi^{0, k}\right)\right]\left[\Xi^{0, k}, \Xi^{0,k}, \Xi^{0,k}\right]\right)\\
		&\quad =\Delta t^3 \mathbb{E}_{\mu_{\mathrm{init}, \eta}}\left(\left[\nabla^3\widetilde{R}_0 \left(X_k^{\eta, \Delta t} + \theta\Xi^{\eta, k}\right) - \nabla^3\widetilde{R}_0\left(X_k^{\eta, \Delta t} + \theta\Xi^{0, k}\right)\right]\left[b^{\eta, k}, b^{\eta, k}, b^{\eta, k}\right]\right)\\
		&\qquad+ 3\Delta t^{5/2}\sqrt{\frac{2}{\beta}} \mathbb{E}_{\mu_{\mathrm{init}, \eta}}\left[\left(\nabla^3\widetilde{R}_0 \left(X_k^{\eta, \Delta t} + \theta\Xi^{\eta, k}\right) - \nabla^3\widetilde{R}_0\left(X_k^{\eta, \Delta t} + \theta\Xi^{0, k}\right)\right)\left[b^{\eta, k}, b^{\eta, k}, G_{k+1}\right]\right]\\
		&\qquad + 3\left(\frac{2}{\beta}\right)\Delta t^2 \mathbb{E}_{\mu_{\mathrm{init}, \eta}}\left(\left[\nabla^3\widetilde{R}_0 \left(X_k^{\eta, \Delta t} + \theta\Xi^{\eta, k}\right) - \nabla^3\widetilde{R}_0\left(X_k^{\eta, \Delta t} + \theta\Xi^{0, k}\right)\right]\left[b^{\eta, k}, G_{k+1}, G_{k+1}\right]\right)\\
		&\qquad+ \Delta t^{3/2}\left(\frac{2}{\beta}\right)^{3/2}\mathbb{E}_{\mu_{\mathrm{init}, \eta}}\left(\left[\nabla^3\widetilde{R}_0 \left(X_k^{\eta, \Delta t} + \theta\Xi^{\eta, k}\right) - \nabla^3\widetilde{R}_0\left(X_k^{\eta, \Delta t} + \theta\Xi^{0, k}\right)\right]\left[G_{k+1}, G_{k+1}, G_{k+1}\right]\right).
	\end{align*}
	Since \(\widetilde{R}_0 \in \mathscr{S}\), we apply \eqref{eq:pseudo_lip_cond} in Lemma~\ref{lm:pseudo_lip} to obtain, for any \(\theta \in [0,1]\),
	\[\begin{aligned}
		&\max_{1\leq i, j, \ell\leq d}\left|\partial_{x_i x_j x_\ell}^3\left[\widetilde{R}_0 \left(X_k^{\eta, \Delta t} + \theta\Xi^{\eta, k}\right) - \widetilde{R}_0\left(X_k^{\eta, \Delta t} + \theta\Xi^{0, k}\right)\right]\right|\\
		 &\qquad\qquad\leq C\eta \Delta t \left\|F\right\|_\infty\left(\mathcal{K}_n\left(X_k^{\eta, \Delta t}\! + \!\theta\Xi^{\eta, k}\right) + \mathcal{K}_n\left(X_k^{\eta, \Delta t}\! + \!\theta\Xi^{0, k}\right)\right).
	\end{aligned}\]
	Consequently, the fact that the functions under consideration belong to \(\mathscr{S}\), the moment growth bound~\eqref{eq:disc_semigroup_estimates}, and the hypotheses on the initial measures then imply that each of the expectations in the expansion of the trilinear form, and hence \((\mathrm{III})\), is of order \(\eta \Delta t^{5/2}\) uniformly in \(k\).
	
	For the last term \((\mathrm{IV})\), we again exchange the expectation and integral to obtain in the integrand the following expectation:
	\[\mathbb{E}_{\mu_{\mathrm{init}, \eta}}\left(\nabla^3\widetilde{R}_0 \left(X_k^{\eta, \Delta t} + \theta\Xi^{0, k}\right)\left[\Xi^{0,k}, \Xi^{0,k}, \Xi^{0, k}\right] - \nabla^3\widetilde{R}_0\left(Y_k^{0, \Delta t} + \theta\widetilde{\Xi}^{0, k}\right)\left[\widetilde{\Xi}^{0, k}, \widetilde{\Xi}^{0,k}, \widetilde{\Xi}^{0,k}\right]\right).\]
	Expanding the trilinear forms in the integrand then gives
	\[\begin{aligned}
		&\mathbb{E}_{\mu_{\mathrm{init}, \eta}}\left(\nabla^3\widetilde{R}_0\left(X_k^{\eta, \Delta t} + \theta \Xi^{0, k}\right)\left[\Xi^{0, k}, \Xi^{0, k}, \Xi^{0, k}\right] - \nabla^3 \widetilde{R}_0\left(Y_k^{0, \Delta t} + \theta \widetilde{\Xi}^{0, k}\right)\left[\widetilde{\Xi}^{0, k}, \widetilde{\Xi}^{0, k}, \widetilde{\Xi}^{0, k}\right]\right)\\
		&= \Delta t^3 \mathbb{E}_{\mu_{\mathrm{init}, \eta}}\left(\nabla^3\widetilde{R}_0\left(X_k^{\eta, \Delta t} + \theta \Xi^{0, k}\right)\left[b^{\eta, k}, b^{\eta, k}, b^{\eta, k}\right] - \nabla^3 \widetilde{R}_0\left(Y_k^{0, \Delta t} + \theta \widetilde{\Xi}^{0, k}\right)\left[b^{0, k}, b^{0, k}, b^{0, k}\right]\right)\\
		&\quad + 3\sqrt{\frac{2}{\beta}}\Delta t^{5/2}\mathbb{E}_{\mu_{\mathrm{init}, \eta}}\left(\nabla^3 \widetilde{R}_0\left(X_k^{\eta, \Delta t} + \theta \Xi^{0, k}\right)\left[b^{\eta, k}, b^{\eta, k}, G_{k+1}\right] - \nabla^3\widetilde{R}_0\left(Y_k^{0, \Delta t} + \theta \widetilde{\Xi}^{0, k}\right)\left[b^{0, k}, b^{0,k}, \widetilde{G}_{k+1}\right]\right)\\
		&\quad + 3\left(\frac{2}{\beta}\right) \Delta t^2\mathbb{E}_{\mu_{\mathrm{init}, \eta}}\left(\nabla^3 \widetilde{R}_0\left(X_k^{\eta, \Delta t} + \theta \Xi^{0, k}\right)\left[b^{\eta, k}, G_{k+1}, G_{k+1}\right] - \nabla^3\widetilde{R}_0\left(Y_k^{0, \Delta t} + \theta \widetilde{\Xi}^{0, k}\right)\left[b^{0, k}, \widetilde{G}_{k+1}, \widetilde{G}_{k+1}\right]\right)\\
		&\quad + \left(\frac{2}{\beta}\right)^{3/2} \Delta t^{3/2}\mathbb{E}_{\mu_{\mathrm{init}, \eta}}\left(\nabla^3 \widetilde{R}_0\left(X_k^{\eta, \Delta t} + \theta \Xi^{0, k}\right)\left[G_{k+1}, G_{k+1}, G_{k+1}\right] - \nabla^3\widetilde{R}_0\left(Y_k^{0, \Delta t} + \theta \widetilde{\Xi}^{0, k}\right)\left[\widetilde{G}_{k+1}, \widetilde{G}_{k+1}, \widetilde{G}_{k+1}\right]\right).
	\end{aligned}\]
	The first three expectations are bounded via the same argument. Let \(G\) be a standard \(d\)-dimensional Gaussian random variable and consider the functions 
	\[x \mapsto \mathbb{E}\left[\nabla^3\widetilde{R}_0\left(x + \theta\Delta tb(x) + \theta\sqrt{\frac{2\Delta t}{\beta}}G\right)\left[b(x), b(x), b(x)\right]\right],\]
	\[x \mapsto \mathbb{E}\left[\nabla^3\widetilde{R}_0\left(x + \theta\Delta tb(x) + \theta\sqrt{\frac{2\Delta t}{\beta}}G\right)\left[G, b(x), b(x)\right]\right],\]
	and 
	\[x \mapsto \mathbb{E}\left[\nabla^3\widetilde{R}_0\left(x + \theta\Delta tb(x) + \theta\sqrt{\frac{2\Delta t}{\beta}}G\right)\left[b(x), G, G\right]\right],\]
	where the expectations are taken with respect to only \(G\). Each of these functions belongs to \(B_n^\infty \subset B^\infty_{V_c}\) uniformly in \(\theta \in [0, 1]\) for \(n\) large enough. Since the marginals of \(\left(G^{k+1}, \widetilde{G}^{k+1}\right)\) conditional on \(\left(X_k^{\eta, \Delta t}, Y_k^{0, \Delta t}\right)\) are standard \(d\)-dimensional Gaussians, \eqref{eq:W_n_bound} in Lemma~\ref{lm:pseudo_lip} and Proposition~\ref{prop:discrete_Wn_norm_bound} plus the moment growth bounds \eqref{eq:disc_semigroup_estimates} and the hypotheses on the initial measures imply that the first three expectations are of order \(\eta\). 
	For the last expectation, a first-order Taylor expansion gives
	\begin{align*}
		&\mathbb{E}_{\mu_{\mathrm{init}, \eta}}\left[\nabla^3 \widetilde{R}_0\left(X_k^{\eta, \Delta t} + \theta \Xi^{0, k}\right)\left[G_{k+1}, G_{k+1}, G_{k+1}\right] - \nabla^3\widetilde{R}_0\left(Y_k^{0, \Delta t} + \theta \widetilde{\Xi}^{0, k}\right)\left[\widetilde{G}_{k+1}, \widetilde{G}_{k+1}, \widetilde{G}_{k+1}\right]\right]\\
		&= \mathbb{E}_{\mu_{\mathrm{init}, \eta}}\left[\sum_{1\leq i, j, \ell\leq d} \partial_{x_i x_j x_\ell}^3\widetilde{R}_0\left(X_k^{\eta, \Delta t} + \theta\Delta tb^{\eta, k}\right)G_{k+1,i} G_{k+1,j} G_{k+1,\ell} \right]\\
		&\quad- \mathbb{E}_{\mu_{\mathrm{init}, \eta}}\left[\sum_{1\leq i, j, \ell\leq d} \partial_{x_i x_j x_\ell}^3\widetilde{R}_0\left(Y_k^{0, \Delta t} + \theta\Delta tb^{0, k}\right)\widetilde{G}_{k+1,i} \widetilde{G}_{k+1,j} \widetilde{G}_{k+1,\ell} \right]\\
		&\quad + \sqrt{\frac{2\Delta t}{\beta}}\int_0^1\mathbb{E}_{\mu_{\mathrm{init}, \eta}}\left[\sum_{1\leq i, j, \ell, q\leq d} \partial_{x_i x_j x_\ell x_q}^4\widetilde{R}_0 \left(X_k^{\eta, \Delta t} + \theta\Delta t b^{\eta, k} + \theta\vartheta\sqrt{\frac{2\Delta t}{\beta}}G_{k+1} \right)G_{k+1,i}G_{k+1,j} G_{k+1,\ell} G_{k+1,q}\right]d\vartheta\\
		&\quad - \sqrt{\frac{2\Delta t}{\beta}}\int_0^1\mathbb{E}_{\mu_{\mathrm{init}, \eta}}\left[\sum_{1\leq i, j, \ell, q\leq d} \partial_{x_i x_j x_\ell x_q}^4\widetilde{R}_0\left(Y^{0, \Delta t}_k + \theta\Delta tb^{0, k} + \theta \vartheta\sqrt{\frac{2\Delta t}{\beta}}\widetilde{G}_{k+1}\right)\widetilde{G}_{k+1,i}\widetilde{G}_{k+1,j} \widetilde{G}_{k+1,\ell} \widetilde{G}_{k+1,q}\right]d\vartheta.
	\end{align*}
	The first two terms are equal to zero since the marginals of \(\left(G^{k+1}, \widetilde{G}^{k+1}\right)\) conditional on \(\left(X_k^{\eta, \Delta t}, Y_k^{0, \Delta t}\right)\) are standard \(d\)-dimensional Gaussians. Similarly as what has been done to control the fourth term of~\((\mathrm{I})\), an application of \eqref{eq:W_n_bound} in Lemma~\ref{lm:W_n_bound} and Proposition~\ref{prop:discrete_Wn_norm_bound} imply that the difference of the last two terms is of order \(\eta\sqrt{\Delta t}\). Consequently, \((\mathrm{IV})\) is of order \(\eta\Delta t^2\).
				
	Together this control of the two terms in \eqref{eq:bias_error_terms} shows that the error term \(\mathscr{E}^{\eta, \Delta t}_k\) is of order \(\eta \Delta t^2\) uniformly in \(k\).
\subsubsection{Control of the Variance}
Using \eqref{eq:clt_asymp_variance} and \eqref{eq:sol_sticky_poisson_eq}, we have that the asymptotic variance is given by
\begin{equation*}
	\begin{aligned}
		&\sigma_{\mathrm{sticky}, R,\eta,\Delta t}^2 = \frac{1}{\eta^2}\mathbb{E}_{\mu_{\eta, \Delta t}}\left(\left\{\widehat{R}_{\eta, \Delta t}\left(X_1^{\eta, \Delta t}\right) - \widehat{R}_{0, \Delta t}\left(Y_1^{0, \Delta t}\right) - T^{\eta, \Delta t}\left[\widehat{R}_{\eta, \Delta t}\left(X_0^{\eta, \Delta t}\right) - \widehat{R}_{0, \Delta t}\left(Y_0^{0, \Delta t}\right)\right]\right\}^2\right)\\
		&\leq \frac{2}{\eta^2}\mathbb{E}_{\mu_{\eta, \Delta t}}\left(\left\{\widehat{R}_{\eta, \Delta t}\left(X_1^{\eta, \Delta t}\right) - \widehat{R}_{0, \Delta t}\left(Y_1^{0, \Delta t}\right)\right\}^2\right) + \frac{2}{\eta^2}\mathbb{E}_{\mu_{\eta, \Delta t}}\left(\left\{T^{\eta, \Delta t}\left[\widehat{R}_{\eta, \Delta t}\left(X_0^{\eta, \Delta t}\right) - \widehat{R}_{0, \Delta t}\left(Y_0^{0, \Delta t}\right)\right]\right\}^2\right)\\
		&\leq \frac{2}{\eta^2}\mathbb{E}_{\mu_{\eta, \Delta t}}\left(\left\{\widehat{R}_{\eta, \Delta t}\left(X_1^{\eta, \Delta t}\right)-\widehat{R}_{0, \Delta t}\left(Y_1^{0, \Delta t}\right)\right\}^2\right) + \frac{2}{\eta^2}\mathbb{E}_{\mu_{\eta, \Delta t}}\left(T^{\eta, \Delta t}\left(\left\{\widehat{R}_{\eta, \Delta t}\left(X_0^{\eta, \Delta t}\right) - \widehat{R}_{0, \Delta t}\left(Y_0^{0, \Delta t}\right)\right\}^2\right)\right)\\
		&= \frac{4}{\eta^2}\mathbb{E}_{\mu_{\eta, \Delta t}}\left(\left\{\widehat{R}_{\eta, \Delta t}\left(X_0^{\eta, \Delta t}\right)-\widehat{R}_{0, \Delta t}\left(Y_0^{0, \Delta t}\right)\right\}^2\right)\\
		&= \frac{4}{\eta^2}\int_{\mathbb{R}^d \times \mathbb{R}^d}\left(\widehat{R}_{\eta, \Delta t}(x) - \widehat{R}_{0, \Delta t}(y)\right)^2 \mu_{\eta, \Delta t}\left(dx\, dy\right)\\
		&\leq \frac{8}{\eta^2}\int_{\mathbb{R}^d \times \mathbb{R}^d}\left(\widehat{R}_{\eta, \Delta t}(x) - \widehat{R}_{0, \Delta t}(x)\right)^2 \mu_{\eta, \Delta t}\left(dx\, dy\right) + \frac{8}{\eta^2}\int_{\mathbb{R}^d \times \mathbb{R}^d}\left(\widehat{R}_{0, \Delta t}(x) - \widehat{R}_{0, \Delta t}(y)\right)^2 \mu_{\eta, \Delta t}\left(dx\, dy\right).
	\end{aligned}
\end{equation*}
The second inequality is due to Jensen's inequality and the subsequent equality is due to stationarity. Lemma~\ref{lm:disc_poisson_sol_approx} implies that the integrand of the first integral in the last line is of order \(\eta^2 + \Delta t^{4n}\) for any \(n \in \mathbb{N}\). Since \(R \in \mathscr{S} \subset B_{\widetilde{V}_c}^\infty\) and \(\Delta t^{-1}\left(\mathrm{Id} - P^{0, \Delta t}\right)\) has a bounded inverse on \(\Pi_0 B_{\widetilde{V}_c}^\infty\), one has \(\widehat{R}_{0, \Delta t} \in B_{\widetilde{V}_c}^\infty\). Consequently, we can use \eqref{eq:W_n_bound} in Lemma~\ref{lm:W_n_bound} to control the second integral as 
\begin{equation*}
	\int_{\mathbb{R}^d \times \mathbb{R}^d}\left(\widehat{R}_{0, \Delta t}(x) - \widehat{R}_{0, \Delta t}(y)\right)^2 \mu_{\eta, \Delta t}\left(dx\, dy\right) \leq C\int_{\mathbb{R}^d \times \mathbb{R}^d}\mathbf{1}_{\left\{x \neq y\right\}}\left(\widetilde{V}_c\left(x\right) + \widetilde{V}_c\left(y\right)\right) \mu_{\eta, \Delta t}\left(dx\, dy\right).
\end{equation*}
Proposition~\ref{prop:discrete_Wn_norm_bound} lets us control the latter right hand side by \(C\eta\). Therefore, putting these bounds together leads to \eqref{eq:sticky_var}.
\begin{flushright}
	\qedsymbol
\end{flushright}
\section{Numerical Results}\label{sec:numerics}
We present in this section some of the results of our numerical investigations of the various coupling strategies. In our numerical experiments, we restrict ourselves to  drifts of the form \(b = -\nabla V\). We are interested in the distribution of coupling distances for each coupling method and their performance on a few representative observables. In Section 5.1, we briefly describe the numerical scheme used to simulate our examples. In Section 5.2, we present simulation results for two simple two-dimensional examples. In Section 5.3, we present simulation results from a more involved example---a cluster of Lennard--Jones particles.

\subsection{Numerical Schemes}
For the discrete-time perturbed process and the discrete-time synchronously coupled process we consider Euler--Maruyama discretizations of \eqref{eq:sde_model}:
\begin{equation}
	X^{\eta, \Delta t}_{k+1} = X_k^{\eta, \Delta t} + \Delta t \left(b\left(X_k^{\eta, \Delta t}\right) + \eta F\left(X_k^{\eta, \Delta t}\right)\right) + \sqrt{\frac{2\Delta t}{\beta}}G_{k+1}, 
\end{equation}
and of \eqref{eq:coupled_dynamics}:
\begin{equation}\label{eq:disc_sync_dynamics}
	\begin{aligned}
		X^{\eta, \Delta t}_{k+1} &= X_k^{\eta, \Delta t} + \Delta t \left(b\left(X_k^{\eta, \Delta t}\right) + \eta F\left(X_k^{\eta, \Delta t}\right)\right) + \sqrt{\frac{2\Delta t}{\beta}}G_{k+1},\\
		Y^{0, \Delta t}_{k+1} &= Y_k^{0, \Delta t} + \Delta t b\left(Y_k^{0, \Delta t}\right) + \sqrt{\frac{2\Delta t}{\beta}}G_{k+1}.
	\end{aligned}
\end{equation}
Here, as in Section~\ref{subsec:disc_dynamics}, \(\left(G_k\right)_{k\geq 1}\) is an i.i.d. sequence of standard \(d\)-dimensional Gaussian random variables. Importantly, in \eqref{eq:disc_sync_dynamics}, the same Gaussian noise is driving the two marginals. For the sticky coupled process, we directly simulate the discrete-time sticky coupled dynamics presented in Section~\ref{subsec:sticky_coupling}. 

The discretization of the standard NEMD estimator \eqref{eq:naive_estimator} is
\begin{equation}
	\widehat{\Phi}_{\eta, N}^{\Delta t} = \frac{1}{\eta N}\sum_{n=1}^NR\left(X_n^{\eta, \Delta t}\right),
\end{equation}
and the discretization of the synchronous coupling based estimator is
\begin{equation} \label{eqn:disc_time_coupled_estim}
	\widehat{\Psi}_{\eta, N}^{\Delta t, \mathrm{sync}} = \frac{1}{\eta N}\sum_{n =1}^N \left[R\left(\hat{X}_{n}^{\eta, \Delta t}\right) - R\left(\hat{Y}_{n}^{0, \Delta t}\right)\right], 
\end{equation}
where \(\left(\hat{X}_{n}^{\eta, \Delta t}, \hat{Y}_{n}^{0, \Delta t}\right)_{n \in \mathbb{N}}\) is evolving according to \eqref{eq:disc_sync_dynamics}.

The coupling distance, i.e. the Euclidean distance between the two coupled trajectories, is a proxy for the performance of the coupling methods as a simple calculation shows that, for any response function~\(R\) with bounded first derivatives,
\begin{equation}
	\mathrm{Var}\left(\Psi_{\eta, N}^{\Delta t}\right) \leq \frac{\left\|\nabla R\right\|_{\infty}^2}{\eta^2} \mathbb{E}\left[\frac{1}{ N}\sum^N_{n = 1} \left|\hat{X}_{n}^{\eta, \Delta t}- \hat{Y}_{n}^{0, \Delta t}\right|^2\right].
\end{equation}

\subsection{Two Dimensional Toy examples}
\subsubsection{Harmonic Potential}\label{subsec:numerics_toy_convex}
We first consider a strongly convex potential
\begin{equation}\label{eqn:quadpot}
	U(x) = \frac{\left|x\right|^2}{2},
\end{equation}
perturbed by a linear shearing force
\begin{equation}\label{eqn:lin_forcing}
	F_1(x) = \left[
	\begin{matrix*}
		x_2\\
		0
	\end{matrix*}
	\right].
\end{equation}
To ensure that (\ref{eq:F_bounded}) is satisfied we could multiply \(F_1\) by \(\chi \in C_{\mathrm{c}}^\infty\left(\mathbb{R}^d\right)\) such that \(\chi \equiv 1\) on \(\left\{x : \|x \| \leq B\right\}\) and \(\chi \equiv 0\) on \(\left\{x : \|x \| \geq B + \varepsilon\right\}\) for some \(\varepsilon > 0\) and very large \(B > 0\). This satisfies the boundedness assumption without having a practical effect on the simulation. As a response function, we consider \(R_{\mathrm{cov}}\left(x_1, x_2\right) = x_1 x_2\), i.e. our observable is the covariance between first and second components. For the situation at hand, the process \eqref{eq:sde_model} is a Gaussian process, so that we can explicitly compute the covariance between first and second components and thereby the linear response \(\alpha_{R_{\mathrm{cov}}}\). More precisely,
\[dX_t^\eta = 
-\begin{bmatrix*}
	1 & -\eta \\
	0 & 1
\end{bmatrix*} 
X_t^\eta dt + \sqrt{\frac{2}{\beta}}dW_t,\]
is an Ornstein--Uhlenbeck process with stationary distribution \(\mathcal{N}\left(0, \Sigma\right)\), where \(\Sigma\) satisfies
\[\begin{bmatrix*}
	1 & -\eta \\
	0 & 1
\end{bmatrix*} \Sigma + \Sigma \begin{bmatrix*}
	1 & 0 \\
	-\eta & 1
\end{bmatrix*}  = \frac{2}{\beta}\mathrm{Id} . \]
A simple calculation shows that
\begin{equation}\label{eqn:Sigma}
	\Sigma = \frac{1}{2\beta}\begin{bmatrix*}
		2 + \eta^2 & \eta \\
		\eta & 2
	\end{bmatrix*},
\end{equation}
and therefore \(\alpha_{R_{\mathrm{cov}}} = \left(2\beta\right)^{-1}\).

To study numerically the behavior of the variance as \(\eta \to 0\), for each \(\eta \in \left\{0.1, 0.05, 0.025, 0.01, 0.005\right\}\) we perform 500 realizations of the synchronously coupled process and of the sticky coupled process and compute the empirical variance of \(\hat{\Psi}_{\eta, N}^{\Delta t}\) over the realizations. All realizations were run at inverse temperature \(\beta = 1\) and with time step \(\Delta t = 0.005\). Each realization was "burned-in" with the equilbrium dynamics for \(N_{\mathrm{Burn}} = T_{\mathrm{Burn}}/\Delta t\) steps with  \(T_{\mathrm{Burn}} = 10^5\) and then trajectory was simulated up to \(N = T/\Delta t\) steps with \(T = 10^6\). 

\begin{figure}[h]
	\centering
	\includegraphics[width = 0.49\textwidth]{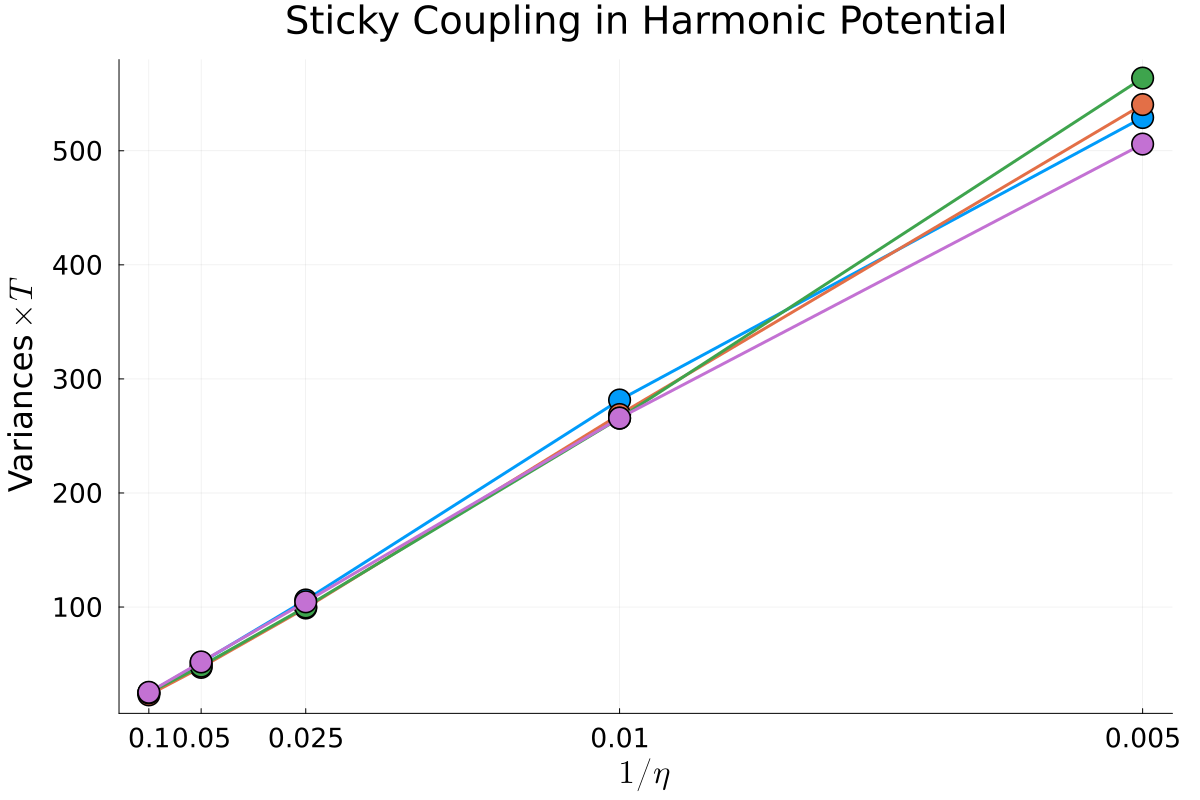}
	\includegraphics[width = 0.49\textwidth]{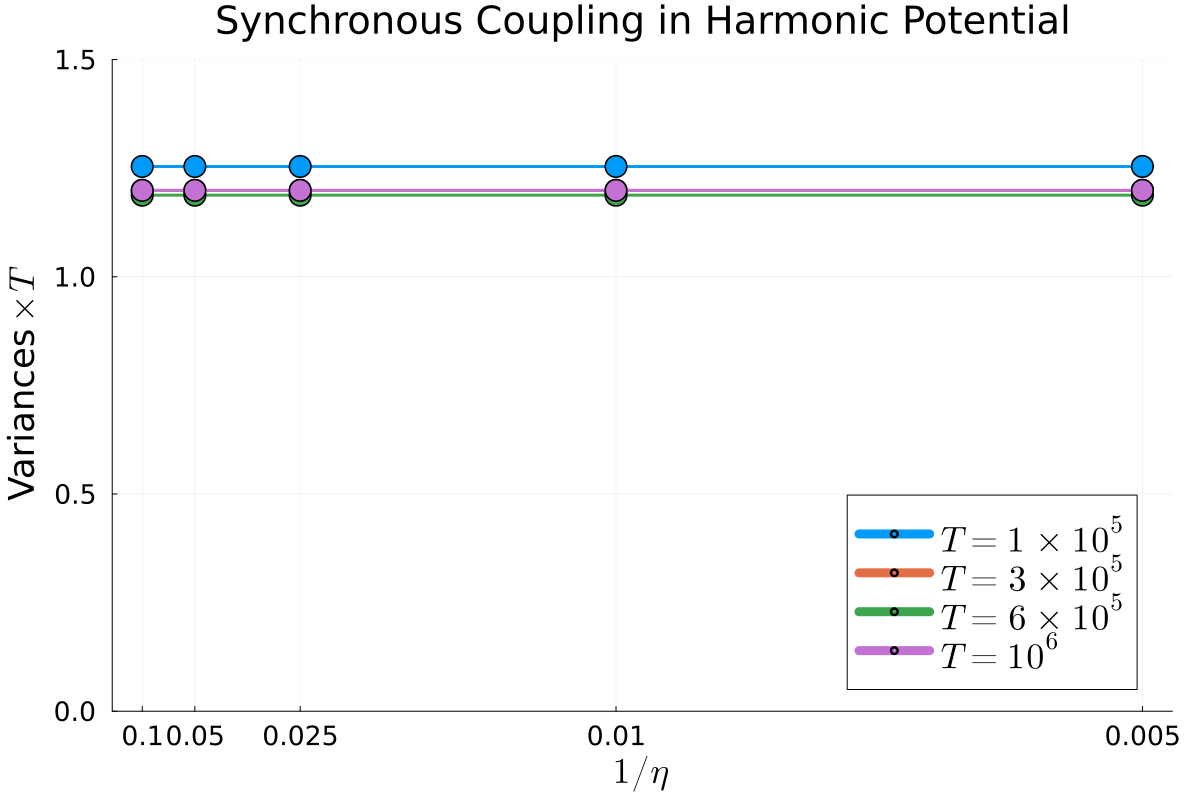}
	\caption{Variance of the coupled estimators for the two-dimensional strongly convex potential \eqref{eqn:quadpot}. We highlight the different scales of the \(y\)-axis in the two plots.}
	\label{fig:convex_coupling_corr}
\end{figure}
For the strongly convex potential \eqref{eqn:quadpot}, we see in Figure~\ref{fig:convex_coupling_corr} that the estimator based on synchronous coupling remains bounded as \(\eta \to 0\) while the estimator based on sticky coupling grows like \(1/\eta\) as predicted by Theorems~\ref{thm:sync_bias_var} and \ref{thm:sticky_bias_var}. The difference in variance between the two coupling methods shows itself when looking at how fast the estimators converge to the analytic value of \(\alpha_{R_{\mathrm{cov}}}\) in Figure~\ref{fig:convex_coupling_corr_convergence}. As expected, the synchronously coupled estimator's rate of convergence does not appear to worsen with smaller \(\eta\). In comparison, we can see that the sticky coupled estimator needs more time to converge as~\(\eta\) gets smaller.

\begin{figure}[h]
	\centering
	\includegraphics[width = 0.49\textwidth]{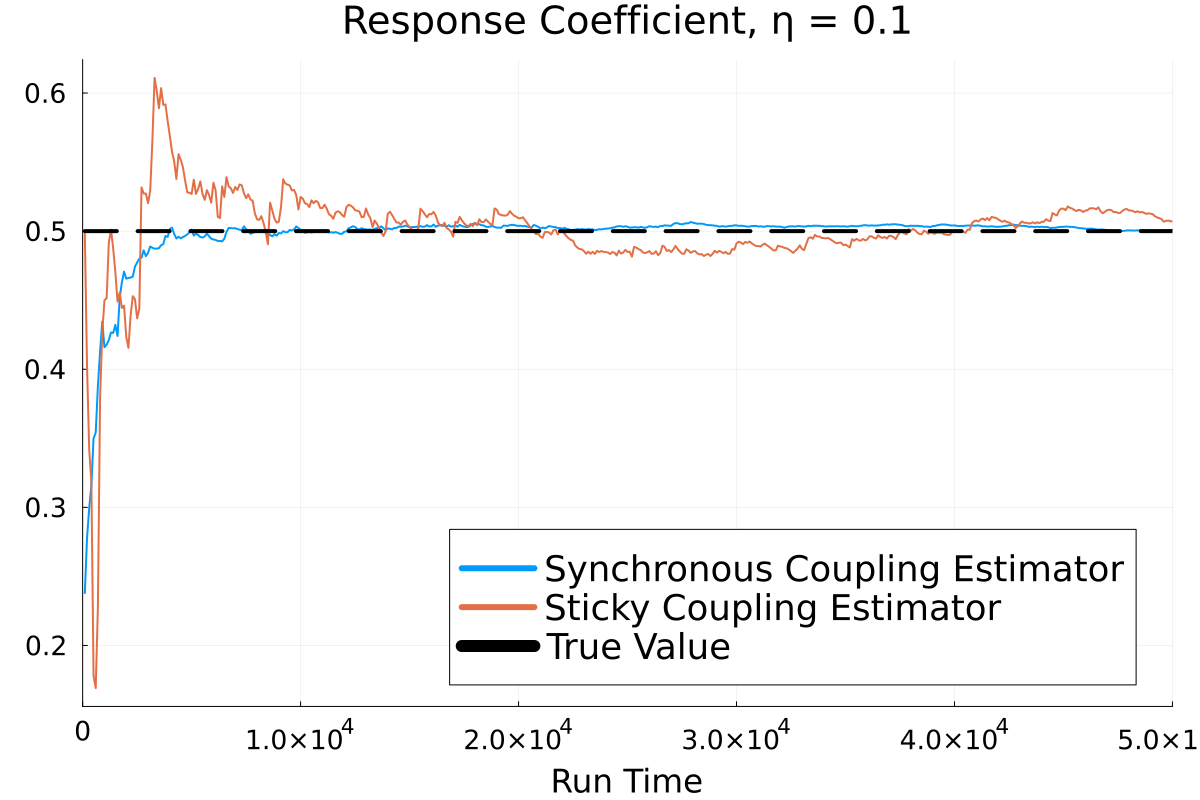}
	\includegraphics[width = 0.49\textwidth]{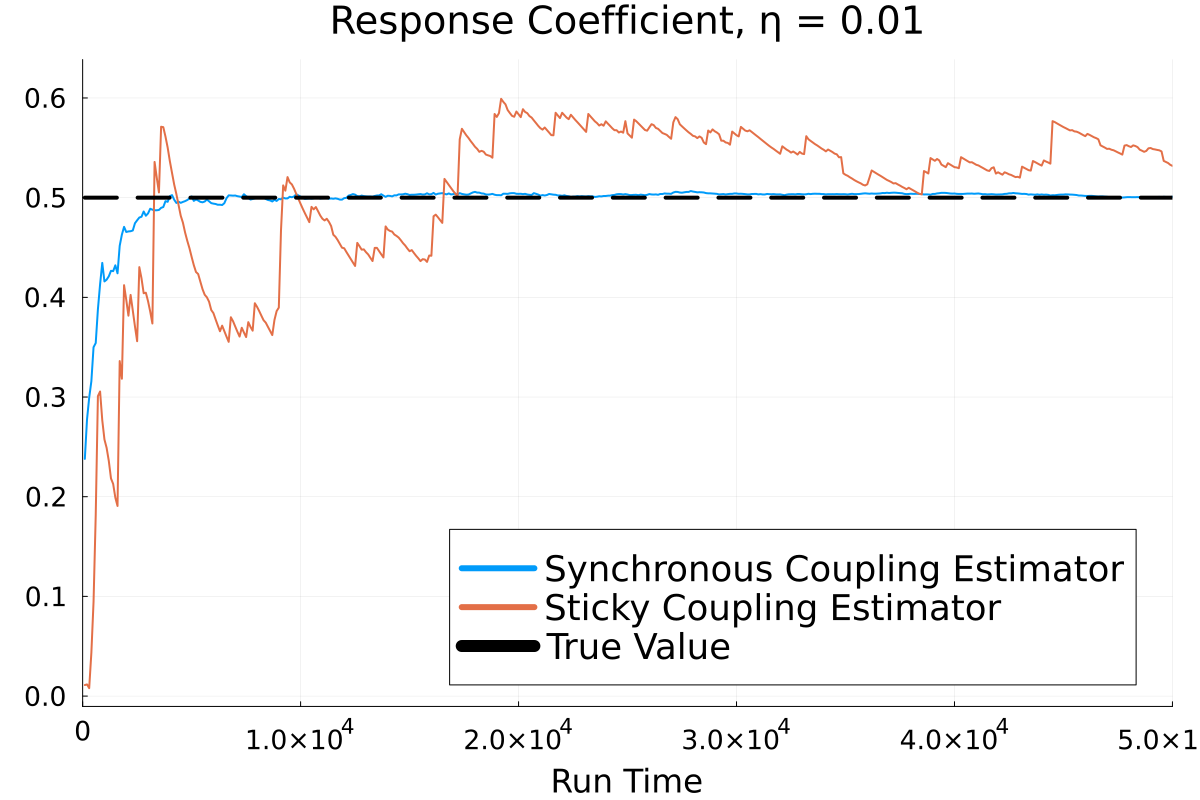}
	\caption{Convergence of coupled estimators to true value of \(\alpha_{R_{\mathrm{cov}}}\).}
	\label{fig:convex_coupling_corr_convergence}
\end{figure}

\subsubsection{Non-Convex Potential}
For our second toy example, we consider a non-convex potential that behaves like the product of cosine functions inside \(\left[-L, L\right]^2\) and like a quadratic function outside:
\begin{equation}
	U(x) = 
	\begin{cases}
		\left(1 - \cos\left(\frac{2\pi x_1}{L}\right)\right)\left(1 - \cos\left(\frac{2\pi x_2}{L}\right)\right), & \left\|x\right\|_\infty < L,\\ 
		\displaystyle \frac{1}{2}\left[\max\left\{0, \left|x_1\right| - L\right\}^2 + \max\left\{0, \left|x_2\right| - L\right\}^2\right], & \left\|x\right\|_\infty \geq L.
	\end{cases}
\end{equation}
For the non-gradient forcing, we consider a sinusoidal shear forcing
\begin{equation}\label{eqn:sin_forcing}
	F_2(x) = \left[
	\begin{matrix*}
		\sin{\left(x_2\right)}\\
		0
	\end{matrix*}
	\right].
\end{equation}
We use the same observable as in the convex case, \(R_{\mathrm{cov}}(x_1, x_2) = x_1x_2\). Since the potential is even in each component, \(R_{\mathrm{cov}}\in \mathscr{S}_0\) in this case as well. We perform 500 realizations of the synchronously coupled process and of the sticky coupled process and compute the empirical variance of \(\widehat{\Psi}_{\eta, N}^{\Delta t}\) with same time step, inverse temperature, burn-in time, and run time as the convex case of Section~\ref{subsec:numerics_toy_convex}.  

In the nonconvex case, we see in Figure~\ref{fig:nonconvex_coupling_dist} that for both coupling methods the variance of the coupling distance grows like \(\frac{1}{\eta}\). Outside the case of a strongly convex potential, the variance of the synchronously coupled estimator is no longer bounded as \(\eta \to 0\). 

\begin{figure}[h]
	\centering
	\includegraphics[width = 0.49\textwidth]{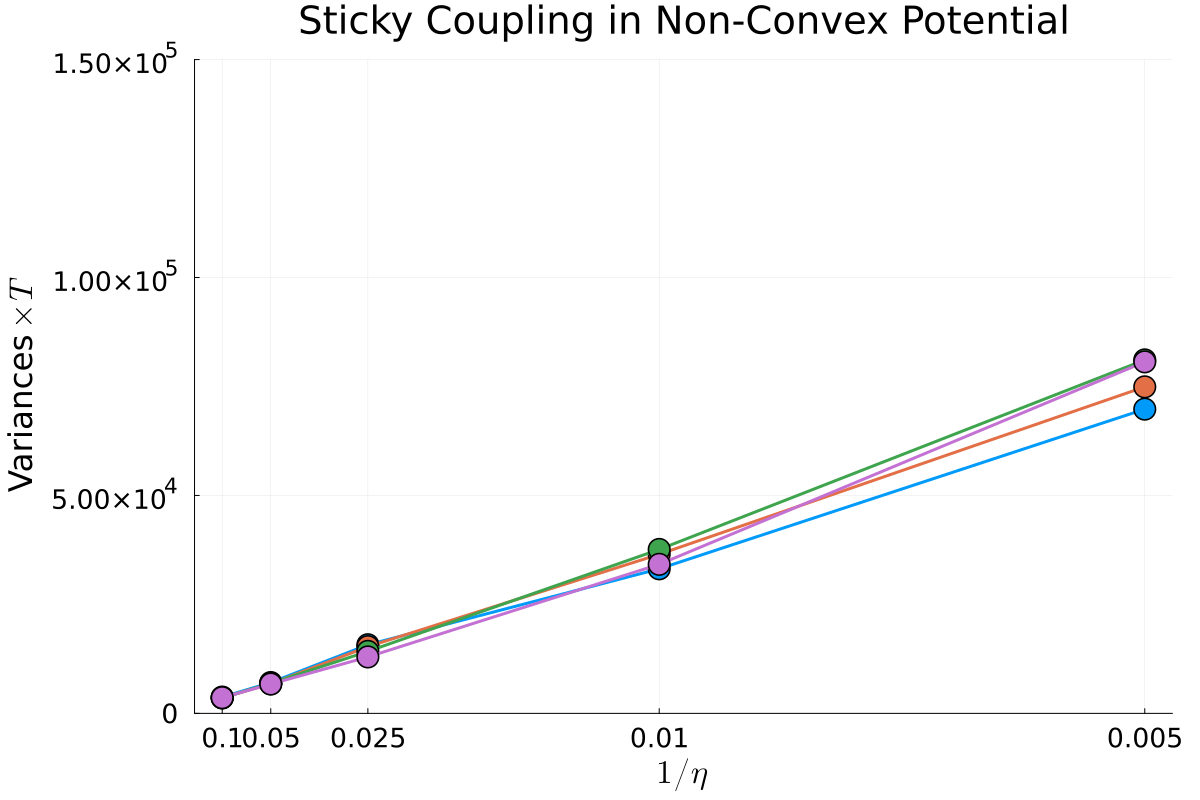}
	\includegraphics[width = 0.49\textwidth]{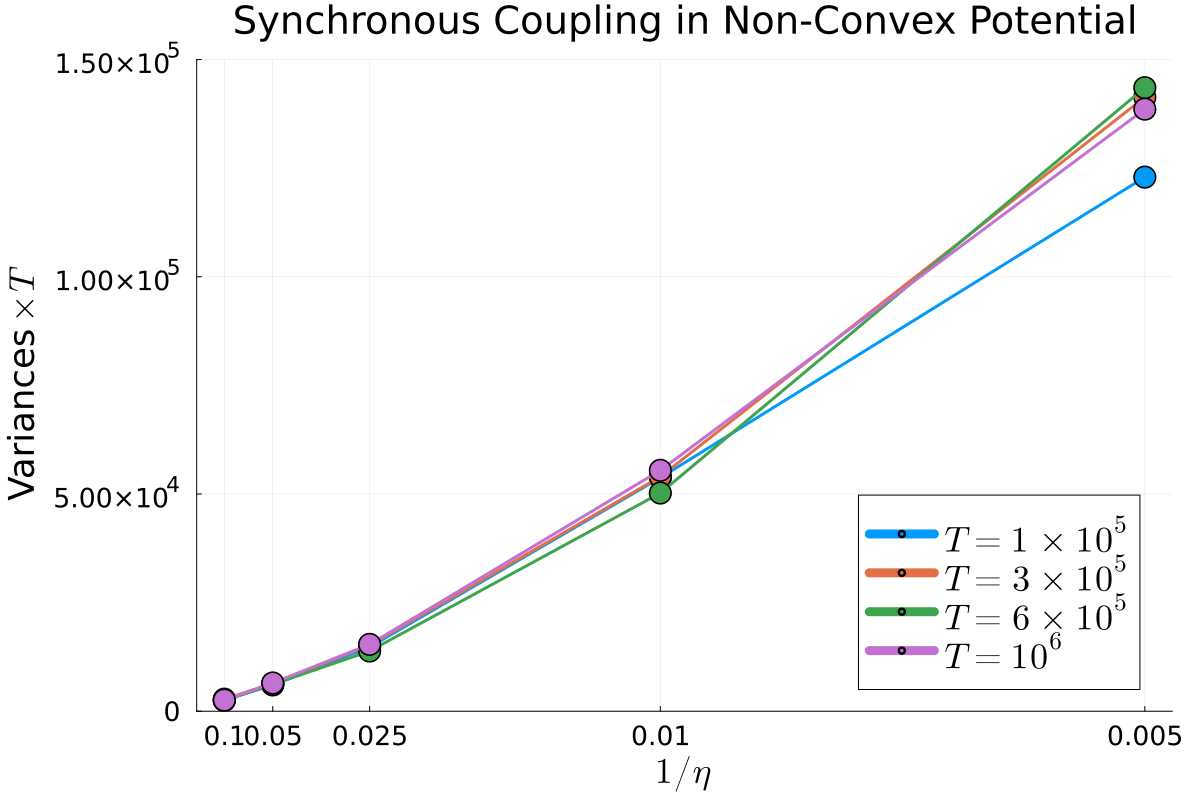}
	\caption{Variance of the coupled estimators in the two-dimensional potential which is only strongly convex outside of a compact set.}
	\label{fig:nonconvex_coupling_dist}
\end{figure}

Despite the fact that the bound on the variance in Theorem~\ref{thm:sync_bias_var} fails, the variance of the synchronously coupling based estimator is still comparable to that of the sticky coupling based estimator. This is emblematic of something we observed in our numerical investigations---in low dimensions synchronous coupling typically remains competitive with sticky coupling. We believe that this is an effect of the synchronous coupling's performance not being too badly harmed by non-convexity in low dimensions while the reflection coupling part of sticky coupling tends to increase the variance.

\subsection{Lennard-Jones Clusters}
A Lennard--Jones cluster is a collection of atoms which interact only in a pairwise manner through the potential
\[v(r) = 4\epsilon\left[\left(\frac{\sigma}{r}\right)^{12} - \left(\frac{\sigma}{r}\right)^{6}\right], \]
where \(\sigma > 0\) is some reference distance and \(\epsilon > 0\) a reference energy. The model is commonly used as a simplified model for molecular interactions \cite{FrenkelSmit,Tuckerman}. For our simulations, we restrict ourselves to the case of particles in two dimensions. The potential is highly non-convex. To ensure that (\ref{eq:contractive_at_inf}) is satisfied, we add a quadratic potential that confines the system in a square box \(\left[-L, L\right]^2\). For 
\(x = \left(x_1^1, x_2^1, x_1^2, x_2^2, \dots, x_1^N, x_2^N\right)^T \in \mathbb{R}^{2N}\), where \(N\) is the number of particles in the cluster, the overall potential of the system is
\begin{equation}\label{eq:lj_pot}
	U(x) = \sum_{0\leq i < j \leq N} v\left(\left|x^i - x^j\right|\right) + \frac{\alpha}{2}\sum_{i = 1}^N \left[\max\left\{0, \left|x^i_1\right| - L\right\}^2 + \max\left\{0, \left|x^i_2\right| - L\right\}^2 \right],
\end{equation}
where \(x_0 \equiv 0\) is an anchor point and \(\alpha \geq 0\) is the strength of the confining potential. For all the simulations \(\alpha = 1\), \(L = 5\), \(\epsilon = 1\), and \(\sigma = 2^{-1/6}\). For the non-gradient forcing we consider a shear force~\(F\) acting only in the \(x_1\) direction, given componentwise by \(F_{2i - 1}(x) = \sin\left(\frac{x_2^i\pi}{L}\right)\) and \(F_{2i}(x) = 0\). We ran simulations of \(N = 18\) particles at several temperatures up to time \(T = 2\times 10^5\) with time step \(\Delta t = 10^{-4}\).
As a response, we measure the mobility 
\[R_{\mathrm{mobility}}(x) = F(x)^T\nabla U(x) = \sum_{i=1}^N \sin\left(\frac{x_2^i\pi}{L}\right)\partial_{x_1^i}U(x),\]
and the tilt of the cluster
\[R_{\mathrm{tilt}}(x) = \sum_{i=1}^N \tanh\left(\frac{x_1^i}{\varepsilon}\right)\tanh\left(\frac{x_2^i}{\varepsilon}\right),\] 
which can be seen as a regularization of \(R(x) = \sum_{i=1}^N\mathrm{sign}\left(x_1^i\right)\mathrm{sign}\left(x_2^i\right)\) with \(\varepsilon > 0\) the regularization parameter. In all our simulations, we choose \(\varepsilon = 1/5\).
We ran our simulations for \(\beta \in \left\{0.5, 1, 2, 4\right\}\) and \(\eta \in \left\{0.0025, 0.005, 0.01, 0.025, 0.05, 0.1, 0.25, 0.5\right\}\).

For each of the temperatures, we check that the observed linear response for the synchronously and sticky coupled systems match that of the standard NEMD system, i.e \(X^\eta\) by itself. We plot a few examples in Figures~\ref{fig:mobility_lin_resp}~and~\ref{fig:tilt_lin_resp} showing the linear response for the two observables, together with a linear fit giving an approximation of the transport coefficient \(\alpha_R\).

\begin{figure}[H]
	\begin{subfigure}{0.5\linewidth}
		\centering
		\includegraphics[width = 0.95\linewidth]{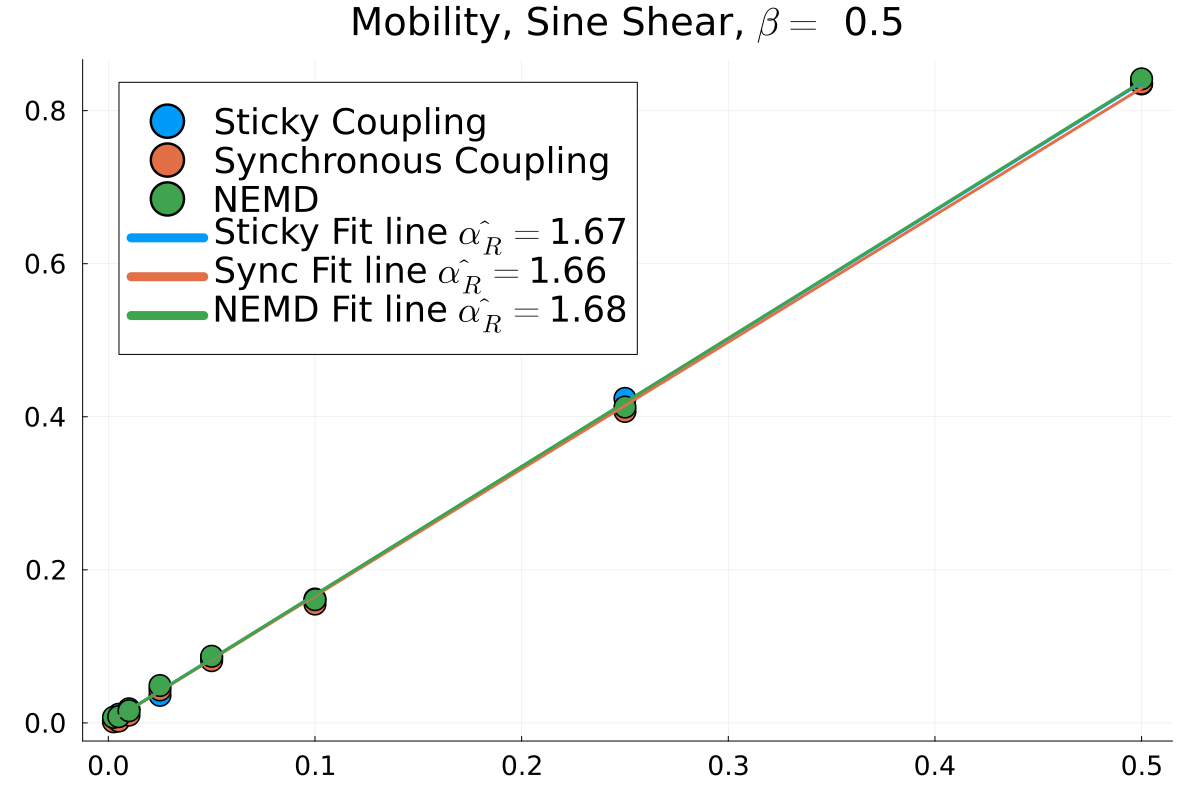}
	\end{subfigure}
	\begin{subfigure}{0.5\linewidth}
		\centering
		\includegraphics[width = 0.95\linewidth]{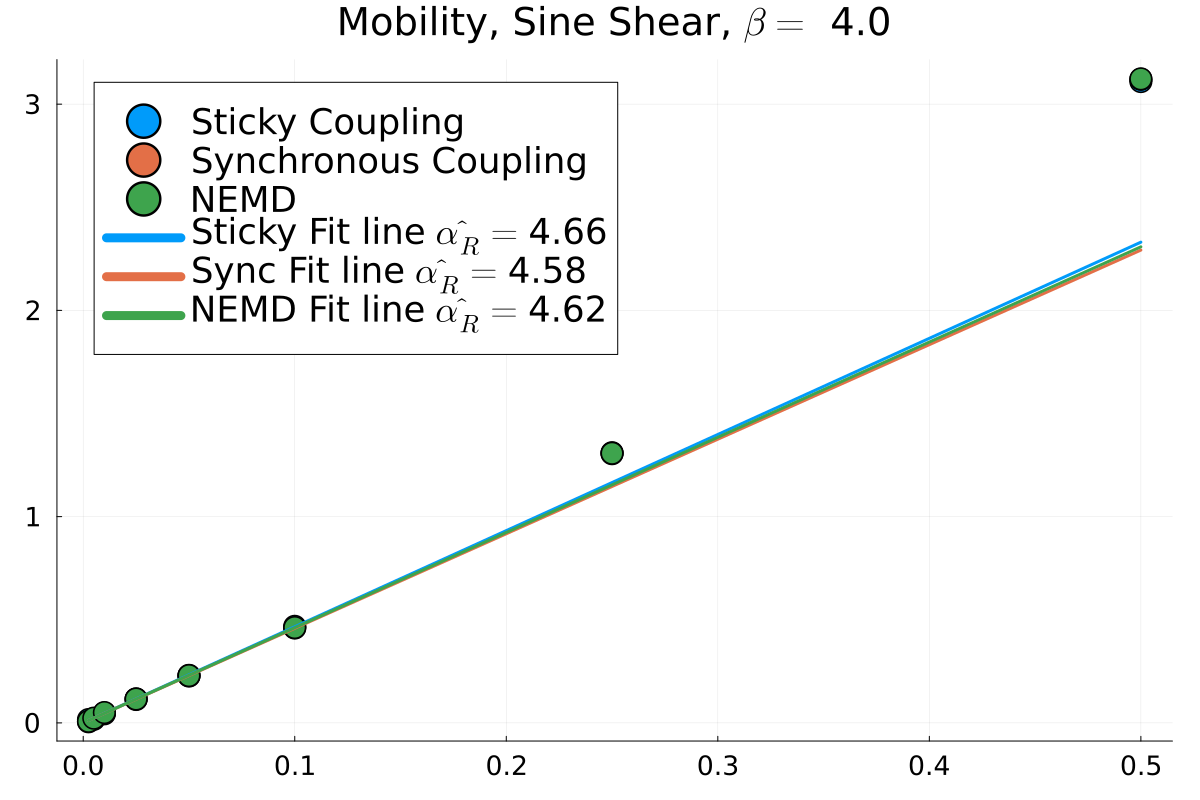}
	\end{subfigure}
		\caption{Observed mobility response with respect to \(\eta\) of the coupled and standard NEMD systems.}\label{fig:mobility_lin_resp}
\end{figure}

\begin{figure}[H]
	\begin{subfigure}{0.5\linewidth}
		\centering
		\includegraphics[width = 0.95\linewidth]{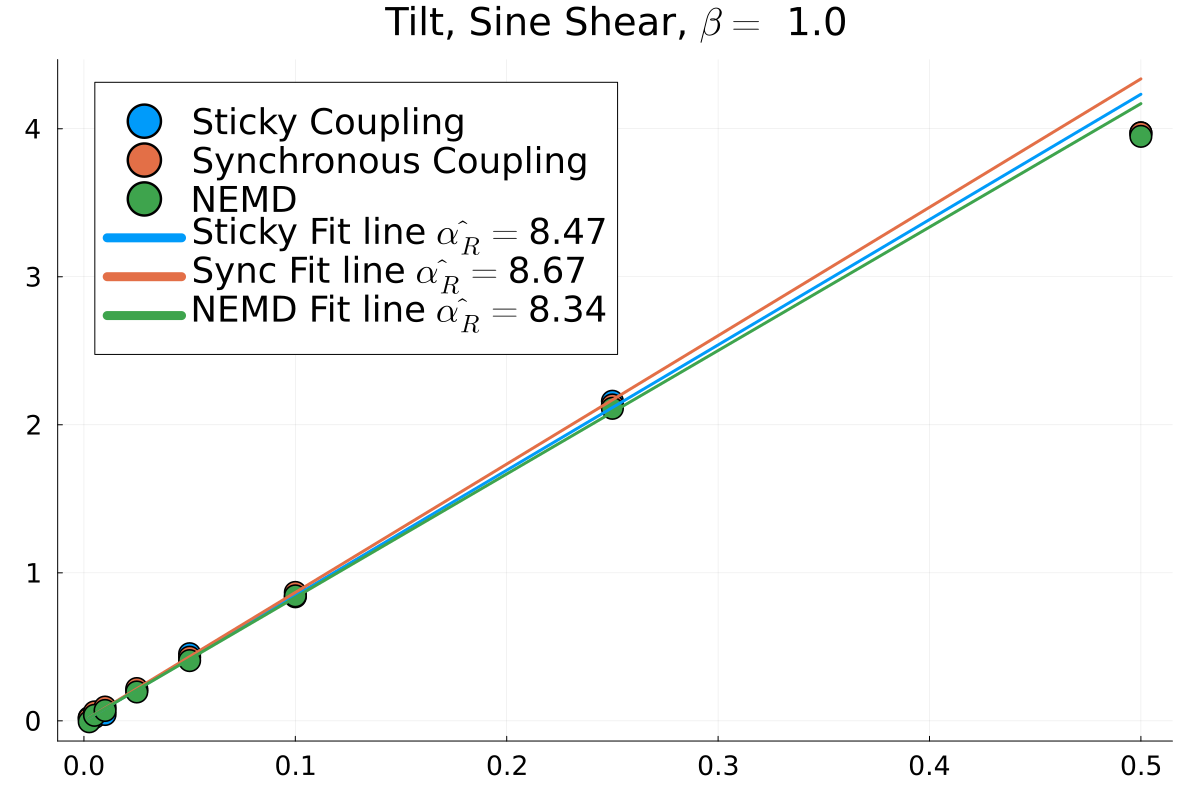}
	\end{subfigure}
	\begin{subfigure}{0.5\linewidth}
		\centering
		\includegraphics[width = 0.95\linewidth]{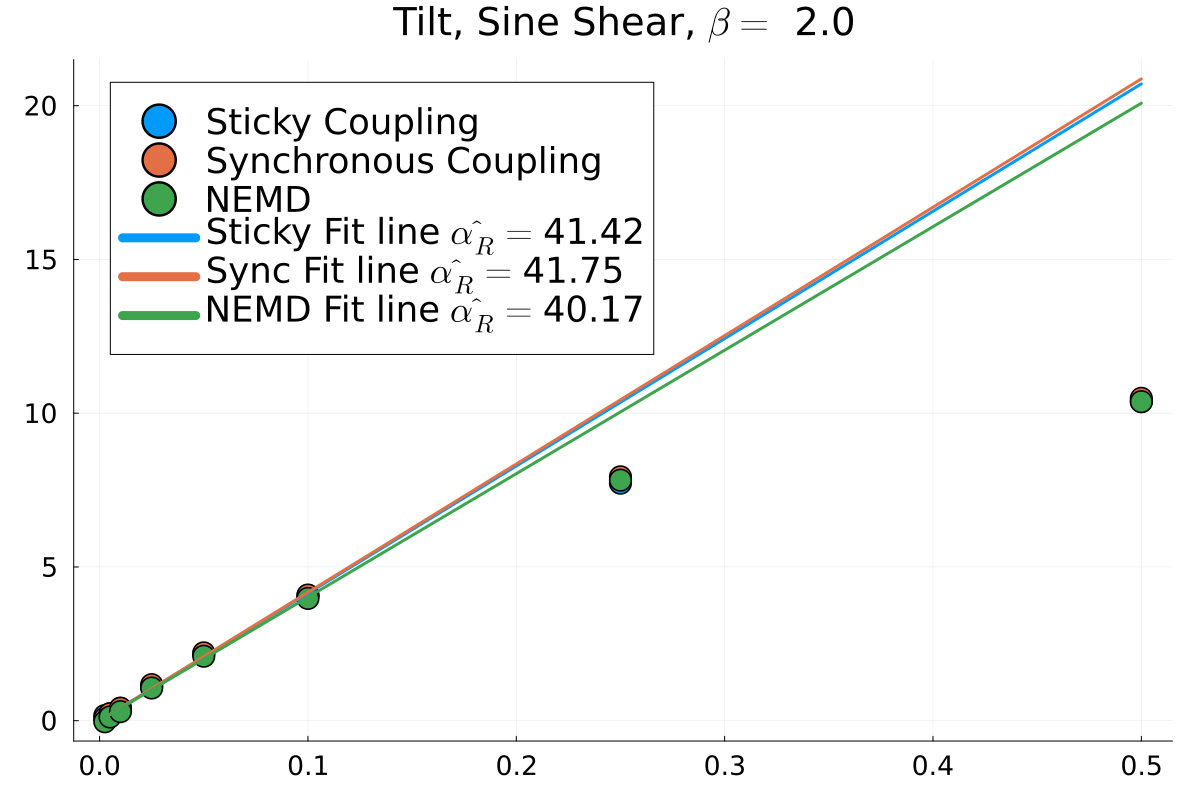}
	\end{subfigure}
	\caption{Observed tilt response with respect to \(\eta\) of the coupled and standard NEMD systems.}\label{fig:tilt_lin_resp}
\end{figure}

In Figure~\ref{fig:mobility_var}, we plot the variance of the summands of the estimators \(\widehat{\Psi}_{\eta, N}^{\Delta t}\) and \(\widehat{\Phi}_{\eta, N}^{\Delta t,}\) for high temperatures \(\beta = 0.5\) and \(1\). In this high temperature regime, the variance of NEMD trajectories and the synchronously coupled trajectories are roughly constant as \(\eta\) goes to zero, the synchronously coupled trajectories having roughly doubled the variance of the NEMD trajectories. Consequently, the variance of the estimators diverge like \(1/\eta^2\) due to the division by \(\eta\). Since trajectories have separated and there is not enough contractivity to bring them back together, the inclusion of the reference dynamics in synchronous coupling based estimator only increases the variance of the estimator.
\begin{figure}[H]
	\begin{subfigure}{0.5\linewidth}
		\includegraphics[width = 0.95\linewidth]{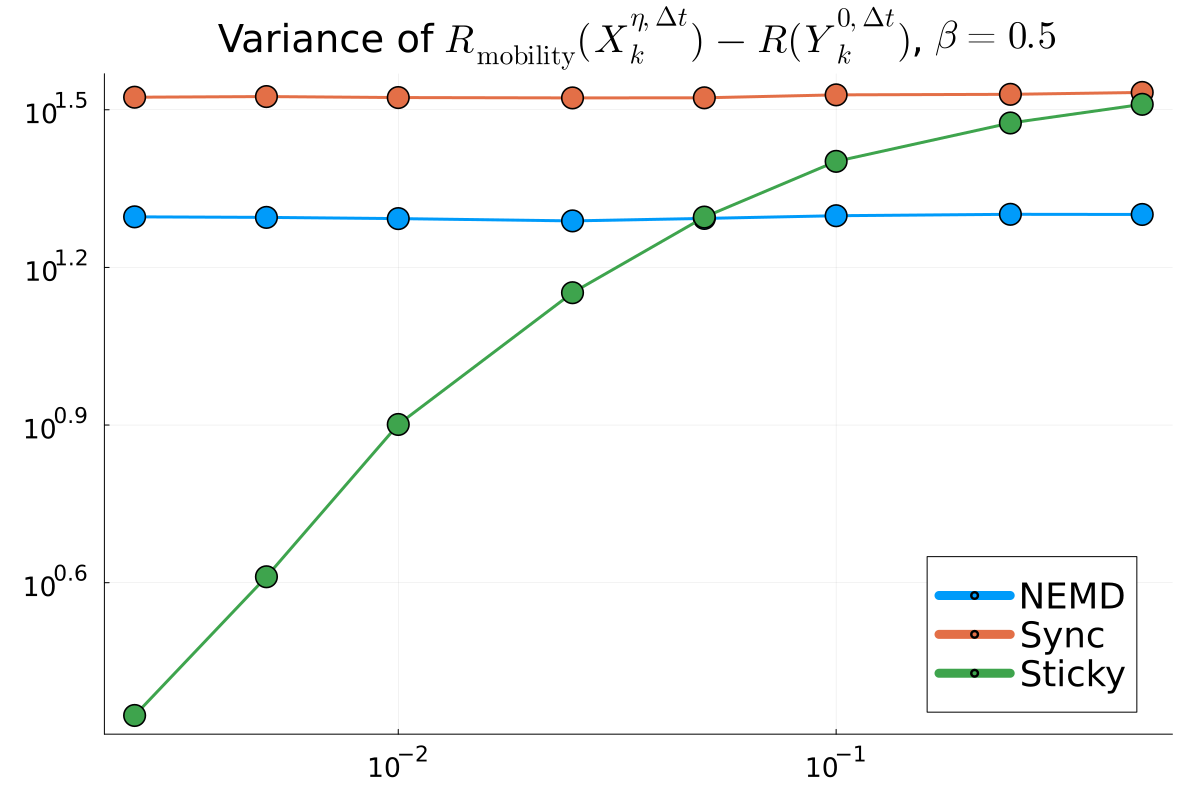}
	\end{subfigure}
	\begin{subfigure}{0.5\linewidth}
		\includegraphics[width = 0.95\linewidth]{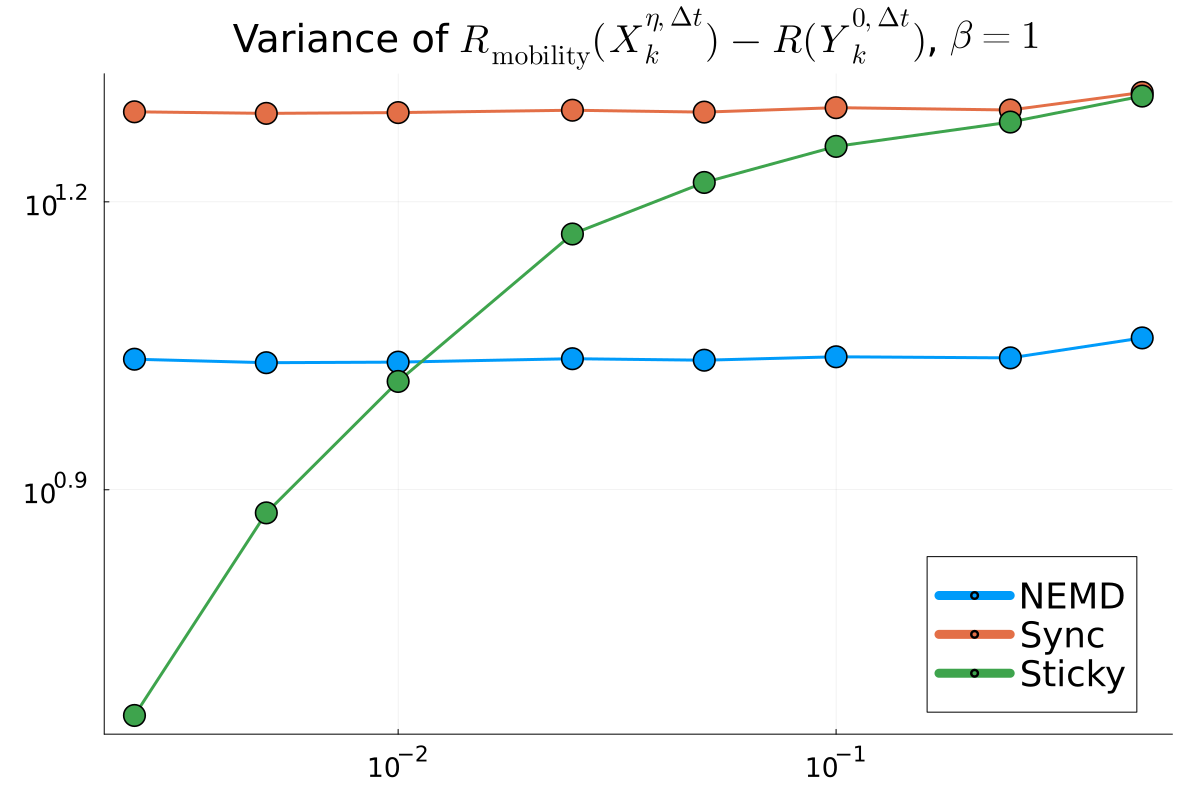}
	\end{subfigure}
	\caption{Empirically observed variance of the summands of each of the estimators with respect to \(\eta\).}\label{fig:mobility_var}
\end{figure}
In contrast, the variance of the summands of the sticky coupling based estimator is roughly proportional to \(\eta\). This is where we get \(\eta\) improvement in the variance of sticky coupling based estimator. We see the effect of this improved asymptotic variance in Figure~\ref{fig:mobility_convergence}. The fluctuations of the synchronous coupling based estimators are more violent than those of the sticky coupling based estimator as \(\eta\) gets smaller.
\begin{figure}[H]
	\includegraphics[width = \linewidth]{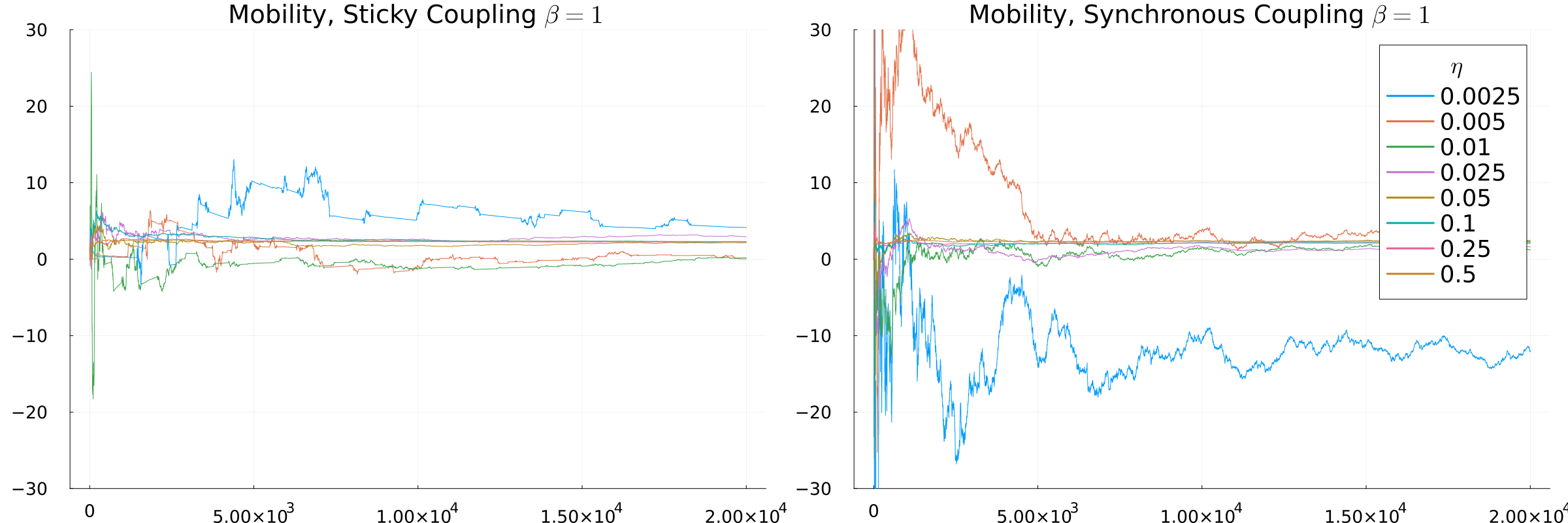}
	\caption{Convergence of the coupling based estimators for mobility response and \(\beta = 1\).}\label{fig:mobility_convergence}
\end{figure}
Similar results are obtained for the tilt response functions, see Figure~\ref{fig:tilt_var}.
\begin{figure}[H]
	\begin{subfigure}{0.5\linewidth}
		\includegraphics[width = 0.95\linewidth]{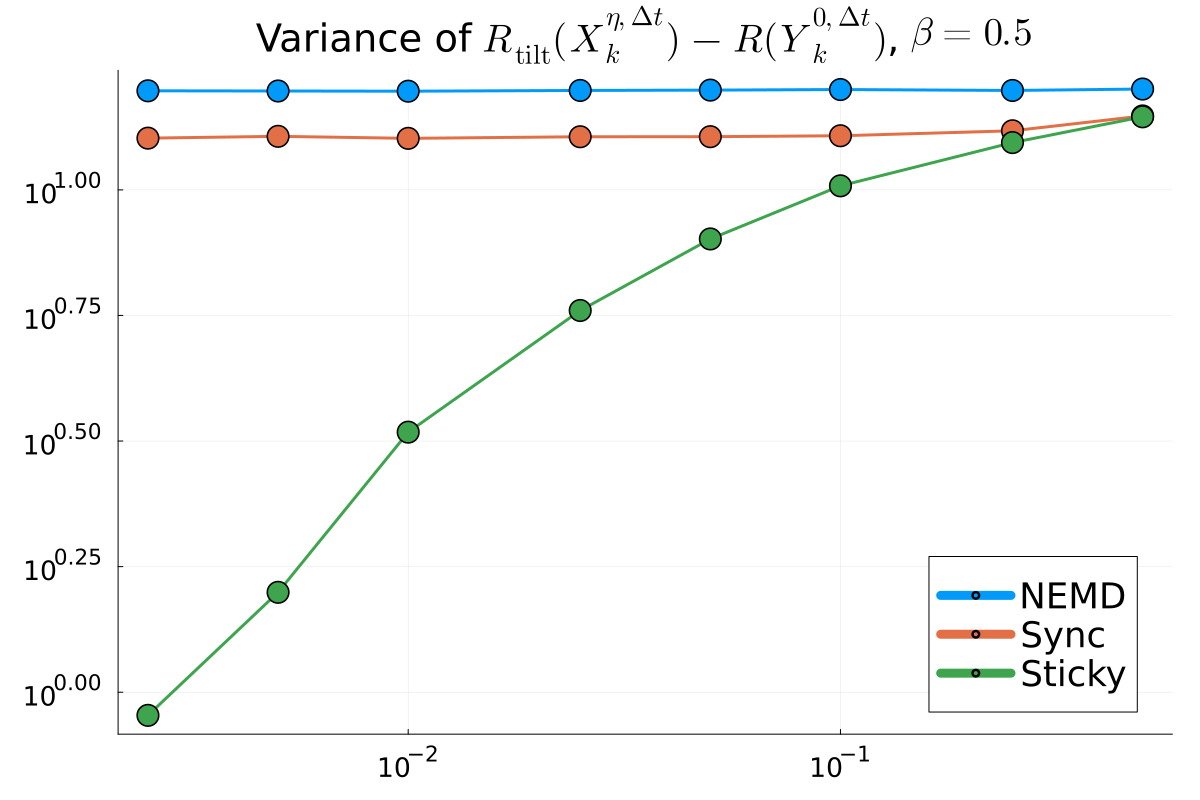}
	\end{subfigure}
	\begin{subfigure}{0.5\linewidth}
		\includegraphics[width = 0.95\linewidth]{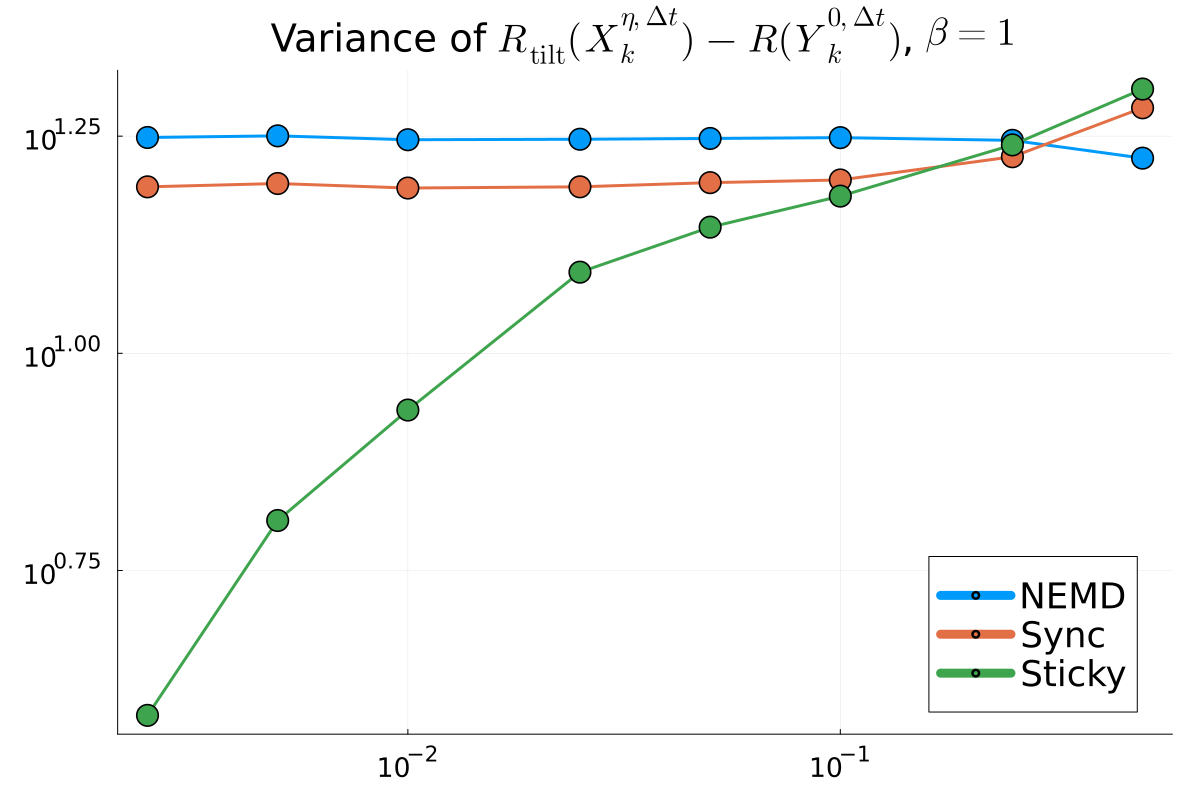}
	\end{subfigure}
	\caption{Empirically observed variance of the summands of each of the estimators with respect to \(\eta\).}\label{fig:tilt_var}
\end{figure}
Notice in Figure~\ref{fig:mobility_var} that for \(\beta = 0.5\) the decay in the variance occurs earlier and faster compared to the one for \(\beta = 1\). We expect that for lower temperatures the decay in variance occurs later. We did not observe a decay in variance for sticky coupling when \(\beta \in \left\{2, 4\right\}\) in the range of values of \(\eta\) we considered, see Figure~\ref{fig:mobility_var_lowtemp}. In this regime, the variance for the sticky coupling based estimator was roughly the same as that of the synchronous coupling based estimator, which was roughly twice the variance of the NEMD estimator.
\begin{figure}[H]
	\begin{subfigure}{0.5\linewidth}
		\includegraphics[width = 0.95\linewidth]{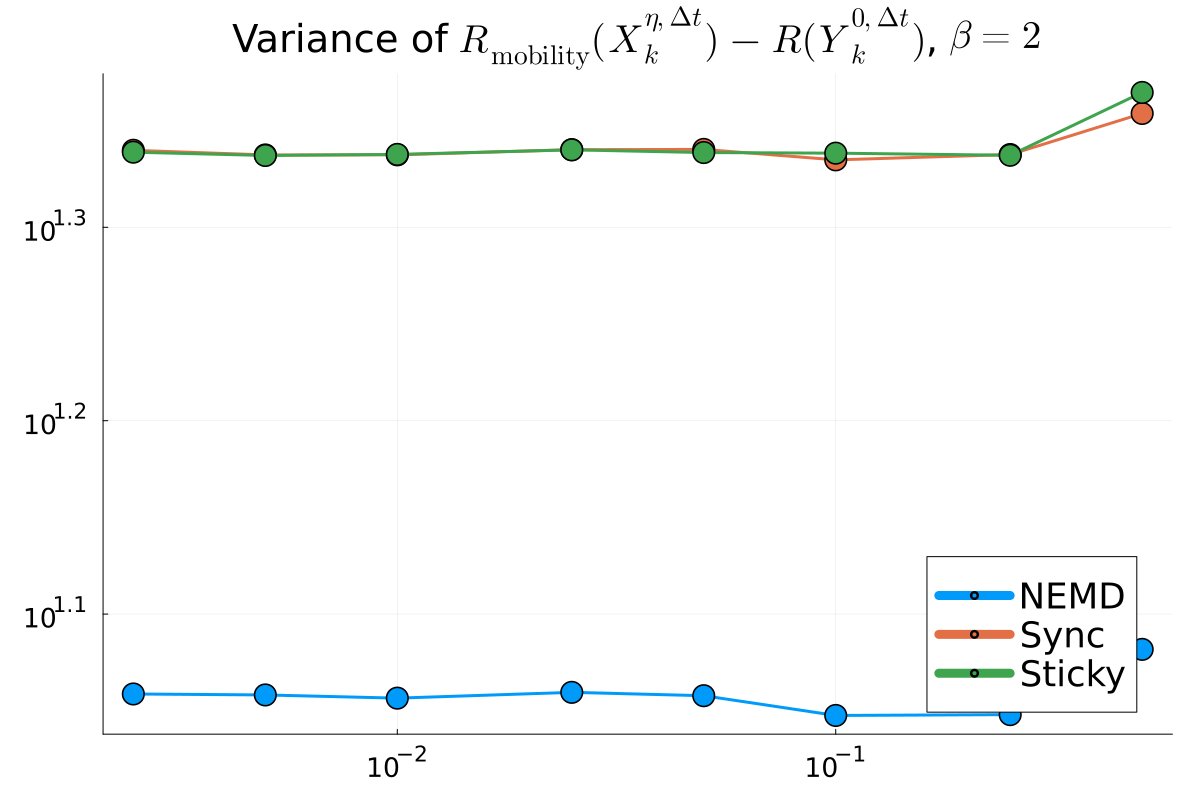}
	\end{subfigure}
	\begin{subfigure}{0.5\linewidth}
		\includegraphics[width = 0.95\linewidth]{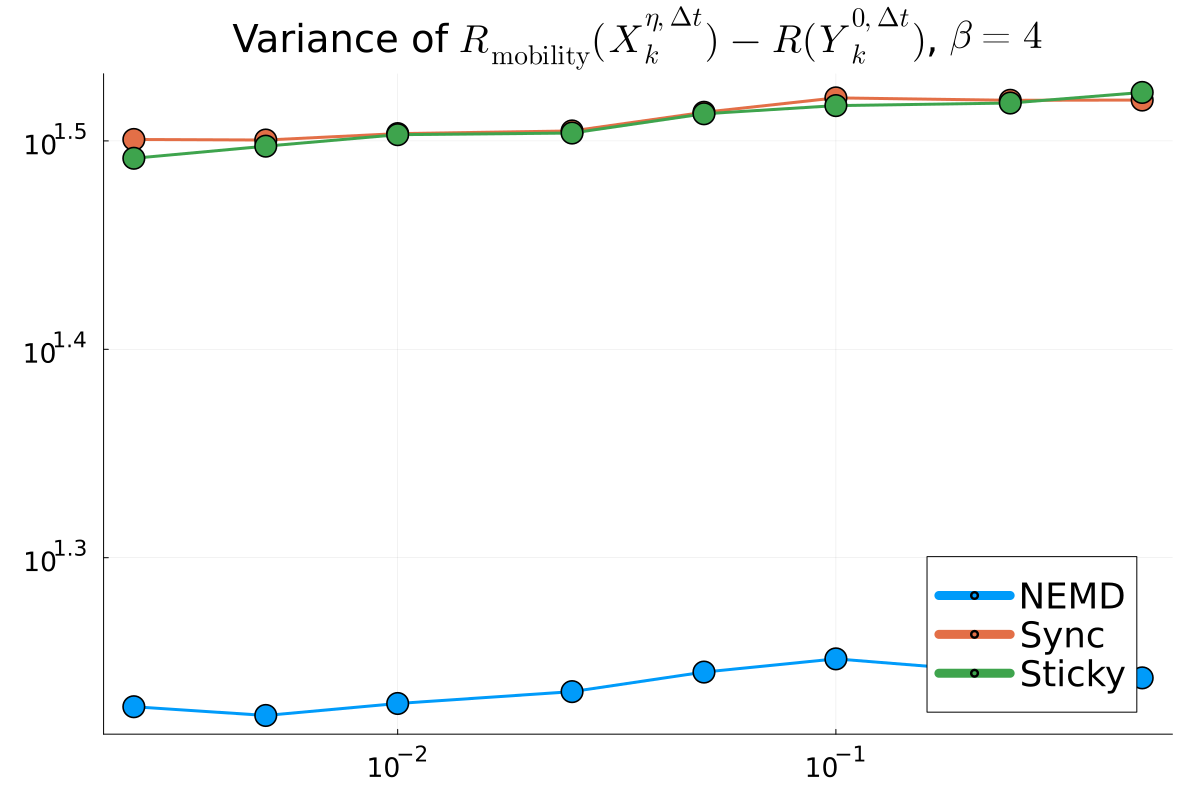}
	\end{subfigure}
	\caption{Empirically observed variance of the summands of each of the estimators with respect to \(\eta\).}\label{fig:mobility_var_lowtemp}
\end{figure}
\section{Perspectives}
We conclude this work with a brief discussion of possible extensions of the coupling methods present in this article. To improve numerical implementation, it may be advantageous to consider a hybrid coupling that mixes synchronous and sticky coupling. Such couplings have already been used in theoretical works such as \cite{EberleGuillinZimmer}, in which the two processes are synchronously coupled at large distances when the deterministic dynamics is contractive and sticky/reflectively coupled otherwise. Algorithms that use different coupling methods at long and short distances have also been suggested in \cite{LiWang}. When the deterministic part of the dynamics is contractive, synchronous coupling is clearly the best choice. In a numerical implementation of this hybrid coupling strategy, one could test at each integration step if the current position is contractive, i.e. for the current configuration \((x, y) \in \mathbb{R}^d \times \mathbb{R}^d\) we test if
\[\left\langle x - y, b(x) + \eta F(x) - b(y)\right\rangle < 0.\]
The force field has to be evaluated at each step, so if the values of \(b(x) + \eta F(x)\) and \(b(y)\) are stored in running memory, which is often done in efficient code, this test does not add extra force field evaluations. Force field evaluations represent the vast majority of the computational effort and this contractivity test only entails cheap vector addition and a scalar product.

In this work, we did not carefully study the behavior of our coupling and the resulting bounds as the dimension increases. However it seems likely that the version of sticky coupling we have presented would suffer from a curse of dimension---the overlap between two Gaussians at the typical distances in large dimensions would be exponentially small and thus so would the meeting probability. At least for a particle system such as a Lennard--Jones cluster presented in Section~\ref{sec:numerics}, coupling the whole cluster does not appear to be a good idea as a meeting event in this case is every single particle being forced to meeting its homologue in one time step. For a large number of particles, this is an exponentially rare event. One could instead use a component wise coupling as in \cite{Schuh}. If we do not need to distinguish between particle as in the case of Lennard--Jones cluster, one could also couple each particle with the nearest particle in the other cluster. Works on propagation of chaos suggest that these two types of coupling should have better scaling properties with dimension.

Another natural extension of the current work would be to the case of kinetic Langevin dynamics, that is to say hypoelliptic diffusions. The sticky coupling presented in the current work uses in an essential way the fact that our noise is elliptic, i.e. in all directions. In a recent article \cite{ChakMonmarche}, the authors present a coupling for unadjusted hybrid Monte Carlo that mixes synchronous coupling at long distances, reflection coupling at intermediate distances, and one-shot coupling at close distances that brings the trajectories together and causes an exponential contraction in expectation of the coupling distance, see also \cite{SchuhWhalley}. The hypotheses of the authors include common discretizations of underdamped Langevin dynamics. In the elliptic case \cite{EberleZimmer, Durmus_etal}, ensuring that the trajectories come back together exponentially fast seems to be the essential ingredient in making sticky coupling work and it seems reasonable to believe that the same would be true in the hypoelliptic case. 

\nocite{book}

\appendix
\section{Proof of Proposition~\ref{prop:exist_poisson_eq}}\label{sec:Kopec_ext}
For \(\varphi \in \mathscr{S}_\eta\), let the function \(u : \mathbb{R}_+ \times \mathbb{R}^d \to \mathbb{R}\) be given by
\begin{equation}
	u(t,x) = P_t^\eta \varphi(x) = \mathbb{E}_x\left[\varphi\left(X_t^\eta\right)\right].   
\end{equation}
Standard results, see for example \cite[Chapter 1]{Cerrai}, show that \(u\) is the unique solution of the Kolmogorov equation associated to \eqref{eq:sde_model}:
\begin{equation}\label{eq:kbe}
	\begin{aligned}
		&\partial_t u(t,x) = \mathcal{L}_\eta u(t,x), \qquad &&x\in \mathbb{R}^d, \, t > 0,\\
		&u(0, x) = \varphi(x), \qquad &&x\in \mathbb{R}^d.
	\end{aligned}
\end{equation}
For \(m, \ell \in \mathbb{N}\), we define the space \(C_\ell^m\) of \(m\)-times continuously differentiable functions whose first \(m\) derivatives are in \(B_\ell^\infty\):
\[C_\ell^m := \left\{\varphi \in C^m\left(\mathbb{R}^d\right) \left| \partial^k \varphi \in B_\ell^\infty, \, \forall |k| \leq m \right.\right\},\]
and equip the space with the norm
\[\left\|\varphi\right\|_{m, \ell} := \sup_{|k| \leq m} \left\|\partial^k \varphi\right\|_{\mathcal{K}_\ell}.\]
In this appendix we prove the following proposition.
\begin{proposition}\label{prop:kolmo_eq}
	Let \(\eta \in \mathbb{R}\). Suppose that Assumptions~\ref{ass:drift}~and~\ref{ass:ref_measure} hold true. 	Then, for any \(m, \ell_m \in \mathbb{N}\), there exist a non-negative integer \(s \geq \ell_m\) and strictly positive constants~\(C_m, \lambda_m > 0\) such that for any initial condition \(\varphi \in C_{\ell_m}^m \left(\mathbb{R}^d\right)\cap \mathscr{S}_\eta\) the solution \(u\) of 
	\begin{equation*}
		\begin{aligned}
			&\partial_t u(t,x) = \mathcal{L}_\eta u(t,x), \qquad &&x\in \mathbb{R}^d, \, t > 0,\\
			&u(0, x) = \varphi(x), \qquad &&x\in \mathbb{R}^d,
		\end{aligned}
	\end{equation*}
	satisfies, for any \(k \in \mathbb{N}^d\) such that \(|k| = m\), 
	\begin{equation}\label{eq:bd_kbe}
		\forall t > 0, \quad \forall x \in \mathbb{R}^d, \qquad\qquad \left|\partial^k u(t,x)\right| \leq C_m \left\|\varphi\right\|_{m, \ell_m}\mathcal{K}_s(x) \, \mathrm{e}^{-\lambda_m t}.
	\end{equation}
\end{proposition}
\noindent Proposition~\ref{prop:exist_poisson_eq} then follows as a corollary by integrating with respect to time, i.e.
\begin{equation}\label{eq:poisson_sol}
	\widetilde{\varphi}_\eta(x) := \int_0^\infty u(t,x) \, dt.
\end{equation}
The convergence of this integral and the fact that \(\widetilde{\varphi}_\eta \in \mathscr{S}_\eta\) follow from the bound \eqref{eq:bd_kbe}. 

The proof of Proposition~\ref{prop:kolmo_eq} is an extension of the proofs of Proposition~2.7 and Lemma~2.6 from~\cite{Kopec}, relying on the fact that Assumption~\ref{ass:drift} implies that, for any \(\eta \in \mathbb{R}\),
\begin{equation}\label{eq:jac_drift_spectral_bound}
	\forall x,h \in \mathbb{R}^d \qquad \nabla \left(b(x) + \eta F(x)\right)\cdot \left(h, h\right) \leq \lambda_{|\eta|} \left|h\right|^2,
\end{equation}
with \(\lambda_{|\eta|} = L_b + |\eta|L_F\) since \(\left\|\nabla b\right\| \leq L_b\) and \(\left\|\nabla F\right\| \leq L_F\). This bound is in fact uniform in \(\eta\) in the sense that for any \(\eta_{\star} > 0\) the above inequality holds for all \(\eta \in \left[-\eta_{\star}, \eta_{\star}\right]\) with constant \(\lambda_{\eta_{\star}} = L_b + \eta_{\star}L_F\). Thanks to this uniformity, we can in fact make the estimate \eqref{eq:bd_kbe} uniform in \(\eta \in \left[-\eta_\star, \eta_\star\right]\). We however do not do this as we do not need uniform estimates. 

For the entirety of this section, Assumptions~\ref{ass:drift}~and~\ref{ass:ref_measure} are assumed to hold. Furthermore for the rest of the this section, \(\eta \in \mathbb{R}\) is fixed. From herein, we use stochastic flow notation, i.e. for \(x \in \mathbb{R}^d\), we write \(\left(X_x^\eta(t)\right)_{t\geq 0}\) for the solution of \eqref{eq:sde_model} with \(X_x^\eta(0) = x\). For a \(k\)-linear form \(A\) on \(\mathbb{R}^d\) evaluated at~\(h_1, \dots, h_k \in \mathbb{R}^d\) we write \(A\cdot \left(h_1, \dots, h_k\right)\). We often identify \(1\)- and \(2\)-forms with vectors and matrices and use standard matrix/vector multiplication rules. We do this in particular for the 1- and 2-forms induced by fixing the first \(k-1\) or \(k-2\) arguments of \(k\)-linear form. In this case we use the notation \(A\cdot(h_1, \dots, h_{k-1}, \cdot)\) and  \(A\cdot(h_1, \dots, h_{k-2}, \cdot, \cdot)\). 

To prove Proposition~\ref{prop:kolmo_eq}, we first prove two lemmas: one on the polynomial growth of \(u\) and its derivatives (compare to \cite[Lemma A.2]{Kopec}); and one providing estimates on the derivatives of~\(u\) up to time \(t = 1\) involving only the norm of the initial condition and importantly not its derivatives (compare to \cite[Lemma A.6]{Kopec}). Finally, we combine these two lemmas with the geometric ergodicity \eqref{eq:geo_ergodicity} of the dynamics to prove the proposition.

\begin{lemma}\label{lm:poly_growth_u}
	For any \(m, \ell_m \in \mathbb{N}\), there exist \(s_m \in \mathbb{N}\), \(C_m > 0\) and \(\gamma_m \in \mathbb{R}\) such that for any \(\psi \in  C_{\ell_m}^m\left(\mathbb{R}^d\right)\), the solution \(u\) of \eqref{eq:kbe} with~\(\psi\) as its initial condition satisfies for any \(k \in \mathbb{N}^d\) with \(|k| = m\)
	\begin{equation}\label{eq:poly_growth_u}
		\forall t > 0, \, x\in\mathbb{R}^d \qquad \left|\partial^k u(t, x)\right| \leq C_m \left\|\psi\right\|_{m, \ell_m} \mathcal{K}_{s_m}(x) \, \mathrm{e}^{\gamma_m t}.
	\end{equation}
\end{lemma}
\begin{proof}
	We start with the case \(m = 0\). Let \(s_0 = \ell_0\). By \eqref{eq:semigroup_estimates}, since \(u(t,x) = \mathbb{E}\left[\psi\left(X_x^\eta(t)\right)\right]\), there exists \(S_{\ell_0} \in \mathbb{R}_+\) such that
	\[\left|u(t,x)\right| \leq \left\|\psi \right\|_{\mathcal{K}_{\ell_0}} \mathbb{E}\left[\mathcal{K}_{\ell_0}\left(X_x^\eta(t)\right)\right] \leq S_{\ell_0} \left\|\psi\right\|_{0,\ell_0} \mathcal{K}_{\ell_0}(x) = S_{\ell_0}\left\|\psi\right\|_{0, \ell_0}\mathcal{K}_{s_m}.\]
	Next we show the result for \(m = 1\). For all \(x, h \in \mathbb{R}^d\) and \(t \geq 0\), we have
	\begin{equation}\label{eq:gradient_kbe_sol}
		\nabla u(t,x) \cdot h = \mathbb{E}\left[\nabla \psi\left(X_x^\eta\left(t\right)\right) \cdot \delta_x^h(t)\right],
	\end{equation}
	where \(\delta_x^h(t) \in \mathbb{R}^d\) is the process defined by
	\begin{equation}\label{eq:flow_first_var}
		\delta_x^h(t) = \nabla_x X_x^\eta(t)h,
	\end{equation}
	with \(\nabla_x X_x^\eta(t) \in \mathbb{R}^{d\times d}\) the derivative process satisfying 
	\begin{equation}\label{eq:derivative_flow}
		d \nabla_x X_x^\eta(t) = \left[\nabla b\left(X_x^\eta(t)\right) + \eta \nabla F\left(X_x^\eta(t)\right)\right]\nabla_x X_x^\eta(t) dt, \qquad \nabla_x X_x^\eta(0) = \mathrm{Id}.
	\end{equation}
	Equations \eqref{eq:gradient_kbe_sol} and \eqref{eq:derivative_flow} are derived by formally differentiating the solution \(\left(X_x^\eta(t)\right)_{t \geq 0}\) to \eqref{eq:sde_model} with respect to its initial condition, which is indeed justified since the coefficients of the SDE are in \(\mathscr{S}\), see for example \cite[Chapter 3]{Kunita}. These equations imply that
	\[\begin{aligned}
		\frac{d}{dt} \left|\delta_x^h\right|^2 &= 2\delta_x^h(t) \cdot \frac{d}{dt}\delta_x^h(t) = \delta_x^h(t) \cdot \left[\left(\nabla b\left(X_x^\eta(t)\right) + \eta \nabla F\left(X_x^\eta(t)\right)\right)\nabla_x X_x^\eta(t)h \right] \\
		&= \left[\nabla b\left(X_x^\eta(t)\right) + \eta \nabla F\left(X_x^\eta(t)\right)\right]\cdot\left(\delta_x^h(t), \delta_x^h(t)\right).
	\end{aligned}\] 
	Then, applying \eqref{eq:jac_drift_spectral_bound}, we obtain
	\[\frac{d}{dt}\left|\delta_x^h(t)\right|^2 \leq 2\lambda_{|\eta|} \left|\delta_x^h(t)\right|^2.\]
	Gr\"onwall's lemma then implies that 
	\begin{equation}\label{eq:first_var_bd}
		\left|\delta_x^h(t)\right|^2 \leq \mathrm{e}^{2\lambda_{|\eta|}t}|h|^2.
	\end{equation}
	Applying this bound and the Cauchy--Schwarz inequality, we obtain
	\[\begin{aligned}
		\left|\nabla u(t,x) \cdot h\right| &\leq \mathbb{E}\left[\left|\nabla \psi\left(X_x^\eta(t)\right)\right|^2\right]^{1/2} \mathbb{E}\left[\left|\delta_x^h\right|^2\right]^{1/2}\\
		&\leq \left\|\psi\right\|_{1, \ell_1} \mathbb{E}\left[\mathcal{K}_{\ell_1}\left(X_x^\eta(t)\right)^2\right]^{1/2}\mathbb{E}\left[\left|\delta_x^h(t)\right|^2\right]^{1/2}\\
		&\leq 2\left\|\psi\right\|_{1, \ell_1} \mathbb{E}\left[\mathcal{K}_{2\ell_1}\left(X_x^\eta(t)\right)\right]^{1/2}\mathbb{E}\left[\left|\delta_x^h(t)\right|^2\right]^{1/2}.
	\end{aligned}\]
	Then, by the moment growth bounds \eqref{eq:semigroup_estimates}, there exists a constant \(S_{2\ell_1} \in \mathbb{R}_+\) such that
	\[\begin{aligned}
		\left|\nabla u(t,x)\cdot h \right| &\leq 2 S_{2\ell_1}\left\|\psi\right\|_{1, \ell_1} \mathcal{K}_{2\ell_1}(x)^{1/2}\mathbb{E}\left[\left|\delta_x^h(t)\right|^2\right]^{1/2}\\
		&\leq 2 S_{2\ell_1}\left\|\psi\right\|_{1, \ell_1} \mathcal{K}_{2\ell_1}(x)|h|\mathrm{e}^{\lambda_{|\eta|}t}.
	\end{aligned}\]
	Specifically, for each \(i \in \left\{1, \dots, d\right\}\),
	\(\left|\partial_{x_i}u(t,x)\right| \leq 2 S_{2\ell_1}\left\|\psi\right\|_{1, \ell_1} \mathcal{K}_{2\ell_1}(x) \, \mathrm{e}^{\lambda_{|\eta|}t}.\)
	
	We now show the result for \(m = 2\). Let \(\widetilde{s}_2 \geq \ell_2\) be large enough such that \(b_i, F_i \in C_{\widetilde{s}_2}^2\left(\mathbb{R}^d\right)\) for all \(i \in \left\{1, \dots, d\right\}\). For \(x, h \in \mathbb{R}^d\) and \(t \geq 0\), we have
	\[\nabla^2 u(t,x) \cdot (h,h) = \mathbb{E}\left[\nabla^2 \psi\left(X_x^\eta(t)\right)\cdot \left(\delta_x^h(t), \delta_x^h(t)\right) + \nabla \psi\left(X_x^\eta(t)\right)\cdot \xi_x^h(t)\right],\]
	where \(\xi_x^h(t) \in \mathbb{R}^d\) is a process defined by 
	\begin{equation}\label{eq:flow_second_var}
		\xi_x^h(t) = \nabla^2X_x^\eta(t)\cdot (h,h, \cdot), \qquad \xi_x^h(0) = 0.
	\end{equation}
	Taking a gradient in \eqref{eq:derivative_flow} and applying the resulting equation for \(\nabla^2X_x^\eta(t)\), we obtain 
	\[\frac{d}{dt}\xi_x^h(t) = \left[\nabla^2b\left(X_x^\eta(t)\right) + \eta \nabla^2F\left(X_x^\eta(t)\right)\right]\cdot \left(\delta_x^h(t), \delta_x^h(t), \cdot\right) + \left[\nabla b\left(X_x^\eta(t)\right) + \eta \nabla\left(X_x^\eta(t)\right)\right]\xi_x^h(t).\]
	Consequently applying \eqref{eq:jac_drift_spectral_bound} and \eqref{eq:first_var_bd} gives
	\[\begin{aligned}
		\frac{d}{dt}\left|\xi_x^h(t)\right|^2 &= 2\xi_x^h(t) \cdot \left(\left[\nabla^2 b\left(X_x^\eta(t)\right) + \eta\nabla^2F\left(X_x^\eta(t)\right)\right]\cdot\left(\delta_x^h(t), \delta_x^h(t), \cdot \right) + \left[\nabla b\left(X_x^\eta(t)\right) + \eta\nabla F\left(X_x^\eta(t)\right)\right]\xi_x^h(t) \right)\\
		&= 2 \left[\nabla^2 b\left(X_x^\eta(t)\right) + \eta\nabla^2F\left(X_x^\eta(t)\right)\right]\cdot\left(\delta_x^h(t), \delta_x^h(t), \xi_x^h(t) \right) + 2 \left[\nabla b\left(X_x^\eta(t)\right) + \eta\nabla F\left(X_x^\eta(t)\right)\right]\cdot\left(\xi_x^h(t), \xi_x^h(t)\right)\\
		&\leq 2\left(\sum_{i=1}^d \left[\|b_i\|_{2, \widetilde{s}_2} + \eta \|F_i\|_{2, \widetilde{s}_2}\right]\right)\left|\delta_x^h(t)\right|^2\left|\xi_x^h(t)\right|\mathcal{K}_{\widetilde{s}_2}\left(X_x^\eta(t)\right) + 2\lambda_{|\eta|}\left|\xi_x^h(t)\right|^2\\
		&\leq \left(\sum_{i=1}^d \left[\|b_i\|_{2, \widetilde{s}_2} + \eta \|F_i\|_{2, \widetilde{s}_2}\right]\right)^2\left|\delta_x^h(t)\right|^4\mathcal{K}_{\widetilde{s}_2}\left(X_x^\eta(t)\right)^2 + \left(1 +2\lambda_{|\eta|}\right)\left|\xi_x^h(t)\right|^2\\
		&\leq 2\left(\sum_{i=1}^d \left[\|b_i\|_{2, \widetilde{s}_2} + \eta \|F_i\|_{2, \widetilde{s}_2}\right]\right)^2\mathcal{K}_{2\widetilde{s}_2}\left(X_x^\eta(t)\right)|h|^4\mathrm{e}^{4\lambda_{|\eta|}t} + \left(1 +2\lambda_{|\eta|}\right)\left|\xi_x^h(t)\right|^2.
	\end{aligned}\]
	Taking expectation and applying Gr\"onwall's lemma with the fact that \(\xi_x^h(0) = 0\) and the moment growth bounds \eqref{eq:semigroup_estimates}, there exists \(S_{2\widetilde{s}_2} \in \mathbb{R}_+\) such that 
	\begin{equation}\label{eq:second_var_bd}
		\begin{aligned}
			\mathbb{E}\left[\left|\xi_x^h(t)\right|^2\right] &\leq 2S_{2\widetilde{s}_2}\left(\sum_{i=1}^d \left[\|b_i\|_{2, \widetilde{s}_2} + \eta \|F_i\|_{2, \widetilde{s}_2}\right]\right)^2\mathcal{K}_{2\widetilde{s}_2}(x)|h|^4\frac{\mathrm{e}^{\left(1+6\lambda_{|\eta|}\right)t}}{4\lambda_{|\eta|}}\\
			&\leq \widetilde{C}_2\mathcal{K}_{2\widetilde{s}_2}(x)|h|^4\mathrm{e}^{\widetilde{\gamma}_2 t}
		\end{aligned}
	\end{equation}
	for some constants \(\widetilde{\gamma}_2, \widetilde{C}_2 > 0\) that do not depend on \(t\), \(x\), or \(h\). Using the latter inequality, the bound \eqref{eq:first_var_bd} we derived earlier for \(\delta_x^h(t)\), and the moment growth bounds \eqref{eq:semigroup_estimates} along with the Cauchy--Schwarz inequality, we obtain
	\[\begin{aligned}
		\left|\nabla^2 u(t,x) \cdot (h, h)\right| &\leq \|\psi\|_{2, \ell_2}\mathbb{E}\left[\mathcal{K}_{\ell_2}\left(X_x^\eta(t)\right) \left|\delta_x^h(t)\right|^2\right] + \mathbb{E}\left[\left|\nabla \psi\left(X_x^\eta(t)\right)\right|^2\right]^{1/2}\mathbb{E}\left[\left|\xi_x^h(t)\right|^2\right]^{1/2}\\
		&\leq \|\psi\|_{2, \ell_2}\mathbb{E}\left[\mathcal{K}_{\ell_2}\left(X_x^\eta(t)\right)^2\right]^{1/2} \mathbb{E}\left[\left|\delta_x^h(t)\right|^4\right]^{1/2} + d\|\psi\|_{2, \ell_2}\mathbb{E}\left[\mathcal{K}_{\ell_2}\left(X_x^\eta\right)^2\right]^{1/2}\mathbb{E}\left[\left|\xi_x^h(t)\right|^2\right]^{1/2}\\
		&\leq \|\psi\|_{2, \ell_2}\mathbb{E}\left[\mathcal{K}_{\widetilde{s}_m}\left(X_x^\eta(t)\right)^2\right]^{1/2} \mathbb{E}\left[\left|\delta_x^h(t)\right|^4\right]^{1/2} + d\|\psi\|_{2, \ell_2}\mathbb{E}\left[\mathcal{K}_{\widetilde{s}_m}\left(X_x^\eta\right)^2\right]^{1/2}\mathbb{E}\left[\left|\xi_x^h(t)\right|^2\right]^{1/2}\\
		&\leq C_2\|\psi\|_{2, \ell_2}\mathcal{K}_{4\widetilde{s}_m}(x)|h|^2\mathrm{e}^{\gamma_2 t},
	\end{aligned}\]
	for some constants \(\gamma_2, C_2 > 0\) that do not depend on \(t\), \(x\), or \(h\), which proves the result for \(m = 2\). 
	
	The argument for \(m > 2\) follows in a similar spirit by induction. Using the same methods as above, we can derive bounds on \(\frac{d}{dt} \mathbb{E}\left[\left|\nabla^m_x X_x^\eta(t) \cdot \left(h, \dots, h, \cdot\right)\right|^2\right]\) in terms of \(\mathbb{E}\left[\left|\nabla^m_x X_x^\eta(t) \cdot \left(h, \dots, h, \cdot\right)\right|^2\right]\) and terms involving lower order derivatives of the solution flow which we know how to control from the previous steps. This permits us to apply Gr\"onwall's lemma to eventually conclude that there exists constants \(\gamma_m, C_m > 0\) and integer \(\widetilde{s}_m \geq \ell_m\) that are independent of \(t\), \(x\), and \(h\) such that 
	\[\mathbb{E}\left[\left|\nabla^m_x X_x^\eta(t) \cdot \left(h, \dots, h, \cdot\right)\right|^2\right] \leq C_m \mathcal{K}_{\widetilde{s}_m}(x)|h|^{2m}\mathrm{e}^{\gamma_m t}.\]
	This control then permits us to bound \(\left|\nabla^m u(t,x) \cdot \left(h, \dots, h\right)\right|\) as we had above and therefore provides bounds on \(\partial^k u\) for \(|k| = m\). 
	\end{proof}
	
	We next apply Bismut--Elworthy--Li type formulae and bounds on the derivatives of the stochastic flow \(\left(X_x^\eta(t)\right)_{t\geq 0}\) derived in the previous proof to derive estimates the derivatives of \(u\) that do not depend on the derivatives of the initial condition.
	\begin{lemma}\label{lm:u_short_time_deriv_estimates}
		Let \(\ell_0 \in \mathbb{N}\). For any \(m \in \mathbb{N}\), there exist constants \(C_m > 0\) and \(s_m \in \mathbb{N}\) with \(s_m \geq \ell_0\) such that, the solution \(u\) of \eqref{eq:kbe} with~\(\psi \in \mathscr{S}_\eta\cap C_{\ell_0}^0\left(\mathbb{R}^d\right)\) as its initial condition satisfies for any \(k \in \mathbb{N}^d\) with \(|k| = m\)
		\begin{equation}
			\forall t \in (0, 1], \, \forall x \in \mathbb{R}^d, \qquad \left|\partial^k u(t,x)\right| \leq C_m \|\psi\|_{0, \ell_0}\mathcal{K}_{s_m}(x) t^{-k/2}.
		\end{equation}
	\end{lemma}
	\begin{proof}
		We prove the result for the first two derivatives and then discuss how the result for higher-order derivatives follows from induction and analogous arguments. 
		
		For \(m = 1\), the Bismut--Elworthy--Li formula \cite[Theorem 2.1]{ElworthyLi} gives
		\begin{equation}\label{eq:1st_order_bel_formula}
			\nabla u(t,x)\cdot h = \sqrt{\frac{\beta}{2}}\frac{1}{t}\mathbb{E}\left[\psi\left(X_x^\eta(t)\right)\int_0^t\delta_x^h(s)\cdot dW_s\right],
		\end{equation}
		where \(\delta_x^h\) is the process defined in \eqref{eq:flow_first_var}. 
		Using the Cauchy--Schwarz inequality and It\^o's isometry, we obtain
		\[\begin{aligned}
			\left|\nabla u(t,x) \cdot h\right| &\leq \sqrt{\frac{\beta}{2}} \frac{1}{t} \mathbb{E}\left[\psi\left(X_x^\eta(t)\right)^2\right]^{1/2}\mathbb{E}\left[\left(\int_0^t \delta_x^h(s)\cdot dW_s\right)^2\right]^{1/2}\\
			&= \sqrt{\frac{\beta}{2}} \frac{1}{t} \mathbb{E}\left[\psi\left(X_x^\eta(t)\right)^2\right]^{1/2}\mathbb{E}\left[\int_0^t \left|\delta_x^h(s)\right|^2 ds\right]^{1/2}\\
			&= \sqrt{\frac{\beta}{2}} t^{-1/2}   \mathbb{E}\left[\psi\left(X_x^\eta(t)\right)^2\right]^{1/2}\mathbb{E}\left[\frac{1}{t}\int_0^t \left|\delta_x^h(s)\right|^2 ds\right]^{1/2}\\
			&\leq \sqrt{\frac{\beta}{2}}\|\psi\|_{0, \ell_0} t^{-1/2}   \mathbb{E}\left[\mathcal{K}_{\ell_0}\left(X_x^\eta(t)\right)^2\right]^{1/2}\mathbb{E}\left[\frac{1}{t}\int_0^t \left|\delta_x^h(s)\right|^2 ds\right]^{1/2}
		\end{aligned}\]
		The first expectation is bounded using the moment growth bounds \eqref{eq:semigroup_estimates} and the second is uniformly bounded for \(t \in \left(0, 1\right]\) due to the bound \eqref{eq:first_var_bd} for \(\left|\delta_x^h(t)\right|^2\). Thus, for some constant \(C_1 > 0\),
		\begin{equation}\label{eq:1st_short_time_deriv_estimates}
			\forall t \in \left(0, 1\right], \quad \left|\nabla u(t,x) \cdot h\right| \leq C_1\|\psi\|_{0, \ell_0} \mathcal{K}_{\ell_0}(x) t^{-1/2}|h|.
		\end{equation}
		
		For \(m = 2\), the Bismuth-Elworthy-Li formula at second order (see \cite[Theorem 2.3]{ElworthyLi}) gives 
		\begin{equation}\label{eq:2nd_order_bel_formula}
			\begin{aligned}
				\nabla^2 u(t,x)\cdot (h,h) &=  \frac{4}{t^2} \mathbb{E}\left[\psi\left(X_x^\eta(t)\right) \sqrt{\frac{\beta}{2}}\int_{t/2}^t \delta_x^h(s)\cdot dW_s  \sqrt{\frac{\beta}{2}}\int_0^{t/2}\delta_x^h\left(s\right)\cdot dW_s\right] \\
				&\qquad + \frac{2}{t}\mathbb{E}\left[\int_0^{t/2}\nabla u\left(t-s, X_x^\eta(s)\right)\cdot \xi_x^h\left(s\right) ds\right],
			\end{aligned}
		\end{equation}
		where \(\xi_x^h\) is the process defined in \eqref{eq:flow_second_var}. Note that in our case the second term in \cite[Theorem 2.3]{ElworthyLi} vanishes as we have additive noise. Let \(\left(\mathcal{F}_t\right)_{t\geq 0}\) be the filtration generated by \(\left(X_x^\eta(t), \nabla_x X_x^\eta(t)\right)_{t\geq 0}\). Since \(\left(X_x^\eta(t), \nabla_xX_x^\eta(t)\right)_{t\geq 0}\) is a Markov process with respect to \(\left(\mathcal{F}_t\right)_{t\geq 0}\), conditioning on \(\mathcal{F}_{t/2}\) and using the tower property gives for the first expectation in the above equality 
		\begin{equation*}
			\begin{aligned}
				&\mathbb{E}\left[\psi\left(X_x^\eta(t)\right)\int_{t/2}^t \delta_x^h(s)\cdot dW_s \int_0^{t/2}\delta_x^h(s)\cdot dW_s\right] = \mathbb{E}\left[\mathbb{E}\left[\left. \psi\left(X_x^\eta(t)\right) \int_{t/2}^t \delta_x^h(s)\cdot dW_s \int_0^{t/2}\delta_x(s)\cdot dW_s \right| \mathcal{F}_{t/2}\right]\right]\\
				&\qquad = \mathbb{E}\left[\int_0^{t/2}\delta_x(s)\cdot dW_s\mathbb{E}\left[\left. \psi\left(X_x^\eta(t)\right) \int_{t/2}^t \delta_x^h(s)\cdot dW_s  \right| \mathcal{F}_{t/2}\right]\right]\\
				&\qquad = \mathbb{E}\left[\int_0^{t/2}\delta_x^h(s)\cdot dW_s\mathbb{E}\left[\left. \psi\left(X_x^\eta(t)\right) \int_0^{t/2} \delta_{X_x^\eta\left(t/2\right)}^{\delta_x^h(t/2)}(s)\cdot dW_{s + t/2}  \right| X_x^\eta\left(t/2\right), \nabla_x X_x^\eta\left(t/2\right)\right]\right]
			\end{aligned}
		\end{equation*}
		By \eqref{eq:1st_order_bel_formula},
		\[\mathbb{E}\left[\left. \psi\left(X_x^\eta(t)\right) \int_0^{t/2} \delta_{X_x^\eta\left(t/2\right)}^{\delta_x^h\left(t/2\right)}(s)\cdot dW_{s + t/2}  \right| X_x^\eta\left(t/2\right), \nabla_x X_x^\eta\left(t/2\right)\right] = \sqrt{\frac{2}{\beta}} \frac{t}{2}\nabla u\left(t/2, X_x^\eta\left(t/2\right)\right) \cdot \delta_x^h\left(t/2\right).\]
		Consequently, equation \eqref{eq:2nd_order_bel_formula} becomes
		\[\begin{aligned}
			\nabla_x^2 u(t,x)\cdot(h, h) = \frac{2}{t}\Bigg(\sqrt{\frac{\beta}{2}}\mathbb{E}\left[\nabla u\left(t/2, X_x^\eta\left(t/2\right)\right) \cdot \delta_x^h\left(t/2\right) \int_0^{t/2}\delta_x(s)\cdot dW_s \right]&\\
			\qquad + \mathbb{E}\left[\int_0^{t/2}\nabla u\left(t-s, X_x^\eta(s)\right)\cdot \xi_x^h(s)ds\right]\Bigg).& 
		\end{aligned}\]
		Using the Cauchy-Schwarz inequality and It\^o's isometry, we have 
		\[\begin{aligned}
			\left|\nabla^2 u(t,x)\cdot (h,h)\right| &\leq  \frac{2}{t}\Bigg(
			\sqrt{\frac{\beta}{2}}\mathbb{E}\left[\left|\nabla u\left(t/2, X_x^\eta\left(t/2\right)\right)\right|^4\right]^{1/4} \mathbb{E}\left[\left|\delta_x^h\left(t/2\right)\right|^4\right]^{1/4} \mathbb{E}\left[\left(\int_0^{t/2}\delta_x^h\left(s\right) \cdot dW_s\right)^2\right]^{1/2} \\
			&\qquad + \int_0^{t/2}\mathbb{E}\left[\left|\nabla u\left(t-s, X_x^\eta(s)\right)\right|^2\right]^{1/2} \mathbb{E}\left[\left|\xi_x^h(s)\right|^2\right]^{1/2} ds\Bigg)\\
			&=  \frac{2}{t}\Bigg(
			\sqrt{\frac{\beta}{2}}\mathbb{E}\left[\left|\nabla u\left(t/2, X_x^\eta\left(t/2\right)\right)\right|^4\right]^{1/4} \mathbb{E}\left[\left|\delta_x^h\left(t/2\right)\right|^4\right]^{1/4} \mathbb{E}\left[\int_0^{t/2}\left|\delta_x^h(s) \right|^2 ds\right]^{1/2} \\
			&\qquad + \int_0^{t/2}\mathbb{E}\left[\left|\nabla u\left(t-s, X_x^\eta(s)\right)\right|^2\right]^{1/2} \mathbb{E}\left[\left|\xi_x^h(s)\right|^2\right]^{1/2} ds\Bigg).
		\end{aligned}\]
		Using the bounds \eqref{eq:first_var_bd} and \eqref{eq:second_var_bd} on \(\left|\delta_x^h(t)\right|^2\) and \(\mathbb{E}\left[\left|\xi_x^h(t)\right|^2\right]\) respectively, the above inequality becomes, for \(t \in \left(0, 1\right]\):
		\[\left|\nabla^2 u(t,x)\cdot (h,h)\right| \leq C \mathcal{K}_s(x)  \frac{1}{t} |h|^2\left(t^{1/2}\mathbb{E}\left[\left|\nabla u\left(t/2, X_x^\eta\left(t/2\right)\right)\right|^4\right]^{1/4} + \int_0^{t/2}\mathbb{E}\left[\left|\nabla u\left(t-s, X_x^\eta(s)\right)\right|^2\right]^{1/2} ds\right),\]
		for some integer \(s \geq \ell_0\) and constant \(C > 0\).
		Then using the bound \eqref{eq:1st_short_time_deriv_estimates}, we can control the two terms involving \(\left|\nabla u (\cdot, \cdot)\right|\) in the above inequality to obtain
		\begin{equation}
			\left|\nabla^2 u(t,x)\cdot (h,h)\right| \leq C_2 \|\psi\|_{0, \ell_0}\mathcal{K}_{s_2}(x)t^{-1}|h|^2,
		\end{equation}
		for some integer \(s_2 \geq \ell_0\) and constant \(C_2 > 0\). 
		
		For \(m > 2\), following the proof \cite[Theorem 2.3]{ElworthyLi}, one can derive higher order Bismut--Elworthy--Li type formulae. Then with these formulae one can bound \(\left|\nabla^m u(t,x) \cdot \left(h, \dots, h\right)\right|\) in terms of \(\left|\nabla^q u(t,x) \cdot \left(h, \dots, h\right)\right|\) for \(q < m\) and \(\mathbb{E}\left[\left|\nabla^p_x X_x^\eta(t)\cdot(h, \dots, h)\right|^2\right]\) for \(p \leq m\). Using the bounds on \(\mathbb{E}\left[\left|\nabla^p_x X_x^\eta(t)\cdot(h, \dots, h)\right|^2\right]\) derived in the proof of Lemma~\ref{lm:poly_growth_u}, we can then conclude that bounds on \(\left|\nabla^q u(t,x) \cdot \left(h, \dots, h\right)\right|\) for \(q < m\) imply the bound on \(\left|\nabla^m u(t,x) \cdot \left(h, \dots, h\right)\right|\), thus completing the induction. 
 	\end{proof}
	
	We can now prove the proposition. 
	\begin{proof}[Proof of Proposition~\ref{prop:kolmo_eq}]
		Let \(\ell_0 \in \mathbb{N}\) be such that \(\varphi \in C_{\ell_0}^0\left(\mathbb{R}^d\right)\). For all \(t \geq 1\), we have \[u(t,x) = P^\eta_1\left(P^\eta_{t-1}\varphi\right)(x).\]
		For \(t\geq 1\), we define functions \(v_t(s, x) = P_s^\eta\left(P_{t-1}^\eta\varphi\right)(x)\) parameterized by \(t\) so that \(v_t(1, x) = u(t,x)\). The function \(v_t\) is the solution to the Kolmogorov equation \eqref{eq:kbe} with initial condition \(P^{\eta}_{t-1}\varphi\). By the moment growth bounds \eqref{eq:semigroup_estimates} and the fact that \(\left(P_t^\eta\right)_{t\geq 0}\) preserves continuity, it holds \(P_{t-1}^\eta \varphi \in C_{\ell_0}^0\left(\mathbb{R}^d\right)\). As a result, we can apply Lemma~\ref{lm:u_short_time_deriv_estimates} at time \(t = 1\) to obtain that, for any \(m \in \mathbb{N}\), there exist constants \(C > 0\) and \(s \in \mathbb{N}\) that depend only on \(m\), \(\ell_0\) and \(b, F, \eta\) such that, for any \(k\in \mathbb{N}^d\) with \(|k| = m\),
		\[\left|\partial^k u(t, x)\right| = \left|\partial^k v_t(1, x)\right| \leq C \left\|P_{t-1}^\eta\varphi\right\|_{0,\ell_0}\mathcal{K}_{s}(x).\]
		Since \(\left(P^\eta_t\right)_{t\geq 0}\) is geometrically ergodic with respect to the \(\mathcal{K}_{\ell_0}\)-norm (recall \eqref{eq:geo_ergodicity}) and \(\varphi\) has average zero with respect to \(\nu_\eta\), we have
		\[\left\|P_{t-1}^\eta\varphi\right\|_{0, \ell_0} \leq C \|\varphi\|_{0, \ell_0}\mathrm{e}^{-\lambda t},\]
		for some \(C, \lambda > 0\). Thus with the same \(\lambda\) but possibly a different \(C\), we have, for \(t \geq 1\),
		\[\left|\partial^k u(t,x)\right| \leq C\|\varphi\|_{0, \ell_0}K_s(x)\mathrm{e}^{-\lambda t}. \]
		For \(t < 1\), the bound \eqref{eq:poly_growth_u} implies the same sort of bound upon replacing \(\|\varphi\|_{0, \ell_0}\) with \(\|\varphi\|_{m, \ell_m}\), and increasing \(C\) and \(s\). Combining these two bounds gives \eqref{eq:bd_kbe}.
	\end{proof}

\section{Ergodicity of the Synchronously Coupled Dynamics}\label{sec:sync_contractive}
We denote by \(\mathcal{W}^1\) the 1-Wasserstein or Kantorovich distance:
\begin{equation}
	\mathcal{W}^1\left(\nu, \mu\right) = \inf_{\pi \in \Pi\left(\nu, \mu\right)} \int_{\mathbb{R}^{2d} \times \mathbb{R}^{2d}} \left|x - y\right| \pi\left(dx\, dy\right),
\end{equation}
where \(\Pi\left(\nu, \mu\right)\) is the set of couplings of \(\nu\) and \(\mu\).
We construct a two stage coupling of two synchronously coupled solutions to \eqref{eq:coupled_dynamics}, \(\left(Z_t^\eta\right)_{t\geq 0} := \left(X_t^\eta, Y_t^0\right)_{t\geq 0}\) and \(\left(\widecheck{Z}_t^\eta\right)_{t\geq 0} := \left(\widecheck{X}_t^\eta, \widecheck{Y}_t^0\right)_{t\geq 0}\). We denote by \(\tau_x := \inf\left\{t \geq 0 : \, X_t^\eta = \widecheck{X}_t^\eta\right\}\) the meeting time of the first components of \(Z^\eta\) and \(\widetilde{Z}^\eta\). The process \(Z^\eta\) follows
\begin{equation}\label{eq:sync_z}
	\begin{aligned}
		dX_t^\eta &= \left(b\left(X_t^\eta\right) + \eta F\left(X_t^\eta\right)\right)dt + \sqrt{\frac{2}{\beta}}dW_t,\\
		dY_t^0 &= b\left(Y_t^0\right)dt + \sqrt{\frac{2}{\beta}}dW_t.
	\end{aligned}
\end{equation}
For \(t < \tau_x\), the process \(\widecheck{Z}^\eta\) follows
\begin{equation}\label{eq:sync_z_til1}
	\begin{aligned}
		d\widecheck{X}_t^\eta &= \left(b\left(\widecheck{X}_t^\eta\right) + \eta F\left(\widecheck{X}_t^\eta\right)\right)dt + \sqrt{\frac{2}{\beta}}\left(\mathrm{Id} - 2e_te_t^T\right)dW_t,\\
		d\widecheck{Y}_t^0 &= b\left(\widecheck{Y}_t^0\right)dt + \sqrt{\frac{2}{\beta}}\left(\mathrm{Id} - 2e_te_t^T\right)dW_t,
	\end{aligned}
\end{equation}
where \(e_t\) is the unit vector
\[e_t := \frac{X_t^\eta - \widecheck{X}_t^\eta}{\left|X_t^\eta - \widecheck{X}_t^\eta\right|};\]
while for \(t \geq \tau_x\),
\begin{equation}\label{eq:sync_z_til2}
	\begin{aligned}
		\widecheck{X}_t^\eta &= X_t^\eta\\
		d\widecheck{Y}_t^0 &= b\left(\widecheck{Y}_t^0\right)dt + \sqrt{\frac{2}{\beta}}dW_t.
	\end{aligned}
\end{equation}
Denote by \(\left(T_t^\eta\right)_{t\geq 0}\) the semigroup of the synchronously coupled dynamics \eqref{eq:sync_z}.

\begin{proposition}
	Consider two probability measures \(\mu\) and \(\widecheck{\mu}\) on \(\mathbb{R}^{2d}\) with finite second moments and \(\left(Z_t^\eta, \widecheck{Z}_t^\eta\right)_{t\geq 0}\) satisfying \eqref{eq:sync_z}--\eqref{eq:sync_z_til2} with initial conditions \(\left(Z_0^\eta, \widecheck{Z}_0^\eta\right)\) such that \(Z_0^\eta \sim \mu\) and \(\widecheck{Z}_0^\eta \sim \widecheck{\mu}\). Assume that Assumption~\ref{ass:drift} holds with \(M = 0\). Then, there exist constants \(C(\mu, \widecheck{\mu}) \in \mathbb{R}_+\) and \(\gamma > 0\) such that 
	\begin{equation}\label{eq:L1_contraction}
		\mathbb{E}\left[\left|Z_t^\eta - \widecheck{Z}_t^\eta\right|\right] \leq C(\mu, \nu)\, \mathrm{e}^{-\gamma t},
	\end{equation}
	where \(C(\mu, \widecheck{\mu})\) only depends on the marginals of the initial condition in a way that is made explicit in the proof, see \eqref{eq:sync_contraction_prefactor}.
	As a consequence,
	\begin{equation}\label{eq:W1_contraction}
		\mathcal{W}^1\left(\mu T^\eta_t, \widecheck{\mu} T^\eta_t\right) \leq C(\mu, \widecheck{\mu})\, \mathrm{e}^{-\gamma t}.
	\end{equation}
\end{proposition}
\begin{proof}
	We proceed in two steps: first we follow the strategy of \cite{Eberle} to show exponential contractivity in~\(L^1\) for the \(x\)-components, namely
	\[\mathbb{E}\left[\left|X_t^\eta - \widecheck{X}_t^\eta\right|\right] \leq C \mathrm{e}^{-ct}\mathbb{E}\left[\left|X_0^\eta - \widecheck{X}_t^\eta \right|\right]\]
	for some constants \(C, c > 0\). We also prove that the coupling time of the first components \(\tau_x\) has an exponential tail, namely \(\mathbb{P}\left(\tau_x > t\right) \leq C \mathrm{e}^{-ct}\) for some constants \(C, c > 0\). 
	Second, we verify that the strong contractivity of the drift ensures that the \(y\)-components are exponentially contractive when \(t > \tau_x\). Combining these two results gives the desired contractivity for \(\mathbb{E}\left|Z_t^\eta - \widecheck{Z}_t^\eta\right|\).
	
	Applying It\^o's formula we obtain
	\[d\left|X_t^\eta - \widecheck{X}_t^\eta\right|^2 = \left[2\left\langle X_t^\eta - \widecheck{X}_t^\eta, b\left(X_t^\eta\right) + \eta F\left(X_t^\eta\right) - b\left(\widecheck{X}_t^\eta\right) - \eta F\left(\widecheck{X}_t^\eta\right)\right\rangle + \frac{8}{\beta}\right]dt +  \sqrt{\frac{32}{\beta}} \left|X_t^\eta - \widecheck{X}_t^\eta\right|e^T_tdW_t.\]
	We recall that the It\^o differential of a continuous strictly positive real-valued semi-martingale \(\left(\rho_t\right)_{t\geq 0}\) is given by 
	\[d\left(\sqrt{\rho_t}\right) = \frac{1}{2\sqrt{\rho_t}}d\rho_t - \frac{1}{8}\frac{d\langle \rho\rangle_t}{\rho^{3/2}_t},\]
	where \(\left(\langle \rho \rangle_t\right)_{t\geq 0}\) is the quadratic variation of \(\left(\rho\right)_{t\geq}\). Applying this fact to \(\left|X_t^\eta - \widecheck{X}_t^\eta\right|^2\) for \(t < \tau_x\) (so that \(\left|X_t^\eta - \widecheck{X}_t^\eta\right|^2 > 0\)) we have,
	\[d\left|X_t^\eta - \widecheck{X}_t^\eta\right| = \frac{\left\langle X_t^\eta - \widecheck{X}_t^\eta, b\left(X_t^\eta\right) + \eta F\left(X_t^\eta\right) - b\left(\widecheck{X}_t^\eta\right) - \eta F\left(\widecheck{X}_t^\eta\right) \right\rangle}{\left|X_t^\eta - \widecheck{X}_t^\eta\right|}dt + \sqrt{\frac{8}{\beta}}e_t^T dW_t.\]
	Since \(b\) is strongly contractive everywhere, we can bound the finite variation part as follows
	\[\begin{aligned}
		&\frac{\left\langle X_t^\eta - \widecheck{X}_t^\eta, b\left(X_t^\eta\right) + \eta F\left(X_t^\eta\right) - b\left(\widecheck{X}_t^\eta\right) - \eta F\left(\widecheck{X}_t^\eta\right) \right\rangle}{\left|X_t^\eta - \widecheck{X}_t^\eta\right|} \\
		&\qquad\qquad= \frac{\left\langle X_t^\eta - \widecheck{X}_t^\eta, b\left(X_t^\eta\right) - b\left(\widecheck{X}_t^\eta\right) \right\rangle}{\left|X_t^\eta - \widecheck{X}_t^\eta\right|} + \eta\frac{\left\langle X_t^\eta - \widecheck{X}_t^\eta, F\left(X_t^\eta\right) - F\left(\widecheck{X}_t^\eta\right) \right\rangle}{\left|X_t^\eta - \widecheck{X}_t^\eta\right|}\\
		&\qquad\qquad\leq -m\left|X_t^\eta - \widecheck{X}_t^\eta\right| + 2\eta\left\|F\right\|_\infty.
	\end{aligned}\] 
	Thus, by the comparison theorem for SDEs (see \cite[Chapter 6]{IkedaWatanabe}), we have almost surely that \(\left|X_t^\eta - \widecheck{X}_t^\eta\right| \leq r_t\) for all \(t \geq 0\), where \(\left(r_t\right)_{t\geq 0}\) is a one-dimensional diffusion satisfying, for \(t \in \left[0, \tau_r\right]\),
	\[dr_t = \left(-m r_t + 2\eta \left\|F\right\|_\infty\right)dt + \sqrt{\frac{8}{\beta}}e_t^TdW_t, \qquad r_0 = \left|X_0^\eta - \widecheck{X}_0\right|,\]
	and \(r_t \equiv 0\) for \(t \geq \tau_r\), where \(\tau_r\) is the first hitting time of zero for the process \(\left(r_t\right)_{t\geq 0}\), namely \(\tau_r = \inf\left\{t \geq 0\left| r_t = 0 \right.\right\}\). Note that, \(\tau_r \geq \tau_x\) almost surely.
	
	To show that \(\tau_r\) (and thereby \(\tau_x\)) has an exponential tail and that the \(x\)-components are exponentially contractive in \(L^1\), we follow the strategy in \cite{Eberle, EberleZimmer} of constructing a function \(f\) such that \(d_f(x,y) = f\left(|x - y|\right)\) is a distance on \(\mathbb{R}^d\) equivalent to the standard Euclidean distance, with the property that \(\left(\mathrm{e}^{ct}f(r_t)\right)_{t\in \left[0, \tau_r\right]}\) is a supermartingale. We define constants \(R_0, R_1 \in \left(0, \infty\right)\) by
	\begin{align}
		R_0 &= \inf\Big\{R \geq 0 \, \Big| \left(-m r + 2\eta\left\|F\right\|_\infty\right) \leq 0, \quad \forall r \geq R \Big\} = \frac{2\eta \|F\|_\infty}{m},\\
		R_1 &= \inf\Big\{R \geq R_0 \, \Big| \, R\left(R - R_0\right)\left(-mr + 2\eta\|F\|_\infty\right)/r \leq -4, \quad \forall r \geq R \Big\} = \frac{2\eta \|F\|_\infty}{m} + \frac{2}{\sqrt{m}}.
	\end{align}
	We then define a concave strictly increasing function \(f:\mathbb{R}_+ \to \mathbb{R}_+\) by
	\begin{equation}
		f(r) = \int_0^r \phi(s)g(s) \, ds,
	\end{equation}
	where 
	\begin{equation}
		\begin{aligned}
			\phi(r) &= \exp\left(-\frac{\beta}{4}\int_0^r\max\left\{0, -ms + 2\eta \|F\|_\infty\right\}ds\right),\\
			\Phi(r) &= \int_0^r \phi(s) \, ds,\\
			g(r) &= 1 - \frac{1}{4}\int_0^{\min\left\{r, R_1\right\}} \frac{\Phi(s)}{\phi(s)} \, ds \Bigg/ \int_0^{R_1} \frac{\Phi(s)}{\phi(s)} \, ds - \frac{1}{4}\int_0^{\min\left\{r, R_1\right\}} \frac{ds}{\phi(s)} \Bigg/ \int_0^{R_1}\frac{ds}{\phi(s)}.
		\end{aligned}
	\end{equation}
	Observe that \(g\) and \(\phi\) are both strictly positive and decreasing functions, so \(f\) is concave and strictly increasing. Furthermore \(g(0) = \phi(0) = 1\), for \(r \geq R_0\) the function \(\phi\) is constant, and for \(r \geq R_1\) the function \(g \equiv \frac{1}{2}\). Consequently, 
	\begin{equation}\label{eq:dist_f_equivalence}
		\frac{\phi\left(R_0\right)}{2} r \leq \frac{\Phi(r)}{2} \leq f(r) \leq \Phi(r) \leq r,
	\end{equation}
	so that \(d_f\) and the standard Euclidean distance are equivalent. Additionally, \(d_f\left(X_t^\eta, \widecheck{X}_t^\eta\right) \leq f\!\left(r_t\right)\) and \(f\left(r_t\right) = 0\) implies that \(X_t^\eta = \widecheck{X}_t^\eta\).
	 
	Since \(f\) is twice differentiable, by It\^o's formula we have
	\begin{equation}\label{eq:ito_diff_f}
		df\!\left(r_t\right) = \left[f'\left(r_t\right)\left(-mr_t + 2\eta \|F\|_\infty\right) + \frac{4}{\beta}f''\left(r_t\right)\right]dt + \sqrt{\frac{8}{\beta}}f'\left(r_t\right)e_t^T dW_t.
	\end{equation}
	For \(r \in \left(0, R_1\right)\),
	\[
	\begin{aligned}
		f'(r)\left(-mr + 2\eta \|F\|_\infty\right) + \frac{4}{\beta}f''(r) &= \phi(r)g(r) \left(-mr + 2\eta\|F\|_\infty\right)+ \frac{4}{\beta}\left(\phi'(r)g(r) + \phi(r)g'(r)\right)\\
		&= \phi(r)g(r) \left(-mr + 2\eta\|F\|_\infty\right) -  \phi(r)g(r) \min\left\{0, -mr + 2\eta\|F\|_\infty\right\}\\
		&\quad -\frac{1}{\beta} \phi(r)\left[\frac{\Phi(r)}{\phi(r)}\Bigg/\int_0^{R_1} \frac{\Phi(s)}{\phi(s)}ds + \frac{1}{\phi(r)}\Bigg/ \int_0^{R_1}\frac{ds}{\phi(s)} \right]\\
		&\leq - \left(\beta\int_0^{R_1} \frac{\Phi(s)}{\phi(s)}ds\right)^{-1}\Phi(r) - \left(\beta\int_0^{R_1}\frac{ds}{\phi(s)}\right)^{-1}\\
		&\leq - \left(\beta\int_0^{R_1} \frac{\Phi(s)}{\phi(s)}ds\right)^{-1}f\left(r\right) - \left(\beta\int_0^{R_1}\frac{ds}{\phi(s)}\right)^{-1}
	\end{aligned}
	\]
	For \(r \geq R_1\), \(f\) is affine with \(f' \equiv \phi\left(R_0\right)/2\) and thus \(f'' \equiv 0\). Furthermore, by our definition of \(R_1\), 
	\[-mr + 2\eta \|F\|_\infty \leq -\frac{4r}{R_1\left(R_1 - R_0\right)}.\]
	Thus we can bound the drift in \eqref{eq:ito_diff_f} when \(r \geq R_1\) by
	\begin{equation}\label{eq:far_contraction_1}
		\frac{\phi\left(R_0\right)}{2}\left(-mr + 2\eta \|F\|_\infty\right)  \leq -2\frac{\phi\left(R_0\right)}{R_1 - R_0}\frac{r}{R_1} \leq -2\frac{\phi\left(R_0\right)}{R_1 - R_0}\frac{\Phi(r)}{\Phi\left(R_1\right)},
	\end{equation}
	where the second inequality follow from the fact that \(\Phi(r)/r \leq \Phi\left(R_1\right)/R_1\) since \(\Phi\) is concave, i.e. concavity implies that
	\[\frac{\Phi(r)}{r} = \frac{\Phi(r)-\Phi(0)}{r} \leq \frac{\Phi\left(R_1\right) - \Phi(0)}{R_1} = \frac{\Phi\left(R_1\right)}{R_1}.\]
	This above inequality also implies
	\[\int_{R_0}^{R_1}\frac{\Phi(s)}{\phi(s)}ds = \frac{1}{\phi\left(R_0\right)}\int_{R_0}^{R_1}\Phi(s)ds \geq \frac{\Phi\left(R_1\right)}{\phi\left(R_0\right)}\frac{1}{R_1}\int_{R_0}^{R_1}sds = \frac{\Phi\left(R_1\right)}{\phi\left(R_0\right)} \frac{R_1^2 - R_0^2}{2R_1} \geq \frac{\Phi\left(R_1\right)}{\phi\left(R_0\right)} \frac{R_1 - R_0}{2}.\]
	Plugging this inequality into the right hand side of of \eqref{eq:far_contraction_1}, we obtain
	\begin{equation*}
		\begin{aligned}
			\frac{\phi\left(R_0\right)}{2}\left(-mr + 2\eta \|F\|_\infty\right)  &\leq  -2\frac{\phi\left(R_0\right)}{R_1 - R_0}\frac{\Phi(r)}{\Phi\left(R_1\right)} = -2\Phi(r)\left[(R_1 - R_0)\frac{\Phi\left(R_1\right)}{\phi\left(R_0\right)}\right]^{-1}\\ 
			&= -2\Phi(r)\left[2\frac{R_1 - R_0}{2}\frac{\Phi\left(R_1\right)}{\phi\left(R_0\right)}\right]^{-1} \leq -2\Phi(r)\left(2\int_{R_0}^{R_1}\frac{\Phi(s)}{\phi(s)}ds\right)^{-1}\\
			&\leq - 2\Phi(r)\left(2\int_0^{R_1}\frac{\Phi(s)}{\phi(s)}ds\right)^{-1}\leq -\left(\Phi(R_1) + f(r)\right)\left(2\int_0^{R_1}\frac{\Phi(s)}{\phi(s)}ds\right)^{-1},
		\end{aligned}
	\end{equation*}
	Defining the constants
	\begin{equation}\label{eq:ref_coupling_constants}
		\begin{aligned}
		c &= \min\left\{\frac{1}{\beta}, \frac{1}{2}\right\}\left(\int_0^{R_1}\frac{\Phi(s)}{\phi(s)}ds\right)^{-1},\\
		\epsilon &= \min\left\{\Phi\left(R_1\right)\left(2\int_0^{R_1}\frac{\Phi(s)}{\phi(s)}ds\right)^{-1}, \left(\beta\int_0^{R_1}\frac{ds}{\phi(s)}\right)^{-1}\right\},
		\end{aligned}
	\end{equation}
	the above bounds on the drift imply that 
	\[df\!\left(r_t\right)\leq -\left(\epsilon + cf\!\left(r_t\right)\right)dt + \sqrt{\frac{8}{\beta}}f'\left(r_t\right)e_t^TdW_t.\]
	Therefore, 
	\[d\left(\mathrm{e}^{ct}f\!\left(r_t\right)\right) \leq -\epsilon \mathrm{e}^{ct}dt + \mathrm{e}^{ct}\sqrt{\frac{8}{\beta}}f'\left(r_t\right)e_t^TdW_t.\]
	Integrating this inequality and taking expectations gives
	\begin{equation}\label{eq:f_dist_bound}
		\mathbb{E}\left[\mathrm{e}^{c\left(t\wedge \tau_r\right)}f\!\left(r_{t\wedge \tau_r}\right)\right] \leq \mathbb{E}\left[f\!\left(r_0\right)\right] - \frac{\epsilon}{c}\mathbb{E}\left[\mathrm{e}^{c\left(t\wedge \tau_r\right)} - 1\right] \leq \mathbb{E}\left[f\!\left(r_0\right)\right].
	\end{equation}
	For any \(t \geq 0\), separating between the event \(t < \tau_r\) and the event \(t \geq \tau_r\) gives
	\begin{equation*}
		\mathbb{E}\left[\mathrm{e}^{ct}f\!\left(r_t\right)\right] = \mathbb{E}\left[\mathrm{e}^{ct}f\!\left(r_t\right)\mathbf{1}_{\left\{\tau_r < t\right\}}\right] + \mathbb{E}\left[\mathrm{e}^{ct}f\!\left(r_t\right)\mathbf{1}_{\left\{\tau_r \geq t\right\}}\right].
	\end{equation*}
	The first term vanishes by construction (since \(f \!\left(r_t\right) = f(0) = 0\) for \(t \geq \tau_r\)) while the second term is bounded using \eqref{eq:f_dist_bound}. Thus, using \eqref{eq:dist_f_equivalence},
	\begin{equation*}
		\mathbb{E}\left[f\!\left(r_t\right)\right] \leq \mathrm{e}^{-ct}\mathbb{E}\left[f\!\left(r_0\right)\right] \leq \mathrm{e}^{-ct}\mathbb{E}\left[\left|X_0^\eta - \widecheck{X}_0^\eta\right|\right].
	\end{equation*}
	Consequently, using again \eqref{eq:dist_f_equivalence}, we obtain for the distance between the \(x\) marginals that
	\begin{equation}\label{eq:x_marginal_bd}
		\mathbb{E}\left[\left|X_t^\eta - \widecheck{X}_t^\eta\right|\right] \leq \frac{2}{\phi\left(R_0\right)}\mathrm{e}^{-ct}\mathbb{E}\left[\left|X_0^\eta - \widecheck{X}_0^\eta\right|\right].
	\end{equation}
	Observe that, since \(\mathrm{e}^{ct}f\left(r_t\right) \geq 0\), the intermediate inequality in \eqref{eq:f_dist_bound} implies that
	\[\mathbb{E}\left[\mathrm{e}^{c\left(t \wedge \tau_r\right)}\right] \leq \frac{c}{\epsilon}\mathbb{E}\left[f\left(r_0\right)\right] + 1 \leq  \frac{c}{\epsilon}\mathbb{E}\left[\left|X_0^\eta - \widecheck{X}_0^\eta\right|\right] + 1.\]
	We can take the limit of the left hand side as \(t\) goes to infinity to obtain, by monotone convergence,
	\[\mathbb{E}\left[\mathrm{e}^{c\tau_r}\right] \leq  \frac{c}{\epsilon}\mathbb{E}\left[\left|X_0^\eta - \widecheck{X}_0^\eta\right|\right] + 1.\]
	This inequality and Markov's inequality imply that 
	\begin{equation}\label{eq:x_meeting_time_bd}
		\mathbb{P}\left(\tau_x > t\right) \leq \mathbb{P}\left(\tau_r > t\right) \leq \mathrm{e}^{-ct}\mathbb{E}\left[\mathrm{e}^{c\tau_r}\right] \leq  \mathrm{e}^{-ct}\left(\frac{c}{\epsilon}\mathbb{E}\left[\left|X_0^\eta - \widecheck{X}_0^\eta\right|\right] + 1\right).
	\end{equation}
	
	Moving on to the \(y\)-components, \eqref{eq:sync_coupling_dist} from Lemma~\ref{lm:sync_coupling_dist} implies that for \(\left(Y_{\tau_x +t}^0, \widecheck{Y}_{\tau_x + t}^0\right)_{t \geq 0}\)
	\[\left|Y_{\tau_x + t}^0 - \widecheck{Y}_{\tau_x + t}^0\right| \leq \mathrm{e}^{-mt}\left|Y_{\tau_x}^0 - \widecheck{Y}_{\tau_x}^0\right|.\]
	Splitting up the expectation of the distance between the \(y\)-components into the expectation in the event \(\left\{\tau_x > t\right\}\) and the event \(\left\{\tau_x \leq t\right\}\), and using the above inequality and a Cauchy-Schwarz inequality, we obtain
	\begin{equation}\label{eq:y_dist_bd}
		\begin{aligned}
			\mathbb{E}\left[\left|Y_t^0 - \widecheck{Y}_t^0\right|\right] &= \mathbb{E}\left[\left|Y_t^0 - \widecheck{Y}_t^0\right|\mathbf{1}_{\left\{\tau_x > t\right\}}\right] + \mathbb{E}\left[\left|Y_t^0 - \widecheck{Y}_t^0\right|\mathbf{1}_{\left\{\tau_x \leq t\right\}}\right]\\
			&\leq \mathbb{E}\left[\left|Y_t^0 - \widecheck{Y}_t^0\right|^2\right]^{1/2}\mathbb{E}\left[\mathbf{1}_{\left\{\tau_x > t\right\}}\right]^{1/2} + \mathbb{E}\left[\left|Y_t^0 - \widecheck{Y}_t^0\right|\mathbf{1}_{\left\{\tau_x \leq t\right\}}\right]\\
			&= \mathbb{E}\left[\left|Y_t^0 - \widecheck{Y}_t^0\right|^2\right]^{1/2}\mathbb{P}\left(\tau_x > t\right)^{1/2} + \mathbb{E}\left[\left|Y_t^0 - \widecheck{Y}_t^0\right|\mathbf{1}_{\left\{\tau_x \leq t\right\}}\right]\\
			&\leq \mathbb{E}\left[\left|Y_t^0 - \widecheck{Y}_t^0\right|^2\right]^{1/2}\mathbb{P}\left(\tau_x > t\right)^{1/2} + \mathbb{E}\left[\mathrm{e}^{-m(t - \tau_x)}\left|Y_{\tau_x}^0 - \widecheck{Y}_{\tau_x}^0\right|\mathbf{1}_{\left\{\tau_x \leq t\right\}}\right].\\
		\end{aligned}
	\end{equation}
	The first term is controlled using the moment growth bounds \eqref{eq:semigroup_estimates} and the inequality \eqref{eq:x_meeting_time_bd},
	\begin{equation*}
		\begin{aligned}
			\mathbb{E}\left[\left|Y_t^0 - \widecheck{Y}_t^0\right|^2\right]^{1/2}\mathbb{P}\left(\tau_x > t\right)^{1/2} &\leq 	\mathbb{E}\left[\sqrt{\mathcal{K}_2\left(Y_t^0\right)} + \sqrt{\mathcal{K}_2\left(\widecheck{Y}_t^0\right)}\right]\mathbb{P}\left(\tau_x > t\right)^{1/2}\\
			&\leq \sqrt{S_2} \left(\mathbb{E}\left[ \mathcal{K}_2\left(Y_0^0\right)\right]^{1/2} + \mathbb{E}\left[ \mathcal{K}_2\left(\widecheck{Y}_0^0\right)\right]^{1/2}\right)\mathrm{e}^{-ct/2}\sqrt{\frac{c}{\epsilon}\mathbb{E}\left[\left|X_0^\eta - \widecheck{X}_0^\eta\right|\right] + 1},
		\end{aligned}
	\end{equation*}
	where \(S_1 > 0\) is the constant from \eqref{eq:semigroup_estimates}.
	For the second term in \eqref{eq:y_dist_bd}, we use the Cauchy--Schwarz inequality to obtain
	\begin{equation*}
		\begin{aligned}
			\mathbb{E}\left[\mathrm{e}^{-m(t - \tau_x)}\left|Y_{\tau_x}^0 - \widecheck{Y}_{\tau_x}^0\right|\mathbf{1}_{\left\{\tau_x \leq t\right\}}\right] &\leq \mathbb{E}\left[\mathrm{e}^{-2m(t - \tau_x)}\mathbf{1}_{\left\{\tau_x \leq t\right\}}\right]^{1/2}\mathbb{E}\left[\left|Y_{\tau_x}^0 - \widecheck{Y}_{\tau_x}^0\right|^2\right]^{1/2}\\
			&\leq \mathbb{E}\left[\mathrm{e}^{-2m(t - \tau_x)}\mathbf{1}_{\left\{\tau_x \leq t\right\}}\right]^{1/2}\left(\mathbb{E}\left[\left|Y_{\tau_x}^0\right|^2\right]^{1/2} + \mathbb{E}\left[\left|\widecheck{Y}_{\tau_x}^0\right|^2\right]^{1/2}\right)\\
			&\leq \mathbb{E}\left[\mathrm{e}^{-2m(t - \tau_x)}\mathbf{1}_{\left\{\tau_x \leq t\right\}}\right]^{1/2}\left(\mathbb{E}\left[\mathcal{K}_2\left(Y_{\tau_x}^0\right)^2\right]^{1/2} + \mathbb{E}\left[\mathcal{K}_2\left(\widecheck{Y}_{\tau_x}^0\right)^2\right]^{1/2}\right)\\
			&\leq \mathbb{E}\left[\mathrm{e}^{-2m(t - \tau_x)}\mathbf{1}_{\left\{\tau_x \leq t\right\}}\right]^{1/2}\sqrt{S_2}\left(\mathbb{E}\left[\mathcal{K}_2\left(Y_0^0\right)\right]^{1/2} + \mathbb{E}\left[\mathcal{K}_2\left(\widecheck{Y}_0^0\right)\right]^{1/2}\right),
		\end{aligned}
	\end{equation*}
	where \(S_2 > 0\) is the constant from \eqref{eq:semigroup_estimates}.
	To compute the first expectation in the last line, we use the fact, which follows from the layer cake representation of an integral (see for example \cite[Section 1.13]{LiebLoss}), that, for a non-negative random variable \(S \geq 0\) and \(C^1\) function \(h\),
	\[\mathbb{E}\left[h\left(\max\{S, t\}\right)\right] = \int_0^\infty \mathbf{1}_{\left\{s \leq t\right\}} h'\left(s\right) \mathbb{P}\left(S > s\right)ds + h(0).\]
	Thus, in view of \eqref{eq:x_meeting_time_bd},
	\[\begin{aligned}
		\mathbb{E}\left[\mathrm{e}^{-2m(t - \tau_x)}\mathbf{1}_{\left\{\tau_x \leq t\right\}}\right] &\leq \mathbf{E}\left[\mathrm{e}^{-2\left(t - \max\left\{\tau_x, t\right\}\right)}\right]= 2m\int_0^t \mathrm{e}^{-2m\left(t - s\right)}\mathbb{P}\left(\tau_x > s\right)ds  + \mathrm{e}^{-2mt}\\
		&\leq 2m\left(\frac{c}{\epsilon}\mathbb{E}\left[\left|X_0^\eta - \widecheck{X}_0^\eta\right|\right] + 1\right)\mathrm{e}^{-2mt}\int_0^t \mathrm{e}^{\left(2m - c\right)s}ds + \mathrm{e}^{-2mt}\\
		&= 2m\left(\frac{c}{\epsilon}\mathbb{E}\left[\left|X_0^\eta - \widecheck{X}_0^\eta\right|\right] + 1\right)\mathrm{e}^{-2mt}\frac{ \mathrm{e}^{\left(2m-c\right)t} - 1}{2m -c} + \mathrm{e}^{-2mt}\\
		&= 2m\left(\frac{c}{\epsilon}\mathbb{E}\left[\left|X_0^\eta - \widecheck{X}_0^\eta\right|\right] + 1\right)\frac{ \mathrm{e}^{-ct} - \mathrm{e}^{-2mt}}{2m -c} + \mathrm{e}^{-2mt}.
	\end{aligned} \]
	Note that \(2m - c > 0\) as \(m \geq 4c\). Indeed by the definition of \(c\) in \eqref{eq:ref_coupling_constants}, we have
	\[\frac{1}{c} \geq 2\int_0^{R_1}\frac{\Phi(s)}{\phi(s)}ds \geq 2\int_{R_0}^{R_1}\frac{\Phi(s)}{\phi(s)}ds =  2\int_{R_0}^{R_1}\frac{\phi\left(R_0\right)\left(s - R_0\right)}{\phi\left(R_0\right)}ds = 2\int_{R_0}^{R_1}\left(s - R_0\right)ds = \left(R_1 - R_0\right)^2 = \frac{4}{m}.\]
	Consequently, the distance between the \(y\)-components is bounded as follows:
	\[\begin{aligned}
			\mathbb{E}\left[\left|Y_t^0 - \widecheck{Y}_t^0\right|\right] &\leq \sqrt{S_2} \left(\mathbb{E}\left[ \mathcal{K}_2\left(Y_0^0\right)\right]^{1/2} + \mathbb{E}\left[ \mathcal{K}_2\left(\widecheck{Y}_0^0\right)\right]^{1/2}\right)\mathrm{e}^{-ct/2}\sqrt{\frac{c}{\epsilon}\mathbb{E}\left[\left|X_0^\eta - \widecheck{X}_0^\eta\right|\right] + 1} \\
			&\quad+ \sqrt{S_2}\left(\mathbb{E}\left[\mathcal{K}_2\left(Y_0^0\right)\right]^{1/2} + \mathbb{E}\left[\mathcal{K}_2\left(\widecheck{Y}_0^0\right)\right]^{1/2}\right) \sqrt{2m\left(\frac{c}{\epsilon}\mathbb{E}\left[\left|X_0^\eta - \widecheck{X}_0^\eta\right|\right] + 1\right)\frac{ \mathrm{e}^{-ct} - \mathrm{e}^{-2mt}}{2m -c} + \mathrm{e}^{-2mt}}\\
			&\leq \mathrm{e}^{-ct/2}\sqrt{S_2} \left(\mathbb{E}\left[ \mathcal{K}_2\left(Y_0^0\right)\right]^{1/2} + \mathbb{E}\left[ \mathcal{K}_2\left(\widecheck{Y}_0^0\right)\right]^{1/2}\right)\\
			&\qquad \times \left[\sqrt{\frac{c}{\epsilon}\mathbb{E}\left[\left|X_0^\eta - \widecheck{X}_0^\eta\right|\right] + 1} +\sqrt{2m\left(\frac{c}{\epsilon}\mathbb{E}\left[\left|X_0^\eta - \widecheck{X}_0^\eta\right|\right] + 1\right)\frac{ \mathrm{e}^{-ct} - \mathrm{e}^{-2mt}}{2m -c} + \mathrm{e}^{-2mt}}\right]\\
			&\leq \mathrm{e}^{-ct/2}\sqrt{S_2} \left(\mathbb{E}\left[ \mathcal{K}_2\left(Y_0^0\right)\right]^{1/2} + \mathbb{E}\left[ \mathcal{K}_2\left(\widecheck{Y}_0^0\right)\right]^{1/2}\right) \left[\sqrt{1 + 2m}\sqrt{\frac{c}{\epsilon}\mathbb{E}\left[\left|X_0^\eta - \widecheck{X}_0^\eta\right|\right] + 1} + 1\right]\\
			&\leq \mathrm{e}^{-ct/2}\sqrt{S_2} \left(\mathbb{E}\left[ \mathcal{K}_2\left(Y_0^0\right)\right]^{1/2} + \mathbb{E}\left[ \mathcal{K}_2\left(\widecheck{Y}_0^0\right)\right]^{1/2}\right) \left[\sqrt{1 + 2m}\sqrt{\frac{c}{\epsilon}\mathbb{E}\left[\mathcal{K}_1\left(X_0^\eta\right) + \mathcal{K}_1\left(\widecheck{X}_0^\eta\right)\right] + 1} + 1\right]\\
	\end{aligned}\]
	where the third inequality is due to subadditivity of the square root and the fact that \(\frac{ \mathrm{e}^{-ct} - \mathrm{e}^{-2mt}}{2m -c} \leq 1\) which follows from the convexity of \(x \mapsto \mathrm{e}^{-xt}\). 
	Recalling the bound \eqref{eq:x_marginal_bd} on the \(x\) component, we obtain 
	\begin{equation*}
		\begin{aligned}
			\mathbb{E}\left[\left|Z_t^\eta - \widecheck{Z}_t^\eta\right|\right] &\leq \mathbb{E}\left[\left|X_t^\eta - \widecheck{X}_t^\eta\right|\right] + \mathbb{E}\left[\left|Y_t^0 - \widecheck{Y}_t^0\right|\right] \leq C\left(\mu, \widetilde{\mu}\right)\,\mathrm{e}^{-ct/2},
		\end{aligned}
	\end{equation*}
	where
	\begin{equation}\label{eq:sync_contraction_prefactor}
		\begin{aligned}
		C\left(\mu, \widetilde{\mu}\right) &= \frac{2}{\phi\left(R_0\right)}\left(\mu\left(\mathcal{K}_1\oplus \mathbf{0}\right) + \widetilde{\mu}\left(\mathcal{K}_1 \oplus \mathbf{0}\right)\right)\\
		 &\quad+ \sqrt{S_2}\left(\mu\left(\mathbf{0} \oplus\mathcal{K}_2\right)^{1/2} + \widetilde{\mu}\left(\mathbf{0} \oplus\mathcal{K}_2\right)^{1/2}\right)\left(\sqrt{1+2m}\sqrt{\frac{c}{\epsilon}\left[\mu\left(\mathbf{0}\oplus\mathcal{K}_1\right) + \widetilde{\mu}\left(\mathbf{0}\oplus\mathcal{K}_1\right)\right] + 1} + 1\right)
		\end{aligned},
	\end{equation}
	thus giving the claimed bound \eqref{eq:L1_contraction}.
\end{proof}
A corollary of this proposition and its interest for us is that the synchronously coupled dynamics admits a unique ergodic invariant probability measure, which we previously denoted by \(\mu_{\eta, \mathrm{sync}}\) in Section~\ref{sec:coupling}. We finish this section with a proof of this corollary.
\begin{corollary}
	For any \(\eta \in \mathbb{R}\), the synchronously coupled dynamics \eqref{eq:sync_z} admits a unique invariant probability measure. 
\end{corollary}

\begin{proof}
	Let \(\mu\) be a probability measure on \(\mathbb{R}^{2d}\) which admits second moments. This assumption is in fact not restrictive as any invariant measure of the synchronously coupled dynamics necessarily admits moments of all orders since it is a coupling of \(\nu_\eta\) and \(\nu_0\) and the moment bounds \eqref{eq:moment_bounds} imply that these two measures admit moments of all orders. Using the explicit expression \eqref{eq:sync_contraction_prefactor} of the prefactor in \eqref{eq:W1_contraction} and the growth estimates on the semigroups of the marginal dynamics \eqref{eq:semigroup_estimates}, it is clear that there exists a constant \(K_\mu > 0 \) that may depend on \(\mu\) such that
	\[\forall t, s \geq 0, \qquad C\left(\mu T_t^\eta, \mu T_s^\eta\right) \leq K_\mu.\]
	 
	 Fix a probability measure \(\mu\) on \(\mathbb{R}^{2d}\) with finite second moments. For any two times \(t \geq s \geq 0\), we have that 
	 \[\mathcal{W}^1\left(\mu T_t^\eta, \mu T_s^\eta\right) = \mathcal{W}^1\left(\mu T_{t-s}^\eta T_s, \mu T_s^\eta\right) \leq C\left(\mu T_{t-s}^\eta, \mu\right) \mathrm{e}^{-\gamma s} \leq K_\mu \mathrm{e}^{-\gamma s}. \]
	 Consequently the sequence \(\left(\mu T_t^\eta\right)_{t\geq 0}\) is a Cauchy sequence for the \(\mathcal{W}^1\) distance and therefore has a unique limit, denoted by \(\mu_\infty\). By the triangle inequality, we have for any \(t, s \geq 0\)
	 \[\mathcal{W}^1\left(\mu_\infty T_t^\eta,\mu_\infty\right) \leq \mathcal{W}^1\left(\mu_\infty T_t^\eta, \mu T_s^\eta T_t^\eta\right) + \mathcal{W}^1\left(\mu T_{s+t}^\eta, \mu_\infty\right). \]
	 Since the map \(\rho \mapsto \rho T_t^\eta\) is continuous with respect to the \(\mathcal{W}^1\) distance for \(t \geq 0\) fixed and \(\mu T_s\) converges to \(\mu_\infty \) for \(\mathcal{W}^1\) distance, the right hand side of the above inequality tends to zeros as \(s \to \infty\). Thus, \(\mu_\infty T_t = \mu_\infty\) for any~\(t \geq 0\), i.e \(\mu_\infty\) is an invariant probability measure with respect to~\(\left(T_t\right)_{t\geq 0}\).
	 
	 To see that \(\mu_\infty\) does not depend on the choice of initial measure, we can repeat this construction for another probability measure \(\nu\) with finite second moments to obtain another invariant probability measure \(\nu_\infty\) on \(\mathbb{R}^{2d}\). We claim that \(\nu_\infty = \mu_\infty\). Indeed, since both \(\mu_\infty\) and \(\nu_\infty\) are invariant with respect to \(\left(T_t^\eta\right)_{t\geq 0}\) we have, for any~\(t \geq 0\),
	 \[\mathcal{W}^1\left(\mu_\infty, \nu_\infty\right) = \mathcal{W}^1\left(\mu_\infty T_t^\eta, \nu_\infty T_t^\eta\right) \leq C\left(\mu_\infty, \nu_\infty\right) \, \mathrm{e}^{-\gamma t}.\]
	 Thus \(\mathcal{W}^1\left(\mu_\infty, \nu_\infty\right) = 0\) and \(\mu_\infty = \nu_\infty\). Furthermore, by the same reasoning using invariance, any probability measure \(\rho\) with finite second moments that is invariant with respect to \(\left(T_t^\eta\right)_{t\geq 0}\) is necessarily equal to \(\mu_\infty\), i.e. \(\mu_\infty\) is the unique invariant probability measure of the synchronously coupled dynamics.
\end{proof}

\section{Proofs of Lemmas \ref{lm:disc_lin_resp} and \ref{lm:disc_poisson_sol_approx}}\label{sec:discrete_lin_resp_proofs}
In this appendix, we prove Lemmas~\ref{lm:disc_lin_resp}~and~\ref{lm:disc_poisson_sol_approx} using the strategy from \cite[Section 4]{Leimkuhler} of approximating the transition kernel and the inverse discrete generator. We first present some of the tools necessary for the arguments for proving the two lemmas and then conclude with their proofs. These arguments result in some explicit correction terms and remainders involving higher powers of \(\Delta t\) and/or \(\eta\) and what we call well-behaved operators. Precisely, we call an operator \(\mathcal{D}_{\Delta t, \eta, \theta}\) (possibly depending on \(\Delta t\), \(\eta \) and a parameter \(\theta\) belonging to a compact parameter space \(\Theta\)) well behaved if its domain contains \(\mathscr{S}\), it stabilizes \(\mathscr{S}\), and  there exists \(p \in \mathbb{N}\), such that for any \(\eta_\star, \Delta t^\star > 0\) and \(n \in \mathbb{N}\), there exist \(m_n \in \mathbb{N}\) and \(C_{n, \eta_\star, \Delta t^{\star}} > 0\) such that, for any \(\varphi \in \mathscr{S} \cap C_{n}^{p}\left(\mathbb{R}^d\right)\),
\begin{equation}\label{eq:well_behaved_def}
	\forall \eta \in \left[-\eta_\star, \eta_\star\right], \quad \Delta t \in \left(0, \Delta t^*\right), \quad \theta \in \Theta, \qquad \left\|\mathcal{D}_{\Delta t, \eta, \theta}\varphi\right\|_{\mathcal{K}_{m_n}} \leq C_{n, \eta_\star, \Delta t^{\star}} \sum_{|\alpha| \leq p} \left\|\partial^\alpha \varphi\right\|_{\mathcal{K}_n}.
\end{equation}
Recall from Appendix~\ref{sec:Kopec_ext} that \(C_m^p\left(\mathbb{R}^d\right)\) is the space of \(p\)-times continuously differentiable functions that belong to \(B_m^\infty\) and that have derivatives in \(B_m^\infty\).
Similarly, we call a functional \(\mathcal{A}_{\Delta t,\eta, \theta}: \mathscr{S} \to \mathbb{R}\) (possibly depending on \(\Delta t\), \(\eta \) and a parameter \(\theta\) belonging to a compact parameter space \(\Theta\)) well behaved if it satisfies an inequality similar to the one above but with an absolute value on the left hand side. More precisely, there exists \(p \in \mathbb{N}\) such that for any \(\eta_\star, \Delta t^* > 0\) and \(n \in \mathbb{N}\), there exists a constant \(C_{n, \eta_\star, \Delta t^{\star}} > 0\) such that, for any \(\varphi \in \mathscr{S} \cap C_n^p\left(\mathbb{R}^d\right)\),
\begin{equation}\label{eq:well_behaved_functional_def}
	\forall \eta \in \left[-\eta_\star, \eta_\star\right], \quad \Delta t \in \left(0, \Delta t^*\right), \quad \theta \in \Theta, \qquad \left|\mathcal{A}_{\Delta t, \eta, \theta}\varphi\right| \leq C_{n, \eta_\star, \Delta t^{\star}} \sum_{|\alpha| \leq p} \left\|\partial^\alpha \varphi\right\|_{\mathcal{K}_n}.
\end{equation}
Denote by \(B_\eta\) and \(C_\beta\) the generators of the following semigroups: 
\begin{equation*}
	\begin{aligned}
		\left(\mathrm{e}^{tB_\eta}\varphi\right)(x) = \varphi\left(x + t\left[b(x) + \eta F(x)\right]\right), \qquad
		\left(\mathrm{e}^{tC_\beta}\varphi\right)(x) = \int_{\mathbb{R}^d} \varphi\left(x + \sqrt{\frac{2t}{\beta}}z\right) \frac{\mathrm{e}^{-\frac{|z|^2}{2}}dz}{\left(2\pi\right)^{d/2}}.
	\end{aligned}
\end{equation*}
We write \(B := B_0\) when \(\eta = 0\). 
A simple computation shows that \(B_\eta = B + \eta \widetilde{\mathcal{L}}\) with \(B = b\cdot \nabla\) and \(C_{\beta} = \frac{1}{\beta}\Delta \). The transition kernel \(P^{\eta, \Delta t}\) then can be written as 
\[P^{\eta, \Delta t} = \mathrm{e}^{\Delta t B_\eta}\mathrm{e}^{\Delta t C_\beta}.\]
Viewing the transition kernel as a function of the time step, \(t \mapsto P^{\eta, t}\), we make the following Taylor expansion:
\[P^{\eta, \Delta t} = \mathrm{Id} + \Delta t \frac{\mathrm{d}}{\mathrm{d}t}P^{\eta, t}\big|_{t =0 } + \frac{\Delta t^2}{2} \frac{\mathrm{d}^2}{\mathrm{d}t^2}P^{\eta, t}\big|_{t = 0} + \frac{\Delta t^3}{2}\int_0^1 (1- \theta) \frac{\mathrm{d}^3}{\mathrm{d}t^3}P^{\eta, t}\big|_{t = \theta \Delta t}d\theta.\] 
Computing explicitly these derivatives we get
\begin{align}
	\frac{\mathrm{d}}{\mathrm{d}t}P^{\eta, t}\big|_{t =0 } &= B_\eta + C_\beta = \mathcal{L}_0 + \eta \widetilde{\mathcal{L}},\\
	\frac{\mathrm{d}^2}{\mathrm{d}t^2}P^{\eta, t}\big|_{t = 0} &= B_\eta^2 + 2B_\eta C_\beta + C_\beta^2 = \mathcal{L}_0^2 + S + \eta D + \eta^2 \widetilde{\mathcal{L}}^2,\\
	\frac{\mathrm{d}^3}{\mathrm{d}t^3}P^{\eta, t}\big|_{t =  s} &= B_\eta^3\mathrm{e}^{sB_\eta}\mathrm{e}^{sC_\beta} + 3B_\eta^2 \mathrm{e}^{sB_\eta}C_\beta \mathrm{e}^{sC_\beta} + 3 B_\eta \mathrm{e}^{sB_\eta}C_\beta^2\mathrm{e}^{sC_\beta} + \mathrm{e}^{sB_\eta}C_\beta^3\mathrm{e}^{sC_\beta},\label{eq:P_third_deriv}
\end{align}
where \(S = BC_\beta - C_\beta B\) and \(D = B\widetilde{\mathcal{L}} + \widetilde{\mathcal{L}}(B + 2C_\beta)\). Denote by \(\mathcal{R}_{\eta, s}\) be the operator given by the right hand side of \eqref{eq:P_third_deriv}. Using the equality
\[\mathrm{e}^{sB_\eta} - \mathrm{e}^{sB} = \eta \int_0^1 \mathrm{e}^{\theta s B_\eta}\widetilde{\mathcal{L}}\mathrm{e}^{s(1-\theta)B}d\theta,\]
we can write \(\mathcal{R}_{\eta, s}\) as
\[\mathcal{R}_{\eta, s} = \mathcal{R}_{0, s} + \eta \widetilde{\mathcal{R}}_{\eta, s},\]
where \(\mathcal{R}_{0, s}\) corresponds to case when \(\eta = 0\) and \(\widetilde{\mathcal{R}}_{\eta, s}\) is some well-behaved operator. Indeed, \(B_\eta\), \(C_\beta\), \(\mathrm{e}^{tB_\eta}\), and \(\mathrm{e}^{tC_\beta}\) all stabilize \(\mathscr{S}\) and satisfy \eqref{eq:well_behaved_def} and the composition of well-behaved operators is a well-behaved operator. Consequently the transition kernel \(P^{\eta, \Delta t}\) can be written as
\[P^{\eta, \Delta t} = \mathrm{Id} + \Delta t \left(\mathcal{L}_0 + \eta \widetilde{\mathcal{L}}\right) + \frac{\Delta t^2}{2}\left(\mathcal{L}_0^2 + S + \eta D + \eta^2 \widetilde{\mathcal{L}}^2\right) + \frac{\Delta t^3}{3}\int_0^1 (1- \theta) \left(\mathcal{R}_{0, \theta\Delta t } + \eta \widetilde{\mathcal{R}}_{\eta, \theta\Delta t}\right)d\theta,\]
and the discrete generator can be written as
\begin{equation}\label{eq:disc_generator_approx}
	\frac{\mathrm{Id} - P^{\eta, \Delta t}}{\Delta t} = -\left(\mathcal{L}_0 + \eta \widetilde{\mathcal{L}}\right) - \frac{\Delta t}{2}\left(\mathcal{L}_0^2 + S + \eta D\right) + \Delta t^2 \mathcal{D}_{1, \Delta t} + \eta^2 \Delta t\mathcal{D}_{2,\Delta t} + \eta \Delta t^2\mathcal{D}_{3, \eta, \Delta t},
\end{equation}
where \(\mathcal{D}_{1}\), \(\mathcal{D}_{2, \Delta t}\), and \(\mathcal{D}_{3, \eta, \Delta t}\) some well behaved operators. We approximate the inverse of the discrete generator restricted to the range of \(\Pi_0\) with the operator from \(\mathscr{S}\) to \(\mathscr{S}_0\)
\begin{equation}\label{eq:approx_inverse}
	\begin{aligned}
	\mathcal{Q}_{\eta, \Delta t} &= -\mathcal{L}_0^{-1}\Pi_0 + \eta \mathcal{L}_0^{-1}\Pi_0 \widetilde{\mathcal{L}}\mathcal{L}_0^{-1}\Pi_0 + \frac{\Delta t}{2}\left(\Pi_0 + \mathcal{L}_0^{-1}\Pi_0 \left(S + \eta D\right)\mathcal{L}_0^{-1}\Pi_0\right)\\
	&\qquad + \frac{\eta \Delta t}{2}\mathcal{L}_0^{-1}\Pi_0\widetilde{\mathcal{L}}\left(\Pi_0 + \mathcal{L}_0^{-1}\Pi_0S\mathcal{L}_0^{-1}\Pi_0\right)\\
	&\qquad + \frac{\eta \Delta t}{2}\left(\Pi_0 + \mathcal{L}_0^{-1}\Pi_0 S\mathcal{L}_0^{-1}\right)\Pi_0\widetilde{\mathcal{L}}\mathcal{L}_0^{-1}\Pi_0,
	\end{aligned}
\end{equation}
which is constructed by formally taking the inverse of \(-\mathcal{L}_0 - \eta \widetilde{\mathcal{L}} - \frac{\Delta t}{2}\left(\mathcal{L}_0^2 + S + \eta D\right)\) using the formula \(\left(A + B\right)^{-1} = A^{-1} - A^{-1}BA^{-1} + A^{-1}BA^{-1}BA^{-1} - \dots\) and discarding terms of order \(\eta^2\), \(\Delta t^2 \) or higher. There exist well-behaved operators \(\widetilde{\mathcal{D}}_1\), \(\widetilde{\mathcal{D}}_2\), \(\widetilde{\mathcal{D}}_{3, \Delta t}\), and \(\widetilde{\mathcal{D}}_{4, \eta, \Delta t}\) such that
\begin{equation}\label{eq:approx_inverse_equality}
	\begin{aligned}
	\Pi_0\left(\frac{\mathrm{Id} - P^{\eta, \Delta t}}{\Delta t}\right)\mathcal{Q}_{\eta, \Delta t} &= \Pi_0 + \eta \Pi_0\left(\mathrm{Id} - \Pi_0\right)\widetilde{\mathcal{L}}\mathcal{L}_0^{-1}\Pi_0 + \frac{\Delta t}{2}\Pi_0\left(\mathrm{Id} - \Pi_0\right)S\mathcal{L}_0^{-1}\Pi_0 \\
	&\qquad -\frac{\eta \Delta t}{2}\Pi_0\left(\mathrm{Id} - \Pi_0 \right)S \mathcal{L}_0^{-1}\Pi_0\widetilde{\mathcal{L}}\mathcal{L}_0^{-1}\Pi_0 \\
	&\qquad+ \Delta t^2 \widetilde{\mathcal{D}}_1 + \eta^2 \widetilde{\mathcal{D}}_2 + \eta^2 \Delta t\widetilde{\mathcal{D}}_{3, \Delta t} + \eta \Delta t^2 \widetilde{\mathcal{D}}_{4, \eta, \Delta t}\\
	&= \Pi_0 + \Delta t^2 \left(\widetilde{\mathcal{D}}_1 +  \eta\widetilde{\mathcal{D}}_{4, \eta, \Delta t}\right)+ \eta^2 \left(\widetilde{\mathcal{D}}_2 + \Delta t\widetilde{\mathcal{D}}_{3, \Delta t}\right).
	\end{aligned}
\end{equation}
The last equality is due the fact that \(\Pi_0\left(\mathrm{Id} - \Pi_0\right) \equiv 0\) since the range of the operator \(\left(\mathrm{Id} - \Pi_0\right)\) is the constant functions and is thus contained in the kernel of any projection onto the space of function with mean zero with respect some probability measure. Note that we can write the approximate inverse as \(\mathcal{Q}_{\eta, \Delta t} = \mathcal{Q}_{0, \Delta t} + \eta \widetilde{\mathcal{Q}}_{\eta, \Delta t}\) where \(\mathcal{Q}_{0, \Delta t}\) is given by \eqref{eq:approx_inverse} with \(\eta = 0\) and \(\widetilde{\mathcal{Q}}_{\eta, \Delta t}\) by
\begin{equation}\label{eq:Q_til_def}
	\begin{aligned}
		\widetilde{Q}_{\eta, \Delta t} &= \eta \mathcal{L}_0^{-1}\Pi_0 \widetilde{\mathcal{L}}\mathcal{L}_0^{-1}\Pi_0 - \frac{\eta\Delta t}{2}\mathcal{L}_0^{-1}\Pi_0 D\mathcal{L}_0^{-1}\Pi_0\\
		&\qquad + \frac{\eta \Delta t}{2}\mathcal{L}_0^{-1}\Pi_0\widetilde{\mathcal{L}}\left(\Pi_0 + \mathcal{L}_0^{-1}\Pi_0S\mathcal{L}_0^{-1}\Pi_0\right)\\
		&\qquad + \frac{\eta \Delta t}{2}\left(\Pi_0 + \mathcal{L}_0^{-1}\Pi_0S \mathcal{L}_0^{-1}\right)\Pi_0\widetilde{\mathcal{L}}\mathcal{L}_0^{-1}\Pi_0,
	\end{aligned}
\end{equation} We can now prove Lemmas~\ref{lm:disc_lin_resp}~and~\ref{lm:disc_poisson_sol_approx}.

\begin{proof}[Proof of Lemma~\ref{lm:disc_lin_resp}]
	We choose \(\mathfrak{f}_1\) such that \[-\mathcal{L}_0^* \mathfrak{f}_1 = \frac{1}{2}S^*\mathbf{1},\]
	\(\mathfrak{f}_2\) such that
	\[-\mathcal{L}_0^*\mathfrak{f}_2 = \widetilde{\mathcal{L}}^*\mathbf{1},\]
	\(\mathfrak{f}_3\) such that
	\[-\mathcal{L}_0^*\mathfrak{f}_3 = \widetilde{\mathcal{L}}^*\mathfrak{f}_1 + \frac{1}{2}\left(\mathcal{L}_0^2 + S\right)^*\mathfrak{f}_2 + \frac{1}{2}D^*\mathbf{1}.\]
	The bound \eqref{eq:inverse_bounds} ensures that each of these equations has a unique solution in \(\Pi_0B_n^\infty\) for any \(n \in \mathbb{N}\) large enough.
	This choice of \(\mathfrak{f}_1\), \(\mathfrak{f}_2\), and \(\mathfrak{f}_3\) ensure that for, \(\varphi \in \mathscr{S}\),
	\begin{equation}\label{eq:approx_invariant_measure}
		\int_{\mathbb{R}^d} \left[\left(\frac{\mathrm{Id} - P^{\eta, \Delta t}}{\Delta t}\right)\varphi\right]\left(1 + \Delta t \mathfrak{f}_1 + \eta \mathfrak{f}_2 + \eta \Delta t \mathfrak{f}_3\right)d\nu_0 = \eta^2 \widetilde{\mathcal{A}}_{1, \Delta t}\varphi + \Delta t^2 \left(\widetilde{\mathcal{A}}_{2, \Delta t}\varphi + \eta \widetilde{\mathcal{A}}_{3, \Delta t, \eta}\varphi\right),
	\end{equation}
	with right hand side involving well-behaved functionals \(\widetilde{\mathcal{A}}_{1, \Delta t}\), \(\widetilde{\mathcal{A}}_{2, \Delta t}\psi\), and \(\widetilde{\mathcal{A}}_{3, \Delta t, \eta}\). 
	For the invariant measure of the discretized process, we have by definition, for any \(\varphi \in \mathscr{S}\),
	\[\int_{\mathbb{R}^d} \left(\frac{\mathrm{Id} - P^{\eta, \Delta t}}{\Delta t}\right)\varphi\, d\nu_{\eta, \Delta t} = 0. \]
	Restricting ourselves to \(\varphi \in \mathscr{S}_0\) and applying the projector \(\Pi_0\) to the above integrand, we obtain
	\[\begin{aligned}
		\int_{\mathbb{R}^d} \Pi_0\left(\frac{\mathrm{Id} - P^{\eta, \Delta t}}{\Delta t}\right)\varphi\, d\nu_{\eta, \Delta t} &= \int_{\mathbb{R}^d} \left(\frac{\mathrm{Id} - P^{\eta, \Delta t}}{\Delta t}\right)\varphi\, d\nu_{\eta, \Delta t} - \int_{\mathbb{R}^d} \left(\frac{\mathrm{Id} - P^{\eta, \Delta t}}{\Delta t}\right)\varphi\, d\nu_0\\
		&= \Delta t^{-1}\int_{\mathbb{R}^d} P^{\eta, \Delta t}\varphi \, d\nu_0.
	\end{aligned}\]
	Additionally, using \eqref{eq:approx_invariant_measure}, we obtain, for \(\varphi \in \mathscr{S}_0\),
	\[\begin{aligned}
		&\int_{\mathbb{R}^d} \left[\Pi_0\left(\frac{\mathrm{Id} - P^{\eta, \Delta t}}{\Delta t}\right)\varphi\right]\left(1 + \Delta t \mathfrak{f}_1 + \eta \mathfrak{f}_2 + \eta \Delta t \mathfrak{f}_3\right)d\nu_0\\
		&\qquad = \int_{\mathbb{R}^d} \left(\frac{\mathrm{Id} - P^{\eta, \Delta t}}{\Delta t}\right)\varphi\left(1 + \Delta t \mathfrak{f}_1 + \eta \mathfrak{f}_2 + \eta \Delta t \mathfrak{f}_3\right)d\nu_0 - \int_{\mathbb{R}^d}\left(\frac{\mathrm{Id} - P^{\eta, \Delta t}}{\Delta t}\right)\varphi \, d\nu_0\\
		&\qquad = \eta^2 \widetilde{\mathcal{A}}_{1, \Delta t}\varphi + \Delta t^2 \left(\widetilde{\mathcal{A}}_{2, \Delta t}\varphi + \eta \widetilde{\mathcal{A}}_{3, \Delta t, \eta}\varphi\right) + \Delta t^{-1}\int_{\mathbb{R}^d} P^{\eta, \Delta t}\varphi \, d\nu_0,
	\end{aligned}\]
	where the first equality follows from the fact that \(\mathfrak{f}_1, \mathfrak{f}_2, \mathfrak{f}_3 \in \mathscr{S}_0\). 
	Consequently, combining these two equalities gives, for \(\varphi \in \mathscr{S}_0\),
	\[\begin{aligned}
		\int_{\mathbb{R}^d} \Pi_0\left(\frac{\mathrm{Id} - P^{\eta, \Delta t}}{\Delta t}\right)\varphi\, d\nu_{\eta, \Delta t} &= \int_{\mathbb{R}^d} \left[\Pi_0\left(\frac{\mathrm{Id} - P^{\eta, \Delta t}}{\Delta t}\right)\varphi\right]\left(1 + \Delta t \mathfrak{f}_1 + \eta \mathfrak{f}_2 + \eta \Delta t \mathfrak{f}_3\right)d\nu_0\\
		&\qquad - \eta^2 \widetilde{\mathcal{A}}_{1, \Delta t}\varphi - \Delta t^2 \left(\widetilde{\mathcal{A}}_{2, \Delta t}\varphi + \eta \widetilde{\mathcal{A}}_{3, \Delta t, \eta}\varphi\right)
	\end{aligned}\]
	For \(\psi \in \mathscr{S}\), we substitute \(\varphi\) in the above equality with \(\mathcal{Q}_{\eta, \Delta t}\psi \in \mathscr{S}_0\):
	\begin{equation}\label{eq:integral_of_pseudoinv_equation}
		\begin{aligned}
			\int_{\mathbb{R}^d} \Pi_0\left(\frac{\mathrm{Id} - P^{\eta, \Delta t}}{\Delta t}\right)\mathcal{Q}_{\eta, \Delta t}\psi \, d\nu_{\eta, \Delta t} &= \int_{\mathbb{R}^d} \left[\Pi_0\left(\frac{\mathrm{Id} - P^{\eta, \Delta t}}{\Delta t}\right)\mathcal{Q}_{\eta, \Delta t}\psi\right]\left(1 + \Delta t \mathfrak{f}_1 + \eta \mathfrak{f}_2 + \eta \Delta t \mathfrak{f}_3\right)d\nu_0\\
			&\qquad - \eta^2 \widetilde{\mathcal{A}}_{1, \Delta t}\mathcal{Q}_{\eta, \Delta t}\psi - \Delta t^2 \left(\widetilde{\mathcal{A}}_{2, \Delta t}\mathcal{Q}_{\eta, \Delta t}\psi + \eta \widetilde{\mathcal{A}}_{3, \Delta t, \eta}\mathcal{Q}_{\eta, \Delta t}\psi\right).
		\end{aligned}
	\end{equation}
	By \eqref{eq:approx_inverse_equality}, the left hand side of this equality becomes
	\[\begin{aligned}
		\int_{\mathbb{R}^d} \Pi_0\left(\frac{\mathrm{Id} - P^{\eta, \Delta t}}{\Delta t}\right)\mathcal{Q}_{\eta, \Delta t}\psi \, d\nu_{\eta, \Delta t} &= \int_{\mathbb{R}^d}\psi \, d\nu_{\eta, \Delta t} - \int_{\mathbb{R}^d}\psi \, d\nu_0\\
		&\qquad + \int_{\mathbb{R}^d} \left(\Delta t^2 \widetilde{\mathcal{D}}_1 + \eta^2 \widetilde{\mathcal{D}}_2 + \eta^2 \Delta t\widetilde{\mathcal{D}}_{3, \Delta t} + \eta \Delta t^2 \widetilde{\mathcal{D}}_{4, \eta, \Delta t}\right)\psi \, d\nu_{\eta, \Delta t},
	\end{aligned}\]
	and the integral on the right hand side becomes
	\[\begin{aligned}
		&\int_{\mathbb{R}^d} \left[\Pi_0\left(\frac{\mathrm{Id} - P^{\eta, \Delta t}}{\Delta t}\right)\mathcal{Q}_{\eta, \Delta t}\psi\right]\left(1 + \Delta t \mathfrak{f}_1 + \eta \mathfrak{f}_2 + \eta \Delta t \mathfrak{f}_3\right)d\nu_0\\
		&\qquad = \int_{\mathbb{R}^d} \psi \left(\Delta t \mathfrak{f}_1 + \eta \mathfrak{f}_2 + \eta \Delta t \mathfrak{f}_3\right)d\nu_0\\
		&\qquad\qquad + \int_{\mathbb{R}^d}\left(\Delta t^2 \widetilde{\mathcal{D}}_1 + \eta^2 \widetilde{\mathcal{D}}_2 + \eta^2 \Delta t\widetilde{\mathcal{D}}_{3, \Delta t} + \eta \Delta t^2 \widetilde{\mathcal{D}}_{4, \eta, \Delta t}\right)\psi \left(1 + \Delta t \mathfrak{f}_1 + \eta \mathfrak{f}_2 + \eta \Delta t \mathfrak{f}_3\right)d\nu_0.
	\end{aligned}\]
	Thus \eqref{eq:integral_of_pseudoinv_equation} becomes
	\begin{equation}\label{eq:pre_lin_resp_eq}
		\begin{aligned}
			\int_{\mathbb{R}^d}\psi \, d\nu_{\eta, \Delta t} &= \int_{\mathbb{R}^d} \psi \left(1 + \Delta t \mathfrak{f}_1 + \eta \mathfrak{f}_2 + \eta \Delta t \mathfrak{f}_3\right) d\nu_0\\
			&\qquad - \int_{\mathbb{R}^d} \left(\Delta t^2 \widetilde{\mathcal{D}}_1 + \eta^2 \widetilde{\mathcal{D}}_2 + \eta^2 \Delta t\widetilde{\mathcal{D}}_{3, \Delta t} + \eta \Delta t^2 \widetilde{\mathcal{D}}_{4, \eta, \Delta t}\right)\psi \, d\nu_{\eta, \Delta t}\\
			&\qquad + \int_{\mathbb{R}^d}\left(\Delta t^2 \widetilde{\mathcal{D}}_1 + \eta^2 \widetilde{\mathcal{D}}_2 + \eta^2 \Delta t\widetilde{\mathcal{D}}_{3, \Delta t} + \eta \Delta t^2 \widetilde{\mathcal{D}}_{4, \eta, \Delta t}\right)\psi \left(1 + \Delta t \mathfrak{f}_1 + \eta \mathfrak{f}_2 + \eta \Delta t \mathfrak{f}_3\right)d\nu_0\\
			&\qquad - \eta^2 \widetilde{\mathcal{A}}_{1, \Delta t}\mathcal{Q}_{\eta, \Delta t}\psi - \Delta t^2 \left(\widetilde{\mathcal{A}}_{2, \Delta t}\mathcal{Q}_{\eta, \Delta t}\psi + \eta \widetilde{\mathcal{A}}_{3, \Delta t, \eta}\mathcal{Q}_{\eta, \Delta t}\psi\right).
		\end{aligned}
	\end{equation}
	We would like rewrite the remainder terms in the above equality in terms of well-behaved functionals and higher powers of \(\eta\) and \(\Delta t\) as claimed in the statement of the lemma. Using the fact that \(\mathcal{Q}_{\eta, \Delta t} = \mathcal{Q}_{0, \Delta t} + \eta \widetilde{\mathcal{Q}}_{\eta, \Delta t}\) and grouping terms by powers of \(\eta\) and \(\Delta t\), this is clearly possible for all the remainder terms except for \(-\Delta t \int_{\mathbb{R}^d} \widetilde{\mathcal{D}}_1 \psi \, d\nu_{\eta, \Delta t}\). Writing this term as
	\[\int_{\mathbb{R}^d} \widetilde{\mathcal{D}}_1 \psi \, d\nu_{\eta, \Delta t}  = \int_{\mathbb{R}^d} \widetilde{\mathcal{D}}_1 \psi \, d\nu_{0, \Delta t} + \left(\int_{\mathbb{R}^d} \widetilde{\mathcal{D}}_1 \psi \, d\nu_{\eta, \Delta t} - \int_{\mathbb{R}^d} \widetilde{\mathcal{D}}_1 \psi \, d\nu_{0, \Delta t}\right),\]
	we see that if the difference is of order \(\eta\) then all the remainder terms in \eqref{eq:pre_lin_resp_eq} can be written in terms of well-behaved functionals and higher powers of \(\eta\) and \(\Delta t\) in an appropriate way. Recall that \(\widetilde{V}_c(x) = \exp\left(c|x|^2\right)\) and \(V_c(x,y) = \widetilde{V}_c(x) + \widetilde{V}_c(y)\) with \(c > 0\) having the same value as in \eqref{eq:marginal_lyapunov_ineq}. Since \(\varphi \in \mathscr{S}\), all its derivatives are in \(B^\infty_{\widetilde{V}_c}\) and \(\widetilde{\mathcal{D}}_1 \varphi \in B^\infty_{\widetilde{V}_c}\). Using the fact that \(\mu_{\eta, \Delta t}\) is a coupling of \(\nu_{0, \Delta t}\) and \(\nu_{\eta, \Delta t}\) and applying \eqref{eq:W_n_bound} in Lemma~\ref{lm:W_n_bound}, we have
	\[\left|\int_{\mathbb{R}^d} \widetilde{\mathcal{D}}_1 \psi d\nu_{\eta, \Delta t} - \int_{\mathbb{R}^d} \widetilde{\mathcal{D}}_1 \psi d\nu_{0, \Delta t}\right| = \left|\int_{\mathbb{R}^d\times\mathbb{R}^d} \left(\widetilde{D}_1\psi(x) - \widetilde{D}_1\psi(y)\right)d\mu_{\eta, \Delta t}\right| \leq \left\|\widetilde{D}_1\psi\right\|_{\widetilde{V}_c}\int_{\mathbb{R}^d\times\mathbb{R}^d}\!\!\!\!\!\!\!\!\mathbf{1}_{\left\{x\neq y\right\}}V_c(x,y)d\mu_{\eta, \Delta t}. \]
	The bound \eqref{eq:stationary_discrete_weightedtv_bound} in Proposition~\ref{prop:discrete_Wn_norm_bound} lets us bound the integral on the right hand side by \(C\eta\left[\nu_{\eta, \Delta t}\left(\widetilde{V}_c\right) + \nu_{0, \Delta t}\left(\widetilde{V}_c\right)\right]\). By \eqref{eq:disc_moment_bounds}, \(\nu_{\eta, \Delta t}\left(\widetilde{V}_c\right)\) is uniformly bounded in \(\eta \in \left[-\eta_\star, \eta_\star\right]\). Consequently, the difference \(\int_{\mathbb{R}^d} \widetilde{\mathcal{D}}_1 \psi d\nu_{\eta, \Delta t} - \int_{\mathbb{R}^d} \widetilde{\mathcal{D}}_1 \psi d\nu_{0, \Delta t}\) is of order \(\eta\) and we can safely conclude that there exist well-behaved functionals \(\mathcal{A}_{1, \Delta t, \eta}\), \(\mathcal{A}_{2, \Delta t}\), and \(\mathcal{A}_{3, \Delta t, \eta}\) such that
	\begin{equation}
		\begin{aligned}
			\int_{\mathbb{R}^d}\psi \, d\nu_{\eta, \Delta t} &= \int_{\mathbb{R}^d} \psi \left(1 + \Delta t \mathfrak{f}_1 + \eta \mathfrak{f}_2 + \eta \Delta t \mathfrak{f}_3\right)d\nu_0 + \eta^2 \mathcal{A}_{1, \Delta t, \eta}\psi + \Delta t^2\left(\mathcal{A}_{2, \Delta t} + \eta \mathcal{A}_{3, \Delta t, \eta}\right)\psi,
		\end{aligned}
	\end{equation}
	which concludes the proof.
\end{proof}
\begin{remark}
	In the proof of \cite[Theorem 3.4]{Leimkuhler}, there is also a remainder term of the form
	\[\int_{\mathbb{R}^d \times \mathbb{R}^d} \mathcal{D} \varphi \, d\mu_{\gamma, \Delta t, \eta}\]
	where \(\mu_{\gamma, \Delta t, \eta}\) is the invariant measure of the discretized non-equilibrium underdamped Langevin dynamics with time step \(\Delta t > 0\), friction \(\gamma > 0\), and perturbation of size \(\eta \in \mathbb{R}\). This integral in the remainder term was not properly controlled as \(\eta\) went to zero. Writing
	\[\int_{\mathbb{R}^d \times \mathbb{R}^d} \mathcal{D} \varphi \, d\mu_{\gamma, \Delta t, \eta} = \int_{\mathbb{R}^d \times \mathbb{R}^d} \mathcal{D} \varphi \, d\mu_{\gamma, \Delta t, 0} + \eta \left(\frac{\int_{\mathbb{R}^d \times \mathbb{R}^d} \mathcal{D} \varphi \, d\mu_{\gamma, \Delta t, \eta} - \int_{\mathbb{R}^d \times \mathbb{R}^d} \mathcal{D} \varphi \, d\mu_{\gamma, \Delta t, 0}}{\eta}\right)\]
	we see that to make the proof work, one would have to show that the fraction 
	\[\frac{\int_{\mathbb{R}^d \times \mathbb{R}^d} \mathcal{D} \varphi d\mu_{\gamma, \Delta t, \eta} - \int_{\mathbb{R}^d \times \mathbb{R}^d} \mathcal{D} \varphi d\mu_{\gamma, \Delta t, 0}}{\eta}\]
	is bounded as \(\eta\) goes to zero. One could do this in the same way we did above by appealing to a bound of the form of \eqref{eq:stationary_discrete_weightedtv_bound} for some appropriate Lyaponov function \(V\). However, one would first need to prove an analogues to the results of \cite{Durmus_etal} and Proposition~\ref{prop:discrete_Wn_norm_bound} for the invariant measures of splitting schemes for hypoelliptic dynamics.
\end{remark}

\begin{proof}[Proof of Lemma~\ref{lm:disc_poisson_sol_approx}]
	Using our approximation of the inverse of \(\frac{\mathrm{Id} - P^{\eta, \Delta t}}{\Delta t}\), we write the difference of the two solutions to the discrete Poisson equation as 
	\begin{equation}
		\widehat{R}_{\eta, \Delta t} - \widehat{R}_{0, \Delta t} = \widehat{R}_{\eta, \Delta t} - \mathcal{Q}_{\eta, \Delta t}R - \left(\widehat{R}_{0, \Delta t} - \mathcal{Q}_{0, \Delta t}R \right) + \mathcal{Q}_{\eta, \Delta t}R - \mathcal{Q}_{0, \Delta t}R. 
	\end{equation}
	Recall that
	\(\mathcal{Q}_{\eta, \Delta t} - \mathcal{Q}_{0, \Delta t} = \eta\widetilde{\mathcal{Q}}_{\eta, \Delta t}\)
	with \(\widetilde{\mathcal{Q}}_{\eta, \Delta t}\) the well-behaved operator defined in \eqref{eq:Q_til_def}.
	For the first difference, we use \eqref{eq:approx_inverse_equality} to write
	\[
	\begin{aligned}
		\widehat{R}_{\eta, \Delta t} - \mathcal{Q}_{\eta, \Delta t}R &= \left(\frac{\mathrm{Id} - P^{\eta, \Delta t}}{\Delta t}\right)^{-1}\Pi_{\nu_{\eta, \Delta t}}\left(\frac{\mathrm{Id} - P^{\eta,\Delta t}}{\Delta t}\right)\left[\widehat{R}_{\eta, \Delta t} - \mathcal{Q}_{\eta, \Delta t}R\right] \\
		&= \left(\frac{\mathrm{Id} - P^{\eta, \Delta t}}{\Delta t}\right)^{-1}\Pi_{\nu_{\eta, \Delta t}} \Pi_0\left(\frac{\mathrm{Id} - P^{\eta,\Delta t}}{\Delta t}\right)\left[\widehat{R}_{\eta, \Delta t} - \mathcal{Q}_{\eta, \Delta t}R\right] \\
		&= \left(\frac{\mathrm{Id} - P^{\eta, \Delta t}}{\Delta t}\right)^{-1}\Pi_{\nu_{\eta, \Delta t}}\left[\Pi_0R - \Pi_0 R - \left(\Delta t^2 \widetilde{\mathcal{D}}_1 + \eta^2 \widetilde{\mathcal{D}}_2 + \eta^2 \Delta t\widetilde{\mathcal{D}}_{3, \Delta t} + \eta \Delta t^2 \widetilde{\mathcal{D}}_{4, \eta, \Delta t}\right)R\right]\\
		&= -\left(\frac{\mathrm{Id} - P^{\eta, \Delta t}}{\Delta t}\right)^{-1}\Pi_{\nu_{\eta, \Delta t}}\left[\Delta t^2 \widetilde{\mathcal{D}}_1 + \eta^2 \widetilde{\mathcal{D}}_2 + \eta^2 \Delta t\widetilde{\mathcal{D}}_{3, \Delta t} + \eta \Delta t^2 \widetilde{\mathcal{D}}_{4, \eta, \Delta t}\right]R\\
		&= -\Delta t^2\left(\frac{\mathrm{Id} - P^{\eta, \Delta t}}{\Delta t}\right)^{-1}\Pi_{\nu_{\eta, \Delta t}}\widetilde{\mathcal{D}}_1 R - \eta\left(\frac{\mathrm{Id} - P^{\eta, \Delta t}}{\Delta t}\right)^{-1}\Pi_{\nu_{\eta, \Delta t}}\left[\eta \widetilde{\mathcal{D}}_2 + \eta \Delta t\widetilde{\mathcal{D}}_{3, \Delta t} + \Delta t^2 \widetilde{\mathcal{D}}_{4, \eta, \Delta t}\right]R,
	\end{aligned} \]
	where for the second inequality we used the fact that \(\Pi_{\nu_{\eta, \Delta t}}\Pi_0 = \Pi_{\nu_{\eta, \Delta t}}\). For the second difference, we similarly have
	\[\widehat{R}_{0, \Delta t} - \mathcal{Q}_{0, \Delta}R = -\Delta t^2\left(\frac{\mathrm{Id} - P^{0, \Delta t}}{\Delta t}\right)^{-1}\Pi_{\nu_{0, \Delta t}}\widetilde{\mathcal{D}}_1 R. \]
	Putting this all together we get that 
	\[
	\begin{aligned}
	\widehat{R}_{\eta, \Delta t} - \widehat{R}_{0, \Delta t} &= \eta\widetilde{\mathcal{Q}}_{\eta, \Delta t}R - \eta \left(\frac{\mathrm{Id} - P^{\eta, \Delta t}}{\Delta t}\right)^{-1}\Pi_{\mu_{\eta, \Delta t}} \mathcal{A}_{\eta, \Delta t}R\\ 
	&\qquad - \Delta t^2 \left[\left(\frac{\mathrm{Id} - P^{\eta, \Delta t}}{\Delta t}\right)^{-1}\Pi_{\nu_{\eta, \Delta t}}\widetilde{\mathcal{D}}_1 R - \left(\frac{\mathrm{Id} - P^{0, \Delta t}}{\Delta t}\right)^{-1}\Pi_{\nu_{0, \Delta t}}\widetilde{\mathcal{D}}_1 R\right],
	\end{aligned}\]
	where \(\mathcal{A}_{\eta, \Delta t} := \eta \widetilde{\mathcal{D}}_2 + \eta \Delta t\widetilde{\mathcal{D}}_{3, \Delta t} + \Delta t^2 \widetilde{\mathcal{D}}_{4, \eta, \Delta t}\) is a well-behaved operator. Observe that the coefficient of the order \(\Delta t^2\) term is again the difference of two solutions of the discrete Poisson equation but this time with source term \(\widetilde{\mathcal{D}}_1R\). Since \(\widetilde{\mathcal{D}}_1R \in \mathscr{S}\), we can repeat the above line of reasoning. In fact since \(\widetilde{\mathcal{D}}_1\) stabilizes \(\mathscr{S}\), it holds \(\widetilde{\mathcal{D}}_1^nR \in \mathscr{S}\) for any \(n \in \mathbb{N}\) and we can repeat the above line of reasoning arbitrarily many times. Thus, for any \(n \in \mathbb{N}\), we have 
	\[
	\begin{aligned}
		\widehat{R}_{\eta, \Delta t} - \widehat{R}_{0, \Delta t} &= \eta\left(\sum_{k = 0}^{n-1}\left(-1\right)^k\Delta t^{2k}\widetilde{\mathcal{Q}}_{\eta, \Delta t}\widetilde{\mathcal{D}}_1^kR + \left(\frac{\mathrm{Id} - P^{\eta, \Delta t}}{\Delta t}\right)^{-1}\Pi_{\mu_{\eta, \Delta t}} \sum_{k = 0}^{n-1}\left(-1\right)^k\Delta t^{2k}\mathcal{A}_{\eta, \Delta t}\widetilde{\mathcal{D}}_1^kR\right)\\ 
		&\qquad + \left(-1\right)^n\Delta t^{2n} \left[\left(\frac{\mathrm{Id} - P^{\eta, \Delta t}}{\Delta t}\right)^{-1}\Pi_{\nu_{\eta, \Delta t}}\widetilde{\mathcal{D}}_1^n R - \left(\frac{\mathrm{Id} - P^{0, \Delta t}}{\Delta t}\right)^{-1}\Pi_{\nu_{0, \Delta t}}\widetilde{\mathcal{D}}_1^n R\right].
	\end{aligned}\]
	Everything on the right hand side is uniformly bounded in \(\eta \in \left[-\eta_\star, \eta_\star\right]\) and \(\Delta t \in \left(0, \Delta t^\star\right)\) for the norm~\(\left\|\cdot\right\|_{\widetilde{V}_c}\). Indeed, the first sum in the order \(\eta\) term is uniformly bounded  in \(\eta\) and \(\Delta t\) since \(\mathcal{A}_{1, \eta, \Delta t}\) and \(\widetilde{\mathcal{D}}_1\) are well behaved---this bound may however depend on \(R\). Secondly, since these operators stabilize \(\mathscr{S}\), the second sum is in \(\mathscr{S} \subset B_{\widetilde{V}_c}^\infty\) and remains in \(B_{\widetilde{V}_c}^\infty\) when we apply the inverse of the discrete generator since the discrete generator has bounded inverse in \(\Pi_{\mu_{\eta, \Delta t}}B_{\widetilde{V}_c}^\infty\), see \eqref{eq:disc_inverse_bounds}. For the same reason, the order~\(\Delta t^{2n}\) term also belongs to \(\Pi_{\mu_{\eta, \Delta t}}B_{\widetilde{V}_c}^\infty\). Thus, for any \(n \in \mathbb{N}\), there exists \(K_n > 0\) such that 
	\[\left\|\widehat{R}_{\eta, \Delta t} - \widehat{R}_{0, \Delta t}\right\|_{\widetilde{V}_c} \leq K_n \left(\eta + \Delta t^{2n}\right),\]
	uniformly in \(\eta \in \left[-\eta_\star, \eta_\star\right]\) and \(\Delta t \in \left(0, \Delta t^\star\right)\), giving the desired bound.
\end{proof}

\section{Equivalence of the Two Forms of Discrete Sticky Coupling}\label{sec:equivalence_of_disc_MR_coupling}
For this section we denote by \(\varphi_d\) the density of a $d$-dimensional standard normal distribution and by $\varphi_1$ a one dimensional standard normal distribution. In this section we verify that our definition of the meeting probability \eqref{eq:meeting_prob} is equivalent to that given in \cite[Section 2.2]{Durmus_etal}, namely
\[p(x, y, g) = \min\left\{1, \frac{\varphi_1\left(\sqrt{\frac{\beta}{2\Delta t}}\left|\mathbf{E}(x,y)\right| - \left\langle g, \mathbf{e}\left(x, y\right) \right\rangle\right)}{\varphi_1\left(\left\langle \mathbf{e}\left(x,y\right), g\right\rangle\right)}\right\}.\]
Indeed all we need to show is that the ratio of one-dimensional Gaussian densities above is equal to the ratio of \(d\)-dimensional Gaussian densities in \eqref{eq:meeting_prob}. The following computation affirms this:
\begin{align*}
	\frac{\varphi_d\left(\sqrt{\frac{\beta}{2\Delta t}}\mathbf{E}(x,y) + g\right)}{\varphi_d\left(g\right)} &= \frac{\exp\left(-\frac{1}{2}\left\langle \sqrt{\frac{\beta}{2\Delta t}}\mathbf{E}(x, y) + g, \sqrt{\frac{\beta}{2\Delta t}}\mathbf{E}(x, y) + g\right\rangle\right)}{\exp\left(-\frac{1}{2}\left\langle g, g \right\rangle\right)}\\
	&=  \exp\left(-\frac{1}{2}\left\langle \sqrt{\frac{\beta}{2\Delta t}}\mathbf{E}(x, y) + g, \sqrt{\frac{\beta}{2\Delta t}}\mathbf{E}(x, y) + g\right\rangle + \frac{1}{2}\left\langle g, g\right\rangle\right)\\
	&= \exp\left(-\frac{\beta}{4\Delta t} \left|\mathbf{E}(x,y)\right|^2 - \sqrt{\frac{\beta}{2\Delta t}}\left\langle \mathbf{E}\left(x,y\right), g\right\rangle \right)\\
	&= \exp\left(-\frac{\beta}{4\Delta t} \left|\mathbf{E}(x,y)\right|^2 - \sqrt{\frac{\beta}{2\Delta t}}\left\langle \mathbf{E}\left(x,y\right), g\right\rangle - \frac{1}{2}\left\langle \mathbf{e}\left(x,y\right), g\right\rangle^2 + \frac{1}{2}\left\langle \mathbf{e}\left(x,y\right), g\right\rangle^2\right)\\
	&= \exp\left(-\frac{\beta}{4\Delta t} \left|\mathbf{E}(x,y)\right|^2 - \sqrt{\frac{\beta}{2\Delta t}}\left|\mathbf{E}(x,y)\right|\left\langle \mathbf{e}\left(x,y\right), g\right\rangle - \frac{1}{2}\left\langle \mathbf{e}\left(x,y\right), g\right\rangle^2 + \frac{1}{2}\left\langle \mathbf{e}\left(x,y\right), g\right\rangle^2\right)\\
	&= \frac{\exp\left(-\frac{1}{2}\left[\sqrt{\frac{\beta}{2\Delta t}}\left|\mathbf{E}(x,y)\right| - \left\langle \mathbf{e}(x,y), g\right\rangle\right]^2\right)}{\exp\left(-\frac{1}{2}\left\langle \mathbf{e}(x,y), g\right\rangle^2\right)} = \frac{\varphi_1\left(\sqrt{\frac{\beta}{2\Delta t}}\left|\mathbf{E}(x,y)\right| - \left\langle \mathbf{e}(x,y), g\right\rangle\right)}{\varphi_1\left(\left\langle \mathbf{e}(x,y), g\right\rangle\right)}
\end{align*}
Thus the probability of the discretized sticky coupled trajectories meeting are the same for both coupling methods.

\bigskip

\paragraph{Acknowledgments.}
The authors thank Nawaf Bou--Rabee for helpful discussions and exchanges, as well as Pierre Jacob. S.D thanks Noé Blassel for helpful discussions on the code for the numerical illustrations. The work of S.D. and G.S. was funded by the European Research Council (ERC) under the European Union's Horizon 2020 research and innovation programme (project EMC2, grant agreement No 810367), and by Agence Nationale de la Recherche, under grants ANR-19-CE40-0010-01 (QuAMProcs) and ANR-21-CE40-0006 (SINEQ). The work of A.E. was funded by the Deutsche Forschungsgemeinschaft (DFG, German Research
Foundation) under Germany’s Excellence Strategy – GZ 2047/1, Project-ID 390685813.

\printbibliography

\end{document}